\documentclass[final]{siamart0216}
\usepackage{amsfonts}
\usepackage{todonotes}
\usepackage{amsfonts,amsmath,amssymb,stackengine}
\usepackage{mathrsfs,mathtools,stmaryrd,wasysym}
\usepackage{enumerate}
\usepackage{xspace,mydef} 
\usepackage{esint}
\usepackage{graphicx,psfrag}
\hypersetup{bookmarks=false}
\usepackage{pgf,tikz,pgfplots}
\usepackage{pstricks-add}
\usepackage{enumitem}
\usepackage{yfonts}
\usepackage{rotating}

\usetikzlibrary{arrows}
\DeclareMathAlphabet{\mathpzc}{OT1}{pzc}{m}{it}

\newcommand{\EO}[1]{\noindent{\textcolor{black}{#1}}}
\newcommand{\GC}[1]{\noindent{\textcolor{black}{#1}}}

\newcommand{\AAF}[1]{\noindent{\textcolor{black}{#1}}}

\newsiamremark{remark}{Remark}

\newcommand{\norm}[1]{{\left\vert\kern-0.25ex\left\vert\kern-0.25ex\left\vert #1 
\right\vert\kern-0.25ex\right\vert\kern-0.25ex\right\vert}}

\newcommand{\TheTitle}{
Bilinear optimal control for the Stokes--Brinkman equations: a priori and a posteriori error analyses}
\newcommand{\ShortTitle}{The bilinear optimal control of the Stokes--Brinkman equations}
\newcommand{\TheAuthors}{A. Allendes, G. Campa\~na and E.~Ot\'arola}

\headers{\ShortTitle}{\TheAuthors}

\title{{\TheTitle}\thanks{AA is partially supported by ANID through FONDECYT project \GC{1251531} and \GC{by USM through USM project 2025 PI\_LII\_25\_12}. GC is partially supported by \EO{USM through USM project 2025 PI\_LII\_25\_12.}
EO is partially supported by ANID through FONDECYT project 1220156 and \EO{by USM through USM project 2025 PI\_LII\_25\_12.}}}

\author{Alejandro Allendes\thanks{Departamento de Matem\'atica, Universidad T\'ecnica Federico Santa Mar\'ia, Valpara\'iso, Chile.
(\email{alejandro.allendes@usm.cl}, \url{http://aallendes.mat.utfsm.cl/}).}
\and 
Gilberto Campa\~na\thanks{Departamento de Ciencias, Universidad T\'ecnica Federico Santa Mar\'ia, Valpara\'iso, Chile.
(\email{gilberto.campana@usm.cl}.}
\and
Enrique Ot\'arola\thanks{Departamento de Matem\'atica, Universidad T\'ecnica Federico Santa Mar\'ia, Valpara\'iso, Chile.
(\email{enrique.otarola@usm.cl}, \url{http://eotarola.mat.utfsm.cl/}).}   
}

\ifpdf
\hypersetup{
  pdftitle={\TheTitle},
  pdfauthor={\TheAuthors}
}
\fi

\date{Draft version of \today.}

\begin{document}

\maketitle

\begin{abstract}
We analyze a bilinear optimal control problem for the Stokes--Brinkman equations: the control variable enters the state equations as a coefficient. In two- and three-dimensional Lipschitz domains, we perform a complete continuous analysis that includes the existence of solutions and first- and second-order optimality conditions. We also develop two finite element methods that differ fundamentally in whether the admissible control set is discretized or not. For each of the proposed methods, we perform a convergence analysis and derive a priori error estimates; the latter under the assumption that the domain is convex. Finally, assuming that the domain is Lipschitz, we develop an a posteriori error estimator for each discretization scheme, obtain a global reliability bound, and investigate local efficiency estimates.
\end{abstract}

\begin{keywords}
bilinear optimal control, Stokes--Brinkman equations, first- and second-order optimality conditions, finite elements, convergence, a priori error bounds, a posteriori error bounds.
\end{keywords}

\begin{AMS}
35Q35,  
49K20,  
49M25,  
65K10,  
65N15,  
65N30,  
65N50,  
76D07.  
\end{AMS}


\section{Introduction}
\label{sec:intro}
In this paper we deal with the \emph{analysis} and \emph{discretization} of a bilinear optimal control problem for the Stokes--Brinkman equations: the control variable enters the state equations as a coefficient and not as a source term. To be precise, we let $d\in\{2,3\}$ and $\Omega \subset \mathbb{R}^d$ be an open and bounded domain with Lipschitz boundary $\partial\Omega$. Given a regularization parameter $\alpha>0$ and a desired velocity field $\mathbf{y}_\Omega\in \mathbf{L}^2(\Omega)$, we introduce the cost functional
\begin{equation}\label{eq:functional}
J(\mathbf{y},u) :=
\dfrac{1}{2}\|\mathbf{y}-\mathbf{y}_\Omega\|_{\mathbf{L}^2(\Omega)}^2
+
\dfrac{\alpha}{2}\|u\|^2_{L^2(\Omega)}.
\end{equation}
\EO{Given the body force $\mathbf{f} \in \mathbf{H}^{-1}(\Omega)$ acting on the fluid,} the bilinear optimal control problem is as follows: Find $\min J(\mathbf{y},u)$ subject to the \emph{Stokes--Brinkman} equations
\begin{align}\label{eq:state}
-\Delta \mathbf{y} + u\mathbf{y}+\nabla \mathsf{p} =  \mathbf{f}
\text{ in } \Omega,
\quad
\text{div}~\mathbf{y}  =  0
\text{ in } \Omega,
\quad
\mathbf{y}  =  0
\text{ on } \partial\Omega,
\end{align}
where $\mathbf{y}$ is the velocity field \EO{of the fluid}, $\mathsf{p}$ is the pressure of the fluid, and $u$ corresponds to the control variable, which is chosen so that
\begin{equation}\label{eq:set_adm}
u\in \mathbb{U}_{ad},
\qquad
\GC{\mathbb{U}_{ad}:=
\{
v \in L^2(\Omega): \mathsf{a}\leq v\leq \mathsf{b} \text{ a.e. in } \Omega
\}}.
\end{equation}
The control bounds $\mathsf{a}$ and $\mathsf{b}$ both belong to $\mathbb{R}$ and are chosen so that \GC{$0\leq\mathsf{a} < \mathsf{b}<\infty$}.

When the state equations are replaced by the problem $-\Delta u + q u = f$ in $\Omega$ and $u = 0$ on $\partial \Omega$, there are several works that provide error estimates for finite element discretizations of \eqref{eq:functional}--\eqref{eq:set_adm}. To the best of our knowledge, \cite{MR2536007} is the first paper that provides an analysis. In this paper, the authors investigate discretization methods and derive error bounds \cite[Corollaries 5.6 and 5.10]{MR2536007}. These results were later extended to mixed and stabilized methods in \cite{MR3103238} and \cite{MR3693332}, respectively. Regarding the analysis of \emph{a posteriori} error estimates, we refer the reader to \cite{MR2680928} and \cite{MR4450052}. We would also like to refer to the paper \cite{MR4081184}, in which the author studies an optimal control problem with a bilinear term in the boundary condition of the state equation. Finally, we mention the recent advances in the semilinear scenario presented in \EO{\cite{MR4679130,MR4858133,OSQ:25}.}

As far as we know, this is the \emph{first} paper that investigates optimality conditions and finite element schemes for the control problem \EO{\eqref{eq:functional}--\eqref{eq:set_adm}}. The analysis is not trivial and involves a number of difficulties. To overcome them, we had to provide several results. In the following, we list what we consider to be the main contributions of our work:

$\bullet$ \emph{Existence results and optimality conditions:} We prove that the optimal control problem admits at least one optimal solution and derive first- and necessary and sufficient second-order optimality conditions with a minimal gap.

$\bullet$ \emph{Finite element schemes:} We propose two different schemes: a semidiscrete method in which the control set is not discretized, and a fully discrete method in which such a set is discretized with piecewise constant functions. For each scheme, we prove convergence results and a priori error bounds: $\mathcal{O}(h)$ for the fully discrete scheme and $\mathcal{O}(h^2)$ for the semidiscrete one. We also study a posteriori error bounds.


\section{Notation and preliminary remarks}
\label{sec:notation}
Let us establish the notation and introduce the framework within which we will work on our manuscript.


\subsection{Notation}
Throughout the text, we use the classical notation for Lebesgue and Sobolev spaces and denote the vector-valued counterparts of these spaces with bold letters. In particular, we set
$
\mathbf{V}(\Omega):=\{\mathbf{v}\in\mathbf{H}_0^1(\Omega):\text{div }\mathbf{v}=0\}.
$

Let $\mathcal{X}$ and $\mathcal{Y}$ be Banach function spaces. We write $\mathcal{X} \hookrightarrow \mathcal{Y}$ to indicate that $\mathcal{X}$ is continuously embedded in $\mathcal{Y}$. We denote by $\mathcal{W}'$ and $\|\cdot\|_{\mathcal{W}}$ the dual space and the norm of $\mathcal{W}$, respectively. Given $\mathtt{p}\in(1, \infty)$, we denote by $\mathtt{p}\prime$ its H\"older conjugate, which is such that $1/\mathtt{p} + 1/\mathtt{p}\prime = 1$. The relation $a\lesssim b$ states that $a\leq Cb$ with a positive constant $C$ that does not depend on $a$, $b$, or the discretization parameters. The value of $C$ might change at each occurrence.

\section{The Stokes--Brinkman problem}
\label{sec:Stokes-Brinkman}
We now present some of the most important results related to the well-posedness and regularity of solutions of the Stokes--Brinkman equations \eqref{eq:state}. \EO{For this purpose, we follow \cite{MR4858133} and introduce the set}
\begin{equation}\label{eq:set_A}
\GC{
\mathcal{A}_0:=\{v\in L^2(\Omega): v \geq 0~\mathrm{a.e.}~ \mathrm{in}~\Omega\}.
}
\end{equation}
Let $\boldsymbol{\mathfrak{f}}$ be in $\mathbf{H}^{-1}(\Omega)$, and let $\mathfrak{u}$ be an arbitrary function in  \GC{$\mathcal{A}_{0}$}. Under these assumptions, we introduce the following weak formulation: Find $(\boldsymbol{\mathsf{y}},\mathfrak{p})\in \mathbf{H}_0^1(\Omega)\times L_0^2(\Omega)$ such that
\begin{equation}
\label{eq:brinkman_problem}
(\nabla \boldsymbol{\mathsf{y}},\nabla \mathbf{v})_{\mathbf{L}^2(\Omega)}
+
\EO{( \mathfrak{u}, \boldsymbol{\mathsf{y}} \cdot \mathbf{v})_{L^2(\Omega)}}
-
(\mathfrak{p},\text{div }\mathbf{v})_{L^2(\Omega)}
=
\langle \boldsymbol{\mathfrak{f}},\mathbf{v} \rangle,
\quad
(\mathsf{q},\text{div }\boldsymbol{\mathsf{y}})_{L^2(\Omega)}
=
0
\end{equation}
for all $(\mathbf{v},\mathsf{q}) \in \mathbf{H}_0^1(\Omega) \times L_0^2(\Omega)$. As a first ingredient to analyze \eqref{eq:brinkman_problem}, we recall that the divergence operator is surjective from $\mathbf{H}_0^1(\Omega)$ to $L_0^2(\Omega)$. This implies the existence of a constant $\beta>0$ such that \cite[Chapter I, Section 5.1]{MR851383}, \cite[Corollary B.71]{Guermond-Ern}
\begin{equation}\label{eq:infsup}
\sup_{\mathbf{v}\in \mathbf{H}_0^1(\Omega)}\dfrac{(\mathsf{q},\text{div }\mathbf{v})_{L^2(\Omega)}}{\|\nabla \mathbf{v}\|_{\mathbf{L}^2(\Omega)}}\geq \beta \|\mathsf{q}\|_{L^2(\Omega)}\quad\forall \mathsf{q}\in L_0^2(\Omega).
\end{equation}
\EO{As a second ingredient, we note that if $\mathfrak{u}\in\mathcal{A}_0$, then $(\mathbf{w},\mathbf{v}) \mapsto (\nabla \mathbf{w},\nabla \mathbf{v})_{\mathbf{L}^2(\Omega)}
+
( \mathfrak{u},\mathbf{w} \cdot \mathbf{v})_{L^2(\Omega)}$ defines a bilinear, continuous, and coercive form in $\mathbf{H}_0^1(\Omega) \times \mathbf{H}_0^1(\Omega)$. Thus, we can apply the standard inf-sup theory for saddle point problems to obtain the well-posedness of problem \eqref{eq:brinkman_problem}. Moreover, we have the estimates}
\begin{equation}
\label{eq:well_posedness_brinkman}
\EO{\|\nabla \boldsymbol{\mathsf{y}}\|_{\mathbf{L}^2(\Omega)} \leq \| \boldsymbol{\mathfrak{f}}\|_{\mathbf{H}^{-1}(\Omega)}},
\qquad
\EO{\|\mathfrak{p}\|_{L^2(\Omega)} \leq \beta^{-1} \left[ 2 + C_{4 \hookrightarrow 2}^2 \| \mathfrak{u} \|_{L^2(\Omega)} \right] \| \boldsymbol{\mathfrak{f}}\|_{\mathbf{H}^{-1}(\Omega)},}
\end{equation}
\EO{where $C_{4 \hookrightarrow 2}$ denotes the best constant in the Sobolev embedding $\mathbf{H}_0^1(\Omega) \hookrightarrow \mathbf{L}^4(\Omega)$.}

\subsection{\EO{Regularity results}}
The next result shows that it is possible to obtain an $\mathbf{L}^{\infty}(\Omega)$-regularity result for the velocity component if we assume better regularity properties for the datum $\boldsymbol{\mathfrak{f}}$.

\begin{theorem}[Sobolev and H\"older regularity]\label{theorem_regularity_Linfty}
\EO{Let $\boldsymbol{\mathfrak{f}}$ and $\mathfrak{u}$ be in $L^2(\Omega)$. If the domain $\Omega$ is Lipschitz and $(\boldsymbol{\mathsf{y}},\mathfrak{p})$ is a solution of problem \eqref{eq:brinkman_problem}, then there exists $\kappa$ such that $(\boldsymbol{\mathsf{y}},\mathfrak{p})$ belongs to $\mathbf{H}_0^{1}(\Omega) \cap \mathbf{W}^{1,\kappa}(\Omega) \times L_0^2(\Omega) \cap L^{\kappa}(\Omega)$, where $\kappa > 4$ if $d=2$ and $\kappa >3$ if $d=3$. Moreover,
\begin{equation}\label{eq:estimate_extra}
\| \nabla \boldsymbol{\mathsf{y}}\|_{\mathbf{L}^{\kappa}(\Omega)}
+
\|\mathfrak{p}\|_{L^{\kappa}(\Omega)}
\lesssim
\| \boldsymbol{\mathfrak{f}}\|_{\mathbf{W}^{-1,\kappa}(\Omega)}
\lesssim
\| \boldsymbol{\mathfrak{f}}\|_{\mathbf{L}^{2}(\Omega)},
\end{equation}
where the hidden constants depend on $\| u \|_{L^{2}(\Omega)}$. In particular, $\boldsymbol{\mathsf{y}} \in \mathbf{C}^{0,\varsigma}(\bar \Omega)$, where $\varsigma$ is such that $0 < \varsigma \leq 1-d/\kappa$.}
\end{theorem}
\begin{proof}
\EO{We present a proof in three dimensions; the two-dimensional case is simpler. We begin the analysis by writing the momentum equation of \eqref{eq:brinkman_problem} as $-\Delta \boldsymbol{\mathsf{y}} +\nabla \mathfrak{p} = \boldsymbol{\mathfrak{f}} - \mathfrak{u} \boldsymbol{\mathsf{y}} =: \boldsymbol{\mathfrak{g}}$ in $\Omega$. To derive a first regularity result, we use the Sobolev embeddings
\[
 \mathbf{H}_0^1(\Omega) \hookrightarrow \mathbf{L}^6(\Omega),
 \qquad
 \mathbf{W}_0^{1,3/2}(\Omega) \hookrightarrow \mathbf{L}^3(\Omega),
 \qquad
 (d=3)
\]
(see \cite[Theorem 4.12, Part I, Case C]{MR2424078}) to obtain that $\boldsymbol{\mathfrak{g}}\in \mathbf{W}^{-1,3}(\Omega)$. We also have
\begin{equation}
 \| \boldsymbol{\mathfrak{g}} \|_{\mathbf{W}^{-1,3}(\Omega)} \leq \| \boldsymbol{\mathfrak{f}} \|_{\mathbf{W}^{-1,3}(\Omega)} + C \| \mathfrak{u} \|_{L^2(\Omega)} \| \nabla \boldsymbol{\mathsf{y}}\|_{\mathbf{L}^2(\Omega)}
 \lesssim (1 + \| \mathfrak{u} \|_{L^2(\Omega)}) \| \boldsymbol{\mathfrak{f}} \|_{\mathbf{W}^{-1,3}(\Omega)}.
\end{equation}
Since $\Omega$ is Lipschitz, an application of \cite[(1.52)]{MR2987056} shows that $\boldsymbol{\mathsf{y}} \in \mathbf{W}^{1,3}(\Omega)$. Based on a Sobolev embedding \cite[Theorem 4.12, Part I, Case B]{MR2424078}, this regularity result for $\boldsymbol{\mathsf{y}}$ allows us to conclude that $\boldsymbol{\mathsf{y}} \in \mathbf{L}^{\sigma}(\Omega)$ for every $\sigma < \infty$. This, in turns, implies the existence of some $\kappa > 3$, so that $\boldsymbol{\mathfrak{g}}\in \mathbf{W}^{-1,\kappa}(\Omega)$ together with the bound
\begin{multline*}
 \| \boldsymbol{\mathfrak{g}} \|_{\mathbf{W}^{-1,\kappa}(\Omega)}
 \leq \| \boldsymbol{\mathfrak{f}} \|_{\mathbf{W}^{-1,\kappa}(\Omega)} + C \| \mathfrak{u} \|_{L^2(\Omega)} \| \nabla \boldsymbol{\mathsf{y}}\|_{\mathbf{L}^3(\Omega)}
 \lesssim
 \| \boldsymbol{\mathfrak{f}} \|_{\mathbf{W}^{-1,\kappa}(\Omega)}
 \\
 + \| \mathfrak{u} \|_{L^2(\Omega)} (1 + \| \mathfrak{u} \|_{L^2(\Omega)}) \| \boldsymbol{\mathfrak{f}} \|_{\mathbf{W}^{-1,3}(\Omega)}
 \lesssim \left[ 1 + \| \mathfrak{u} \|_{L^2(\Omega)} (1 + \| \mathfrak{u} \|_{L^2(\Omega)}) \right] \| \boldsymbol{\mathfrak{f}} \|_{\mathbf{W}^{-1,\kappa}(\Omega)}.
\end{multline*}
To obtain the last bound, we have used that $\mathbf{W}^{-1,\kappa}(\Omega) \hookrightarrow \mathbf{W}^{-1,3}(\Omega)$. We can therefore apply \cite[(1.52)]{MR2987056} again to conclude that $\boldsymbol{\mathsf{y}} \in \mathbf{W}^{1,\kappa}(\Omega)$ for some $\kappa > 3$ and that $\| \nabla \boldsymbol{\mathsf{y}}\|_{\mathbf{L}^{\kappa}(\Omega)} \lesssim \| \boldsymbol{\mathfrak{g}} \|_{\mathbf{W}^{-1,\kappa}(\Omega)}$; $\kappa$ may be further restricted. The first stament thus follows from \cite[Corollary 1.2.1]{MR2987056}. Finally, H\"older's regularity result follows from \cite[Theorem 4.12, Part II]{MR2424078}. This concludes the proof.}
\end{proof}

We also present the following regularity result for \EO{a pair $(\boldsymbol{\mathsf{y}},\mathfrak{p})$ that solves \eqref{eq:brinkman_problem} in the case that the domain $\Omega$ is convex.}

\begin{theorem}[\GC{$\mathbf{H}^2(\Omega)\times H^1(\Omega)$-regularity}]
\label{eq:theorem_regularity}
\EO{Let $\boldsymbol{\mathfrak{f}}$ and $\mathfrak{u}$ be in $L^2(\Omega)$. If the domain $\Omega$ is a convex polytope and $(\boldsymbol{\mathsf{y}},\mathfrak{p})$ is a solution of problem \eqref{eq:brinkman_problem}, then $(\boldsymbol{\mathsf{y}},\mathfrak{p})$ belongs to $\mathbf{H}^2(\Omega)\cap\mathbf{H}_0^1(\Omega) \times H^1(\Omega)\cap L_0^2(\Omega)$. In addition, we have the bound}
\begin{align}\label{eq:regularity}
\|\boldsymbol{\mathsf{y}}\|_{\mathbf{H}^2(\Omega)}+\|\mathfrak{p}\|_{H^1(\Omega)}\lesssim \|\boldsymbol{\mathfrak{f}}\|_{\mathbf{L}^2(\Omega)},
\end{align}
\EO{with a hidden constant that is independent of $(\boldsymbol{\mathsf{y}},\mathfrak{p})$ but depends on $\| \mathfrak{u} \|_{L^2(\Omega)}$.}
\end{theorem}
\begin{proof}
\EO{We begin the proof by noting that the results of Theorem \ref{theorem_regularity_Linfty} guarantee that $\boldsymbol{\mathsf{y}} \in \mathbf{C}(\bar \Omega)$ together with $\|\boldsymbol{\mathsf{y}}\|_{\mathbf{C}(\bar \Omega)} \leq C \| \boldsymbol{\mathfrak{f}} \|_{\mathbf{L}^2(\Omega)}$, where $C$ depends on $\| \mathfrak{u} \|_{L^{2}(\Omega)}$. It follows directly from this that $\boldsymbol{\mathfrak{g}} = \boldsymbol{\mathfrak{f}} - \mathfrak{u} \boldsymbol{\mathsf{y}} \in \mathbf{L}^2(\Omega)$ and that
\[
 \| \boldsymbol{\mathfrak{g}} \|_{\mathbf{L^2(\Omega)}} \leq \| \boldsymbol{\mathfrak{f}} \|_{\mathbf{L^2(\Omega)}} + \| \mathfrak{u} \|_{L^2(\Omega)} \|\boldsymbol{\mathsf{y}}\|_{\mathbf{C}(\bar \Omega)} \leq (1 + C  \| \mathfrak{u} \|_{L^2(\Omega)} )\| \boldsymbol{\mathfrak{f}} \|_{\mathbf{L^2(\Omega)}}.
\]
Since $\Omega$ is convex, the desired result} follows directly from an application of \cite{MR0977489,MR1301452} for $d=3$ and \cite[Theorem 2]{MR0404849} for $d=2$.
\end{proof}


\subsection{\EO{Differentiability properties of the map $\mathfrak{u} \to (\boldsymbol{\mathsf{y}},\mathfrak{p})$}}
\EO{We study an instrumental result that establishes differentiability properties of the map $\mathfrak{u} \mapsto (\boldsymbol{\mathsf{y}},\mathfrak{p})$, where $(\boldsymbol{\mathsf{y}},\mathfrak{p})$ corresponds to the solution of \eqref{eq:brinkman_problem}. For this purpose, we define the space
\begin{equation*}
D(\Omega):=\{(\mathbf{v},\mathsf{q}) \in \mathbf{V}(\Omega)\times L_0^2(\Omega): -\Delta \mathbf{v}+\nabla \mathsf{q}\in\mathbf{L}^2(\Omega)\}
\end{equation*}
and the norm $\| \cdot \|_{D(\Omega)}: D(\Omega) \rightarrow \mathbb{R}$ as $\|(\mathbf{v},\mathsf{q})\|_{D(\Omega)}:=\|\nabla\mathbf{v}\|_{\mathbf{L}^2(\Omega)}+\|\mathsf{q}\|_{L^2(\Omega)}+\|-\Delta \mathbf{v}+\nabla\mathsf{q}\|_{\mathbf{L}^2(\Omega)}$. The space $D(\Omega)$ equipped with $\| \cdot \|_{D(\Omega)}$ is a Banach space.}

\begin{lemma}[embedding]
\EO{We have that $D(\Omega) \hookrightarrow \mathbf{V}(\Omega) \cap \mathbf{W}^{1,\kappa}(\Omega) \times L_0^2(\Omega) \cap L^{\kappa}(\Omega)$, for some $\kappa > 4$ when $d=2$ and some $\kappa >3$ when $d=3$.}
\end{lemma}
\begin{proof}
\EO{Let $(\mathbf{v},\mathsf{q}) \in D(\Omega)$. We note that the pair $(\mathbf{v},\mathsf{q})$ can be regarded as the unique weak solution in $\mathbf{H}_0^{1}(\Omega) \times L_0^2(\Omega)$ of the following Stokes problem:
\[
 (\nabla \mathbf{v}, \nabla \mathbf{w})_{\mathbf{L}^2(\Omega)} -  (\mathsf{q}, \mathrm{div} \mathbf{w})_{L^2(\Omega)} = (\mathbf{g}, \mathbf{w})_{\mathbf{L}^2(\Omega)},
 \quad
 (\mathsf{r},\mathrm{div} \mathbf{v})_{L^2(\Omega)} = 0
\]
for all $\mathbf{w} \in \mathbf{H}_0^1(\Omega)$ and $\mathsf{r} \in L_0^2(\Omega)$, respectively. Here, $\mathbf{g} \in \mathbf{L}^2(\Omega)$. Since $\Omega$ is Lipschitz, an application of \cite[Corollary 1.2.1]{MR2987056} shows the existence of $\kappa > 4$ when $d = 2$ and $\kappa > 3$ when $d = 3$ such that $(\mathbf{v},\mathsf{q})\in \mathbf{V}(\Omega) \cap \mathbf{W}^{1,\kappa}(\Omega) \times L_0^2(\Omega) \cap L^{\kappa}(\Omega)$  together with
\[
 \| \nabla\mathbf{v} \|_{\mathbf{L}^{\kappa}(\Omega)}
 +
 \|\mathsf{q}\|_{L^{\kappa}(\Omega)}
 \leq C
 \| \mathbf{g} \|_{\mathbf{L}^{2}(\Omega)}
 =
 C \|-\Delta \mathbf{v}+\nabla\mathsf{q}\|_{\mathbf{L}^2(\Omega)}
 \leq
 C \|(\mathbf{v},\mathsf{q})\|_{D(\Omega)}.
\]
We have thus proved the desired embedding. This concludes the proof.}
\end{proof}

\begin{theorem}[differentiability properties of $\mathfrak{u} \to (\boldsymbol{\mathsf{y}},\mathfrak{p})$]
\label{theorem:differentiablity_stokes_brinkman}
\EO{Let $\mathbf{f} \in \mathbf{L}^2(\Omega)$. Then, there exists an open set $\mathcal{A}$ in $L^2(\Omega)$ such that $\mathcal{A}_0\subset\mathcal{A}$ and for every $\mathfrak{u} \in\mathcal{A}$ the problem \eqref{eq:brinkman_problem} has a unique solution $(\boldsymbol{\mathsf{y}},\mathfrak{p}) \in \mathbf{H}_0^{1}(\Omega) \cap \mathbf{W}^{1,\kappa}(\Omega) \times L_0^2(\Omega) \cap L^{\kappa}(\Omega)$, where $\kappa > 4$ if $d=2$ and $\kappa >3$ if $d=3$. Moreover, there exists a map $\mathcal{S}:\mathcal{A} \to D(\Omega)$ of class $C^2$ so that for every $\mathfrak{u}\in\mathcal{A}$ we have the following properties:}
\begin{itemize}
\item[(i)] \EO{$\mathcal{S}\GC{(\mathfrak{u})} = (\boldsymbol{\mathsf{y}},\mathfrak{p}) \in \EO{\mathbf{H}_0^{1}(\Omega) \cap \mathbf{W}^{1,\kappa}(\Omega) \times L_0^2(\Omega) \cap L^{\kappa}(\Omega)}$, where $(\boldsymbol{\mathsf{y}},\mathfrak{p})$ is the unique solution to \eqref{eq:brinkman_problem},}
\item[(ii)] \EO{\GC{for every $u\in\mathcal{A}$, and for every $v\in L^{2}(\Omega)$, the pair} $(\boldsymbol{\varphi},\zeta) = \mathcal{S}'(u)v \in \EO{\mathbf{H}_0^{1}(\Omega) \cap \mathbf{W}^{1,\kappa}(\Omega) \times L_0^2(\Omega) \cap L^{\kappa}(\Omega)}$ is the unique solution to the problem}
\begin{equation}\label{eq:problem_S_first}
\begin{array}{rl}
(\nabla \boldsymbol{\varphi},\nabla \mathbf{v})_{\mathbf{L}^2(\Omega)}+( u\boldsymbol{\varphi}, \mathbf{v})_{\mathbf{L}^2(\Omega)}-(\zeta,\mathrm{div~} \mathbf{v})_{L^2(\Omega)}
=
&
\hspace{-2mm}
-(v\mathbf{y},\mathbf{v})_{\mathbf{L}^2(\Omega)},
\\
(\mathsf{s},\mathrm{div~} \boldsymbol{\varphi})_{L^2(\Omega)}
=
&
\hspace{-2mm}
0
\end{array}
\end{equation}
for all $(\mathbf{v},\mathsf{s}) \in \mathbf{H}_0^1(\Omega) \times L_0^2(\Omega)$.
\item[(iii)] \EO{for every $v_1,v_2 \in L^{2}(\Omega)$, \EO{the pair $(\boldsymbol{\psi},\xi)=\mathcal{S}''(u)(v_1,v_2)$ belongs to $\mathbf{H}_0^{1}(\Omega) \cap \mathbf{W}^{1,\kappa}(\Omega) \times L_0^2(\Omega) \cap L^{\kappa}(\Omega)$ and is the unique solution to the problem}}
\begin{multline}\label{eq:problem_S_second}
(\nabla \boldsymbol{\psi},\nabla \mathbf{v})_{\mathbf{L}^2(\Omega)}+( u\boldsymbol{\psi},\mathbf{v})_{\mathbf{L}^2(\Omega)}-(\xi,\mathrm{div~} \mathbf{v})_{L^2(\Omega)}
\\
=
-(v_2\boldsymbol{\varphi}_{v_1},\mathbf{v})_{\mathbf{L}^2(\Omega)}-(v_1\boldsymbol{\varphi}_{v_2},\mathbf{v})_{\mathbf{L}^2(\Omega)},
\qquad
(\mathsf{s},\mathrm{div~} \boldsymbol{\psi})_{L^2(\Omega)}  = 0
\end{multline}
for all $(\mathbf{v},\mathsf{s}) \in \mathbf{H}_0^1(\Omega) \times L_0^2(\Omega)$. Here, $(\boldsymbol{\varphi}_{v_i},\zeta_{v_i})=\mathcal{S}'(u)v_i$ and $i\in\{1,2\}$.
\end{itemize}
\end{theorem}

\begin{proof}
\GC{Following the arguments presented in the proof of \cite[Theorem 2.5]{MR4858133}, we proceed on the basis of the implicit function theorem \cite[Theorem 7.13-1]{MR3136903}. We will apply such a theorem to the map $\mathcal{F}$ defined as follows:
\begin{align*}
\mathcal{F}: D(\Omega)\times L^2(\Omega)\to \mathbf{L}^{2}(\Omega),\quad \mathcal{F}(\EO{(\mathbf{y},\mathsf{p})},u):=-\Delta \mathbf{y}+u\mathbf{y}+\nabla \mathsf{p}-\mathbf{f}.
\end{align*}
\EO{It follows directly from the definition that $\mathcal{F}$ is well-defined and that $\mathcal{F}$ is of class $C^2$. Now let $\bar{u}\in\mathcal{A}_0$, and let $(\bar{\mathbf{y}},\bar{\mathsf{p}}) \in D(\Omega)$ be such that it solves problem \eqref{eq:brinkman_problem}, where $\mathfrak{u}$ is replaced by $\bar u$. We note that $\mathcal{F}((\bar{\mathbf{y}},\bar{\mathsf{p}}),\bar{u})=\mathbf{0}$ and that the derivative $T := \partial \mathcal{F}((\bar{\mathbf{y}},\bar{\mathsf{p}}),\bar{u}) / (\mathbf{y},\mathsf{p})$ is such that}
\begin{align*}
\dfrac{\partial \mathcal{F}}{\partial(\mathbf{y},\mathsf{p})}((\bar{\mathbf{y}},\bar{\mathsf{p}}),\bar{u}):D(\Omega)\to \mathbf{L}^2(\Omega),\quad \dfrac{\partial \mathcal{F}}{\partial(\mathbf{y},\mathsf{p})}(\bar{\mathbf{y}},\bar{\mathsf{p}},\bar{u})(\boldsymbol{\varphi},\zeta):=-\Delta \boldsymbol{\varphi} + \bar{u}\boldsymbol{\varphi}+ \nabla \zeta.
\end{align*}
\EO{It is clear that $T$ is well-defined and linear. It can also be shown that $T$ is continuous. Since the Stokes--Brinkman problem underlying the definition of $T$ is well-posed (see Section \ref{sec:Stokes-Brinkman}), we can also deduce that $T$ is a bijection. An application of a corollary of the open mapping theorem \cite[Theorem 5.6-2]{MR3136903} shows that $T$ is an isomorphism. Thus we have all the prerequisites to apply the implicit function theorem and deduce the existence of} open neighborhoods $B_{\epsilon_{\bar{u}}}(\bar{u})\subset L^2(\Omega)$ and $B_{\epsilon_{\bar{\mathbf{y}}}}((\bar{\mathbf{y}},\bar{\mathsf{p}}))\subset D(\Omega)$, where $\epsilon_{\bar{u}}>0$ and $\epsilon_{\bar{\mathbf{y}}}>0$, such that for every $u \in B_{\epsilon_{\bar{u}}}(\bar{u})$ the equation
\begin{align}
\mathcal{F}((\mathbf{y},\mathsf{p}),u)=\mathbf{0}
\end{align}
has a unique solution $(\mathbf{y},\mathsf{p})$ in $B_{\epsilon_{\bar{\mathbf{y}}}}((\bar{\mathbf{y}},\bar{\mathsf{p}}))$. Furthermore, $\mathcal{S}: B_{\epsilon_{\bar{u}}}(\bar{u}) \rightarrow B_{\epsilon_{\bar{\mathbf{y}}}}((\bar{\mathbf{y}},\bar{\mathsf{p}}))$ defined as $u \mapsto (\mathbf{y},\mathsf{p})$ is of class $C^2$. \EO{We would now like to extend the definition of $\mathcal{S}$. To  do this, we must prove that for every $u \in B_{\epsilon_{\bar{u}}}(\bar{u})$, problem \eqref{eq:brinkman_problem}, where $\mathfrak{u}$ is replaced by $u$, admits a unique solution. To prove this, we let $(\mathbf{y}_1,\mathsf{p}_1)$ and $(\mathbf{y}_2,\mathsf{p}_2)$ be two solutions and note that $(\mathbf{y}_1 - \mathbf{y}_2,\mathsf{p}_1-\mathsf{p}_2)$ solves
\begin{multline}
 (\nabla (\mathbf{y}_1 - \mathbf{y}_2), \nabla \mathbf{v})_{\mathbf{L}^2(\Omega)}
 +
 (\bar{u}(\mathbf{y}_1 - \mathbf{y}_2),\mathbf{v})_{\mathbf{L}^2(\Omega)}
 \\
 +
 ((u-\bar{u})(\mathbf{y}_1 - \mathbf{y}_2),\mathbf{v})_{\mathbf{L}^2(\Omega)}
 -
 (\mathsf{p}_1-\mathsf{p}_2, \mathrm{div} \mathbf{v})_{L^2(\Omega)}
 = 0
\end{multline}
and $(\mathsf{q}, \mathrm{div} (\mathbf{y}_1 - \mathbf{y}_2))_{L^2(\Omega)} = 0$ for all $\mathbf{v} \in \mathbf{H}_0^1(\Omega)$ and for all $\mathsf{q} \in L_0^2(\Omega)$. If we set $\mathbf{v} = \mathbf{y}_1 - \mathbf{y}_2 \in \mathbf{H}_0^1(\Omega)$, we can conclude that
\[
 \| \nabla (\mathbf{y}_1 - \mathbf{y}_2) \|_{\mathbf{L}^2(\Omega)}^2 \leq C_{4 \hookrightarrow 2}^2 \| u - \bar{u} \|_{L^2(\Omega)}
 \| \nabla (\mathbf{y}_1 - \mathbf{y}_2) \|_{\mathbf{L}^2(\Omega)}^2.
\]
If $\epsilon_{\bar u} < C_{4 \hookrightarrow 2}^{-2}$, we can tus conclude that $\mathbf{y}_1 = \mathbf{y}_2$. The fact that $\mathsf{p}_1 = \mathsf{p}_2$ follows from the inf-sup condition \eqref{eq:infsup}. Uniqueness can thus be guaranteed if $\epsilon_{\bar u} < C_{4 \hookrightarrow 2}^{-2}$. With this result at hand, we can therefore extend the definition of operator $\mathcal{S}$ as follows:
}
\begin{equation}
\label{eq:S}
\EO{\mathcal{S}: \mathcal{A} \rightarrow D(\Omega)},
\quad
\mathcal{A} \subset L^2(\Omega),
\quad
\EO{\mathcal{A}:= \bigcup \{ B_{\epsilon_{\bar{u}}}(\bar{u}): \bar{u} \in \mathcal{A}_0 \}},
\quad
\epsilon_{\bar u} < C_{4 \hookrightarrow 2}^{-2}.
\end{equation}
For $u \in \mathcal{A}$, the equation $\mathcal{F}((\mathbf{y},\mathsf{p}),u)=\mathbf{0}$ has a unique solution $(\mathbf{y},\mathsf{p}) = \mathcal{S}(u)$ in $D(\Omega)$, i.e., $\mathcal{F}(\mathcal{S}(u),u)=\mathbf{0}$, and the map $\mathcal{S}$ defined in \eqref{eq:S} is of class $C^2$.
}

The fact that $(\boldsymbol{\psi},\xi)\in \mathbf{H}_0^1(\Omega)\times L_0^2(\Omega)$ solves \eqref{eq:problem_S_second} follows from the arguments developed in \cite[Theorem 4.24]{MR2583281}. This concludes the proof.

\end{proof}

\section{The optimal control problem}
In this section, we introduce a weak formulation for the optimal control problem \EO{\eqref{eq:functional}--\eqref{eq:set_adm}} and analyze an existence result and first- and second-order optimality conditions.

\subsection{Weak formulation}\label{section:weak_formulation}
\EO{Given $\mathbf{f} \in \mathbf{L}^{2}(\Omega)$, we propose the following weak formulation for the optimal control problem \eqref{eq:functional}--\eqref{eq:set_adm}}: Find
\begin{equation}
\label{eq:min_functional}
\min \{J(\mathbf{y},u):(\mathbf{y},u)\in\mathbf{H}_0^1(\Omega)\times \mathbb{U}_{ad}\}
\end{equation}
subject to the \emph{state equations} 
\begin{equation}\label{eq:brinkamn_problem_state}
(\nabla \mathbf{y},\nabla \mathbf{v})_{\mathbf{L}^2(\Omega)}
+
\EO{( u, \mathbf{y} \cdot \mathbf{v})_{L^2(\Omega)}}
-
(\mathsf{p},\text{div }\mathbf{v})_{L^2(\Omega)}
=
\langle \mathbf{f},\mathbf{v} \rangle,
\quad
(\mathsf{q},\text{div }\mathbf{y})_{L^2(\Omega)}
=
0
\end{equation}
for all $\mathbf{v}\in \mathbf{H}_0^1(\Omega)$ and $\mathsf{q}\in L_0^2(\Omega)$, respectively. We recall that $J$ is the cost functional defined in \eqref{eq:functional}. We note that the state equations \eqref{eq:brinkamn_problem_state} are equivalent to the following \emph{constrained formulation} due to de Rham's theorem \cite[Section 4.1.3]{Guermond-Ern}:
\begin{equation}\label{eq:brinkamn_problem_reduced}
\mathbf{y}\in \mathbf{V}(\Omega):
\quad
(\nabla \mathbf{y},\nabla \mathbf{v})_{\mathbf{L}^2(\Omega)}
+
\EO{( u, \mathbf{y} \cdot \mathbf{v})_{L^2(\Omega)}}
=  
\langle \mathbf{f},\mathbf{v} \rangle
\quad 
\forall\mathbf{v}\in \mathbf{V}(\Omega).
\end{equation}

\begin{remark}[the admissible control set]
\EO{We note that $\mathbb{U}_{ad} \subset \mathcal{A}_0 \subset \mathcal{A}$. We recall that the sets $\mathbb{U}_{ad}$, $\mathcal{A}_0$, and $\mathcal{A}$ are defined in \eqref{eq:set_adm}, \eqref{eq:set_A}, and \eqref{eq:S}, respectively.}
\end{remark}

\begin{remark}[the state equations]
The \EO{results of Theorem \ref{theorem:differentiablity_stokes_brinkman} guarantee that for $u \in \mathcal{A}$, the state equation \eqref{eq:brinkamn_problem_state} are well-posed. Moreover, a direct application of Theorem \ref{theorem_regularity_Linfty} shows that $\mathbf{y} \in \mathbf{C}(\bar \Omega)$, so that for $\mathbf{v} \in \mathbf{H}_0^1(\Omega)$ the term $( u, \mathbf{y} \cdot \mathbf{v})_{L^2(\Omega)}$ in \eqref{eq:brinkamn_problem_state} can be rewritten as $( u \mathbf{y}, \mathbf{v})_{\mathbf{L}^2(\Omega)}$.}
\end{remark}

\subsection{Existence of solutions}\label{section:existence}
The existence of an optimal solution is as follows.
 
\begin{theorem}[existence of an optimal solution]\label{theorem_existence_optimalcontrol} The optimal control problem \eqref{eq:min_functional}--\eqref{eq:brinkamn_problem_state} admits at least one global solution $(\bar{\mathbf{y}},\bar{\mathsf{p}},\bar{u})\in \mathbf{H}_0^1(\Omega)\times L_0^2(\Omega)\times \mathbb{U}_{ad}$.
\end{theorem}
\begin{proof}
The proof follows from the direct method of calculus of variations. Let $\{(\mathbf{y}_{k},\mathsf{p}_k,u_{k})\}_{k\in\mathbb{N}}$ be a minimizing sequence, i.e., for $k\in \mathbb{N}$, $(\mathbf{y}_{k},\mathsf{p}_k) \in \mathbf{H}_0^1(\Omega)\times L_0^2(\Omega)$ solves  \eqref{eq:brinkamn_problem_state} where $u$ is replaced by $u_{k}$ and $\{ (\mathbf{y}_{k}, \mathsf{p}_k,u_k) \}_{k \in \mathbb{N}}$ is such that $J(\mathbf{y}_{k},u_{k}) \rightarrow \mathfrak{i}:=\inf\{ J(\mathbf{y}, u):\EO{(\mathbf{y}, \mathsf{p}) \in \mathbf{H}_0^1(\Omega)\times L_0^2(\Omega)~\mathrm{solves}~\eqref{eq:brinkamn_problem_state}}\}$ as $k\uparrow \infty$. Since $\mathbb{U}_{ad}$ is weakly sequentially compact in $L^2(\Omega)$, there exists a nonrelabeled subsequence $\{u_{k}\}_{k\in\mathbb{N}} \subset \mathbb{U}_{ad}$ such that $u_{k}\rightharpoonup \bar{u}$ in $L^2(\Omega)$ as $ k \uparrow \infty$ and $\bar{u} \in \mathbb{U}_{ad}$. On the other hand, since $u_k\in\mathbb{U}_{ad}$ for every $k\in\mathbb{N}$, the bound \eqref{eq:well_posedness_brinkman} shows that $\{(\mathbf{y}_k,\mathsf{p}_k)\}_{k\in\mathbb{N}}$ is uniformly bounded in $\mathbf{H}_0^1(\Omega)\times L_0^2(\Omega)$. We can therefore conclude the existence of a nonrelabeled subsequence $\{(\mathbf{y}_k,\mathsf{p}_k)\}_{k\in\mathbb{N}}$ such that $(\mathbf{y}_k,\mathsf{p}_k)\rightharpoonup (\bar{\mathbf{y}},\bar{\mathsf{p}})$ in $\mathbf{H}_0^1(\Omega)\times L_0^2(\Omega)$ as $k\uparrow\infty$; $(\bar{\mathbf{y}},\bar{\mathsf{p}})$ is the natural candidate for an optimal state. In the following, we prove that $(\bar{\mathbf{y}},\bar{\mathsf{p}})$ solves problem \eqref{eq:brinkamn_problem_state}, where $u$ is replaced by $\bar{u}$, and that $(\bar{\mathbf{y}},\bar{\mathsf{p}},\bar{u})$ is optimal.

With the weak convergence $(\mathbf{y}_k,\mathsf{p}_k)\rightharpoonup (\bar{\mathbf{y}},\bar{\mathsf{p}})$ in $\mathbf{H}_0^1(\Omega)\times L_0^2(\Omega)$ as $k\uparrow\infty$ at hand, we immediately obtain that for every $\mathbf{v}\in\mathbf{H}_0^1(\Omega)$ and $\mathsf{q}\in L_0^2(\Omega)$,
\begin{align*}
(\nabla (\bar{\mathbf{y}}-\mathbf{y}_k),\nabla \mathbf{v})_{\mathbf{L}^2(\Omega)} \to 0,~ (\bar{\mathsf{p}}-\mathsf{p}_k,\text{div }\mathbf{v})_{L^2(\Omega)}\to 0,~ (\mathsf{q},\text{div}(\bar{\mathbf{y}}-\mathbf{y}_k))_{L^2(\Omega)}\to 0
\end{align*}
as $k\uparrow \infty$. To analyze the convergence of the bilinear term, we use $u_k \rightharpoonup\bar{u}$ in $L^2 (\Omega)$ as $k\uparrow\infty$, the fact that $\bar{\mathbf{y}}\cdot\mathbf{v}\in L^2(\Omega)$, $\mathbf{y}_k\rightharpoonup \bar{\mathbf{y}}$ in $\mathbf{H}_0^1(\Omega)$ as $k\uparrow\infty$, and the compact embedding $\mathbf{H}_0^1(\Omega) \hookrightarrow\mathbf{L}^4(\Omega)$ \cite[Theorem 6.3, Part I]{MR2424078} to conclude that
\begin{multline*}
|(\bar{u}\bar{\mathbf{y}},\mathbf{v})_{\mathbf{L}^2(\Omega)} - (u_k\mathbf{y}_k,\mathbf{v})_{\mathbf{L}^2(\Omega)}|
\leq 
|\left((\bar{u} - u_k)\bar{\mathbf{y}},\mathbf{v}\right)_{\mathbf{L}^2(\Omega)}|
\\
+ 
\| u_k \|_{L^2(\Omega)}
\| \bar{\mathbf{y}} - \mathbf{y}_k \|_{\mathbf{L}^4(\Omega)}
\| \mathbf{v} \|_{\mathbf{L}^4(\Omega)}
\to 0, 
\qquad k\uparrow\infty,
\end{multline*}
where we have also used that $\| u_k \|_{L^2(\Omega)} \leq \mathcal{M}$ for every $k \in \mathbb{N}$. We have thus proved that $(\bar{\mathbf{y}},\bar{\mathsf{p}})$ solves problem \eqref{eq:brinkamn_problem_state}, where $u$ is replaced by $\bar{u}$. Finally, the optimality of $(\bar{\mathbf{y}},\bar{\mathsf{p}},\bar{u})$ follows from the strong convergence $\mathbf{y}_k\to\bar{\mathbf{y}}$ in $\mathbf{L}^2(\Omega)$ as $k\uparrow\infty$ and the fact that the square of $\|\cdot\|_{L^2(\Omega)}$ es weakly lower semicontinuous in $L^2(\Omega)$.
\end{proof}

\subsection{\EO{First-order optimality conditions}}\label{section:firstconditions}
In this section, we obtain first-order optimality conditions for the optimal control problem \eqref{eq:min_functional}--\eqref{eq:brinkamn_problem_state}.
In the absence of convexity, we proceed as usual in the framework of \emph{local solutions} in $L^2(\Omega)$ \cite[page 207]{MR2583281}: \EO{A control $\bar{u}\in\mathbb{U}_{ad}$ is said to be \emph{locally optimal} in $L^2(\Omega)$ for \eqref{eq:min_functional}--\eqref{eq:brinkamn_problem_state} if there exists $\delta>0$ such that $J(\bar{\mathbf{y}},\bar{u})\leq J(\mathbf{y},u)$ for all $u\in \mathbb{U}_{ad}$ such that $\|u-\bar{u}\|_{L^2(\Omega)}\leq \delta$. Here, $\bar{\mathbf{y}}$ and $\mathbf{y}$ denote the velocity fields associated with $\bar{u}$ and $u$, respectively.} \EO{A control $\bar{u}\in\mathbb{U}_{ad}$ is called a \emph{strict local solution} in $L^2(\Omega)$ for \eqref{eq:min_functional}--\eqref{eq:brinkamn_problem_state} if there exists $\delta>0$ such that $J(\bar{\mathbf{y}},\bar{u}) < J(\mathbf{y},u)$ for all $u\in \mathbb{U}_{ad} \setminus \{ \bar{u} \}$ such that $\|u-\bar{u}\|_{L^2(\Omega)}\leq \delta$.}

From \EO{the results of Theorem \ref{theorem:differentiablity_stokes_brinkman}, it follows that the control-to-state map $\mathcal{S}:\mathcal{A} \subset L^2(\Omega) \to \mathbf{H}_0^1(\Omega)\times L_0^2(\Omega)$, which assigns to $u \in \mathcal{A}$ the unique pair $(\mathbf{y},\mathsf{p}) \in \mathbf{H}_0^1(\Omega) \times L_0^2(\Omega)$ that solves \eqref{eq:brinkamn_problem_state}, is of class $C^2$. Since it will be useful, we introduce the map}
\begin{equation}\label{eq:mapG}
\mathcal{G}: \EO{\mathcal{A}}\to \mathbf{H}_0^1(\Omega) :
\quad u\to \mathbf{y},
\end{equation}
where $\mathbf{y}$ corresponds to the velocity component of $(\mathbf{y},\mathsf{p})=\mathcal{S}(u)$. With these ingredients, we introduce the reduced cost functional $j:\EO{\mathcal{A}}\to\mathbb{R}_0^+$ as $j(u) := J(\GC{\mathcal{G}(u)},u)$.

To obtain first-order optimality conditions, we present the following basic result: If $\bar{u} \in \mathbb{U}_{ad}$ is locally optimal for problem \eqref{eq:min_functional}--\eqref{eq:brinkamn_problem_state}, then $\bar{u}$ satisfies \cite[Lemma 4.18]{MR2583281}
\begin{equation}\label{eq:ineq_variational}
j'(\bar{u})(u-\bar{u})\geq 0\quad \forall u\in \mathbb{U}_{ad},
\end{equation}
where $j'(\bar{u})$ denotes the Gate\^aux derivative of $j$ at $\bar{u}$. To explore the inequality \eqref{eq:ineq_variational}, we introduce $(\mathbf{z},\mathsf{r})\in\mathbf{H}_0^1(\Omega)\times L_0^2(\Omega)$ as the solution to the \emph{adjoint equations}
\begin{equation}\label{eq:brinkamn_problem_adjoint}
\begin{array}{rl}
(\nabla \mathbf{v}, \nabla \mathbf{z})_{\mathbf{L}^2(\Omega)}
+
\EO{( u, \mathbf{z} \cdot \mathbf{v})_{\mathbf{L}^2(\Omega)}}
+
(\mathsf{r},\text{div }\mathbf{v})_{L^2(\Omega)} = &\hspace{-2mm} (\mathbf{y}-\mathbf{y}_{\Omega},\mathbf{v})_{\mathbf{L}^2(\Omega)},
\\
(\mathsf{s},\text{div }\mathbf{z})_{L^2(\Omega)}  = & \hspace{-2mm} 0
\end{array}
\end{equation}
for all $(\mathbf{v},\mathsf{s}) \in \mathbf{H}_0^1(\Omega) \times L_0^2(\Omega)$. Here, $(\mathbf{y},\mathsf{p})=\mathcal{S}(u)$ is the unique solution to \eqref{eq:brinkamn_problem_state}.

\begin{theorem}[well-posedness]
\EO{The adjoint equations are well-posed.}
\label{thm:adj_equations_well-posed}
\begin{proof}
 \EO{Let $u \in \mathcal{A}$. We define
 $
  a: \mathbf{H}_0^1(\Omega) \times \mathbf{H}_0^1(\Omega) \rightarrow \mathbb{R}
 $
by
$
 a(\mathbf{z},\mathbf{v}) := (\nabla \mathbf{v},\nabla \mathbf{z})_{\mathbf{L}^2(\Omega)}
+
( u,\mathbf{z} \cdot \mathbf{v})_{L^2(\Omega)}.
$
It is clear that $a$ is a bilinear and bounded form in $\mathbf{H}_0^1(\Omega) \times \mathbf{H}_0^1(\Omega)$. On the other hand, since $u \in \mathcal{A}$, there exists $\bar{u} \in \mathcal{A}_0$ such that $u \in B_{\epsilon_{\bar u}}(\bar u)$.
With $\bar{u}$ at hand, we rewrite the term $a(\mathbf{z},\mathbf{v})$ as follows:
\[
 a(\mathbf{z},\mathbf{v}) := (\nabla \mathbf{v},\nabla \mathbf{z})_{\mathbf{L}^2(\Omega)}
+
( \bar{u},\mathbf{z} \cdot \mathbf{v})_{L^2(\Omega)}
+
( (u - \bar{u}),\mathbf{z} \cdot \mathbf{v})_{L^2(\Omega)}.
\]
If we set $\mathbf{z} = \mathbf{v}$ and use that $\epsilon_{\bar u} < C_{4 \hookrightarrow 2}^{-2}$, we arrive at the coercivity of $a$ in $\mathbf{H}_0^1(\Omega) \times \mathbf{H}_0^1(\Omega)$: $a(\mathbf{v}, \mathbf{v}) \geq (1 - \epsilon_{\bar{u}} C_{4 \hookrightarrow 2}^{2}) \| \nabla \mathbf{v}\|_{\mathbf{L}^2(\Omega)}$. As a final ingredient, we note that $\mathbf{y} - \mathbf{y}_{\Omega} \in \mathbf{L}^2(\Omega)$. The result follows from the standard inf-sup theory for saddle point problems.}
\end{proof}
\end{theorem}

\begin{remark}[regularity properties of the adjoint pair $(\mathbf{z},\mathsf{r})$]
\EO{Let $u \in \mathcal{A}$. Since $\mathbf{y} - \mathbf{y}_{\Omega} \in \mathbf{L}^2(\Omega)$ and $\Omega$ is Lipschitz, an application of Theorem \ref{theorem_regularity_Linfty} shows the existence of $\kappa> 4$ when $d=2$ and $\kappa >3$ when $d=3$ such that $(\mathbf{z},\mathsf{r})$ belongs to $\mathbf{H}_0^{1}(\Omega) \cap \mathbf{W}^{1,\kappa}(\Omega) \times L_0^2(\Omega) \cap L^{\kappa}(\Omega)$. Assuming that $\Omega$ is a convex polytope, it can be concluded from the results of Theorem \ref{eq:theorem_regularity} that $(\mathbf{z},\mathsf{r})$ belongs to $\mathbf{H}^2(\Omega)\cap\mathbf{H}_0^1(\Omega) \times H^1(\Omega)\cap L_0^2(\Omega)$.}
\end{remark}

We are now in a position to establish necessary first-order optimality conditions as in the scalar case described in \cite[identity (2.6)]{MR2536007}.

\begin{theorem}[first-order optimality conditions]
Every locally optimal control $\bar{u}\in \mathbb{U}_{ad}$ for the optimal control problem \eqref{eq:min_functional}--\eqref{eq:brinkamn_problem_state} satisfies the variational inequality
\begin{eqnarray}\label{eq:variational_inequality}
(\alpha \bar{u}-\bar{\mathbf{y}}\cdot \bar{\mathbf{z}},u-\bar{u})_{L^2(\Omega)}\geq 0\quad \forall u\in \mathbb{U}_{ad}.
\end{eqnarray}
Here, $(\bar{\mathbf{z}},\bar{\mathsf{r}})$ denotes the solution of the adjoint equations \eqref{eq:brinkamn_problem_adjoint}, where $u$ and $\mathbf{y}$ are replaced by $\bar{u}$ and $\bar{\mathbf{y}}=\mathcal{G}\GC{(\bar{u})}$, respectively.
\end{theorem}

\begin{proof}
In a first step, we rewrite the variational inequality \eqref{eq:ineq_variational} as follows:
\begin{eqnarray}\label{eq:variat_ineq_01}
(\mathcal{G}'(\bar{u})(u-\bar{u}),\bar{\mathbf{y}}-\mathbf{y}_{\Omega})_{\mathbf{L}^2(\Omega)}
+
(\alpha \bar{u},u-\bar{u})_{L^2(\Omega)}\geq 0
\quad
\forall u\in \mathbb{U}_{ad},
\end{eqnarray}
where $\bar{\mathbf{y}}=\mathcal{G}\GC{(\bar{u})}$. Define $(\boldsymbol{\varphi},\zeta):=\mathcal{S}'(\bar{u})(u-\bar{u})$ and note that $(\boldsymbol{\varphi},\zeta)$ solves problem \eqref{eq:problem_S_first} with $u$ and $v$ replaced by $\bar{u}$ and $u-\bar{u}$, respectively. We now set $\mathbf{v}=\bar{\mathbf{z}}$ in the first equation of this system to obtain
$
(\nabla \boldsymbol{\varphi},\nabla \bar{\mathbf{z}})_{\mathbf{L}^2(\Omega)}
+
( \bar{u}\boldsymbol{\varphi}, \bar{\mathbf{z}})_{\mathbf{L}^2(\Omega)}
-
(\zeta,\mathrm{div~} \bar{\mathbf{z}})_{L^2(\Omega)}
=  
-((u-\bar{u})\bar{\mathbf{y}},\bar{\mathbf{z}})_{\mathbf{L}^2(\Omega)}.
$
Similarly, we set $\mathbf{v} = \boldsymbol{\varphi}$ as a test function in the the first equation of \eqref{eq:brinkamn_problem_adjoint} to obtain
$
(\nabla \boldsymbol{\varphi},\nabla \bar{\mathbf{z}})_{\mathbf{L}^2(\Omega)}
+
( \bar{u}\bar{\mathbf{z}}, \boldsymbol{\varphi})_{\mathbf{L}^2(\Omega)}
+
(\mathsf{r},\text{div }\boldsymbol{\varphi})_{L^2(\Omega)} = (\bar{\mathbf{y}}-\mathbf{y}_{\Omega},\boldsymbol{\varphi})_{\mathbf{L}^2(\Omega)}.
$
Since $(\zeta,\mathrm{div~} \bar{\mathbf{z}})_{L^2(\Omega)} = 0$ and $(\mathsf{r},\text{div }\boldsymbol{\varphi})_{L^2(\Omega)} = 0$ for $\zeta \in L_0^2(\Omega)$ and $\mathsf{r} \in L_0^2(\Omega)$, respectively, we can conclude that $(\bar{\mathbf{y}}-\mathbf{y}_{\Omega},\boldsymbol{\varphi})_{\mathbf{L}^2(\Omega)} = -((u-\bar{u})\bar{\mathbf{y}},\bar{\mathbf{z}})_{\mathbf{L}^2(\Omega)}$. The desired variational inequality \eqref{eq:variational_inequality} thus follows from \eqref{eq:variat_ineq_01}. Note that $\boldsymbol{\varphi} = \mathcal{G}'(\bar u)(u - \bar u)$.
\end{proof}

The following projection formula will be important for the derivation of regularity estimates: If $\bar{u} \in \mathbb{U}_{ad}$ is a locally optimal control for  \eqref{eq:min_functional}--\eqref{eq:brinkamn_problem_state}, then \cite[section 4.6.1]{MR2583281}
\begin{equation}\label{eq:representation}
\bar{u}(x):=\Pi_{[\mathsf{a},\mathsf{b}]}(\alpha^{-1}\bar{\mathbf{y}}\cdot\bar{\mathbf{z}})\text{ a.e. } x \in \Omega,
\end{equation}
where $\Pi_{[\mathsf{a},\mathsf{{b}]}}:L^1(\Omega)\to\mathbb{U}_{ad}$ is defined by $\Pi_{[\mathsf{a},\mathsf{b}]}(v):=\min\{\mathsf{b},\max\{v,\mathsf{a}\}\}$ a.e.~in $\Omega$.

The Sobolev regularity of an optimal control variable is as follows.

\begin{theorem}[regularity of an optimal control]
\label{thm:regularity_optimal_control}
Let $\bar{u}$ be a locally optimal control for problem \eqref{eq:min_functional}--\eqref{eq:brinkamn_problem_state}.
\EO{Then, $\bar{u}\in H^1(\Omega)$ and}
\begin{align}\label{eq:control_regul_estimate}
\|\bar{u}\|_{H^1(\Omega)}
\lesssim
\|\mathbf{f}\|_{\mathbf{L}^2(\Omega)}
\left[
\|\mathbf{f}\|_{\mathbf{L}^2(\Omega)}+\|\mathbf{y}_{\Omega}\|_{\mathbf{L}^2(\Omega)}
\right].
\end{align}
\end{theorem} 
\begin{proof}
Let $(\bar{\mathbf{y}},\bar{\mathsf{p}}) = \mathcal{S}\GC{(\bar{u})}$, and let $(\bar{\mathbf{z}},\bar{\mathsf{r}}) \in \mathbf{H}_0^1(\Omega)\times L_0^2(\Omega)$ be the solution of  \eqref{eq:brinkamn_problem_adjoint}, where $u$ and $\mathbf{y}$ are replaced by $\bar{u}$ and $\bar{\mathbf{y}}$, respectively. Since $\mathbf{f},\mathbf{y}_{\Omega}\in\mathbf{L}^2(\Omega)$, \EO{we can immediately deduce from the results of Theorem \ref{theorem_regularity_Linfty} that $\bar{\mathbf{y}},\bar{\mathbf{z}}\in\mathbf{H}_0^1(\Omega)\cap \mathbf{C}(\bar \Omega)$}. A simple calculation therefore shows that $\nabla (\bar{\mathbf{y}}\cdot\bar{\mathbf{z}}) \in \mathbf{L}^2(\Omega)$ Consequently, the desired regularity property and the bound \eqref{eq:control_regul_estimate} for $\bar{u}$ follow from the projection formula \eqref{eq:representation}.

It \EO{is important to note that, as shown in the proof of Theorem \ref{theorem_regularity_Linfty}, the control of the norms $\|\bar{\mathbf{y}}\|_{\mathbf{C}(\bar \Omega)}$ and $\|\bar{\mathbf{z}}\|_{\mathbf{C}(\bar \Omega)}$ involve hidden constants that depend on $\|\bar{u}\|_{L^2(\Omega)}$. Given the structure of the admissible control set $\mathbb{U}_{ad}$, we can control $\|\bar{u}\|_{L^2(\Omega)}$ with a constant that depends only on $\Omega$, $\mathsf{a}$, and $\mathsf{b}$.}
\end{proof}

\subsection{Second-order optimality conditions}\label{section:secondconditions}
In this section, we derive necessary and sufficient second-order optimality conditions. To this end, we begin our analysis with the following instrumental result.

\begin{theorem}[$j$ is of class $C^2$ and $j''$ is Lipschitz]
\label{eq:theorem_properties_j}
The reduced cost $j: \GC{\mathcal{A}}\to\mathbb{R}_0^+$ is of class $C^2$. Moreover, for each $u\in\GC{\mathcal{A}}$ and $v\in \GC{L^{2}(\Omega)}$, it holds that
\begin{align}
\label{eq:identity_j}
j''(u)v^2=\alpha\|v\|_{L^2(\Omega)}^2-2(v\boldsymbol{\varphi},\mathbf{z})_{\mathbf{L}^2(\Omega)}+\|\boldsymbol{\varphi}\|_{\mathbf{L}^2(\Omega)}^2.
\end{align}
Here, $(\mathbf{z},\mathsf{r})\in\mathbf{H}_0^1(\Omega)\times L_0^2(\Omega)$ denotes the solution to \eqref{eq:brinkamn_problem_adjoint} and $(\boldsymbol{\varphi},\zeta)=\mathcal{S}'(u)v$. Finally, for every $u_1,u_2\in\GC{\mathcal{A}}$ and $v\in \GC{L^{2}(\Omega)}$, we have the following bound:
\begin{eqnarray}
\label{eq:lemma_j}
|j''(u_1)v^2-j''(u_2)v^2|\leq \mathfrak{D}\|u_1-u_2\|_{L^2(\Omega)}\|v\|_{L^2(\Omega)}^2,
\end{eqnarray}
where $\mathfrak{D}>0$ is a constant that depends on \EO{$\|\mathbf{f}\|_{\mathbf{L}^2(\Omega)}$} and $\|\mathbf{y}_\Omega\|_{\mathbf{L}^2(\Omega)}$.
\end{theorem}

\begin{proof}
The fact that $j$ belongs to the class $C^2$ is a consequence of the differentiability properties of the control-to-state map $\mathcal{S}$ from \GC{Theorem \ref{theorem:differentiablity_stokes_brinkman}}. We now prove \eqref{eq:identity_j} and \eqref{eq:lemma_j} in three steps.

\emph{Step 1.} \emph{The identity \eqref{eq:identity_j}:} For every $u\in\GC{\mathcal{A}}$ and $v \in \GC{L^{2}(\Omega)}$, we have
\begin{align}\label{eq:second_derivative_aux}
j''(u)v^2 =\alpha\|v\|_{L^2(\Omega)}^2+(\mathbf{y}-\mathbf{y}_{\Omega},\boldsymbol{\psi})_{\mathbf{L}^2(\Omega)}+\|\boldsymbol{\varphi}\|_{\mathbf{L}^2(\Omega)}^2;
\end{align}
\EO{compare with \cite[page 151]{MR4858133}. In \eqref{eq:second_derivative_aux},} $(\boldsymbol{\varphi},\zeta) = \mathcal{S}'(u)v$, i.e., $(\boldsymbol{\varphi},\zeta)$ corresponds to the solution to \eqref{eq:problem_S_first} and
$(\boldsymbol{\psi},\xi)$ is defined as the solution to \eqref{eq:problem_S_second} with $v_1 = v_2 = v$ and $\boldsymbol{\varphi}_{v} = \boldsymbol{\varphi}$. Note that $\boldsymbol{\varphi}=\mathcal{G}'(u)v$ and $\boldsymbol{\psi}=\mathcal{G}''(u)v^2$. We now set $(\mathbf{v},\mathsf{s})=(\boldsymbol{\psi},0)$ in \eqref{eq:brinkamn_problem_adjoint} and $(\mathbf{v},\mathsf{s}) = (\mathbf{z},0)$ in \eqref{eq:problem_S_second} to obtain the identity $(\mathbf{y}-\mathbf{y}_{\Omega},\boldsymbol{\psi})_{\mathbf{L}^2(\Omega)}=-2(v\boldsymbol{\varphi},\mathbf{z})_{\mathbf{L}^2(\Omega)}$. Replacing this relation into \eqref{eq:second_derivative_aux} gives the desired identity \eqref{eq:identity_j}.

\emph{Step 2:} \emph{The bound \eqref{eq:lemma_j}:} Let $u_1,u_2\in \GC{\mathcal{A}}$ and $v\in \GC{L^{2}(\Omega)}$. Define $(\boldsymbol{\varphi}_i,\zeta_i):=\mathcal{S}'(u_i)v$ with $i \in \{ 1,2 \}$. We note that $(\boldsymbol{\varphi}_i,\zeta_i)$ is the solution to \eqref{eq:problem_S_first}, where $u$ and $\mathbf{y}$ are replaced by $u_i$ and $\mathbf{y}_i$, respectively. A simple application of the identity \eqref{eq:identity_j} leads to
\begin{multline}
\label{eq:aux_identity_propj}
j''(u_1)v^2-j''(u_2)v^2=2(v(\boldsymbol{\varphi}_2-\boldsymbol{\varphi}_1),\mathbf{z}_2)_{\mathbf{L}^2(\Omega)}+2(v\boldsymbol{\varphi}_1,\mathbf{z}_2-\mathbf{z}_1)_{\mathbf{L}^2(\Omega)}\\
(\boldsymbol{\varphi}_1-\boldsymbol{\varphi}_2,\boldsymbol{\varphi}_1+\boldsymbol{\varphi}_2)_{\mathbf{L}^2(\Omega)}=:\text{I}+\text{II}+\text{III}.
\end{multline}
Here, $(\mathbf{z}_i,\mathsf{r}_i) \in\mathbf{H}_0^1(\Omega) \times L_0^2(\Omega)$ corresponds to the solution of the adjoint problem \eqref{eq:brinkamn_problem_adjoint}, where $\mathbf{y}$ and $u$ are replaced by $\mathbf{y}_i$ and $u_i$, respectively. In the following, we analyze each of the terms \text{I}, \text{II}, and \text{III} involved in the identity \eqref{eq:aux_identity_propj}.

We first bound \text{I}. Given the embedding $\mathbf{H}_0^1(\Omega) \hookrightarrow \mathbf{L}^{r}(\Omega)$ with $r<\infty$ when $d=2$ and $r\leq6$ when $d=3$, and the fact that \EO{$\mathbf{f} \in \mathbf{L}^{2}(\Omega)$} and $\mathbf{y}_{\Omega}\in\mathbf{L}^2(\Omega)$, we obtain
\begin{align}\label{eq:estimates_extra}
\|\mathbf{y}_i\|_{\mathbf{L}^{r}(\Omega)}\lesssim \EO{\|\mathbf{f}\|_{\mathbf{L}^{2}(\Omega)}},
\quad 
\|\mathbf{z}_i\|_{\mathbf{L}^{r}(\Omega)}\lesssim \EO{\|\mathbf{f}\|_{\mathbf{L}^{2}(\Omega)}} + \|\mathbf{y}_{\Omega}\|_{\mathbf{L}^{2}(\Omega)},
\quad
i\in\{1,2\}.
\end{align}
We thus bound $\text{I}$ as follows: $\text{I} \leq 2 \|v\|_{L^2(\Omega)}\|\boldsymbol{\varphi}_2-\boldsymbol{\varphi}_1\|_{\mathbf{L}^4(\Omega)}\|\mathbf{z}_2\|_{\mathbf{L}^4(\Omega)}$ $\lesssim$ $\|v\|_{L^2(\Omega)}\|\nabla(\boldsymbol{\varphi}_2-\boldsymbol{\varphi}_1)\|_{\mathbf{L}^2(\Omega)}$, where the hidden constant depends on \EO{$\|\mathbf{f}\|_{\mathbf{L}^{2}(\Omega)}$} and $\|\mathbf{y}_\Omega\|_{\mathbf{L}^2(\Omega)}$. To control $\|\nabla(\boldsymbol{\varphi}_2-\boldsymbol{\varphi}_1)\|_{\mathbf{L}^2(\Omega)}$, we invoke the problem that solves $(\boldsymbol{\varphi}_2-\boldsymbol{\varphi}_1,\zeta_2-\zeta_1)$:
\begin{multline*}
(\nabla (\boldsymbol{\varphi}_2-\boldsymbol{\varphi}_1),\nabla \mathbf{v})_{\mathbf{L}^2(\Omega)}+( u_2(\boldsymbol{\varphi}_2-\boldsymbol{\varphi}_1), \mathbf{v})_{\mathbf{L}^2(\Omega)}-(\zeta_2-\zeta_1,\text{div }\mathbf{v})_{L^2(\Omega)} \\= (\boldsymbol{\varphi}_1(u_1-u_2)-v(\mathbf{y}_2-\mathbf{y}_1),\mathbf{v})_{\mathbf{L}^2(\Omega)},\qquad
(\mathsf{s},\text{div }(\boldsymbol{\varphi}_2-\boldsymbol{\varphi}_1))_{L^2(\Omega)}  =  0,
\end{multline*}
for all $(\mathbf{v},\mathsf{s}) \in \mathbf{H}_0^1(\Omega) \times L_0^2(\Omega)$. If we set $\mathbf{v}=\boldsymbol{\varphi}_2-\boldsymbol{\varphi}_1$ and $\mathsf{s}=0$ \EO{and proceed as in the proof of Theorem} \ref{thm:adj_equations_well-posed}, we obtain
\begin{multline}\label{eq:estimate_aux1_theorem_second}
\|\nabla (\boldsymbol{\varphi}_2-\boldsymbol{\varphi}_1)\|_{\mathbf{L}^2(\Omega)}\lesssim \|\nabla\boldsymbol{\varphi}_1\|_{\mathbf{L}^{2}(\Omega)}\|u_1-u_2\|_{L^2(\Omega)}\\+\|v\|_{L^2(\Omega)}\|\nabla(\mathbf{y}_2-\mathbf{y}_1)\|_{\mathbf{L}^{2}(\Omega)}.
\end{multline}
A bound for $\|\nabla\boldsymbol{\varphi}_1\|_{\mathbf{L}^{2}(\Omega)}$ can be derived using stability estimates for the problems that $(\boldsymbol{\varphi}_1,\zeta_1)$ and $(\mathbf{y}_1,\mathsf{p}_1)$ solve. In fact,
\begin{align}\label{eq:estimate_aux2_theorem_second}
\|\nabla\boldsymbol{\varphi}_1\|_{\mathbf{L}^{2}(\Omega)}
\lesssim 
\|v \mathbf{y}_1\|_{\mathbf{H}^{-1}(\Omega)}
\lesssim \|v\|_{L^2(\Omega)} \|\mathbf{y}_1\|_{\mathbf{L}^{4}(\Omega)}
\lesssim \|v\|_{L^2(\Omega)}\EO{\|\mathbf{f}\|_{\mathbf{L}^{2}(\Omega)}}.
\end{align}
On the other hand, we note that $(\mathbf{y}_2-\mathbf{y}_1,\mathsf{p}_2-\mathsf{p}_1)\in\mathbf{H}_0^1(\Omega)\times L_0^2(\Omega)$ solves
\begin{multline*}
(\nabla (\mathbf{y}_2-\mathbf{y}_1),\nabla \mathbf{v})_{\mathbf{L}^2(\Omega)}+( u_2(\mathbf{y}_2-\mathbf{y}_1), \mathbf{v})_{\mathbf{L}^2(\Omega)}-(\mathsf{p}_2-\mathsf{p}_1,\text{div }\mathbf{v})_{L^2(\Omega)} \\
= (\mathbf{y}_1(u_1-u_2),\mathbf{v})_{\mathbf{L}^2(\Omega)},
\quad
(\mathsf{s},\text{div }(\mathbf{y}_2-\mathbf{y}_1))_{L^2(\Omega)}  =  0
\quad
\forall (\mathbf{v},\mathsf{s}) \in \mathbf{H}_0^1(\Omega) \times L_0^2(\Omega).
\end{multline*}
If we set $(\mathbf{v},\mathsf{s}) = (\mathbf{y}_2-\mathbf{y}_1,0)$ \EO{and proceed as in the proof of Theorem \ref{thm:adj_equations_well-posed}, we obtain}
\begin{multline}\label{eq:estimate_aux3_theorem_second}
\|\nabla(\mathbf{y}_2-\mathbf{y}_1)\|_{\mathbf{L}^{2}(\Omega)}
\lesssim 
\|\nabla\mathbf{y}_1\|_{\mathbf{L}^{2}(\Omega)} \|u_1-u_2\|_{L^2(\Omega)}
\lesssim 
\EO{\|\mathbf{f}\|_{\mathbf{L}^{2}(\Omega)}}
\|u_1-u_2\|_{L^2(\Omega)}.
\end{multline} 
A collection of the estimates \eqref{eq:estimate_aux1_theorem_second}, \eqref{eq:estimate_aux2_theorem_second}, and \eqref{eq:estimate_aux3_theorem_second} leads to the conclusion
\begin{align}\label{eq:I_estimates}
\text{I}\lesssim\|v\|^2_{L^2(\Omega)}\|u_1-u_2\|_{L^2(\Omega)},
\end{align}
with a hidden constant that depends on \EO{$\|\mathbf{f}\|_{\mathbf{L}^{2}(\Omega)}$} and $\|\mathbf{y}_\Omega\|_{\mathbf{L}^2(\Omega)}$.

In the following, we control $\text{II}$ in \eqref{eq:aux_identity_propj}. For this, we use H\"older's inequality and \eqref{eq:estimate_aux2_theorem_second} and obtain $\text{II}\lesssim \|v\|_{L^2(\Omega)}\|\boldsymbol{\varphi}_1\|_{\mathbf{L}^4(\Omega)}\|\mathbf{z}_2-\mathbf{z}_1\|_{\mathbf{L}^{4}(\Omega)}\lesssim \|v\|_{L^2(\Omega)}^2 \|\nabla(\mathbf{z}_2-\mathbf{z}_1)\|_{\mathbf{L}^{2}(\Omega)}$. To control $\|\nabla(\mathbf{z}_2-\mathbf{z}_1)\|_{\mathbf{L}^{2}(\Omega)}$, we first note that  $(\mathbf{z}_2-\mathbf{z}_1,\mathsf{r}_2-\mathsf{r}_1)$ solves
\begin{multline*}
(\nabla \mathbf{v},\nabla (\mathbf{z}_2-\mathbf{z}_1))_{\mathbf{L}^2(\Omega)}+( u_2(\mathbf{z}_2-\mathbf{z}_1), \mathbf{v})_{\mathbf{L}^2(\Omega)}+(\mathsf{r}_2-\mathsf{r}_1,\text{div }\mathbf{v})_{L^2(\Omega)} \\= ((\mathbf{y}_2-\mathbf{y}_1)+\mathbf{z}_1(u_1-u_2),\mathbf{v})_{\mathbf{L}^2(\Omega)},\qquad
(\mathsf{s},\text{div }(\mathbf{z}_2-\mathbf{z}_1))_{L^2(\Omega)}  =  0,
\end{multline*}
for all $(\mathbf{v},\mathsf{s}) \in \mathbf{H}_0^1(\Omega) \times L_0^2(\Omega)$. If we set $\mathbf{v}=\mathbf{z}_2-\mathbf{z}_1$ and $\mathsf{s}=0$ \EO{and proceed as in the proof of Theorem} \ref{thm:adj_equations_well-posed}, we obtain
\begin{align}
\|\nabla(\mathbf{z}_2-\mathbf{z}_1)\|_{\mathbf{L}^{2}(\Omega)}\lesssim \|\nabla(\mathbf{y}_2-\mathbf{y}_1)\|_{\mathbf{L}^{2}(\Omega)}+\|\nabla\mathbf{z}_1\|_{\mathbf{L}^{2}(\Omega)}\|u_2-u_1\|_{L^2(\Omega)}.
\end{align}
The bound $\|\nabla(\mathbf{y}_2-\mathbf{y}_1)\|_{\mathbf{L}^{2}(\Omega)}\lesssim\|u_2-u_1\|_{L^2(\Omega)}$, which follows from \eqref{eq:estimate_aux3_theorem_second}, and the stability estimate $\|\nabla\mathbf{z}_1\|_{\mathbf{L}^{2}(\Omega)} \lesssim \EO{\|\mathbf{f}\|_{\mathbf{L}^{2}(\Omega)}} + \|\mathbf{y}_{\Omega}\|_{\mathbf{L}^{2}(\Omega)}$ allow us to obtain
\begin{align}\label{eq:II_estimates}
\text{II}\lesssim \|v\|_{L^2(\Omega)}^2 \|\nabla(\mathbf{z}_2-\mathbf{z}_1)\|_{\mathbf{L}^{2}(\Omega)}\lesssim \|v\|_{L^2(\Omega)}^2 \|u_2-u_1\|_{L^2(\Omega)}.
\end{align}

Finally, we control $\text{III}$. To accomplish this task, we use basic inequalities, \eqref{eq:estimate_aux2_theorem_second}, an adaptation of \eqref{eq:estimate_aux2_theorem_second} for $\boldsymbol{\varphi}_2$, and $\|\boldsymbol{\varphi}_2-\boldsymbol{\varphi}_1\|_{\mathbf{L}^2(\Omega)}\lesssim\|v\|_{L^2(\Omega)}\|u_1-u_2\|_{L^2(\Omega)}$:
\begin{align}\label{eq:III_estimates}
\text{III}\lesssim \|v\|_{L^2(\Omega)}\|\boldsymbol{\varphi}_2-\boldsymbol{\varphi}_1\|_{\mathbf{L}^2(\Omega)}\lesssim\|v\|_{L^2(\Omega)}^2 \|u_2-u_1\|_{L^2(\Omega)}.
\end{align}

\emph{Step 3:} Given \eqref{eq:aux_identity_propj}, we derive the desired estimate \eqref{eq:lemma_j} by collecting the bounds derived in \eqref{eq:I_estimates}, \eqref{eq:II_estimates}, and \eqref{eq:III_estimates}. This concludes the proof.
\end{proof}

\subsubsection{Necessary second-order optimality conditions}

We \EO{adapt} the arguments from the proof of \cite[Theorem 23]{MR3586845} and develop necessary second-order optimality conditions. For this purpose, we introduce some concepts: Let \EO{$(\bar{u},(\bar{\mathbf{y}},\bar{\mathsf{p}}),(\bar{\mathbf{z}},\bar{\mathsf{r}}))\in\mathbb{U}_{ad} \times (\mathbf{H}_0^1(\Omega)\times L_0^2(\Omega)) \times (\mathbf{H}_0^1(\Omega) \times L_0^2(\Omega))$} satisfy the first-order optimality conditions \eqref{eq:brinkamn_problem_state}, \eqref{eq:brinkamn_problem_adjoint}, and \eqref{eq:variational_inequality}. We define $\bar{\mathfrak{d}}:=\alpha\bar{u}-\bar{\mathbf{y}}\cdot\bar{\mathbf{z}}$ and the cone of critical directions
\begin{align}\label{eq:cone}
C_{\bar{u}}:=\{v\in L^{2}(\Omega) \text{ satisfying } \eqref{eq:cone_condition} \text{ and } v(x)=0 \text{ if } \bar{\mathfrak{d}}(x)\neq 0 \text{ a.e.~in } \Omega \},
\end{align}
where condition \eqref{eq:cone_condition} reads as follows:
\begin{align}\label{eq:cone_condition}
v(x)\geq 0 \text{ a.e. } x\in\Omega	\text{ if } \bar{u}(x)=\mathsf{a},\quad v(x)\leq 0 \text{ a.e. } x\in\Omega	\text{ if } \bar{u}(x)= \mathsf{b}.
\end{align}

We are now in a position to present necessary second-order optimality conditions.

\begin{theorem}[necessary second-order optimality conditions]\label{theorem_neccesary_conditions} If $\bar{u}\in \mathbb{U}_{ad}$ denotes a locally optimal control for \eqref{eq:min_functional}--\eqref{eq:brinkamn_problem_state}, then $j''(\bar{u})v^2\geq 0$ for all $v\in C_{\bar{u}}$.
\end{theorem}
\begin{proof}
The proof follows from an adaptation of the arguments from \cite[Theorem 23]{MR3586845} and an application of Theorem \ref{eq:theorem_properties_j}. In a first step, we define for $w\in C_{\bar{u}}$, for a.e.~$x\in\Omega$, and for $k\in\mathbb{N}$, the function 
\begin{eqnarray*}
w_k(x):=\left\{
\begin{array}{cl}
0 & \text{if } x: \mathsf{a}<\bar{u}(x)<\mathsf{a}+k^{-1},\quad \mathsf{b}-k^{-1}<\bar{u}(x)<\mathsf{b},\\
\Pi_{[-k,k]}(w(x)) & \text{otherwise}.
\end{array}
\right.
\end{eqnarray*}
Starting from the fact that $w\in C_{\bar{u}}$, we obtain that \EO{$w_k\in C_{\bar{u}}$ for every} $k \in \mathbb{N}$. We also note that $w_k(x)\to w(x)$ as $k\uparrow\infty$ for a.e.~$x\in\Omega$ and that $|w_k(x)|\leq |w(x)|$ for a.e.~$x\in\Omega$. As a consequence of Lebesgue's dominated convergence theorem, we have that $w_k\to w$ in $L^2(\Omega)$ as $k\uparrow\infty$. We now define $\rho^{\star}=\min\{k^{-2},(b-a)k^{-1}\}$ and note that for $\rho \in (0, \rho^{\star}]$ we have $\bar{u}+\rho w_k \in \mathbb{U}_{ad}$. Thus, since $\bar{u}+\rho w_k$ is admissible, we use the fact that $\bar{u}$ is a local minimizer to obtain the inequality $j(\bar{u})\leq j(\bar{u}+\rho w_k)$ for $\rho$ sufficiently small. We now apply Taylor's theorem for $j$ at $\bar{u}$, and the relation $j'(\bar{u})w_k=0$, which follows from the fact that $w_k\in C_{\bar{u}}$, to obtain
\begin{align*}
0\leq j(\bar{u}+\rho w_k)-j(\bar{u})=\rho j'(\bar{u})w_k+\frac{\rho^2}{2}j''(\bar{u}+\rho \theta_k w_k)w_k^2 =\frac{\rho^2}{2}j''(\bar{u}+\rho\theta_k w_k)w_k^2,
\end{align*}
where $\theta_k\in(0,1)$. Based on the estimate \eqref{eq:lemma_j}, we obtain the following result as $\rho\downarrow 0$:
\begin{align*}
|j''(\bar{u}+\rho\theta_k w_k)w_k^2-j''(\bar{u})w_k^2|\lesssim\rho \|w_k\|_{L^2(\Omega)}^3\to 0,
\qquad \rho\downarrow 0.
\end{align*}
This implies that $j''(\bar{u})w_k^2\geq 0$ for every $k\in\mathbb{N}$. We now use the convergence $w_k\to w$ in $L^2(\Omega)$ as $k\uparrow\infty$ and the identity \eqref{eq:identity_j} to obtain that $j''(\bar{u})w^2 \geq 0$ for all $w \in C_{\bar{u}}$.
\end{proof}

\subsubsection{Sufficient second-order optimality conditions}

We now establish sufficient second-order optimality conditions \cite[Theorem 2.3]{MR2902693}, \cite[Theorem 23]{MR3586845}.

\begin{theorem}[sufficient second-order optimality conditions]\label{theorem_second_order_optimal}
\EO{Let $(\bar{u},(\bar{\mathbf{y}},\bar{\mathsf{p}}),(\bar{\mathbf{z}},\bar{\mathsf{r}}))$ $\in$ $\mathbb{U}_{ad}\times (\mathbf{H}_0^1(\Omega)\times L_0^2(\Omega)) \times (\mathbf{H}_0^1(\Omega) \times L_0^2(\Omega))$} satisfy the first-order optimality conditions \eqref{eq:brinkamn_problem_state}, \eqref{eq:brinkamn_problem_adjoint}, and \eqref{eq:variational_inequality}. If $j''(\bar{u})v^2> 0$ for all $v\in C_{\bar{u}}\setminus\{0\}$, then there exist $\delta>0$ and $\sigma>0$ such that
\begin{eqnarray}\label{eq:second_order_condition}
j(u)\geq j(\bar{u})+\frac{\sigma}{2}\|u-\bar{u}\|_{L^2(\Omega)}^2\quad\forall u\in\mathbb{U}_{ad}\cap B_{\delta}(\bar{u}),
\end{eqnarray}
where $B_{\delta}(\bar{u})$ denotes the closed ball in $L^2(\Omega)$ of radius $\delta$ and center $\bar{u}$.
\end{theorem}

\begin{proof}
We proceed by contradiction and assume that for every $k\in\mathbb{N}$ we can find an element $u_k\in\mathbb{U}_{ad}$ such that
\begin{align}\label{eq:estimate_contradiction}
\|\bar{u}-u_k\|_{L^2(\Omega)}<k^{-1},
\quad
j(u_k)<j(\bar{u})+(2k)^{-1}\|\bar{u}-u_k\|_{L^2(\Omega)}^2.
\end{align}
For $k\in\mathbb{N}$, we define $\rho_k := \|u_k - \bar{u}\|_{L^2(\Omega)}$ and $w_k:=(u_k-\bar{u})/\rho_k$. It is clear that, for every $k\in\mathbb{N}$, $\|w_k\|_{L^2(\Omega)}=1$. We can therefore deduce the existence of a nonrelabeled subsequence $\{w_k\}_{k\geq 1}\subset L^2(\Omega)$ such that $w_k \rightharpoonup w$ in $L^2(\Omega)$ as $k \uparrow\infty$. In the following, we prove that the limit point $w$ belongs to the cone $C_{\bar{u}}$ and that $w\equiv 0$.

\emph{Step 1:} $w\in C_{\bar{u}}$. We first note that the set of elements satisfying \eqref{eq:cone_condition} is closed and convex in $L^2(\Omega)$. Consequently, this set is weakly sequentially closed in $L^2(\Omega)$, and the weak limit $w$ satisfies \eqref{eq:cone_condition}. Thus, to prove that $w\in C_{\bar{u}}$, we need to verify the remaining condition in \eqref{eq:cone}. To do so, we apply the mean value theorem and the estimate of the right-hand side in \eqref{eq:estimate_contradiction} and obtain
\begin{align}\label{eq:estimate_jprime}
j'(\tilde{u}_k)w_k=\rho_k^{-1}(j(u_k)-j(\bar{u}))< (2k)^{-1}\rho_k\to 0,\quad k\uparrow \infty,
\end{align}
where $\tilde{u}_k=\bar{u}+\theta_k (u_k-\bar{u})$ and $\theta_k\in (0,1)$. We now define $(\tilde{\mathbf{y}}_k,\tilde{\mathsf{p}}_k):=\mathcal{S}\GC{(\tilde{u}_k)}$ and $(\tilde{\mathbf{z}}_k,\tilde{\mathsf{r}}_k)\in\mathbf{H}_0^1(\Omega)\times L_0^2(\Omega)$ as the solution of problem \eqref{eq:brinkamn_problem_adjoint}, where $\mathbf{y}$ and $u$ are replaced by $\tilde{\mathbf{y}}_k$ and $\tilde{u}_k$, respectively. A basic stability bound for the problem that $(\bar{\mathbf{y}}-\tilde{\mathbf{y}}_k,\bar{\mathsf{p}}-\tilde{\mathsf{p}}_k)$ solves combined with the fact that $\tilde{u}_k \to \bar{u}$ in $L^2(\Omega)$ as $k\uparrow \infty$ results in $\tilde{\mathbf{y}}_k \to \bar{\mathbf{y}}$ in $\mathbf{H}_0^1(\Omega)$ as $k\uparrow \infty$. Using similar arguments as above, we can conclude that $\tilde{\mathbf{z}}_k\to \bar{\mathbf{z}}$ in $\mathbf{H}_0^1(\Omega)$ as $k\uparrow \infty$. With these convergence results at hand, we obtain  $\alpha\tilde{u}_k-\tilde{\mathbf{y}}_k\cdot\tilde{\mathbf{z}}_k=:\tilde{\mathfrak{d}}_k\to\bar{\mathfrak{d}} = \alpha\bar{u}-\bar{\mathbf{y}}\cdot\bar{\mathbf{z}}$ in $L^2(\Omega)$ as $k\uparrow \infty$. In fact, we have
$
\|\tilde{\mathfrak{d}}_k-\bar{\mathfrak{d}}\|_{L^2(\Omega)}
\lesssim
\alpha \|\tilde{u}_k-\bar{u}\|_{L^2(\Omega)}
+
\|\nabla\tilde{\mathbf{y}}_k\|_{\mathbf{L}^2(\Omega)}\|\nabla( \tilde{\mathbf{z}}_k-\bar{\mathbf{z}})\|_{\mathbf{L}^2(\Omega)}
+ 
\|\nabla\bar{\mathbf{z}}\|_{\mathbf{L}^2(\Omega)}\|\nabla( \tilde{\mathbf{y}}_k-\bar{\mathbf{y}})\|_{\mathbf{L}^2(\Omega)} \rightarrow 0$ as $k \uparrow \infty$, where we have used $\|\nabla\bar{\mathbf{z}}\|_{\mathbf{L}^2(\Omega)}\lesssim \EO{\|\mathbf{f}\|_{\mathbf{L}^{2}(\Omega)}} + \|\mathbf{y}_\Omega\|_{\mathbf{L}^2(\Omega)}$ and $\|\nabla\tilde{\mathbf{y}}_k\|_{\mathbf{L}^2(\Omega)}\lesssim \EO{\|\mathbf{f}\|_{\mathbf{L}^{2}(\Omega)}}$, which is uniform in $k$. We now use $w_k \rightharpoonup w$ in $L^2(\Omega)$ as $k \uparrow\infty$ and \eqref{eq:estimate_jprime} to arrive at $j'(\bar{u})w = ( \bar{\mathfrak{d}}, w )_{L^2(\Omega)} = \lim_{k\uparrow \infty} ( \tilde{\mathfrak{d}}_k,w_k)_{L^2(\Omega)} =\lim_{k\uparrow \infty}j'(\tilde{u}_k)w_k\leq 0$. On the other hand, the first-order condition \eqref{eq:variational_inequality} shows that $\int_{\Omega}\bar{\mathfrak{d}}(x)w(x)\mathrm{d}x \geq 0$. As a result, we deduce that $\int_{\Omega}|\bar{\mathfrak{d}}(x)w(x)|\mathrm{d}x=\int_{\Omega}\bar{\mathfrak{d}}(x)w(x)\mathrm{d}x=0$. We can thus conclude that $\bar{\mathfrak{d}} \neq 0$ implies that $w=0$ a.e.~in $\Omega$. Therefore, $w\in C_{\bar{u}}$.

\emph{Step 2:} $w\equiv 0$. Using Taylor's theorem, $j'(\bar{u})(u_k-\bar{u})\geq 0$, and the right-hand side estimate in \eqref{eq:estimate_contradiction}, we obtain $j''(\hat{u}_k)w_k^2 < k^{-1}\to 0$ as $k \uparrow \infty$, where $\hat{u}_k:= \bar{u}+\hat{\theta}_k (u_k-\bar{u})$ and $\theta_k \in (0,1).$ We now prove the following result: $j''(\bar{u})w^2\leq \liminf_{k\uparrow \infty} j''(\hat{u}_k)w_k^2$. To do this, we refer to \eqref{eq:identity_j} and write
\begin{equation*}
j''(\hat{u}_k)w_k^2=\alpha\|w_k\|_{L^2(\Omega)}^2-2(w_k\hat{\boldsymbol{\varphi}}_k,\hat{\mathbf{z}}_k)_{\mathbf{L}^2(\Omega)}+\|\hat{\boldsymbol{\varphi}}_k\|_{\mathbf{L}^2(\Omega)}^2.
\end{equation*}
Here, $(\hat{\mathbf{y}}_k,\hat{\mathsf{p}}_k):=\mathcal{S}\GC{(\hat{u}_k)}$, $(\hat{\mathbf{z}}_k,\hat{\mathsf{r}}_k)$ denotes the solution to problem \eqref{eq:brinkamn_problem_adjoint}, where $u$ and $\mathbf{y}$ are replaced by $\hat{u}_k$ and $\hat{\mathbf{y}}_k$, respectively, and $(\hat{\boldsymbol{\varphi}}_k,\hat{\zeta}_k)$ solves \eqref{eq:problem_S_first}, where $u$, $v$ and $\mathbf{y}$ are replaced by $\hat{u}_k$, $w_k$ and $\hat{\mathbf{y}}_k$, respectively. We now note that
\begin{multline*}
\|\nabla(\bar{\mathbf{y}}-\hat{\mathbf{y}}_k)\|_{\mathbf{L}^2(\Omega)}+\|\nabla(\bar{\mathbf{z}}-\hat{\mathbf{z}}_k)\|_{\mathbf{L}^2(\Omega)}
\\
\lesssim
\left( \EO{\|\mathbf{f}\|_{\mathbf{L}^{2}(\Omega)}}
+
\|\mathbf{y}_{\Omega}\|_{\mathbf{L}^2(\Omega)}
\right)
\|\bar{u}-\hat{u}_k\|_{L^2(\Omega)}\to 0,
\quad k\uparrow\infty.
\end{multline*}

We now examine the convergence properties of the sequence $\{ ( \hat{\boldsymbol{\varphi}}_k,\hat{\zeta}_k) \}_{k \geq 1}$. For this purpose, we first write the problem that solves $(\boldsymbol{\varphi}-\hat{\boldsymbol{\varphi}}_k, \zeta-\hat{\zeta}_k)$:
\begin{multline*}
(\nabla (\boldsymbol{\varphi}-\hat{\boldsymbol{\varphi}}_k),\nabla \mathbf{v})_{\mathbf{L}^2(\Omega)}+( \bar{u}(\boldsymbol{\varphi}-\hat{\boldsymbol{\varphi}}_k), \mathbf{v})_{\mathbf{L}^2(\Omega)}-(\zeta-\hat{\zeta}_k,\text{div }\mathbf{v})_{L^2(\Omega)} \\=- ((w-w_k)\bar{\mathbf{y}},\mathbf{v})_{\mathbf{L}^2(\Omega)}+(w_k(\hat{\mathbf{y}}_k-\bar{\mathbf{y}}),\mathbf{v})_{\mathbf{L}^2(\Omega)}+((\hat{u}_k-\bar{u})\hat{\boldsymbol{\varphi}}_k,\mathbf{v})_{\mathbf{L}^2(\Omega)}
=:\text{I}_k+\text{II}_k+\text{III}_k,
\end{multline*}
and $(\mathsf{s},\text{div }(\boldsymbol{\varphi} - \hat{\boldsymbol{\varphi}}_k))_{L^2(\Omega)}  =  0$ for all $\mathbf{v}\in \mathbf{H}_0^1(\Omega)$ and $\mathsf{s}\in L_0^2(\Omega)$, respectively. Since $\bar{\mathbf{y}}\cdot\mathbf{v}\in L^{2}(\Omega)$, the weak convergence $w_k \rightharpoonup w$ in $L^2(\Omega)$ as $k \uparrow\infty$ allows us to obtain that $|\text{I}_k|\to 0$. We control the term $\text{II}_k$ using the Sobolev embedding $\mathbf{H}_0^1(\Omega) \hookrightarrow \mathbf{L}^{4}(\Omega)$:
\begin{equation*}
|\text{II}_k|
\lesssim
\|w_k\|_{L^2(\Omega)}\|\nabla(\bar{\mathbf{y}}-\hat{\mathbf{y}}_k)\|_{\mathbf{L}^{2}(\Omega)}\|\nabla\mathbf{v}\|_{\mathbf{L}^{2}(\Omega)}
\to 0,
\quad k\uparrow\infty.
\end{equation*}
The control of the term $\text{III}_k$ follows similar arguments. In fact, we have
\begin{align*}
|\text{III}_k|\leq
\| \hat{u}_k-\bar{u}\|_{L^{2}(\Omega)}
\|\nabla\hat{\boldsymbol{\varphi}}_k\|_{\mathbf{L}^{2}(\Omega)} \|\nabla\mathbf{v}\|_{\mathbf{L}^{2}(\Omega)}
\lesssim
\| \hat{u}_k-\bar{u}\|_{L^{2}(\Omega)} \|\nabla\mathbf{v}\|_{\mathbf{L}^{2}(\Omega)} \rightarrow 0
\end{align*}
as $k \uparrow \infty$, where we used $\hat{u}_k\to\bar{u}$ in $L^2(\Omega)$ as $k\uparrow\infty$ and $\|\nabla\hat{\boldsymbol{\varphi}}_k\|_{\mathbf{L}^2(\Omega)}\lesssim\|\EO{\mathbf{f}\|_{\mathbf{L}^{2}(\Omega)}}$, which is uniform in $k$. We can therefore conclude that $\hat{\boldsymbol{\varphi}}_k\rightharpoonup\boldsymbol{\varphi}$ in $\mathbf{H}_0^1(\Omega)$ as $k \uparrow \infty$ and that $\hat{\boldsymbol{\varphi}}_k \rightarrow \boldsymbol{\varphi}$ in $\mathbf{L}^r(\Omega)$ for every $r < \infty$ when $d = 2$ and for every $r < 6$ when $d = 3$. In particular, $\|\hat{\boldsymbol{\varphi}}_k\|_{\mathbf{L}^2(\Omega)} \to \|\boldsymbol{\varphi}\|_{\mathbf{L}^2(\Omega)}$ as $k\uparrow\infty$.

The convergence of $\{ (w_k\hat{\boldsymbol{\varphi}}_k,\hat{\mathbf{z}}_k)_{\mathbf{L}^2(\Omega)} \}_{k \geq 1}$ to $(w\boldsymbol{\varphi},\bar{\mathbf{z}})_{\mathbf{L}^2(\Omega)}$ in $\mathbb{R}$ is as follows:
\begin{multline*}
\left| (w\boldsymbol{\varphi},\bar{\mathbf{z}})_{\mathbf{L}^2(\Omega)}-(w_k\hat{\boldsymbol{\varphi}}_k,\hat{\mathbf{z}}_k)_{\mathbf{L}^2(\Omega)} \right|
\leq |((w-w_k)\boldsymbol{\varphi},\bar{\mathbf{z}})_{\mathbf{L}^2(\Omega)}|
\\
+\left| (w_k(\boldsymbol{\varphi}-\hat{\boldsymbol{\varphi}}_k),\bar{\mathbf{z}})_{\mathbf{L}^2(\Omega)} \right|
+
\left| (w_k\hat{\boldsymbol{\varphi}}_k,\bar{\mathbf{z}}-\hat{\mathbf{z}}_k)_{\mathbf{L}^2(\Omega)} \right|
\to 0,
\quad k\uparrow\infty,
\end{multline*}
where we used the weak convergence $w_k \rightharpoonup w$ in $L^2(\Omega)$ as $k\uparrow\infty$, the properties $\boldsymbol{\varphi}\cdot \bar{\mathbf{z}}\in L^2(\Omega)$ and $\|w_k\|_{L^2(\Omega)}=1$, the convergence result $\|\boldsymbol{\varphi}-\hat{\boldsymbol{\varphi}}_k \|_{\mathbf{L}^4(\Omega)}\to 0$ as $k\uparrow\infty$, the bound $\|\nabla \bar{\mathbf{z}}\|_{\mathbf{L}^2(\Omega)}\lesssim \EO{\|\mathbf{f}\|_{\mathbf{L}^{2}(\Omega)}}+\|\mathbf{y}_{\Omega}\|_{\mathbf{L}^2(\Omega)}$, and $\|\nabla(\bar{\mathbf{z}}-\hat{\mathbf{z}}_k)\|_{\mathbf{L}^{2}(\Omega)}\to 0$ as $k\uparrow\infty$.

A collection of the two previous convergence results, namely $(w_k\hat{\boldsymbol{\varphi}}_k,\hat{\mathbf{z}}_k)_{\mathbf{L}^2(\Omega)} \to (w\boldsymbol{\varphi},\bar{\mathbf{z}})_{\mathbf{L}^2(\Omega)}$ in $\mathbb{R}$ and $\|\hat{\boldsymbol{\varphi}}_k\|_{\mathbf{L}^2(\Omega)} \to \|\boldsymbol{\varphi}\|_{\mathbf{L}^2(\Omega)}$ in $\mathbb{R}$ as $k\uparrow\infty$, in conjunction with the fact that the square of $\| \cdot \|_{L^2(\Omega)}$ is weakly lower semicontinuous in $L^2(\Omega)$, show that $j''(\bar{u})w^2 \leq \liminf_{k\uparrow \infty} j''(\hat{u}_k)w_k^2$. Since $\liminf_{k\uparrow \infty} j''(\hat{u}_k)w_k^2 \leq 0$, we conclude that $j''(\bar{u})w^2 \leq 0$. Finally, we use that $w \in C_{\bar u}$ and that $\bar u$ satisfies the second-order condition $j''(\bar u) v^2 > 0$ for all $v \in C_{\bar u} \setminus \{ 0 \}$ to conclude that $w \equiv 0$.

\emph{Step 3:} \emph{The contradiction}. Since $w \equiv 0$, it is immediate that $\hat{\boldsymbol{\varphi}}_k \rightharpoonup 0$ in $\mathbf{H}_0^1(\Omega)$ as $k\uparrow\infty$. We therefore use the limit $\liminf_{k\uparrow \infty} j''(\hat{u}_k)w_k^2 \leq 0$ and the identities $\alpha = \alpha \|w_k\|_{L^2(\Omega)}^2 = j''(\hat{u}_k)w_k^2 + 2(w_k\hat{\boldsymbol{\varphi}}_k,\hat{\mathbf{z}}_k)_{\mathbf{L}^2(\Omega)}-\|\hat{\boldsymbol{\varphi}}_k\|_{\mathbf{L}^{2}(\Omega)}^2$ to conclude that $\alpha\leq 0$. This is a contradiction and concludes the proof.
\end{proof}

We conclude with an equivalent result that is important for deriving error bounds for the numerical methods that we will propose in our paper. To present such a result, we introduce the set 
\begin{equation}
\label{eq:Cutau}
 \EO{C_{\bar{u}}^{\tau}:=\{v \in L^2(\Omega) \text{ satisfying } \eqref{eq:cone_condition} \text{ and } v(x)=0 \text{ if } |\bar{\mathfrak{d}}(x)| > \tau \textrm{~a.e.~in~} \Omega \}.}
\end{equation}

The \EO{aforementioned result is as follows.}

\begin{theorem}[equivalent optimality conditions]\label{theorem_equivalent_condition}
\EO{Let $(\bar{u},(\bar{\mathbf{y}},\bar{\mathsf{p}}),(\bar{\mathbf{z}},\bar{\mathsf{r}}))$ $\in$ $\mathbb{U}_{ad}\times (\mathbf{H}_0^1(\Omega)\times L_0^2(\Omega)) \times (\mathbf{H}_0^1(\Omega) \times L_0^2(\Omega))$} satisfy the first-order optimality conditions \eqref{eq:brinkamn_problem_state}, \eqref{eq:brinkamn_problem_adjoint}, and \eqref{eq:variational_inequality}. Then, the following statements are equivalent:
\begin{align}\label{equivalence_condition}
j''(\bar{u})v^2 >0\quad \forall v \in  C_{\bar{u}}\setminus\{0\} \iff \exists \mu,\tau>0: j''(\bar{u})v^2  \geq \mu\|v\|_{L^2{(\Omega)}}^2 \quad \forall v\in C_{\bar{u}}^{\tau}.
\end{align}
\end{theorem}
\begin{proof}
	The result follows from the arguments developed in the proofs of \cite[Theorem 25]{MR3586845} and Theorem \ref{theorem_second_order_optimal}. For the sake of brevity, we omit the details.
\end{proof}

\section{Finite element approximation of the state equations}\label{fem}
In this section, we present a standard finite element method for approximating the solution of the Stokes--Brinkman equations \eqref{eq:brinkman_problem}. Before describing the numerical method, we introduce some terminology and basic ingredients \cite{MR2373954,CiarletBook,Guermond-Ern}.

We denote by $\mathscr{T} = \{ K\}$ a conforming partition or mesh of $\bar{\Omega}$ into closed simplices $K$ with size $h_K = \text{diam}(K)$. Define $h:=\max \{ h_K: K \in \mathscr{T} \}$. We denote by $\mathbb{T}$ a collection of conforming meshes $\mathscr{T}$, which are refinements of an initial mesh $\mathscr{T}_0$. We assume that the collection $\mathbb{T}$ satisfies the so-called shape regularity condition \cite{MR2373954,CiarletBook,Guermond-Ern}. We define $\mathscr{S}$ as the set of interelement boundaries $\gamma$ of $\T$. For $K \in \T$, we define $\mathscr{S}_K$ as the subset of $\mathscr{S}$ that contains the sides of $K$. We denote by $\mathcal{N}_{\gamma}$, for $\gamma \in \mathscr{S}$, the subset of $\T$ that contains the two elements that have $\gamma$ as a side. Finally, we define the following \emph{stars} or \emph{patches} associated with an element $K \in \T$:
\begin{equation}
\label{eq:patch}
\mathcal{N}_K=  \{K' \in \T: \mathscr{S}_K \cap \mathscr{S}_{K'} \neq \emptyset \},
\qquad 
\mathcal{N}_K^*= \{K' \in \T: K \cap {K'} \neq \emptyset \}.
\end{equation}
In an abuse of notation, we use $\mathcal{N}_K$ and $\mathcal{N}_K^*$ in the following to denote either the sets themselves or the union of their elements.

To perform an a posteriori error analysis, we introduce the following notation. If $K^{+}, K^{-} \in \T$ are such that $K^{+} \neq K^{-}$ and $\partial K^{+} \cap \partial K^{-} = \gamma$, we define the \emph{jump} of the discrete tensor-valued function $\mathbf{w}_{h}$ on $\gamma$ by $\llbracket \mathbf{w}_{h} \cdot \mathbf{n} \rrbracket:= \mathbf{w}_{h} \cdot \mathbf{n}^{+} |^{}_{K^{+}} + \mathbf{w}_{h} \cdot \mathbf{n}^{-}  |^{}_{K^{-}}$, where $\mathbf{n}^{+}$ and $\mathbf{n}^{-}$ denote the unit normals on $\gamma$ pointing towards $K^{+}$ and $K^{-}$, respectively.

\subsection{Finite element spaces}
Let $\mathscr{T}$ be a mesh in $\mathbb{T}$. To approximate the velocity field $\boldsymbol{\mathsf{y}}$ and the pressure $\mathfrak{p}$ of the fluid, we consider a pair $(\mathbf{X}_h,M_h)$ of finite element spaces satisfying a uniform discrete inf-sup condition: There exists a constant $\tilde{\beta}>0$ independent of $h$ such that
\begin{equation}\label{eq:discrete_infsup}
\inf _{\mathsf{q}_h \in M_h} \sup _{\mathbf{v}_h \in \mathbf{X}_h} \frac{ \int_{\Omega} \mathsf{q}_h\operatorname{div} \mathbf{v}_h \mathrm{d}x }{\left\|\nabla\mathbf{v}_h\right\|_{\mathbf{L}^2(\Omega)}\left\|\mathsf{q}_h\right\|_{L^2(\Omega)}} \geqslant \tilde{\beta}>0.
\end{equation}
We will look in particular at the following pairs, which are significant in the literature:

\begin{enumerate}
\item[(a)] The \emph{mini element} \cite[Section 4.2.4]{Guermond-Ern}: In this case,
\begin{eqnarray}\label{ME:vel_space}
\mathbf{X}_h&=&\{\mathbf{v}_{h}\in\mathbf{C}(\bar{\Omega}): \mathbf{v}_{h}|_{K}\in[\mathbb{W}(K)]^{d} \,\,\forall K\in\T \} \cap \mathbf{H}_0^1(\Omega),
\\ 
\label{ME:press_space}
M_h&=&\{\mathsf{q}_{h}\in L_0^2(\Omega)\cap C(\bar{\Omega}): \mathsf{q}_{h}|_{K}\in\mathbb{P}_1(K) \,\,\forall K\in\T \},
\end{eqnarray}
where $\mathbb{W}(K):=\mathbb{P}_1(K)\oplus \mathbb{B}(K)$, and $\mathbb{B}(K)$ denotes the space spanned by a local bubble function.

\item[(b)] The lowest order \emph{Taylor--Hood element} \cite[Section 4.2.5]{Guermond-Ern}: In this scenario,
\begin{eqnarray}\label{TH:vel_space}
\mathbf{X}_h
&=&
\{\mathbf{v}_{h}\in\mathbf{C}(\bar{\Omega}): \mathbf{v}_{h}|_{K}\in[\mathbb{P}_2(K)]^{d} \,\,\forall K\in\T \} \cap \mathbf{H}_0^1(\Omega),\\ \label{TH:press_space}
M_h
&=&
\{\mathsf{q}_{h}\in L_0^2(\Omega)\cap C(\bar{\Omega}): \mathsf{q}_{h}|_{K}\in\mathbb{P}_1(K) \,\,\forall K\in\T \}.
\end{eqnarray}
\end{enumerate}

A proof of the inf-sup condition \eqref{eq:discrete_infsup} for the mini element can be found in \cite[Lemma 4.20]{Guermond-Ern}. Provided that the mesh $\T_h$ contains at least three triangles in two dimensions and that each tetrahedron has at least one internal vertex in three dimensions, a proof of \eqref{eq:discrete_infsup} for the Taylor--Hood element can be found in \cite[Theorem 8.8.1]{MR3097958} and \cite[Theorem 8.8.2]{MR3097958}, respectively.

\subsection{The finite element method}
\EO{Given $\boldsymbol{\mathfrak{f}} \in \mathbf{L}^{2}(\Omega)$ and $\mathfrak{u} \in \mathcal{A}_0$,} we introduce the following approximation of problem \eqref{eq:brinkman_problem}: Find $(\boldsymbol{\mathsf{y}}_{h},\mathfrak{p}_{h})\in \mathbf{X}_h\times M_h$ such that
\begin{equation}\label{eq:modeldiscrete}
\begin{array}{rcl}
(\nabla \boldsymbol{\mathsf{y}}_{h},\nabla \mathbf{v}_{h})_{\mathbf{L}^2(\Omega)}\!+\!( \mathfrak{u}\boldsymbol{\mathsf{y}}_{h}, \mathbf{v}_{h})_{\mathbf{L}^2(\Omega)}\!-\!(\mathfrak{p}_{h},\text{div }\mathbf{v}_{h})_{L^2(\Omega)}& = & \EO{( \boldsymbol{\mathfrak{f}},\mathbf{v}_{h} )_{\mathbf{L}^2(\Omega)}},\\
(\mathsf{q}_{h},\text{div }\boldsymbol{\mathsf{y}}_{h})_{L^2(\Omega)} & = & 0
\end{array}
\end{equation}
for all $\mathbf{v}_{h}\in \mathbf{X}_h$ and $\mathsf{q}_{h}\in M_h$, respectively. The existence and uniqueness of $(\boldsymbol{\mathsf{y}}_{h},\mathfrak{p}_{h})$ $\in \mathbf{X}_h\times M_h$ is standard. In addition, we have $\|\nabla \boldsymbol{\mathsf{y}}_{h}\|_{\mathbf{L}^2(\Omega)} + \|\mathfrak{p}_{h}\|_{L^2(\Omega)}\lesssim \EO{\| \boldsymbol{\mathfrak{f}}\|_{\mathbf{L}^{2}(\Omega)}}$.

In the following, we present a priori error bounds for the approximation \eqref{eq:modeldiscrete}.

\begin{theorem}[error estimates]\label{eq:theorem_error_estimates}
Let $\Omega\subset \mathbb{R}^d$ be a convex polytope, let $\boldsymbol{\mathfrak{f}}\in\mathbf{L}^2(\Omega)$, and let \EO{$\mathfrak{u} \in \mathcal{A}_0$}. Let $(\boldsymbol{\mathsf{y}},\mathfrak{p})\in\mathbf{H}^2(\Omega)\cap\mathbf{H}_0^1(\Omega) \times H^1(\Omega)\cap L_0^2(\Omega)$ be the solution to the problem \eqref{eq:brinkman_problem}, and let $(\boldsymbol{\mathsf{y}}_{h},\mathfrak{p}_{h})\in\mathbf{X}_h\times  M_h$ be its finite element approximation obtained as the solution to \eqref{eq:modeldiscrete}. Then, we have the following error estimate:
\begin{align}\label{eq:error_estimate_Fem}
\|\nabla(\boldsymbol{\mathsf{y}}-\boldsymbol{\mathsf{y}}_{h})\|_{\mathbf{L}^2(\Omega)}+\|\mathfrak{p}-\mathfrak{p}_{h}\|_{L^2(\Omega)}
\lesssim 
h
\|\boldsymbol{\mathfrak{f}}\|_{\mathbf{L}^2(\Omega)}.
\end{align}
The hidden constant is independent of $(\boldsymbol{\mathsf{y}},\mathfrak{p})$, $(\boldsymbol{\mathsf{y}}_{h},\mathfrak{p}_{h})$, and $h$, \EO{but depends on $\| \mathfrak{u} \|_{L^2(\Omega)}$.}
\end{theorem}
\begin{proof}
\EO{A quasi-best approximation result follows directly from \cite[Lemma 2.44]{Guermond-Ern}, with a hidden constant that depends on $ \mathfrak{u} $ as $ 1 + \| \mathfrak{u} \|_{L^2(\Omega)}$.} With this result, the bound \eqref{eq:error_estimate_Fem} follows from the regularity results of Theorem \ref{eq:theorem_regularity} and standard interpolation error estimates. \EO{The hidden constant in \eqref{eq:error_estimate_Fem} depends on $\mathfrak{u}$ as $(1 + \| \mathfrak{u} \|_{L^2(\Omega)})^2$. For brevity, we omit the details of the proof.}
\end{proof}

\begin{remark}[Taylor--Hood approximation]
If the solution $(\mathbf{y},\mathsf{p})$ of problem \eqref{eq:brinkman_problem} belongs to $\mathbf{H}^3(\Omega) \cap\mathbf{H}_0^1(\Omega)\times H^2(\Omega)\cap L_0^2(\Omega)$, then we have the following error bound:
\begin{align*}
\|\nabla(\boldsymbol{\mathsf{y}}-\boldsymbol{\mathsf{y}}_{h})\|_{\mathbf{L}^2(\Omega)}+\|\mathfrak{p}-\mathfrak{p}_{h}\|_{L^2(\Omega)}
\lesssim 
h^2 \left( \|\boldsymbol{\mathsf{y}}\|_{\mathbf{H}^3(\Omega)} + \|\mathfrak{p}\|_{H^2(\Omega)} \right).
\end{align*}
\end{remark}

In \EO{the following, we analyze the differentiability properties of the discrete map $\mathfrak{u} \mapsto (\boldsymbol{\mathsf{y}}_h,\mathfrak{p}_h)$, where $(\boldsymbol{\mathsf{y}}_h,\mathfrak{p}_h)$ corresponds to the solution of the discrete problem \eqref{eq:modeldiscrete}.}

\EO{\begin{theorem}[differentiability properties of $\mathfrak{u} \mapsto (\boldsymbol{\mathsf{y}}_h,\mathfrak{p}_h)$]
\label{theorem_diff_discrete}
For every $h>0$, there exists an open set $\mathcal{A}_h$ in $L^2(\Omega)$ such that $\mathcal{A}_0\subset \mathcal{A}_h\subset \mathcal{A}$, and for all $\mathfrak{u}\in \mathcal{A}_h$ the problem \eqref{eq:modeldiscrete} has a unique solution $(\boldsymbol{\mathsf{y}}_h,\mathfrak{p}_h)\in\mathbf{X}_h\times M_h$. Moreover, the discrete control-to-state map $\mathcal{S}_h:\mathcal{A}_h \rightarrow \mathbf{X}_h\times M_h$ defined by $S_h(\mathfrak{u})=(\boldsymbol{\mathsf{y}}_h,\mathfrak{p}_h)$, is of class $C^2$. Finally, for every $u\in\mathcal{A}_h$ and for every $v\in L^2(\Omega)$, $\mathcal{S}_h'(u)v=(\boldsymbol{\varphi}_h,\zeta_h)$ is the unique solution to the problem
\begin{equation}\label{eq:problem_S_first_discrete}
\begin{array}{rl}
(\nabla \boldsymbol{\varphi}_h,\nabla \mathbf{v}_h)_{\mathbf{L}^2(\Omega)}+( u\boldsymbol{\varphi}_h, \mathbf{v}_h)_{\mathbf{L}^2(\Omega)}-(\zeta_h,\mathrm{div~} \mathbf{v}_h)_{L^2(\Omega)}
=
&
\hspace{-2mm}
-(v\mathbf{y}_h,\mathbf{v}_h)_{\mathbf{L}^2(\Omega)},
\\
(\mathsf{s}_h,\mathrm{div~} \boldsymbol{\varphi}_h)_{L^2(\Omega)}
=
&
\hspace{-2mm}
0
\end{array}
\end{equation}
for all $(\mathbf{v}_h,\mathsf{s}_h) \in \mathbf{X}_h\times M_h$.
\end{theorem}
}

\begin{proof}
The \EO{proof follows from an adaptation of the arguments presented in the proofs of Theorem \ref{theorem:differentiablity_stokes_brinkman} and \cite[Theorem 4.1]{MR4956345}.}

Let \EO{us introduce the operators 
$\mathbb{A}_h: \mathbf{X}_h \rightarrow \mathbf{X}_h'$, 
$\mathbb{B}_h:\mathbf{X}_h \rightarrow M_h'$, 
$\mathbb{C}_h:L^2(\Omega) \times \mathbf{X}_h \rightarrow \mathbf{X}_h'$, 
$\mathbb{F}_h \in \mathbf{X}_h'$, and
$\mathcal{F}_{h}: \mathbf{X}_h\times M_h\times L^2(\Omega) \rightarrow \mathbf{X}_h'\times M_h'$
by
\begin{align*}
& \langle \mathbb{A}_h \boldsymbol{\mathsf{y}}_{h},\mathbf{v}_h\rangle := (\nabla \boldsymbol{\mathsf{y}}_{h},\nabla \mathbf{v}_{h})_{\mathbf{L}^2(\Omega)}
\, \forall \mathbf{v}_h \in \mathbf{X}_h,
\\
& \langle \mathbb{B}_h \mathbf{v}_h,\mathsf{q}_{h}\rangle :=-(\mathsf{q}_h,\mathrm{div~}\mathbf{v}_h)_{L^2(\Omega)}
\, \forall \mathsf{q}_{h} \in M_h,
\\
& \langle\mathbb{C}_h (\mathfrak{u},\boldsymbol{\mathsf{y}}_{h}),\mathbf{v}_h\rangle:= (\mathfrak{u},\boldsymbol{\mathsf{y}}_{h} \cdot \mathbf{v}_h)_{L^2(\Omega)}
\, \forall \mathbf{v}_h \in \mathbf{X}_h,
\quad
\langle \mathbb{F}_h, \mathbf{v}_h \rangle := (\boldsymbol{\mathfrak{f}},\mathbf{v}_{h})_{\mathbf{L}^2(\Omega)} \, \forall \mathbf{v}_h \in \mathbf{X}_h,
\\
& \mathcal{F}_{h}((\boldsymbol{\mathsf{y}}_{h},\GC{\mathfrak{p}_h}),\mathfrak{u}) := [\mathbb{A}_h\boldsymbol{\mathsf{y}}_{h} + \mathbb{C}_h(\mathfrak{u},\boldsymbol{\mathsf{y}}_{h}) + \mathbb{B}_h^T \GC{\mathfrak{p}_h} -\mathbb{F}_h, \mathbb{B}_h \boldsymbol{\mathsf{y}}_{h}]^{\intercal}.
\end{align*}
It follows directly from the definition that $\mathcal{F}_h$ is well-defined and that $\mathcal{F}_h$ is of class $C^2$. We also note that}
\begin{align*}
\dfrac{\partial\mathcal{F}_h}{\partial (\boldsymbol{\mathsf{y}}_{h},\GC{\mathfrak{p}_h})}((\boldsymbol{\mathsf{y}}_{h},\GC{\mathfrak{p}_h}),\mathfrak{u}):
 \mathbf{X}_h\times M_h \rightarrow \mathcal{L}(\mathbf{X}_h\times M_h,\mathbf{X}_h'\times M_h')
 \\
 (\mathbf{z}_h,\mathsf{r}_h) \mapsto [\mathbb{A}_h\mathbf{z}_{h} + \mathbb{C}_h(\mathfrak{u},\mathbf{z}_{h}) + \mathbb{B}_h^T\mathsf{r}_h, \mathbb{B}_h\mathbf{z}_{h}]^{\intercal}.
\end{align*}
%

We \EO{now let $\bar{\mathfrak{u}}\in\mathcal{A}_0$ and define $(\bar{\boldsymbol{\mathsf{y}}}_{h},\bar{\GC{\mathfrak{p}}}_h) := (\boldsymbol{\mathsf{y}}_{h}(\bar{\mathfrak{u}}),\GC{\mathfrak{p}}_h(\bar{\mathfrak{u}}))$, i.e., $(\bar{\boldsymbol{\mathsf{y}}}_{h},\bar{\GC{\mathfrak{p}}}_h)$ is the solution to \eqref{eq:modeldiscrete}, where $\mathfrak{u}$ is replaced $\bar{\mathfrak{u}}$. We note that $\mathcal{F}_h((\bar{\boldsymbol{\mathsf{y}}}_{h},\bar{\GC{\mathfrak{p}}}_h),\bar{\mathfrak{u}})=\mathbf{0}$. We also note that $\partial\mathcal{F}_h / \partial (\boldsymbol{\mathsf{y}}_{h},\GC{\mathfrak{p}}_h) (( \bar{\boldsymbol{\mathsf{y}}}_{h}, \bar{\GC{\mathfrak{p}}}_h), \bar{\mathfrak{u}})$ is an isomorphism. In fact, for every $\mathbb{G}_h \in \mathbf{X}_h'$ and for every $\mathbb{H}_h \in M_h'$, the discrete problem 
\begin{align*}
\mathbb{A}_h\mathbf{z}_{h} + \mathbb{C}_h(\bar{\mathfrak{u}},\mathbf{z}_{h}) + \mathbb{B}_h^T\mathsf{r}_h = \mathbb{G}_h,\\
\mathbb{B}_h\mathbf{z}_{h}= \mathbb{H}_h,
\end{align*}
has a unique solution $(\mathbf{z}_{h},\mathsf{r}_h)\in\mathbf{X}_h\times M_h$; see \cite[Section 4.4]{Guermond-Ern} and \cite[Theorem 3.2.1]{MR3097958}. Here, we have used that $\bar{\mathfrak{u}} \in \mathcal{A}_0$. 
We are thus in a position to apply the implicit function theorem to deduce the existence of open neighborhoods 
$B_{\epsilon_{\bar{\mathfrak{u}}},h}(\bar{\mathfrak{u}}) \subset L^2(\Omega)$, $B_{\epsilon_{\bar{\boldsymbol{\mathsf{y}}}},h}(\bar{\boldsymbol{\mathsf{y}}}_{h})\subset \mathbf{X}_h$, and $B_{\epsilon_{\bar{\mathfrak{p}}},h}(\bar{\mathfrak{p}}_h) \subset M_h$ 
such that for every $\mathfrak{u} \in B_{\epsilon_{\bar{\mathfrak{u}}},h}(\bar{\mathfrak{u}})$ the equation $\mathcal{F}_h((\boldsymbol{\mathsf{y}}_{h},\mathfrak{p}_h),\mathfrak{u})=\mathbf{0}$ has a unique solution $(\boldsymbol{\mathsf{y}}_{h},\mathfrak{p}_h)\in B_{\epsilon_{\bar{\boldsymbol{\mathsf{y}}}},h}(\bar{\boldsymbol{\mathsf{y}}}_{h}) \times B_{\epsilon_{\bar{\mathfrak{p}}},h}(\bar{\mathfrak{p}}_h)$. 
We now proceed as in the proof of Theorem \ref{theorem:differentiablity_stokes_brinkman} to deduce that uniqueness of $(\boldsymbol{\mathsf{y}}_{h},\mathfrak{p}_h)$ in $\mathbf{X}_h \times M_h$ can be guaranteed if $\epsilon_{\bar{\mathfrak{u}},h} < C_{4 \hookrightarrow 2}^{-2}$. With this result at hand, we can thus define
\begin{equation}
\label{eq:mathcalS_h}
\mathcal{S}_h: \mathcal{A}_h \rightarrow \mathbf{X}_h \times M_h,
\qquad
\mathcal{A}_h:= \bigcup \{ B_{\epsilon_{\bar{\mathfrak{u}}},h} (\bar{\mathfrak{u}}): \bar{\mathfrak{u}} \in \mathcal{A}_0 \}\bigcap \mathcal{A},
\qquad
\epsilon_{\bar{\mathfrak{u}},h} < C_{4 \hookrightarrow 2}^{-2}.
\end{equation}
For $u \in \mathcal{A}_h$, the equation $\mathcal{F}_h((\boldsymbol{\mathsf{y}}_{h},\mathfrak{p}_h),\mathfrak{u}) = \mathbf{0}$ has a unique solution $(\boldsymbol{\mathsf{y}}_{h},\mathfrak{p}_h) = \mathcal{S}_h(u)$ in $\mathbf{X}_h\times M_h$, and the map $\mathcal{S}_h$ defined in \eqref{eq:mathcalS_h} is of class $C^2$. In addition, for $v\in L^2(\Omega)$, we have that $\mathcal{S}_h'(u)v=(\boldsymbol{\varphi}_h,\zeta_h)$ corresponds to the unique solution of problem \eqref{eq:problem_S_first_discrete}. This completes the proof.}
\end{proof}

\subsection{A posteriori error estimates}

Following \cite{MR1445736,MR1259620}, we present an a posteriori error estimator for the approximation \eqref{eq:modeldiscrete} of problem \eqref{eq:brinkman_problem}: For $K\in \T$ and an internal side $\gamma\in\mathscr{S}$, we define the \emph{element residual} $\mathcal{R}_K$ and the \emph{interelement residual} $\mathcal{J}_{\gamma}$ as
\begin{align}\label{eq:residuals}
\mathcal{R}_K:= (\boldsymbol{\mathfrak{f}}+\Delta \boldsymbol{\mathsf{y}}_{h}\!-\mathfrak{u}\boldsymbol{\mathsf{y}}_{h}\!-\!\nabla\mathfrak{p}_{h})|_K,
\qquad 
\mathcal{J}_{\gamma}:= \llbracket(\nabla \boldsymbol{\mathsf{y}}_{h}-\mathfrak{p}_{h}\mathbf{I})\cdot \mathbf{n}\rrbracket,
\end{align}
respectively. Here, $(\boldsymbol{\mathsf{y}}_{h},\mathfrak{p}_{h})$ denotes the solution to the discrete problem \eqref{eq:modeldiscrete} and $\mathbf{I}$ denotes the identity matrix in $\mathbb{R}^{d\times d}$. We thus define the \emph{local error indicators} $\mathcal{E}_{K}$ and the corresponding a posteriori error estimator $\mathcal{E}_{h}$ as follows:
\begin{align}\label{eq:indicator_estimator_e}
\mathcal{E}_{K}^2:=h_K^2  \|\mathcal{R}_K\|_{\mathbf{L}^2(K)}^2
+h_K  \|\mathcal{J}_{\gamma} \|_{\mathbf{L}^2(\partial K\setminus\partial \Omega)}^2+\|\text{div }\boldsymbol{\mathsf{y}}_{h}\|_{L^2(K)}^2,
\quad 
\mathcal{E}_{h}^2:= \sum_{K\in\T}\mathcal{E}_{K}^2.
\end{align}

We present the following global reliability result.

\begin{theorem}[global reliability]
\label{eq:reliability} 
\EO{Let $\boldsymbol{\mathfrak{f}}\in\mathbf{L}^2(\Omega)$, and let $\mathfrak{u} \in \mathcal{A}_0$.} Let $(\boldsymbol{\mathsf{y}},\mathfrak{p})\in\mathbf{H}_0^1(\Omega)\times L_0^2(\Omega)$ be the solution to problem \eqref{eq:brinkman_problem}, and let $(\boldsymbol{\mathsf{y}}_{h},\mathfrak{p}_{h})\in \mathbf{X}_h\times M_h$ be its finite element approximation obtained as the solution to \eqref{eq:modeldiscrete}. Then, we have
\begin{eqnarray}\label{eq:bound_realiability}
\|\nabla(\boldsymbol{\mathsf{y}}-\boldsymbol{\mathsf{y}}_{h})\|_{\mathbf{L}^2(\Omega)}^2+\|\mathfrak{p}-\mathfrak{p}_{h}\|_{L^2(\Omega)}^2\lesssim \mathcal{E}_{h}^2,
\end{eqnarray}
The hidden constant is independent of $(\boldsymbol{\mathsf{y}},\mathfrak{p})$ and $(\boldsymbol{\mathsf{y}}_{h},\mathfrak{p}_{h})$, the size of the elements in the mesh $\T$, and $\# \T$, \EO{but depends on $\| \mathfrak{u} \|_{L^2(\Omega)}$.}
\end{theorem}
\begin{proof}
The proof follows the same arguments as in \cite[Section 5]{MR1259620} and \cite[Section 3]{MR1445736}. \EO{The hidden constant depends on $\| \mathfrak{u} \|_{L^2(\Omega)}$ as $(1 + \| \mathfrak{u} \|_{L^2(\Omega)})^4$.} For the sake of brevity, we omit the details.
\end{proof}

\section{Finite element discretizations for the control problem}
We propose two discretization methods to approximate solutions of the control problem \eqref{eq:min_functional}--\eqref{eq:brinkamn_problem_state}: a semidiscrete method in which the control set is not discretized, and a fully discrete one in which the admissible control set is discretized with piecewise constant functions.

\subsection{The fully discrete scheme}
\label{fully_discrete_framework}
To present the fully discrete scheme, we introduce
$
\mathbb{U}_h:=\{u_{h}\in \GC{L^{2}(\Omega)}:u_{h}|_{K}\in\mathbb{P}_0(K)~\forall K\in\T\}
$
and $\mathbb{U}_{ad,h}:= \mathbb{U}_h\cap \mathbb{U}_{ad}$. The scheme is as follows: Find $\min J(\mathbf{y}_{h},u_{h})$ subject to the \emph{discrete state equations}
\begin{equation}\label{eq:brinkamn_problem_state_fd}
\begin{array}{rl}
(\nabla \mathbf{y}_{h},\nabla \mathbf{v}_{h})_{\mathbf{L}^2(\Omega)}
+
( u_{h}\mathbf{y}_{h}, \mathbf{v}_{h})_{\mathbf{L}^2(\Omega)}
-
(\mathsf{p}_{h},\text{div }\mathbf{v}_{h})_{L^2(\Omega)}
=& \EO{( \mathbf{f},\mathbf{v}_{h} )_{\mathbf{L}^2(\Omega)}},\\
(\mathsf{q}_{h},\text{div }\mathbf{y}_{h})_{L^2(\Omega)} =&  0,
\end{array}
\end{equation}
for all $(\mathbf{v}_{h},\mathsf{q}_{h}) \in \mathbf{X}_h \times M_h$, and the \emph{discrete control constraints} $u_{h}\in\mathbb{U}_{ad,h}$.

The scheme admits at least one solution. To present first-order optimality conditions, we introduce the discrete map \EO{$\mathcal{G}_h: \mathcal{A}_h \ni u \mapsto \mathbf{y}_h \in \mathbf{X}_h$, where $\mathbf{y}_h$ corresponds to the velocity component of the pair $(\mathbf{y}_h,\mathsf{p}_h) = \mathcal{S}_h(u)$ that solves problem \eqref{eq:modeldiscrete}. We recall that $\mathcal{S}_h$ is defined in \eqref{eq:mathcalS_h}.} We also introduce 
\[
\EO{j_h: \mathcal{A}_h \rightarrow \mathbb{R},}
\qquad
j_h(u):=J(\mathcal{G}_h\GC{(u)},u).
\]
\EO{We note that the results of Theorem \ref{theorem_diff_discrete} guarantee that $j_h$ is of class $C^2$. In addition, for every $u \in \mathcal{A}_h$ and every $v \in L^2(\Omega)$, we have}
\[
 \EO{j_h'(u) v = (\alpha u - \mathbf{y}_{h} \cdot \mathbf{z}_{h}, v)_{L^2(\Omega)},}
\]
\EO{where $(\mathbf{y}_h,\mathsf{p}_h) = \mathcal{S}_h(u)$ and $(\mathbf{z}_{h},\mathsf{r}_{h}) \in \mathbf{X}_h \times M_h$ corresponds to the solution of the \emph{discrete adjoint equations}: Find $(\mathbf{z}_{h},\mathsf{r}_{h})\in \mathbf{X}_h\times M_h$ such that
\begin{multline}
 \label{eq:brinkman_problem_adjoint_fully}
(\nabla \mathbf{v}_{h},\nabla \mathbf{z}_{h})_{\mathbf{L}^2(\Omega)}
+
(u, \mathbf{z}_{h} \cdot \mathbf{v}_{h})_{\mathbf{L}^2(\Omega)}
+
(\mathsf{r}_{h},\text{div }\mathbf{v}_{h})_{L^2(\Omega)}
= (\mathbf{y}_{h}
-
\mathbf{y}_{\Omega},\mathbf{v}_{h})_{\mathbf{L}^2(\Omega)},
\\
(\mathsf{s}_{h},\text{div }\mathbf{z}_{h})_{L^2(\Omega)} =  0 \qquad \forall (\mathbf{v}_{h},\mathsf{s}_{h})\in \mathbf{X}_h\times M_h.
\end{multline}
After introducing these ingredients, we can derive first-order optimality conditions for the fully discrete scheme: If $\bar{u}_{h}$ is locally optimal for the fully discrete scheme, then}
\begin{eqnarray}\label{eq:variational_inequality_discrete}
j'_h(\bar{u}_h)(u_h-\bar{u}_h)=(\alpha \bar{u}_{h}-\bar{\mathbf{y}}_{h}\cdot \bar{\mathbf{z}}_{h},u_{h}-\bar{u}_{h})_{L^2(\Omega)}\geq 0\quad\forall u_{h}\in \mathbb{U}_{ad,h}.
\end{eqnarray}
Here, \EO{$\bar{\mathbf{y}}_h = \mathcal{G}_h (\bar{u}_{h})$ and $(\bar{\mathbf{z}}_{h},\bar{\mathsf{r}}_{h}) \in \mathbf{X}_h \times M_h$ corresponds to the solution of \eqref{eq:brinkman_problem_adjoint_fully}, where $u$ is replaced by $\bar{u}_h$.}

\subsection{The semidiscrete scheme}
\label{semi_discrete_framework}
\label{sectionsemi-discrete}
\EO{In this section, we propose a semidiscrete scheme based on the so-called variational discretization approach \cite{MR2122182}.} In this scheme, only the state space $\mathbf{H}_0^1(\Omega) \times L_0^2(\Omega)$ is discretized; $\mathbb{U}_{ad}$ is not discretized. The semidiscrete scheme reads as follows: Find $\min J(\mathbf{y}_{h},\mathsf{u})$ subject to the \emph{discrete state equations}
\begin{multline}\label{eq:brinkamn_problem_state_sd}
(\nabla \mathbf{y}_{h},\nabla \mathbf{v}_{h})_{\mathbf{L}^2(\Omega)}
+
( \mathsf{u}\mathbf{y}_{h}, \mathbf{v}_{h})_{\mathbf{L}^2(\Omega)}
-
(\mathsf{p}_{h},\text{div }\mathbf{v}_{h})_{L^2(\Omega)} =  \EO{( \mathbf{f},\mathbf{v}_{h} )_{\mathbf{L}^2(\Omega)}},
\\
(\mathsf{q}_{h},\text{div }\mathbf{y}_{h})_{L^2(\Omega)} =  0 \qquad \forall (\mathbf{v}_{h},\mathsf{q}_{h}) \in \mathbf{X}_h \times M_h,
\end{multline}
and the \emph{control constraints} $\mathsf{u}\in\mathbb{U}_{ad}$. The existence of an optimal solution follows from standard arguments. Moreover, if $\bar{\mathsf{u}}\in\mathbb{U}_{ad}$ denotes a locally optimal control, then
\begin{eqnarray}\label{eq:variational_inequality_semi-discrete}
j'_h(\bar{\mathsf{u}})(u-\bar{\mathsf{u}})=(\alpha \bar{\mathsf{u}}-\bar{\mathbf{y}}_{h}\cdot \bar{\mathbf{z}}_{h},u-\bar{\mathsf{u}})_{L^2(\Omega)}\geq 0\quad\forall u\in \mathbb{U}_{ad}.
\end{eqnarray}
Here, $\bar{\mathbf{y}}_h = \mathcal{G}_h \GC{(\bar{\mathsf{u}})}$ and $(\bar{\mathbf{z}}_{h},\bar{\mathsf{r}}_{h})\in \mathbf{X}_h\times M_h $ is the solution to the following problem:
\begin{multline}\label{eq:brinkman_problem_adjoint_semi}
(\nabla \mathbf{v}_{h},\nabla \bar{\mathbf{z}}_{h})_{\mathbf{L}^2(\Omega)}
+
(\bar{\mathsf{u}}, \bar{\mathbf{z}}_{h} \cdot \mathbf{v}_{h})_{\mathbf{L}^2(\Omega)}
+
(\bar{\mathsf{r}}_{h},\text{div }\mathbf{v}_{h})_{L^2(\Omega)} =  (\bar{\mathbf{y}}_{h}-\mathbf{y}_{\Omega},\mathbf{v}_{h})_{\mathbf{L}^2(\Omega)},
\\
(\mathsf{s}_{h},\text{div }\bar{\mathbf{z}}_{h})_{L^2(\Omega)}  =  0 \qquad \forall (\mathbf{v}_{h},\mathsf{s}_{h})\in \mathbf{X}_h\times M_h.
\end{multline}

We \EO{note that in view of the variational inequality \eqref{eq:variational_inequality_semi-discrete}, the following projection formula holds: $\bar{\mathsf{u}}(x) = \Pi_{[\mathsf{a},\mathsf{b}]}(\alpha^{-1} \bar{\mathbf{y}}_{h} (x) \cdot \bar{\mathbf{z}}_{h}(x))$. The semidiscrete scheme thus induces a discretization of an optimal control $\bar{\mathsf{u}}$ by projecting $\alpha^{-1} \bar{\mathbf{y}}_{h}\cdot \bar{\mathbf{z}}_{h}$ under $\Pi_{[\mathsf{a},\mathsf{b}]}$.}

\subsection{Convergence of discretizations}

In \EO{this section, we analyze convergence properties of
the fully discrete and semidiscrete methods.} To simplify the presentation, we define
\begin{equation}
\label{eq:energy_norm}
 \| (\mathbf{v},\mathsf{q}) \|_{\Omega} := \|\nabla \mathbf{v} \|_{\mathbf{L}^2(\Omega)}+\|\mathsf{q}\|_{L^2(\Omega)},
 \quad
 (\mathbf{v},\mathsf{q}) \in \mathbf{H}_0^1(\Omega) \times L^2(\Omega).
\end{equation}

\subsubsection{Discretization of the state equations}
We have the following result.

\begin{proposition}[convergence and error bound]\label{theorem_convergence_01}
Let $\Omega \subset \mathbb{R}^d$ be a Lipschitz polytope, and \EO{let $\mathbf{f} \in \mathbf{L}^{2}(\Omega)$}. Let $(\mathbf{y},\mathsf{p})$ and $(\mathbf{y}_{h},\mathsf{p}_{h})$ be the unique solutions of problems \eqref{eq:brinkamn_problem_state} and \eqref{eq:brinkamn_problem_state_fd}, respectively. \EO{If the sequence $\{ u_{h} \}_{h>0}$ is such that
$u_h \in \mathbb{U}_{ad,h}$ and $u_{h} \rightharpoonup u$ in $L^2(\Omega)$ as $h \rightarrow 0$, with $u \in \mathbb{U}_{ad}$, then we have that}
\begin{align}\label{eq:convergence_prop}
u_{h}\rightharpoonup u
\textrm{ in }
L^2(\Omega) 
\implies 
\EO{(\mathbf{y}_{h},\mathsf{p}_h) 
\to 
(\mathbf{y},\mathsf{p})
\textrm{ in }
\mathbf{H}_0^1(\Omega) \times L^2(\Omega)} \textrm{ as } h \rightarrow 0.
\end{align}
If, in addition, $\Omega$ is convex, then we have the error bound
\begin{align}\label{eq:errorbound}
\|\nabla(\mathbf{y}-\mathbf{y}_{h})\|_{\mathbf{L}^2(\Omega)}+\|\mathsf{p}-\mathsf{p}_{h}\|_{L^2(\Omega)}\lesssim h \|\mathbf{f}\|_{\mathbf{L}^2(\Omega)}
+
\EO{\|\mathbf{f}\|_{\mathbf{L}^2(\Omega)}} \|u-u_{h}\|_{L^2(\Omega)}.
\end{align}
\end{proposition}
\begin{proof}

We divide the proof into two steps.

\emph{Step 1:} \emph{The property \eqref{eq:convergence_prop}.} Let $(\mathbf{y}_{h}(u),\mathsf{p}_h(u))\in \mathbf{X}_h \times M_h$ be the solution to \eqref{eq:brinkamn_problem_state_fd}, where $u_{h}$ is replaced by $u$. \EO{We recall that $(\mathbf{y}_{h},\mathsf{p}_h) = (\mathbf{y}_{h}(u_h),\mathsf{p}_h(u_h))\in \mathbf{X}_h \times M_h$ is the solution to \eqref{eq:brinkamn_problem_state_fd}.} A quasi-best approximation property, which follows from \cite[Lemma 2.44]{Guermond-Ern}, combined with a density argument, as the one developed in \cite[Corollary 1.109]{Guermond-Ern}, show that
$
\|(\mathbf{y}-\mathbf{y}_{h}(u),\mathsf{p}-\mathsf{p}_h(u))\|_{\Omega}
\rightarrow 0
$
as 
$h \rightarrow 0$. To control $\|(\mathbf{y}_{h}(u)-\mathbf{y}_{h},\mathsf{p}_h(u)-\mathsf{p}_h)\|_{\Omega}$, we note that $(\mathbf{y}_{h}(u)-\mathbf{y}_{h},\mathsf{p}_{h}(u)-\mathsf{p}_{h}) \in \mathbf{X}_h \times M_h$ solves
\begin{multline}
\label{eq:Stokes_convergence}
(\nabla (\mathbf{y}_h(u)-\mathbf{y}_h),\nabla \mathbf{v}_h)_{\mathbf{L}^2(\Omega)}
+
( u_h(\mathbf{y}_h(u)-\mathbf{y}_h), \mathbf{v}_h)_{\mathbf{L}^2(\Omega)}
\\
-
(\mathsf{p}_h(u)-\mathsf{p}_h,\text{div }\mathbf{v}_h)_{L^2(\Omega)} =((u_{h}-u)\mathbf{y}_h(u),\mathbf{v}_h)_{\mathbf{L}^2(\Omega)},
\quad
(\mathsf{q}_h,\text{div}(\mathbf{y}_h(u)-\mathbf{y}_h))_{L^2(\Omega)}  =  0
\end{multline}
for all $(\mathbf{v}_{h},\mathsf{q}_{h}) \in \mathbf{X}_h  \times M_h$. \EO{If we set $\mathbf{v}_{h}=\mathbf{y}_h(u)-\mathbf{y}_h$ and $\mathsf{q}_{h}=0$, we obtain 
\[
 \int_{\Omega} | \nabla(\mathbf{y}_{h}(u)-\mathbf{y}_{h}) |^2\mathrm{d}x 
 + 
 \int_{\Omega} u_h (\mathbf{y}_{h}(u)-\mathbf{y}_{h})^2\mathrm{d}x = \int_{\Omega} (u_h - u)\mathbf{y}_{h}(u) (\mathbf{y}_{h}(u)-\mathbf{y}_{h}) \mathrm{d}x.
\]
Since $\{ u_h \}_{h >0} \subset \mathbb{U}_{ad,h} \subset \mathbb{U}_{ad} \subset \mathcal{A}_0$, we can thus derive that $\|\nabla(\mathbf{y}_{h}(u)-\mathbf{y}_{h})\|_{\mathbf{L}^2(\Omega)} \leq\| (u_{h}-u)\mathbf{y}_h(u) \|_{\mathbf{H}^{-1}(\Omega)}$. We now use the convergence property $ \| \nabla (\mathbf{y}-\mathbf{y}_{h}(u))\|_{\mathbf{L}^2(\Omega)} \rightarrow 0 $ as $h \rightarrow 0$
and $\mathbf{H}_0^1(\Omega) \hookrightarrow \mathbf{L}^4(\Omega)$ to conclude that $\{ \mathbf{y}_h (u) \cdot \mathbf{v} \}_{h>0}$ converges in $L^2(\Omega)$ as $h \rightarrow 0$, for a given $\mathbf{v} \in \mathbf{H}_0^1(\Omega)$. From this and $u_{h}\rightharpoonup u$ in $L^2(\Omega)$, it follows that}
\begin{equation}
\|\nabla(\mathbf{y}_{h}(u)-\mathbf{y}_{h})\|_{\mathbf{L}^2(\Omega)} \lesssim \| (u_{h}-u)\mathbf{y}_h(u) \|_{\mathbf{H}^{-1}(\Omega)} \rightarrow 0,
\quad h \rightarrow 0.
\label{eq:convergence_aux}
\end{equation}
\EO{A simple application of the triangle inequality thus shows that $\mathbf{y}_{h} \rightarrow \mathbf{y}$ in $\mathbf{H}_0^1(\Omega)$ as $h \rightarrow 0$. We now prove that $\mathsf{p}_h(u) - \mathsf{p}_h \rightarrow 0$ in $L^2(\Omega)$ as $h \rightarrow 0$. To do this, we use the dicrete inf-sup condition \eqref{eq:discrete_infsup}, the fact that $(\mathbf{y}_{h}(u)-\mathbf{y}_{h},\mathsf{p}_{h}(u)-\mathsf{p}_{h}) \in \mathbf{X}_h \times M_h$ solves problem \eqref{eq:Stokes_convergence}, and the convergence result \eqref{eq:convergence_aux}. The desired convergence result $\mathsf{p}_{h} \rightarrow  \mathsf{p}$ in $L^2(\Omega)$ as $h \rightarrow 0$ thus follows from the triangle inequality.}

\emph{Step 2:} \emph{The bound \eqref{eq:errorbound}}. Let $(\mathbf{y}(u_{h}),\mathsf{p}(u_{h}))\in\mathbf{H}_0^1(\Omega)\times L_0^2(\Omega)$ be the solution \EO{to} \eqref{eq:brinkamn_problem_state}, where $u$ is replaced by $u_{h}$. The control of $\| (\mathbf{y}(u_{h})-\mathbf{y}_{h}, \mathsf{p}(u_h)-\mathsf{p}_h) \|_{\Omega}$ follows directly from Theorem \ref{eq:theorem_error_estimates}:
$
\| (\mathbf{y}(u_{h})-\mathbf{y}_{h}, \mathsf{p}(u_h)-\mathsf{p}_h ) \|_{\Omega} \lesssim h\|\mathbf{f}\|_{\mathbf{L}^2(\Omega)}.
$
\EO{We note that here we have used the fact that the sequence $\{ u_h \}_{h>0}$ is such that $u_h \in \mathbb{U}_{ad,h} \subset \mathbb{U}_{ad} \subset \mathcal{A}_0$ and is uniformly bounded in $L^2(\Omega)$ by a constant that depends only on $\Omega$ and the control bounds.} In the following, we bound $\|\nabla(\mathbf{y}-\mathbf{y}(u_{h}))\|_{\mathbf{L}^2(\Omega)}$. To this end, we note that $(\mathbf{y}-\mathbf{y}(u_h),\mathsf{p}-\mathsf{p}(u_h))$
solves
\begin{multline*}
(\nabla (\mathbf{y}-\mathbf{y}(u_{h})),\nabla \mathbf{v})_{\mathbf{L}^2(\Omega)}+( u(\mathbf{y}-\mathbf{y}(u_{h})), \mathbf{v})_{\mathbf{L}^2(\Omega)}-(\mathsf{p}-\mathsf{p}(u_{h}),\text{div }\mathbf{v})_{L^2(\Omega)} \\=((u_{h}-u)\mathbf{y}(u_{h}),\mathbf{v})_{\mathbf{L}^2(\Omega)},
\qquad 
(\mathsf{q},\text{div }(\mathbf{y}-\mathbf{y}(u_{h})))_{L^2(\Omega)}  =  0
\end{multline*}
for all $(\mathbf{v},\mathsf{q}) \in \mathbf{H}_0^1(\Omega) \times L_0^2(\Omega)$. If we set $\mathbf{v} = \mathbf{y}-\mathbf{y}(u_{h})$ and $\mathsf{q}=0$, \EO{and use that $u \in \mathbb{U}_{ad} \subset \mathcal{A}_0$}, we obtain $\|\nabla(\mathbf{y}-\mathbf{y}(u_{h}))\|_{\mathbf{L}^2(\Omega)}\lesssim \|(u-u_{h})\mathbf{y}(u_{h})\|_{\mathbf{H}^{-1}(\Omega)} \lesssim \| \nabla \mathbf{y}(u_{h}) \|_{\mathbf{L}^2(\Omega)} \|u-u_{h}\|_{L^{2}(\Omega)}$, and as a result
\begin{align}\label{eq:error_aux_est02}
\|\nabla(\mathbf{y}-\mathbf{y}_{h})\|_{\mathbf{L}^2(\Omega)}\lesssim h \|\mathbf{f}\|_{\mathbf{L}^2(\Omega)} + \|\mathbf{f}\|_{\mathbf{L}^2(\Omega)} \|u-u_{h}\|_{L^2(\Omega)}.
\end{align}
To bound $\|\mathsf{p}-\mathsf{p}(u_h)\|_{L^2(\Omega)}$, we use the inf--sup condition \eqref{eq:infsup} and the problem that $(\mathbf{y}-\mathbf{y}(u_h),\mathsf{p}-\mathsf{p}(u_h))\in\mathbf{H}_0^1(\Omega)\times L_0^2(\Omega)$ solves. In fact, we have
\begin{multline*}
\label{eq:error_aux_est03}
\|\mathsf{p}-\mathsf{p}(u_h)\|_{L^2(\Omega)}
\lesssim 
\EO{(1 + \| u \|_{L^2(\Omega)})} \|\nabla(\mathbf{y}-\mathbf{y}(u_h))\|_{\mathbf{L}^2(\Omega)}
\\
+ \EO{\|(u-u_{h})\mathbf{y}(u_{h})\|_{\mathbf{H}^{-1}(\Omega)}}
\lesssim
\|\mathbf{f}\|_{\mathbf{L}^2(\Omega)} \|u-u_{h}\|_{L^2(\Omega)},
\end{multline*}
where we have used the bound $\|\nabla(\mathbf{y}-\mathbf{y}(u_{h}))\|_{\mathbf{L}^2(\Omega)} \lesssim \|\mathbf{f}\|_{\mathbf{L}^2(\Omega)} \|u-u_{h}\|_{L^{2}(\Omega)}$ \EO{and the fact that $\| u \|_{L^2(\Omega)}$ can be controlled in terms of $\Omega$ and the control bounds.}
\end{proof}

\subsubsection{\EO{Discretization of the adjoint equations}}
Before deriving error bounds, we introduce auxiliary variables that are \EO{particularly important} for our analysis. Let $(\boldsymbol{\mathfrak{z}},\mathfrak{r})\in \mathbf{H}_0^1(\Omega)\times L_0^2(\Omega)$ be the solution of
\begin{multline}\label{eq:aux_problem_adj_1}
(\nabla \mathbf{v},\nabla \boldsymbol{\mathfrak{z}})_{\mathbf{L}^2(\Omega)}+( u_{h}\boldsymbol{\mathfrak{z}}, \mathbf{v})_{\mathbf{L}^2(\Omega)}+(\mathfrak{r},\text{div }\mathbf{v})_{L^2(\Omega)} =  	(\mathbf{y}_{h}-\mathbf{y}_{\Omega},\mathbf{v})_{\mathbf{L}^2(\Omega)},
\\
(\mathsf{q},\text{div }\boldsymbol{\mathfrak{z}})_{L^2(\Omega)}  =  0 \qquad \forall(\mathbf{v},\mathsf{q})\in \mathbf{H}_0^1(\Omega)\times L_0^2(\Omega).
\end{multline}
We also define $(\mathbf{z}_{h},\mathsf{r}_{h})\in \mathbf{X}_h\times M_h$ as the solution to
\begin{multline}\label{eq:aux_problem_adj_2}
(\nabla \mathbf{v}_{h},\nabla \mathbf{z}_{h})_{\mathbf{L}^2(\Omega)}+( u_{h}\mathbf{z}_{h}, \mathbf{v}_{h})_{\mathbf{L}^2(\Omega)}+(\mathsf{r}_{h},\text{div }\mathbf{v}_{h})_{L^2(\Omega)} \\=  	(\mathbf{y}_{h}-\mathbf{y}_{\Omega},\mathbf{v}_{h})_{\mathbf{L}^2(\Omega)},
\quad
(\mathsf{q}_{h},\text{div }\mathbf{z}_h)_{L^2(\Omega)}  =  0 \qquad \forall (\mathbf{v}_{h},\mathsf{q}_{h})\in \mathbf{X}_h\times M_h.
\end{multline}

\begin{proposition}[error bound]
\label{eq:theorem_adjoint_aux_estimates}
Let $\Omega \subset \mathbb{R}^d$ be a convex polytope. Let $\mathbf{f},\mathbf{y}_{\Omega}\in \mathbf{L}^2(\Omega)$, and let $(\mathbf{z},\mathsf{r})$ and $(\mathbf{z}_{h},\mathsf{r}_{h})$ be the solutions of \eqref{eq:brinkamn_problem_adjoint} and \eqref{eq:aux_problem_adj_2}, respectively. Then,
\begin{multline}\label{eq:errorbound_adj}
\|\nabla(\mathbf{z}-\mathbf{z}_{h})\|_{\mathbf{L}^2(\Omega)}+\|\mathsf{r}-\mathsf{r}_{h}\|_{L^2(\Omega)}
\lesssim 
h [ \|\mathbf{f}\|_{\mathbf{L}^2(\Omega)}+\|\mathbf{y}_{\Omega}\|_{\mathbf{L}^2(\Omega)}] 
\\
+ \|u-u_{h}\|_{L^2(\Omega)}[ \|\mathbf{f}\|_{\mathbf{L}^{2}(\Omega)} +\|\mathbf{y}_{\Omega}\|_{\mathbf{L}^2(\Omega)}].
\end{multline}
\end{proposition}
\begin{proof}
We begin the proof by controlling the error $\| (\boldsymbol{\mathfrak{z}} - \mathbf{z}_{h}, \mathfrak{r}-\mathsf{r}_{h}) \|_{\Omega}$. To do this, we use the error bound \eqref{eq:error_estimate_Fem} and also that $\| \nabla \mathbf{y}_{h}\|_{\mathbf{L}^2(\Omega)} \lesssim \| \mathbf{f} \|_{\GC{\mathbf{L}^{2}(\Omega)}}$.
In fact, we have
\begin{align}\label{eq:bound_aux_conv1}
\|\nabla(\boldsymbol{\mathfrak{z}}-\mathbf{z}_{h})\|_{\mathbf{L}^2(\Omega)}+\|\mathfrak{r}-\mathsf{r}_{h}\|_{L^2(\Omega)}
\lesssim 
h( \| \mathbf{f} \|_{\GC{\mathbf{L}^{2}(\Omega)}} + \|\mathbf{y}_{\Omega}\|_{\mathbf{L}^2(\Omega)}).
\end{align}
Let us now note that $(\mathbf{z}- \boldsymbol{\mathfrak{z}},\mathsf{r}-\mathfrak{r})\in\mathbf{H}_0^1(\Omega)\times L_0^2(\Omega)$ solves the following problem:
\begin{multline*}
(\nabla (\mathbf{z}- \boldsymbol{\mathfrak{z}}),\nabla \mathbf{v})_{\mathbf{L}^2(\Omega)}+( u(\mathbf{z}- \boldsymbol{\mathfrak{z}}), \mathbf{v})_{\mathbf{L}^2(\Omega)}+(\mathsf{r}-\mathfrak{r},\text{div }\mathbf{v})_{L^2(\Omega)} \\=((u_{h}-u)\boldsymbol{\mathfrak{z}},\mathbf{v})_{\mathbf{L}^2(\Omega)}+(\mathbf{y}-\mathbf{y}_{h},\mathbf{v})_{\mathbf{L}^2(\Omega)},\quad (\mathsf{q},\text{div }(\mathbf{z}- \boldsymbol{\mathfrak{z}}))_{L^2(\Omega)}  =  0,
\end{multline*}
for all $(\mathbf{v},\mathsf{q}) \in \mathbf{H}_0^1(\Omega) \times L_0^2(\Omega)$. If we set $\mathbf{v} = \mathbf{z}- \boldsymbol{\mathfrak{z}}$ and $\mathsf{q}=0$, 
\EO{and use that $u \in \mathcal{A}_0$ and the error bound \eqref{eq:errorbound},} we obtain
\begin{equation}
\label{eq:bound_aux_conv2}
\begin{aligned}
\|\nabla(\mathbf{z}- \boldsymbol{\mathfrak{z}})\|_{\mathbf{L}^2(\Omega)}
& \lesssim  \|u-u_{h}\|_{L^2(\Omega)}\|\nabla\boldsymbol{\mathfrak{z}}\|_{\mathbf{L}^{2}(\Omega)}+\|\nabla (\mathbf{y}-\mathbf{y}_{h})\|_{\mathbf{L}^2(\Omega)}\\
& \lesssim  \|u-u_{h}\|_{L^2(\Omega)}[\GC{\|\mathbf{f}\|_{\mathbf{L}^{2}(\Omega)} +\|\mathbf{y}_{\Omega}\|_{\mathbf{L}^2(\Omega)}}] +h\|\mathbf{f}\|_{\mathbf{L}^2(\Omega)}.
 \end{aligned}
\end{equation}
The bound for $\|\mathsf{r}-\mathfrak{r}\|_{L^2(\Omega)}$ follows from the inf-sup condition \eqref{eq:infsup}, \eqref{eq:errorbound}, and \eqref{eq:bound_aux_conv2}:
\begin{equation}\label{eq:bound_aux_conv3}
\|\mathsf{r}-\mathfrak{r}\|_{L^2(\Omega)}\lesssim 
\|u-u_{h}\|_{L^2(\Omega)}[ \GC{ \|\mathbf{f}\|_{\mathbf{L}^{2}(\Omega)} +\|\mathbf{y}_{\Omega}\|_{\mathbf{L}^2(\Omega)}}] +h\|\mathbf{f}\|_{\mathbf{L}^2(\Omega)}.\end{equation}
\EO{Here, we have used that $\| u \|_{L^2(\Omega)}$ can be controlled in terms of $\Omega$ and the control bounds.} The bound \eqref{eq:errorbound_adj} follows from \eqref{eq:bound_aux_conv1}, \eqref{eq:bound_aux_conv2}, and \eqref{eq:bound_aux_conv3}. This concludes the proof.
\end{proof}

\subsubsection{Convergence results for the fully discrete scheme}

In this section, we derive two convergence results for the fully discrete scheme. The first result shows that a sequence of discrete global solutions contains a subsequence that converges to a global solution of the continuous problem \eqref{eq:min_functional}--\eqref{eq:brinkamn_problem_state} as $h \rightarrow 0$. The second result establishes that strict local solutions of the optimal control problem \eqref{eq:min_functional}--\eqref{eq:brinkamn_problem_state} can be approximated by local solutions of the fully discrete optimal control problems.

\begin{theorem}[convergence of global solutions]
\label{theorem_convergence_global_fully}
Let $\mathbf{f}$ and $\mathbf{y}_{\Omega}$ be in $\mathbf{L}^2(\Omega)$. Let $h>0$, and let $\bar{u}_{h}\in \mathbb{U}_{ad,h}$ be a global solution of the fully discrete scheme. \GC{Then, there exist nonrelabeled subsequences of $\{\bar{u}_{h}\}_{h>0}$ such that $\bar{u}_{h}\rightharpoonup \bar{u}$ in $L^{2}(\Omega)$ as $h \rightarrow 0$, where $\bar{u}$ is a global solution of \eqref{eq:min_functional}--\eqref{eq:brinkamn_problem_state}. Moreover, we have}
\begin{align}\label{eq:convergence_globalsolutions_fully}
\lim_{h \rightarrow 0} \|\bar{u} - \bar{u}_h\|_{L^2(\Omega)}=0,
\qquad \lim_{h \rightarrow 0} j_h(\bar{u}_{h})=j(\bar{u}).
\end{align}
\end{theorem}
\begin{proof}
	We begin the proof by noting that, since $\{\bar{u}_{h}\}_{h>0}\subset \mathbb{U}_{ad,h}$ is uniformly bounded in \GC{$L^{2}(\Omega)$}, we can extract a nonrelabeled subsequence such that \GC{$\bar{u}_{h}\rightharpoonup \bar{u}$ in $L^{2}(\Omega)$} as $h\rightarrow 0$. We must now prove that $\bar{u}$ is optimal for \eqref{eq:min_functional}--\eqref{eq:brinkamn_problem_state} and that the limits in \eqref{eq:convergence_globalsolutions_fully} hold. We divide the proof into two steps.
	
	\emph{Step 1:} \emph{$\bar{u}\in \mathbb{U}_{ad}$ is a global solution of problem \eqref{eq:min_functional}--\eqref{eq:brinkamn_problem_state}}. Let $\mathfrak{u}\in \mathbb{U}_{ad}$ be a global solution of \eqref{eq:min_functional}--\eqref{eq:brinkamn_problem_state}. Let $\mathsf{P}_h: L^2(\Omega)\to \mathbb{U}_h$ be the orthogonal projection operator and define $\mathfrak{u}_{h}:=\mathsf{P}_h(\mathfrak{u})$. Since $\mathfrak{u}\in \mathbb{U}_{ad}$, it is clear that $\mathfrak{u}_{h}\in \mathbb{U}_{ad,h}$. On the other hand, the results of Theorem \ref{thm:regularity_optimal_control} guarantee that $\mathfrak{u}\in H^1(\Omega)$. Consequently, 
	$\|\mathfrak{u}-\mathfrak{u}_{h}\|_{L^2(\Omega)} \rightarrow 0$ as $h \rightarrow 0$. We now use the global optimality of $\mathfrak{u}$, Proposition \ref{theorem_convergence_01}, the global optimality of $\bar{u}_{h}$, and the strong convergence $\mathfrak{u}_{h}\to\mathfrak{u}$ in $L^2(\Omega)$ as $h \rightarrow 0$ to obtain $j(\mathfrak{u})\leq j(\bar{u})\leq \liminf_{h\to 0} j_h (\bar{u}_{h}) \leq \limsup_{h\to 0}j_h (\bar{u}_{h})\leq \limsup_{h \to 0}j_h (\mathfrak{u}_{h})=j(\mathfrak{u})$. As a result, $\bar{u}$ is a global solution of \eqref{eq:min_functional}--\eqref{eq:brinkamn_problem_state} and $j_h (\bar{u}_{h}) \rightarrow j(\bar{u})$ as $h \rightarrow 0$.

\emph{Step 2:} \emph{$\| \bar{u} - \bar{u}_h \|_{L^2(\Omega)}\to 0$ as $h \rightarrow 0$.} In view of the results of Proposition \ref{theorem_convergence_01}, we have that $\bar{\mathbf{y}}_h\to \bar{\mathbf{y}}$ in $\mathbf{H}_0^1(\Omega)$ as $h \rightarrow 0$. This, the fact that $j_h(\bar{u}_{h})\to j(\bar{u})$ as $h\rightarrow 0$, and the weak convergence $\bar{u}_{h}\rightharpoonup \bar{u}$ in $L^2(\Omega)$ as $h\rightarrow 0$ allow us to conclude that $\{\bar{u}_h\}_{h>0}$ converges to $\bar{u}$ in $L^2(\Omega)$ as $h\rightarrow 0$. This completes the proof.
\end{proof}

\begin{theorem}[convergence of local solutions]\label{theorem_convergence_local_fully}
Let $\mathbf{f}$ and $\mathbf{y}_{\Omega}$ be in $\mathbf{L}^2(\Omega)$. Let $\bar{u} \in \mathbb{U}_{ad}$ be a strict local minimum of \eqref{eq:min_functional}--\eqref{eq:brinkamn_problem_state}. Then, there exists a sequence $\{\bar{u}_{h}\}_{0 < h \leq h_{\Box}}$ of local minima of the fully discrete scheme such that \eqref{eq:convergence_globalsolutions_fully} holds.
\end{theorem}

\begin{proof}
Since $\bar{u}$ is a strict local minimum of \eqref{eq:min_functional}--\eqref{eq:brinkamn_problem_state}, there exists $\epsilon>0$, so that the problem: Find
$
\min \{j(u) \,| \, u \in \mathbb{U}_{ad}: \|\bar{u}-u\|_{L^2(\Omega)}\leq\epsilon\},
$
admits $\bar{u}$ as the unique solution. We now introduce the discrete problem for $h>0$: Find
$
\min \{j_h(u_{h}) \,| \, u_{h}\in \mathbb{U}_{ad,h} : \|\bar{u}-u_{h}\|_{L^2(\Omega)}\leq\epsilon\}.
$
In the following, we show that this problem admits at least one solution. To this end, we must prove that the set in which the minimum is sought is nonempty. Since $\mathsf{P}_h(\bar{u}) \in \mathbb{U}_{ad,h}$ and $\|\bar{u}-\mathsf{P}_h(\bar{u})\|_{L^2(\Omega)} \to 0$ as $h \rightarrow 0$, there exists $h_{\epsilon}>0$ so that for every $h \leq h_{\epsilon}$ it holds that $\|\bar{u}-\mathsf{P}_h(\bar{u})\|_{L^2(\Omega)}\leq\epsilon$. Consequently, the set $\{ u_{h}\in \mathbb{U}_{ad,h}: \|\bar{u}-u_{h}\|_{L^2(\Omega)}\leq\epsilon\}$ is nonempty for every $h \leq h_{\epsilon}$.

Let $h \leq h_{\epsilon}$, and let $\bar{u}_{h}$ be a global solution to the discrete problem mentioned above. Applying the same arguments as the ones developed in the proof of Theorem \ref{theorem_convergence_global_fully} we obtain the existence of a nonrelabeled subsequence of $\{\bar{u}_{h}\}_{0 < h \leq h_{\epsilon}}$ such that it converges strongly in $L^2(\Omega)$ to a solution of the problem: Find
$
\min \{j(u) \,| \, u \in \mathbb{U}_{ad}: \|\bar{u}-u\|_{L^2(\Omega)}\leq\epsilon\}.
$ 
Since this problem admits $\bar{u}$ as the unique solution, we have that $\bar{u}_{h} \to \bar{u}$ in $L^2(\Omega)$ as $h \rightarrow 0$ for the whole sequence. This guarantees that the constraint $\|\bar{u}-\bar{u}_{h}\|_{L^2(\Omega)}\leq \epsilon$ is not active for $h$ sufficiently small and thus that $\bar{u}_{h}$ solves the fully discrete scheme and \eqref{eq:convergence_globalsolutions_fully} holds. This concludes the proof.
\end{proof}

%
%
%

\subsubsection{Convergence results for the semidiscrete scheme}

As \EO{with the fully discrete scheme, we have the following convergence results. The proofs of these results follow from minor adjustments to the proofs of the Theorems \ref{theorem_convergence_global_fully} and \ref{theorem_convergence_local_fully}.}

\begin{theorem}[\EO{convergence of global solutions}]
\GC{Let $\mathbf{f}$ and $\mathbf{y}_{\Omega}$ be in $\mathbf{L}^2(\Omega)$. Let $h>0$, and let $\bar{\mathsf{u}}_{h}\in \mathbb{U}_{ad}$ be a global solution of the semidiscrete scheme. Then, there exist nonrelabeled subsequences of $\{\bar{\mathsf{u}}_{h}\}_{h>0}$ such that $\bar{\mathsf{u}}_{h}\rightharpoonup \bar{u}$ in $L^{2}(\Omega)$ as $h \rightarrow 0$, where $\bar{u}$ corresponds to a global solution of the optimal control problem \eqref{eq:min_functional}--\eqref{eq:brinkamn_problem_state}. Moreover, we have $\|\bar{u}-\bar{\mathsf{u}}_{h}\|_{L^2(\Omega)} \rightarrow 0$ and $j_h(\bar{\mathsf{u}}_{h}) \rightarrow j(\bar{u})$ as $h \rightarrow 0$.}
\end{theorem}

\begin{theorem}[convergence of local solutions]
\GC{Let $\mathbf{f}$ and $\mathbf{y}_{\Omega}$ be in $\mathbf{L}^2(\Omega)$. Let $\bar{u}\in\mathbb{U}_{ad}$ be a strict local minimum of \eqref{eq:min_functional}--\eqref{eq:brinkamn_problem_state}. Then, there exists a sequence $\{\bar{\mathsf{u}}_{h}\}_{0 < h \leq h_{\circ}}$ of local minima of the semidiscrete scheme such that $\|\bar{u}-\bar{\mathsf{u}}_{h}\|_{L^2(\Omega)} \rightarrow 0$ and $j_h(\bar{\mathsf{u}}_{h}) \rightarrow j(\bar{u})$ as $h \rightarrow 0$.}
\end{theorem}

\section{A priori error bounds}
The main goal of this section is to derive error bounds for the fully discrete and semidiscrete methods presented in \S \ref{fully_discrete_framework} and \S \ref{semi_discrete_framework}, respectively.

\subsection{The fully discrete scheme}
Let $\{\bar{u}_{h}\}_{h>0}$ be a sequence of local minima of the fully discrete scheme such that $\bar{u}_{h} \to \bar{u}$ in $L^2(\Omega)$ as $h\rightarrow 0$, where $\bar{u}$ corresponds to a local solution of \eqref{eq:min_functional}--\eqref{eq:brinkamn_problem_state}. \GC{In this section, we derive the following error bound:}
\begin{equation}\label{eq:error_estimate_apriori_fully}
\GC{\|\bar{u}-\bar{u}_{h}\|_{L^2(\Omega)}\lesssim h \qquad\forall h\in(0, h_{\dagger}).}
\end{equation}

We begin our analysis with the following instrumental result.

\begin{lemma}[auxiliary error estimate]\label{lemma_auxiliary_error_fd}
\EO{Let $\Omega \subset \mathbb{R}^d$ be a convex polytope, and let $\boldsymbol{\mathbf{f}}$ and $\mathbf{y}_{\Omega}$ be in $\mathbf{L}^2(\Omega)$.} Let us assume that $\bar{u}\in \mathbb{U}_{ad}$ satisfies the second-order optimality conditions \eqref{equivalence_condition}. If \eqref{eq:error_estimate_apriori_fully} is false, then there exists $h_{\star}>0$ such that
\begin{align}\label{eq:auxiliary_error_estimate_fd}
\mathfrak{C}\|\bar{u}-\bar{u}_{h}\|_{L^2(\Omega)}^2\leq [j'(\bar{u}_{h})-j'(\bar{u})](\bar{u}_{h}-\bar{u})
 \quad
\forall h<h_{\star},
\quad
\mathfrak{C}:=2^{-1}\min\{\mu,\alpha\}.
\end{align}
Here, $\alpha$ is the control cost, and $\mu$ denotes the constant appearing in \eqref{equivalence_condition}.
\end{lemma}
\begin{proof}
\EO{We proceed by contradiction as in} \cite[Section 7]{MR2272157}: Since \eqref{eq:error_estimate_apriori_fully} is false, there exists a subsequence $\{h_k\}_{k\in\mathbb{N}}\subset \mathbb{R}^{+}$ such that
\begin{align}\label{eq:auxuliary_error_estimate_fd}
\lim_{h_{k} \rightarrow 0} \|\bar{u}-\bar{u}_{h_k}\|_{L^2(\Omega)} = 0,
\qquad \lim_{h_{k} \rightarrow 0} h_k^{-1} \|\bar{u}-\bar{u}_{h_k}\|_{L^2(\Omega)}=+\infty.
\end{align}
In the following, we omit the subindex $k$ to simplify the notation and denote $\bar{u}_{h_k}=\bar{u}_{h}$. We note that $h \rightarrow 0$ as $k \uparrow \infty$. Define $w_h:=(\bar{u}_h-\bar{u})/\|\bar{u}_h-\bar{u}\|_{L^2(\Omega)}$ for each $h>0$ and note that $\{ w_h \}_{h>0}$ is uniformly bounded in $L^2(\Omega)$. Therefore, there exists a nonrelabeled subsequence such that $w_h \rightharpoonup w$ in $L^2(\Omega)$ as $h\rightarrow 0$. The rest of the proof is divided into two steps.

\emph{Step 1:} \emph{$w\in C_{\bar{u}}$}. Let us first recall that $C_{\bar{u}}$ is defined in \eqref{eq:cone}. Since $\bar{u}_{h} \in \mathbb{U}_{ad,h}$ for any $h>0$, we have that $w_h$ satisfies the sign conditions in \eqref{eq:cone_condition}. Consequently, the weak limit $w$ also satisfies \eqref{eq:cone_condition}. In what follows, we prove that $\bar{\mathfrak{d}}(x) \neq 0$ implies that $w(x)=0$ for a.e.~$x\in\Omega$. To this end, we introduce the variable $\bar{\mathfrak{d}}_h:=\alpha \bar{u}_h-\bar{\mathbf{y}}_h\cdot\bar{\mathbf{z}}_h$ and recall that $\bar{\mathfrak{d}}=\alpha \bar{u}-\bar{\mathbf{y}}\cdot\bar{\mathbf{z}}$. We now control $\|\bar{\mathfrak{d}} - \bar{\mathfrak{d}}_h\|_{L^2(\Omega)}$ in view of H\"older's inequality, the bounds \eqref{eq:errorbound_adj} and \eqref{eq:errorbound}, and the property $\|\bar{u}-\bar{u}_h\|_{L^2(\Omega)}\to 0$ as \EO{$h\rightarrow 0$:
\begin{multline*}
\|\bar{\mathfrak{d}}-\bar{\mathfrak{d}}_h\|_{L^2(\Omega)}\lesssim \|\bar{u}-\bar{u}_h\|_{L^2(\Omega)}+\|\nabla\bar{\mathbf{y}}_h\|_{\mathbf{L}^2(\Omega)}\|\nabla( \bar{\mathbf{z}}-\bar{\mathbf{z}}_h)\|_{\mathbf{L}^2(\Omega)}
+ \|\nabla\bar{\mathbf{z}}\|_{\mathbf{L}^2(\Omega)}
\\
\cdot \|\nabla( \bar{\mathbf{y}}-\bar{\mathbf{y}}_h)\|_{\mathbf{L}^2(\Omega)}\lesssim 
(1 + \|\mathbf{f}\|_{\mathbf{L}^2(\Omega)}[\|\mathbf{f}\|_{\mathbf{L}^2(\Omega)} + \|\mathbf{y}_{\Omega} \|_{\mathbf{L}^2(\Omega)}])
\|\bar{u}-\bar{u}_h\|_{L^2(\Omega)} 
\\
+ h \|\mathbf{f}\|_{\mathbf{L}^2(\Omega)}[\|\mathbf{f}\|_{\mathbf{L}^2(\Omega)} + \| \mathbf{y}_{\Omega} \|_{\mathbf{L}^2(\Omega)}]
\to 0
\end{multline*}
as $h \rightarrow 0$.} We have thus obtained that $\bar{\mathfrak{d}}_h \rightarrow \bar{\mathfrak{d}}$ in $L^2(\Omega)$ as $h \rightarrow 0$. Since $w_h \rightharpoonup w$ in $L^2(\Omega)$ as $h\rightarrow 0$, we have
\begin{eqnarray*}
\int_{\Omega}\bar{\mathfrak{d}}(x)w(x)\mathrm{d}x=
\lim_{h\rightarrow 0} \frac{1}{\|\bar{u}_h-\bar{u}\|_{L^2(\Omega)}}
\left[\int_{\Omega}\bar{\mathfrak{d}}_h(\mathsf{P}_h(\bar{u})-\bar{u})\mathrm{d}x+\int_{\Omega}\bar{\mathfrak{d}}_h(\bar{u}_h-\mathsf{P}_h(\bar{u}))\mathrm{d}x\right],
\end{eqnarray*}
where $\mathsf{P}_h: L^2(\Omega) \rightarrow \mathbb{U}_h$ denotes the $L^2$-orthogonal projection operator. Since $\mathsf{P}_h(\bar{u})\in \mathbb{U}_{ad,h}$, the discrete variational inequality \eqref{eq:variational_inequality_discrete} yields $(\bar{\mathfrak{d}}_h,\bar{u}_h-\mathsf{P}_h(\bar{u}))_{L^2(\Omega)}\leq 0$. Let us now note that $\|\bar{\mathfrak{d}}_h\|_{L^2(\Omega)}\leq\|\bar{\mathfrak{d}}_h-\bar{\mathfrak{d}}\|_{L^2(\Omega)}+\|\bar{\mathfrak{d}}\|_{L^2(\Omega)}\lesssim 1$ for every $h \leq \mathfrak{h}$, which implies that $\{\|\bar{\mathfrak{d}}_h\|_{L^2(\Omega)}\}_{0 < h \leq \mathfrak{h}}$ is uniformly bounded in $\mathbb{R}$. Thus, 
\begin{eqnarray*}
\int_{\Omega}\bar{\mathfrak{d}}(x)w(x)\mathrm{d}x
\leq
\lim_{h \rightarrow 0}
\frac{1}{\|\bar{u}_h-\bar{u}\|_{L^2(\Omega)}}
 \int_{\Omega}\bar{\mathfrak{d}}_h(\mathsf{P}_h(\bar{u})-\bar{u})\mathrm{d}x
\lesssim \lim_{h \rightarrow 0}\dfrac{\|\mathsf{P}_h(\bar{u})-\bar{u}\|_{L^2(\Omega)}}{\|\bar{u}_h-\bar{u}\|_{L^2(\Omega)}}=0.
\end{eqnarray*}
To obtain the last equality, we used the regularity property $\bar{u} \in H^1(\Omega)$ (cf.~Theorem \ref{thm:regularity_optimal_control}), the error bound $\| \bar{u} - \mathsf{P}_h(\bar{u}) \|_{L^2(\Omega)} \lesssim h | \bar{u} |_{H^1(\Omega)}$, and the right-hand side limit in \eqref{eq:auxuliary_error_estimate_fd}. On the other hand, since $w$ satisfies the sign conditions in \eqref{eq:cone_condition} and $\bar u$ satisfies the \EO{projection} formula \eqref{eq:representation}, we obtain that $\bar{\mathfrak{d}}(x)w(x)\geq 0$ for a.e.~$x\in\Omega$. Therefore, $\int_{\Omega}|\bar{\mathfrak{d}}(x)w(x)|\mathrm{d}x = \int_{\Omega}\bar{\mathfrak{d}}(x)w(x)\mathrm{d}x \leq 0$. Consequently, if $\bar{\mathfrak{d}}(x)\neq 0$, then $w(x)=0$ for a.e.~$x\in\Omega$. This proves that $w\in C_{\bar{u}}$.

\emph{Step 2:} \emph{The error bound \eqref{eq:auxiliary_error_estimate_fd}.} We apply the mean value theorem to obtain 
\begin{align}\label{eq:mean_value_theorem}
[j'(\bar{u}_h)-j'(\bar{u})](\bar{u}_h-\bar{u})=j''(\hat{u}_h)(\bar{u}_h-\bar{u})^2,\quad \hat{u}_h:=\bar{u}+\theta_h (\bar{u}_h - \bar{u}),
\end{align}
where $\theta_h \in (0,1)$. Let $(\mathbf{y}(\hat{u}_h),\mathsf{p}(\hat{u}_h))\in \mathbf{H}_0^1(\Omega)\times L_0^2(\Omega)$ be the solution to  \eqref{eq:brinkamn_problem_state}, where $u$ is replaced by $\hat{u}_h$, and let $(\mathbf{z}(\hat{u}_h),\mathsf{r}(\hat{u}_h))\in \mathbf{H}_0^1(\Omega)\times L_0^2(\Omega)$ be the solution to \eqref{eq:brinkamn_problem_adjoint}, where $\mathbf{y}$ and $u$ are replaced by $\mathbf{y}(\hat{u}_h)$ and $\hat{u}_h$, respectively.  We now note that $(\bar{\mathbf{y}}-\mathbf{y}(\hat{u}_h),\bar{\mathsf{p}}-\mathsf{p}(\hat{u}_h))\in\mathbf{H}_0^1(\Omega)\times L_0^2(\Omega)$ solves the following problem:
 \begin{multline*}
(\nabla (\bar{\mathbf{y}}-\mathbf{y}(\hat{u}_{h})),\nabla \mathbf{v})_{\mathbf{L}^2(\Omega)}+( \bar{u}(\bar{\mathbf{y}}-\mathbf{y}(\hat{u}_{h})), \mathbf{v})_{\mathbf{L}^2(\Omega)}-(\bar{\mathsf{p}}-\mathsf{p}(\hat{u}_{h}),\text{div }\mathbf{v})_{L^2(\Omega)} \\=((\hat{u}_{h}-\bar{u})\mathbf{y}(\hat{u}_{h}),\mathbf{v})_{\mathbf{L}^2(\Omega)},\quad (\mathsf{q},\text{div }(\bar{\mathbf{y}}-\mathbf{y}(\hat{u}_{h})))_{L^2(\Omega)}  =  0,
\end{multline*}
for all $(\mathbf{v},\mathsf{q}) \in \mathbf{H}_0^1(\Omega) \times L_0^2(\Omega)$. If we set $\mathbf{v} = \bar{\mathbf{y}}-\mathbf{y}(\hat{u}_{h})$ and $\mathsf{q}=0$, \EO{and use that $\bar{u} \in \mathbb{U}_{ad} \subset \mathcal{A}_0$,} we obtain the bound $\|\nabla(\bar{\mathbf{y}}-\mathbf{y}(\hat{u}_{h}))\|_{\mathbf{L}^2(\Omega)}\lesssim \EO{\| \mathbf{f}\|_{\mathbf{L}^2(\Omega)}} \|\bar{u}-\hat{u}_{h}\|_{L^2(\Omega)}$. Since $\bar{u}_h \to \bar{u}$ in $L^2(\Omega)$ as $h \rightarrow 0$, we can conclude that $\mathbf{y}(\hat{u}_h)\to \bar{\mathbf{y}}$ in $\mathbf{H}_0^1(\Omega)$ and that $\mathsf{p}(\hat{u}_h)\to\bar{\mathsf{p}}$ in $L_0^2(\Omega)$ as $h \rightarrow 0$; the latter follows from the inf-sup condition \eqref{eq:infsup}. Similarly, we can conclude that $\mathbf{z}(\hat{u}_h)\to \bar{\mathbf{z}}$ in $\mathbf{H}_0^1(\Omega)$ and that $\mathsf{r}(\hat{u}_h) \to \bar{\mathsf{r}}$ in $L_0^2(\Omega)$ as $h \rightarrow 0$. Let us now define $(\boldsymbol{\varphi}(w_h),\zeta(w_h))\in \mathbf{H}_0^1(\Omega)\times L_0^2(\Omega)$ as the solution to \eqref{eq:problem_S_first}, where $u$, $\mathbf{y}$, and $v$ are replaced by $\hat{u}_h$, $\mathbf{y}(\hat{u}_{h})$, and $w_h$, respectively. If we apply the arguments from step 2 of the proof of Theorem \ref{theorem_second_order_optimal}, we obtain that $w_h\rightharpoonup w$ in $L^2(\Omega)$ as $h \rightarrow 0$ implies that $\boldsymbol{\varphi}(w_h)\rightharpoonup \boldsymbol{\varphi}$ in $\mathbf{H}_0^1(\Omega)$. 
We thus use the identity \eqref{eq:identity_j}, the definition of $w_h$, and the second-order equivalent conditions in \eqref{equivalence_condition} to arrive at
\begin{multline}
\label{eq:second_derivatives_to_take_the_limit}
\lim_{h\to 0}j''(\hat{u}_h)w_h^2
=
\lim_{h\to 0}
\left[
\alpha\|w_h\|_{L^2(\Omega)}^2-2(w_h\boldsymbol{\varphi}(w_h),\mathbf{z}(\hat{u}_h))_{\mathbf{L}^2(\Omega)}+\|\boldsymbol{\varphi}(w_h)\|_{\mathbf{L}^2(\Omega)}^2
\right]
\\
=\alpha-2(w\boldsymbol{\varphi},\bar{\mathbf{z}})_{\mathbf{L}^2(\Omega)}+\|\boldsymbol{\varphi}\|_{\mathbf{L}^2(\Omega)}^2=\alpha+j''(\bar{u})w^2-\alpha \|w\|_{L^2(\Omega)}^2\geq \alpha + (\mu-\alpha) \|w\|_{L^2(\Omega)}^2.
\end{multline}
We note that we have also used \GC{H\"{o}lder's inequality, \EO{the weak convergence $w_{h} \rightharpoonup w$ in $L^{2}(\Omega)$ as $h \rightarrow 0$}, the fact that by definition $\|w_h\|_{L^2(\Omega)} = 1$ for $h>0$, the Sobolev embedding $\mathbf{H}_0^1(\Omega)\hookrightarrow\mathbf{L}^4(\Omega)$, which is compact, and the fact that} $\boldsymbol{\varphi}(w_h)\rightarrow \boldsymbol{\varphi}$ in $\mathbf{L}^4(\Omega)$ and $\mathbf{z}(\hat{u}_h)\to \bar{\mathbf{z}}$ in $\mathbf{H}_0^1(\Omega)$ as $h \rightarrow 0$. \EO{In particular, the convergence of the sequence $\{ (w_h\boldsymbol{\varphi}(w_h),\mathbf{z}(\hat{u}_h))_{\mathbf{L}^2(\Omega)} \}_h$ to $(w\boldsymbol{\varphi},\bar{\mathbf{z}})_{\mathbf{L}^2(\Omega)}$ can be deduced as follows. First, we write
\begin{multline*}
 (w\boldsymbol{\varphi},\bar{\mathbf{z}})_{\mathbf{L}^2(\Omega)}
 =
 \int_{\Omega} (w - w_h) \boldsymbol{\varphi} \cdot \bar{\mathbf{z}} \mathrm{d}x + \int_{\Omega} w_h (\boldsymbol{\varphi}- \boldsymbol{\varphi}(w_h)) \cdot \bar{\mathbf{z}} \mathrm{d}x
 \\
 + \int_{\Omega} w_h \boldsymbol{\varphi}(w_h) \cdot (\bar{\mathbf{z}} - \mathbf{z}(\hat{u}_h)) \mathrm{d}x
 + (w_h\boldsymbol{\varphi}(w_h),\mathbf{z}(\hat{u}_h))_{\mathbf{L}^2(\Omega)} =: \mathfrak{I}_h + \mathfrak{J}_h + \mathfrak{K}_h + \mathfrak{L}_h.
\end{multline*}
Second, $\mathfrak{I}_h \rightarrow 0$ as $h \rightarrow 0$, $\mathfrak{J}_h \leq C_{4 \hookrightarrow 2} \|  \boldsymbol{\varphi}- \boldsymbol{\varphi}(w_h) \|_{\mathbf{L}^4(\Omega)} \| \nabla \bar{\mathbf{z}}\|_{\mathbf{L}^2(\Omega)} \rightarrow 0$ as $h \rightarrow 0$, and $\mathfrak{K}_h \leq C_{4 \hookrightarrow 2}^2  \| \nabla \boldsymbol{\varphi}(w_h) \|_{\mathbf{L}^2(\Omega)} \| \nabla (\bar{\mathbf{z}} - \mathbf{z}(\hat{u}_h))\|_{\mathbf{L}^2(\Omega)} \rightarrow 0$ as $h \rightarrow 0$. With \eqref{eq:second_derivatives_to_take_the_limit} at hand, we use that} $\|w\|_{L^2(\Omega)}\leq 1$ to obtain $\lim_{h\rightarrow 0}j''(\hat{u}_h)w_h^2\geq \min\{\mu,\alpha\}>0$. We can thus deduce the existence of $h_{\star}>0$ such that $j''(\hat{u}_h)w_h^2\geq \min\{\mu,\alpha\}/2$ for every $h<h_{\star}$. The bound \eqref{eq:auxiliary_error_estimate_fd} therefore follows from the definition of $w_h$ and the identity \eqref{eq:mean_value_theorem}.
\end{proof}

We are now ready to prove the error bound \eqref{eq:error_estimate_apriori_fully}.

\begin{theorem}[a priori error bound]
\label{theorem_estimates_control}
\EO{Let $\Omega \subset \mathbb{R}^d$ be a convex polytope, and let $\boldsymbol{\mathbf{f}}$ and $\mathbf{y}_{\Omega}$ be in $\mathbf{L}^2(\Omega)$.} Let us assume that $\bar{u}\in \mathbb{U}_{ad}$ satisfies the second-order optimality conditions \eqref{equivalence_condition}. Then, there exists $h_{\dagger}>0$ such that
\begin{align}\label{eq:error_estimate_fd}
\|\bar{u}-\bar{u}_{h}\|_{L^2(\Omega)}
\lesssim 
h\|\mathbf{f}\|_{\mathbf{L}^2(\Omega)}
\left[ 
\|\mathbf{f}\|_{\mathbf{L}^2(\Omega)}+\|\mathbf{y}_{\Omega}\|_{\mathbf{L}^2(\Omega)}
\right]
\quad\forall h\in(0, h_{\dagger}).
\end{align}
\end{theorem}
\begin{proof}
We proceed by contradiction: assume that the error bound \eqref{eq:error_estimate_fd} is false. Therefore, we can use the result of Lemma \ref{lemma_auxiliary_error_fd} to obtain that \eqref{eq:auxiliary_error_estimate_fd} holds for every $h<h_{\star}$. On the other hand, if we set $u=\bar{u}_h$ in \eqref{eq:ineq_variational} and $u_h=\mathsf{P}_h(\bar{u})$ in \eqref{eq:variational_inequality_discrete} we can conclude that $-j'(\bar{u})(\bar{u}_h-\bar{u})\leq 0$ and $j'_h(\bar{u}_h)(\mathsf{P}_h(\bar{u})-\bar{u}_h)\geq 0$, respectively. Thus,
\begin{multline}\label{eq:estimates_aux_error_fullydiscrete}
\|\bar{u}-\bar{u}_h\|_{L^2(\Omega)}^2\lesssim j'(\bar{u}_h)(\bar{u}_h-\bar{u})+j'_h (\bar{u}_h)(\mathsf{P}_h(\bar{u})-\bar{u}_h)\\
=j'(\bar{u}_h)(\mathsf{P}_h(\bar{u})-\bar{u})+\left[j'(\bar{u}_h)-j'_h (\bar{u}_h)\right](\bar{u}_h-\mathsf{P}_h(\bar{u}))=:\text{I}_h + \text{II}_h.
\end{multline}

To control the term $\text{I}_h$, we introduce $(\hat{\mathbf{y}},\hat{\mathsf{p}})\in \mathbf{H}_0^1(\Omega)\times L_0^2(\Omega)$ as the solution to
\begin{align}
\label{eq:aux_problem_adj_1_error}
(\nabla\hat{\mathbf{y}},\nabla \mathbf{v})_{\mathbf{L}^2(\Omega)}+( \bar{u}_{h}\hat{\mathbf{y}}, \mathbf{v})_{\mathbf{L}^2(\Omega)}-(\hat{\mathsf{p}},\text{div }\mathbf{v})_{L^2(\Omega)} =  	(\mathbf{f},\mathbf{v})_{\mathbf{L}^2(\Omega)}
\end{align}
and $(\mathsf{q},\text{div }\hat{\mathbf{y}})_{L^2(\Omega)} = 0$ for all $(\mathbf{v},\mathsf{q}) \in \mathbf{H}_0^1(\Omega) \times L_0^2(\Omega)$. We also introduce the pair $(\hat{\mathbf{z}},\hat{\mathsf{r}})\in \mathbf{H}_0^1(\Omega)\times L_0^2(\Omega)$ as the solution to
\begin{align}
\label{eq:aux_problem_adj_2_error}
(\nabla\hat{\mathbf{z}},\nabla \mathbf{v})_{\mathbf{L}^2(\Omega)}+( \bar{u}_{h}\hat{\mathbf{z}}, \mathbf{v})_{\mathbf{L}^2(\Omega)}+(\hat{\mathsf{r}},\text{div }\mathbf{v})_{L^2(\Omega)} = 	(\hat{\mathbf{y}}-\mathbf{y}_{\Omega},\mathbf{v})_{\mathbf{L}^2(\Omega)}
\end{align}
and $(\mathsf{q},\text{div }\hat{\mathbf{z}})_{L^2(\Omega)}  = 0$ for all $(\mathbf{v},\mathsf{q}) \in \mathbf{H}_0^1(\Omega) \times L_0^2(\Omega)$. We thus control $\text{I}_h$ as follows:
\begin{eqnarray*}
\text{I}_h 
&=&(\alpha \bar{u}_h-\hat{\mathbf{y}}\cdot\hat{\mathbf{z}},\mathsf{P}_h(\bar{u})-\bar{u})_{L^2(\Omega)}
=(\mathsf{P}_h(\hat{\mathbf{y}}\cdot\hat{\mathbf{z}})-\hat{\mathbf{y}}\cdot\hat{\mathbf{z}},\mathsf{P}_h(\bar{u})-\bar{u})_{L^2(\Omega)}
\\
&\leq & \|\mathsf{P}_h(\hat{\mathbf{y}}\cdot\hat{\mathbf{z}})-\hat{\mathbf{y}}\cdot\hat{\mathbf{z}}\|_{L^2(\Omega)} \|\mathsf{P}_h(\bar{u})-\bar{u}\|_{L^2(\Omega)},
\end{eqnarray*}
where we used that $(\alpha \bar{u}_h,\mathsf{P}_h(\bar{u})-\bar{u})_{L^2(\Omega)} = 0 = (\mathsf{P}_h(\hat{\mathbf{y}}\cdot\hat{\mathbf{z}}),\mathsf{P}_h(\bar{u})-\bar{u})_{L^2(\Omega)}$. The regularity results of Theorem \ref{thm:regularity_optimal_control}, combined with a basic error bound for $\mathsf{P}_h$, show that $\|\mathsf{P}_h(\bar{u})-\bar{u}\|_{L^2(\Omega)} \lesssim \GC{h}\|\mathbf{f}\|_{\mathbf{L}^2(\Omega)}[\|\mathbf{f}\|_{\mathbf{L}^2(\Omega)}+\|\mathbf{y}_{\Omega}\|_{\mathbf{L}^2(\Omega)}]$. To bound $\|\mathsf{P}_h(\hat{\mathbf{y}}\cdot\hat{\mathbf{z}})-\hat{\mathbf{y}}\cdot\hat{\mathbf{z}}\|_{L^2(\Omega)}$, we first note that since $\mathbf{f}, \hat{\mathbf{y}}-\mathbf{y}_{\Omega} \in \mathbf{L}^2(\Omega)$, we have that $\hat{\mathbf{y}}, \hat{\mathbf{z}} \in \mathbf{C}(\bar \Omega)$, (cf.~Theorem \ref{theorem_regularity_Linfty}). Thus
$
\|\nabla(\hat{\mathbf{y}}\cdot \hat{\mathbf{z}})\|_{\mathbf{L}^2(\Omega)} 
\leq 
\|\hat{\mathbf{z}} \|_{\mathbf{C}(\bar \Omega)} \| \nabla \hat{\mathbf{y}}\|_{\mathbf{L}^2(\Omega)}
+
\|\hat{\mathbf{y}} \|_{\mathbf{C}(\bar \Omega)} \| \nabla \hat{\mathbf{z}}\|_{\mathbf{L}^2(\Omega)}
\lesssim
\|\mathbf{f}\|_{\mathbf{L}^2(\Omega)}
[
\|\mathbf{f}\|_{\mathbf{L}^2(\Omega)}+\|\mathbf{y}_{\Omega}\|_{\mathbf{L}^2(\Omega)}
].
$
\EO{To control $\|\hat{\mathbf{y}} \|_{\mathbf{C}(\bar \Omega)}$ and $\|\hat{\mathbf{z}} \|_{\mathbf{C}(\bar \Omega)}$ we have used that $\bar{u}_h$ belongs to $\mathbb{U}_{ad,h}$ and thus that $\| \bar{u}_h \|_{L^2(\Omega)}$ can be controlled uniformly in terms of $\Omega$ and the control bounds.} A basic error estimate for $\mathsf{P}_h$ thus shows that
$
\|\mathsf{P}_h(\hat{\mathbf{y}}\cdot\hat{\mathbf{z}})-\hat{\mathbf{y}}\cdot\hat{\mathbf{z}}\|_{L^2(\Omega)}
\lesssim h \|\mathbf{f}\|_{\mathbf{L}^2(\Omega)}[\|\mathbf{f}\|_{\mathbf{L}^2(\Omega)}+\|\mathbf{y}_{\Omega}\|_{\mathbf{L}^2(\Omega)}].
$
As a result, we have obtained the following bound:
\[
\text{I}_h \lesssim h^2\|\mathbf{f}\|_{\mathbf{L}^2(\Omega)}^2
\left[ 
\|\mathbf{f}\|_{\mathbf{L}^2(\Omega)}+\|\mathbf{y}_{\Omega}\|_{\mathbf{L}^2(\Omega)}
\right]^2.
\]

We now bound $\text{II}_h$. To this end, we first note \EO{that
$
\text{II}_h = (\bar{\mathbf{y}}_h\cdot\bar{\mathbf{z}}_h - \hat{\mathbf{y}}\cdot\hat{\mathbf{z}},\bar{u}_h-\mathsf{P}_h(\bar{u}))_{L^2(\Omega)}=(\bar{\mathbf{y}}_h\cdot\bar{\mathbf{z}}_h - \hat{\mathbf{y}}\cdot\hat{\mathbf{z}},\mathsf{P}_h(\bar{u}_h-\bar{u}))_{L^2(\Omega)}.$
If} we add and subtract $(\bar{\mathbf{y}}_h\cdot\hat{\mathbf{z}},\mathsf{P}_h(\bar{u}_h-\bar{u}))_{L^2(\Omega)}$ and use a suitable Sobolev embedding and Young's inequality, we obtain
\begin{multline*}
\text{II}_h \lesssim \left[ 
\|\nabla\hat{\mathbf{z}}\|_{\mathbf{L}^{2}(\Omega)}\|\nabla(\hat{\mathbf{y}}-\bar{\mathbf{y}}_h)\|_{\mathbf{L}^{2}(\Omega)} 
+ 
\|\nabla\bar{\mathbf{y}}_h\|_{\mathbf{L}^{2}(\Omega)}\|\nabla(\hat{\mathbf{z}}-\bar{\mathbf{z}}_h)\|_{\mathbf{L}^{2}(\Omega)} \right]
\|\bar{u}-\bar{u}_h\|_{L^2(\Omega)}
\\
\leq C_1 \|\nabla\hat{\mathbf{z}}\|_{\mathbf{L}^{2}(\Omega)}^2\|\nabla(\hat{\mathbf{y}}-\bar{\mathbf{y}}_h)\|_{\mathbf{L}^2(\Omega)}^2
+
C_2 \|\nabla\bar{\mathbf{y}}_h\|_{\mathbf{L}^{2}(\Omega)}^2 \|\nabla(\hat{\mathbf{z}}-\bar{\mathbf{z}}_h)\|_{\mathbf{L}^2(\Omega)}^2+\tfrac{1}{2}\|\bar{u}-\bar{u}_h\|_{L^2(\Omega)}^2,
\end{multline*}
upon using that $\mathsf{P}_h$ is stable in $L^2(\Omega)$. Note that $\|\nabla\hat{\mathbf{z}}\|_{\mathbf{L}^{2}(\Omega)} \lesssim \| \mathbf{f} \|_{\mathbf{L}^2(\Omega)}+\|\mathbf{y}_{\Omega}\|_{\mathbf{L}^2(\Omega)}$ and that $\|\nabla\bar{\mathbf{y}}_h\|_{\mathbf{L}^{2}(\Omega)} \lesssim \| \mathbf{f} \|_{\mathbf{L}^2(\Omega)}$ for every $h>0$. On the other hand, an immediate application of the estimate \eqref{eq:error_estimate_Fem} shows that $\|\nabla(\hat{\mathbf{y}}-\bar{\mathbf{y}}_h)\|_{\mathbf{L}^2(\Omega)}\lesssim h \|\mathbf{f}\|_{\mathbf{L}^2(\Omega)}$. \GC{Note that the hidden constant depends on $\| \bar{u}_h \|_{L^2(\Omega)}$ and that $\| \bar{u}_h \|_{L^2(\Omega)}$ can be uniformly controlled in terms of $\Omega$ and the control bounds.} To control $\|\nabla(\hat{\mathbf{z}}-\bar{\mathbf{z}}_h)\|_{\mathbf{L}^2(\Omega)}$, we introduce the problem: Find $(\tilde{\mathbf{z}},\tilde{\mathsf{r}})\in \mathbf{H}_0^1(\Omega)\times L_0^2(\Omega)$ such that
\begin{align}
\label{eq:aux_problem_adj_3_error}
(\nabla\tilde{\mathbf{z}},\nabla \mathbf{v})_{\mathbf{L}^2(\Omega)}+( \bar{u}_{h}\tilde{\mathbf{z}}, \mathbf{v})_{\mathbf{L}^2(\Omega)}+(\tilde{\mathsf{r}},\text{div }\mathbf{v})_{L^2(\Omega)} = 	(\bar{\mathbf{y}}_h-\mathbf{y}_{\Omega},\mathbf{v})_{\mathbf{L}^2(\Omega)}
\end{align}
and $(\mathsf{q},\text{div }\tilde{\mathbf{z}})_{L^2(\Omega)} =  0$ for all $(\mathbf{v},\mathsf{q}) \in \mathbf{H}_0^1(\Omega) \times L_0^2(\Omega)$. A simple application of the triangular inequality yields $\|\nabla(\hat{\mathbf{z}}-\bar{\mathbf{z}}_h)\|_{\mathbf{L}^2(\Omega)}\leq \|\nabla(\hat{\mathbf{z}}-\tilde{\mathbf{z}})\|_{\mathbf{L}^2(\Omega)}+\|\nabla(\tilde{\mathbf{z}}-\bar{\mathbf{z}}_h)\|_{\mathbf{L}^2(\Omega)}$. The control of $\|\nabla(\hat{\mathbf{z}}-\tilde{\mathbf{z}})\|_{\mathbf{L}^2(\Omega)}$ follows from a stability bound for the problem that $(\hat{\mathbf{z}}-\tilde{\mathbf{z}},\hat{\mathsf{r}}-\tilde{\mathsf{r}})$ solves. \EO{Note that $\bar{u}_h \in \mathcal{A}_0$.} The control of $\|\nabla(\tilde{\mathbf{z}}-\bar{\mathbf{z}}_h)\|_{\mathbf{L}^2(\Omega)}$ follows from \eqref{eq:error_estimate_Fem}.
\GC{We again note $\| \bar{u}_h \|_{L^2(\Omega)}$ can be controlled uniformly in terms of $\Omega$ and the control bounds.} As a result, we obtain the following bound:
$
\|\nabla(\hat{\mathbf{z}}-\bar{\mathbf{z}}_h)\|_{\mathbf{L}^2(\Omega)}
\lesssim \|\hat{\mathbf{y}}-\bar{\mathbf{y}}_h\|_{\mathbf{L}^2(\Omega)}+ h(\|\mathbf{f}\|_{\mathbf{L}^2(\Omega)}+\|\mathbf{y}_{\Omega}\|_{\mathbf{L}^2(\Omega)}).$
With all these ingredients, we can bound the term $\text{II}_h$ as follows:
\begin{align*}
\text{II}_h \lesssim h^2 \|\mathbf{f}\|_{\mathbf{L}^2(\Omega)}^2
\left[\|\mathbf{f}\|_{\mathbf{L}^2(\Omega)}+\|\mathbf{y}_{\Omega}\|_{\mathbf{L}^2(\Omega)}\right]^2
+\tfrac{1}{2}\|\bar{u}-\bar{u}_h\|_{L^2(\Omega)}^2.
\end{align*}

The error bound \eqref{eq:error_estimate_fd} thus follows from substituting the estimates obtained for the terms $\text{I}_h$ and $\text{II}_h$ into \eqref{eq:estimates_aux_error_fullydiscrete}. This, which is a contradiction, concludes the proof.
\end{proof}

Before \GC{presenting a priori error estimates for the finite element approximation of the state and adjoint variables, we introduce the following quantities:}
\begin{equation}\label{eq:constants_estimates_control_apriori}
\begin{aligned}
\mathfrak{M}(\mathbf{f},\mathbf{y}_{\Omega})
&
:=\|\mathbf{f}\|_{\mathbf{L}^2(\Omega)}( 1 + \|\mathbf{f}\|_{\mathbf{L}^2(\Omega)}[\|\mathbf{f}\|_{\mathbf{L}^2(\Omega)}+\|\mathbf{y}_{\Omega}\|_{\mathbf{L}^2(\Omega)}]),\\
\mathfrak{N}(\mathbf{f},\mathbf{y}_{\Omega})
&
:=( \|\mathbf{f}\|_{\mathbf{L}^2(\Omega)}+\|\mathbf{y}_{\Omega}\|_{\mathbf{L}^2(\Omega)})( 1 + \|\mathbf{f}\|_{\mathbf{L}^2(\Omega)}[\|\mathbf{f}\|_{\mathbf{L}^2(\Omega)}+\|\mathbf{y}_{\Omega}\|_{\mathbf{L}^2(\Omega)}]).
\end{aligned}
\end{equation}

We \EO{have the following error bounds.}

\begin{corollary}[a priori error bounds]
Let the assumptions of Theorem \ref{theorem_estimates_control} hold. Then, there exists $h_{\dagger}>0$ such that
\GC{\begin{align*}\label{eq:estimates_error_fully_allvariables}
\|\nabla(\bar{\mathbf{y}}-\bar{\mathbf{y}}_{h})\|_{\mathbf{L}^2(\Omega)}+\|\bar{\mathsf{p}}-\bar{\mathsf{p}}_{h}\|_{L^2(\Omega)}
\lesssim
h \, \mathfrak{M}(\mathbf{f},\mathbf{y}_{\Omega}),
\\
\|\nabla(\bar{\mathbf{z}}-\bar{\mathbf{z}}_{h})\|_{\mathbf{L}^2(\Omega)}+\|\bar{\mathsf{r}}-\bar{\mathsf{r}}_{h}\|_{L^2(\Omega)}
\lesssim 
h \, \mathfrak{N}(\mathbf{f},\mathbf{y}_{\Omega}),
\end{align*}}
for all $h<h_{\dagger}$. In both bounds the hidden constants do not depend on $\mathbf{f}$, $\mathbf{y}_{\Omega}$, and $h$.
\end{corollary}

\begin{proof}
The proof follows directly from a combination of the bounds given in Propositions \ref{theorem_convergence_01} and \ref{eq:theorem_adjoint_aux_estimates} and Theorem \ref{theorem_estimates_control}. This concludes the proof.
\end{proof}

\subsection{The semidiscrete scheme}
Let $\{\bar{\mathsf{u}}_{h}\}_{h>0}$ be a sequence of local minima of the semidiscrete scheme such that $\bar{\mathsf{u}}_{h}\to\bar{u}$ in $L^2(\Omega)$ as $h \rightarrow 0$, where $\bar{u}$ denotes a local solution of \eqref{eq:min_functional}--\eqref{eq:brinkamn_problem_state}. \EO{In this section, we derive the following error bound:} 
\begin{equation}\label{eq:error_estimate_sd_aux}
\EO{\|\bar{u}-\bar{\mathsf{u}}_{h}\|_{L^2(\Omega)}\lesssim h^2
\quad\forall h\in(0, h_{\dagger}).}
\end{equation}

As in the fully discrete scheme, the following instrumental result is important.

\begin{lemma}[auxiliary error estimate]
\label{lemma_auxiliary_error_sd}
\EO{Let $\Omega \subset \mathbb{R}^d$ be a convex polytope, and let $\boldsymbol{\mathbf{f}}$ and $\mathbf{y}_{\Omega}$ be in $\mathbf{L}^2(\Omega)$}. Let us assume that $\bar{u}\in \mathbb{U}_{ad}$ satisfies the second-order optimality conditions \eqref{equivalence_condition}. Then, there exists $h_{\star}>0$ such that
\begin{align}\label{eq:auxiliary_error_estimate_sd}
\mathfrak{C}\|\bar{u}-\bar{\mathsf{u}}_{h}\|_{L^2(\Omega)}^2\leq [j'(\bar{\mathsf{u}}_{h})-j'(\bar{u})](\bar{\mathsf{u}}_{h}-\bar{u})
\quad
\forall h<h_{\star},
\quad
\mathfrak{C}:=2^{-1}\min\{\mu,\alpha\}.
\end{align}
Here, $\alpha$ is the control cost, and $\mu$ denotes the constant appearing in \eqref{equivalence_condition}.
\end{lemma}
\begin{proof}
Define $w_h:=(\bar{\mathsf{u}}_h-\bar{u})/\|\bar{\mathsf{u}}_h-\bar{u}\|_{L^2(\Omega)}$ \EO{for each $h>0$}. We note that $w_h \rightharpoonup w$ in $L^2(\Omega)$ up to a subsequence, if necessary, as $h \to 0$. The arguments developed in the proof of Lemma \ref{lemma_auxiliary_error_fd} show that $w$ satisfies the sign conditions in \eqref{eq:cone_condition} and that $\bar{\mathfrak{d}}_h:= \alpha \bar{\mathsf{u}}_h - \bar{\mathbf{y}}_h \cdot \bar{\mathbf{z}}_h \to \alpha \bar{u} - \bar{\mathbf{y}} \cdot \bar{\mathbf{z}} = \bar{\mathfrak{d}}$ in $L^2(\Omega)$ as $h \to 0$. We now set $\mathsf{u}_h = \bar{u}$ in \eqref{eq:variational_inequality_semi-discrete} to obtain
\[
 \int_{\Omega} \bar{\mathfrak{d}} w \mathrm{d}x = \lim_{h \to 0} \int_{\Omega} \bar{\mathfrak{d}}_h w_h \mathrm{d}x = \lim_{h \to 0} \|\bar{\mathsf{u}}_h-\bar{u}\|^{-1}_{L^2(\Omega)} \int_{\Omega} (\alpha \bar{\mathsf{u}}_h - \bar{\mathbf{y}}_h \cdot \bar{\mathbf{z}}_h)(\bar{\mathsf{u}}_h-\bar{u})  \mathrm{d}x \leq 0.
\]
As in the proof of Lemma \ref{lemma_auxiliary_error_fd}, this shows that if $\bar{\mathfrak{d}}(x)\neq 0$, then $w(x)=0$ for a.e.~$x\in\Omega$. This proves that $w\in C_{\bar{u}}$. The rest of the proof follows the arguments from the proof of Lemma \ref{lemma_auxiliary_error_fd}. For brevity, we omit the details.
\end{proof}

We now derive the error bound \EO{\eqref{eq:error_estimate_sd_aux}}. To this end, we present the following result which guarantees that discrete solutions to problem \eqref{eq:brinkamn_problem_state_sd} are uniformly bounded in $\mathbf{W}^{1,\kappa}(\Omega) \times L^{\kappa}(\Omega)$ for some $\kappa > 4$ when $d=2$ and $\kappa >3$ when $d=3$. Note that as a consequence of a standard Sobolev embedding, we have that such discrete velocities are uniformly bounded with respect to $h$ in $\mathbf{L}^{\infty}(\Omega)$: $\|\bar{\mathbf{y}}_h \|_{\mathbf{L}^{\infty}(\Omega)} \lesssim \| \mathbf{f} \|_{\mathbf{L}^2(\Omega)}$.

\begin{proposition}[discrete stability]
\label{pro:discrete_stability}
Let $\Omega$ be a convex polytope, let $\boldsymbol{\mathbf{f}} \in \mathbf{L}^2(\Omega)$, and let \EO{$\{ \T_h \}_{h>0}$ be a family of quasi-uniform meshes}. Then, there exists $\kappa > 4$ if $d=2$ and $\kappa >3$ if $d=3$ such that
 \begin{equation}
  \| \nabla \bar{\mathbf{y}}_h \|_{\mathbf{L}^{\kappa}(\Omega)} 
  +
  \| \bar{\mathsf{p}}_h \|_{L^\kappa(\Omega)}
  \lesssim \| \mathbf{f} \|_{\mathbf{L}^2(\Omega)}.
 \end{equation}
\end{proposition}
\begin{proof}
 The proof follows directly from \cite[Corollaries 4 and 5]{MR3422453} for $d=3$ and \cite[Theorems 8.2 and 8.4]{MR2121575} for $d=2$, \EO{combined with arguments from the proof of \cite[Lemma 11]{MR3422453},} and the regularity results of Theorem \ref{theorem_regularity_Linfty}. \EO{We note that, since $\bar{\mathsf{u}}_h \in \mathbb{U}_{ad}$, it can be uniformly controlled in terms of $\Omega$ and the control bounds.} 
\end{proof}

\begin{theorem}[a priori error bound]
\label{theorem_estimates_control_semi}
\EO{Let $\Omega \subset \mathbb{R}^d$ be a convex polytope,} let $\boldsymbol{\mathbf{f}}$ and $\mathbf{y}_{\Omega}$ be in $\mathbf{L}^2(\Omega)$, and let \EO{$\{ \T_h \}_{h>0}$ be a family of quasi-uniform meshes.} Let us assume that $\bar{u} \in \mathbb{U}_{ad}$ satisfies the second-order optimality conditions \eqref{equivalence_condition}. Then, there exists $h_{\dagger}>0$ such that
\begin{align}\label{eq:error_estimate_sd}
\|\bar{u}-\bar{\mathsf{u}}_{h}\|_{L^2(\Omega)}\lesssim h^2\|\mathbf{f}\|_{\mathbf{L}^2(\Omega)}\left[\|\mathbf{f}\|_{\mathbf{L}^2(\Omega)}+\|\mathbf{y}_{\Omega}\|_{\mathbf{L}^2(\Omega)}\right]\quad\forall h\in(0, h_{\dagger}).
\end{align}
\end{theorem}
\begin{proof}
Set $u =\bar{\mathsf{u}}_h$ in \eqref{eq:ineq_variational} and $u=\bar{u}$ in \eqref{eq:variational_inequality_semi-discrete} to arrive at $-j'(\bar{u})(\bar{\mathsf{u}}_h-\bar{u})\leq 0$ and $j'_h(\bar{\mathsf{u}}_h)(\bar{u}-\bar{\mathsf{u}}_h)\geq 0$, respectively. With these estimates in hand, we use the instrumental error bound \eqref{eq:auxiliary_error_estimate_sd} to obtain
\begin{multline*}\label{eq:estimates_aux_error_semi-discrete}
\|\bar{u}-\bar{\mathsf{u}}_h\|_{L^2(\Omega)}^2
\lesssim [j'(\bar{\mathsf{u}}_h)-j_h'(\bar{\mathsf{u}}_h)](\bar{\mathsf{u}}_h-\bar{u})
=(\bar{\mathbf{y}}_h \cdot \bar{\mathbf{z}}_h-\hat{\mathbf{y}}\cdot\hat{\mathbf{z}},\bar{\mathsf{u}}_h-\bar{u})_{L^2(\Omega)}
\\
=(\hat{\mathbf{z}}\cdot(\bar{\mathbf{y}}_h-\hat{\mathbf{y}}),\bar{\mathsf{u}}_h-\bar{u})_{L^2(\Omega)}
+(\bar{\mathbf{y}}_h\cdot (\bar{\mathbf{z}}_h-\hat{\mathbf{z}}),\bar{\mathsf{u}}_h-\bar{u})_{L^2(\Omega)}
=:\text{I}_h +\text{II}_h.
\end{multline*}
Here, $(\hat{\mathbf{y}},\hat{\mathsf{p}})\in \mathbf{H}_0^1(\Omega)\times L_0^2(\Omega)$ and $(\hat{\mathbf{z}},\hat{\mathsf{r}})\in \mathbf{H}_0^1(\Omega)\times L_0^2(\Omega)$ correspond to the solutions of problems \eqref{eq:aux_problem_adj_1_error} and \eqref{eq:aux_problem_adj_2_error}, respectively, where $\bar{u}_h$ is replaced by $\bar{\mathsf{u}}_h$.

We now control the term $\text{I}_h$. To do this, we proceed as follows:
\begin{multline}\label{eq:estimates_I1_sd}
\text{I}_h \leq \| \hat{\mathbf{z}} \|_{\mathbf{L}^{\infty}(\Omega)}\|\bar{\mathbf{y}}_h-\hat{\mathbf{y}}\|_{\mathbf{L}^2(\Omega)}\|\bar{\mathsf{u}}_h-\bar{u}\|_{L^2(\Omega)}
\\
\lesssim h^2\|\mathbf{f}\|_{\mathbf{L}^2(\Omega)}\left[\|\mathbf{f}\|_{\mathbf{L}^2(\Omega)}+\|\mathbf{y}_{\Omega}\|_{\mathbf{L}^2(\Omega)}\right]\|\bar{\mathsf{u}}_h-\bar{u}\|_{L^2(\Omega)}.
\end{multline}
\EO{Here, we have used that $(\bar{\mathbf{y}}_h,\bar{\mathsf{p}}_h)$ corresponds to the finite element approximation of $(\hat{\mathbf{y}},\hat{\mathsf{p}})$, so that a direct application of the second estimate in \cite[Proposition 4.18]{Guermond-Ern} (\cite[Theorem 4.21]{Guermond-Ern} in the case of the mini element or \cite[Theorem 4.26]{Guermond-Ern} in the case of the Taylor--Hood element) shows that
\begin{equation}
 \|\bar{\mathbf{y}}_h-\hat{\mathbf{y}}\|_{\mathbf{L}^2(\Omega)} \lesssim h^2 ( \| \hat{\mathbf{y}} \|_{\mathbf{H}^{2}(\Omega)} + \| \hat{\mathsf{p}} \|_{H^1(\Omega)}) \lesssim h^2 \| \mathbf{f} \|_{\mathbf{L}^{2}(\Omega)}.
 \label{eq:L2_error_bound_haty}
\end{equation}
To obtain the last bound, we used the regularity result from Theorem \ref{eq:theorem_regularity} combined with the fact that $\| \bar{\mathsf{u}}_h \|_{L^2(\Omega)}$ can be controlled uniformly in terms of $\Omega$ and the control bounds.
We also note that Theorem \ref{theorem_regularity_Linfty} guarantees $\| \hat{\mathbf{z}} \|_{\mathbf{L}^{\infty}(\Omega)} \lesssim \| \mathbf{f} \|_{\mathbf{L}^{2}(\Omega)} + \| \mathbf{y}_{\Omega} \|_{\mathbf{L}^{2}(\Omega)}$.}

We now control $\text{II}_h$. To this end, we introduce $(\tilde{\mathbf{z}},\tilde{\mathsf{r}})\in \mathbf{H}_0^1(\Omega)\times L_0^2(\Omega)$ as the solution to \eqref{eq:aux_problem_adj_3_error}, where $\bar{u}_h$ is replaced by $\mathsf{\bar{u}}_h$. With these ingredients, we obtain
\[
 \text{II}_h \leq \|\bar{\mathbf{y}}_h\|_{\mathbf{L}^{\infty}(\Omega)}\left(\| \hat{\mathbf{z}}-\tilde{\mathbf{z}} \|_{\mathbf{L}^2(\Omega)}+\|\tilde{\mathbf{z}}-\bar{\mathbf{z}}_h\|_{\mathbf{L}^2(\Omega)}\right)\|\bar{\mathsf{u}}_h-\bar{u}\|_{L^2(\Omega)}.
\]
%
\GC{To bound $\| \hat{\mathbf{z}}-\tilde{\mathbf{z}} \|_{\mathbf{L}^2(\Omega)}$, we first note that $(\hat{\mathbf{z}}-\tilde{\mathbf{z}},\hat{\mathsf{r}}-\tilde{\mathsf{r}}) \in \mathbf{H}_0^1(\Omega)\times L_0^2(\Omega)$ solves: 
\begin{multline*}
(\nabla (\hat{\mathbf{z}}- \tilde{\mathbf{z}}),\nabla \mathbf{v})_{\mathbf{L}^2(\Omega)}+( \bar{\mathsf{u}}_h(\hat{\mathbf{z}}- \tilde{\mathbf{z}}), \mathbf{v})_{\mathbf{L}^2(\Omega)}+(\hat{\mathsf{r}}-\tilde{\mathsf{r}},\text{div }\mathbf{v})_{L^2(\Omega)} \\=(\hat{\mathbf{y}}-\bar{\mathbf{y}}_{h},\mathbf{v})_{\mathbf{L}^2(\Omega)},\quad (\mathsf{q},\text{div }(\hat{\mathbf{z}}- \tilde{\mathbf{z}}))_{L^2(\Omega)}  =  0\quad\forall (\mathbf{v},\mathsf{q})\in\mathbf{H}_0^1(\Omega)\times L_0^2(\Omega).
\end{multline*}
Setting $(\mathbf{v},\mathsf{q}) = (\hat{\mathbf{z}}- \tilde{\mathbf{z}},0)$ yields $\|\nabla(\hat{\mathbf{z}}-\tilde{\mathbf{z}})\|_{\mathbf{L}^2(\Omega)}\lesssim \|\hat{\mathbf{y}}-\bar{\mathbf{y}}_h \|_{\mathbf{L}^2(\Omega)}\lesssim h^2 \|\mathbf{f}\|_{\mathbf{L}^2(\Omega)}$, where the last estimate is obtained using \eqref{eq:L2_error_bound_haty}.} An estimate for the term $\| \tilde{\mathbf{z}}-\bar{\mathbf{z}}_h \|_{\mathbf{L}^2(\Omega)}$ follows from \cite[Proposition 4.18]{Guermond-Ern} (see also the derivation of the bound \eqref{eq:L2_error_bound_haty}): $\| \tilde{\mathbf{z}}-\bar{\mathbf{z}}_h \|_{\mathbf{L}^2(\Omega)}\lesssim h^2(\|\mathbf{f}\|_{\mathbf{L}^2(\Omega)}+\|\mathbf{y}_{\Omega}\|_{\mathbf{L}^2(\Omega)})$. As a final ingredient, we note that the results of Proposition \ref{pro:discrete_stability} guarantee that
$\| \bar{\mathbf{y}}_h\|_{\mathbf{L}^{\infty}(\Omega)} \lesssim \| \mathbf{f} \|_{\mathbf{L}^{2}(\Omega)}$. We can thus obtain that
\begin{align}\label{eq:estimates_I2_sd}
\text{II}_h \lesssim h^2 \| \mathbf{f} \|_{\mathbf{L}^{2}(\Omega)} [\|\mathbf{f}\|_{\mathbf{L}^2(\Omega)}+\|\mathbf{y}_{\Omega}\|_{\mathbf{L}^2(\Omega)}]\|\bar{\mathsf{u}}_h-\bar{u}\|_{L^2(\Omega)}.
\end{align}

The desired estimate \eqref{eq:error_estimate_sd} results from combining \eqref{eq:estimates_I1_sd} and \eqref{eq:estimates_I2_sd}. 
\end{proof}

We conclude this section with the following error estimates.
\begin{corollary}
Let the assumptions of Theorem \ref{theorem_estimates_control_semi} hold. Then, there exists $h_{\dagger}>0$ such that 
\GC{
\begin{align*}\label{eq:estimates_error_semi_allvariables}
\|\nabla(\bar{\mathbf{y}}-\bar{\mathbf{y}}_{h})\|_{\mathbf{L}^2(\Omega)}
+
\|\bar{\mathsf{p}}-\bar{\mathsf{p}}_{h}\|_{L^2(\Omega)}\lesssim h \, \mathfrak{M}(\mathbf{f},\mathbf{y}_{\Omega}),
\\
\|\nabla(\bar{\mathbf{z}}-\bar{\mathbf{z}}_{h})\|_{\mathbf{L}^2(\Omega)}+\|\bar{\mathsf{r}}-\bar{\mathsf{r}}_{h}\|_{L^2(\Omega)}
\lesssim
h \, \mathfrak{N}(\mathbf{f},\mathbf{y}_{\Omega}).
\end{align*}}
for all $h<h_{\dagger}$. In both bounds the hidden constants do not depend on $\mathbf{f}$, $\mathbf{y}_{\Omega}$, and $h$. \EO{The terms $\mathfrak{M}$ and $\mathfrak{N}$ are defined in \eqref{eq:constants_estimates_control_apriori}.} 
\end{corollary}

\begin{proof}
The proof follows directly from a combination of the bounds given in Propositions \ref{theorem_convergence_01} and \ref{eq:theorem_adjoint_aux_estimates} and Theorem \ref{theorem_estimates_control_semi}. This concludes the proof.
\end{proof}

\section{A posteriori error bounds}
In this section, we design a posteriori error estimators for the fully discrete and semidiscrete methods presented in \S\ref{fully_discrete_framework} and \S\ref{semi_discrete_framework}, respectively. \EO{We obtain global reliability bounds and explore efficiency estimates.}


\subsection{An auxiliary result}
To present a posteriori error estimators,
we introduce the auxiliary variable $\breve{u}$ as follows:
\begin{equation}\label{eq:representation_aux}
\breve{u}:= \Pi_{[\mathsf{a},\mathsf{b}]}(\alpha^{-1}\bar{\mathbf{y}}_{h}\cdot\bar{\mathbf{z}}_{h}).
\end{equation}
Here, $(\bar{\mathbf{y}}_{h},\bar{\mathsf{p}}_{h})$ denotes the solution to \eqref{eq:brinkamn_problem_state_fd}, where \GC{$u_h$} is replaced by \GC{$\bar{u}_h$}, and $(\bar{\mathbf{z}}_{h},\bar{\mathsf{r}}_{h})$ denotes the solution to \eqref{eq:brinkman_problem_adjoint_fully}, where \EO{$u$ is replaced by $\bar{u}_h$ and $\mathbf{y}_{h}$ is replaced by $\bar{\mathbf{y}}_{h}$. Since $\breve{u}$ is defined in \eqref{eq:representation_aux} using the projection operator $\Pi_{[\mathsf{a},\mathsf{b}]}$, we can use the same arguments as those given in the proofs of \cite[Lemma 2.26 and Theorems 2.27 and 2.28]{MR2583281} to obtain that $\breve{u}$ satisfies the following variational inequality:} 
\begin{equation}\label{eq:variational_inequality_discrete_aux}
(\alpha \breve{u}-\bar{\mathbf{y}}_{h}\cdot \bar{\mathbf{z}}_{h},u-\breve{u})_{L^2(\Omega)}\geq 0\quad\forall u\in \mathbb{U}_{ad}.
\end{equation}
%

The following bound is crucial for our a posteriori error analysis and extends the results in \cite{MR4450052}. \EO{Specifically, we will prove that the devised a posteriori error estimators are globally reliable using the key estimate \eqref{eq:bound_aux_j_final}. We emphasize that, to derive estimate \eqref{eq:bound_aux_j_final}, we \emph{assume} that both inequalities in \eqref{eq:bound_aux} hold. These assumptions---the inequalities in \eqref{eq:bound_aux}---are not needed for investigating local efficiency estimates.}

\begin{theorem}[auxiliary bound] 
Let $\bar{u}$ be a local solution of \eqref{eq:min_functional}--\eqref{eq:brinkamn_problem_state} such that \eqref{equivalence_condition} holds for $\tau > 0$, and let 
$(\bar{\mathbf{y}},\bar{\mathsf{p}})$ and $(\bar{\mathbf{z}},\bar{\mathsf{r}})$ be the corresponding state and adjoint state variables. Let $\bar{u}_{h}$ be a local minimum of the fully discrete scheme, and let $(\bar{\mathbf{y}}_h,\bar{\mathsf{p}}_h)$ and $(\bar{\mathbf{z}}_h,\bar{\mathsf{r}}_h)$ be the corresponding discrete state and adjoint state variables. If
\begin{align}\label{eq:bound_aux}
\|\bar{\mathbf{y}}\cdot \bar{\mathbf{z}}-\bar{\mathbf{y}}_h\cdot \bar{\mathbf{z}}_h\|_{L^2(\Omega)}\leq \alpha \mu (2\mathfrak{D})^{-1},
\qquad
\|\bar{\mathbf{y}}\cdot \bar{\mathbf{z}}-\bar{\mathbf{y}}_h\cdot \bar{\mathbf{z}}_h\|_{L^{\infty}(\Omega)} < 2^{-1} \tau,
\end{align}
then we have the following error bound:
\begin{align}\label{eq:bound_aux_j_final}
\mu \|\bar{u}-\breve{u}\|_{L^2(\Omega)}^2 \leq 2(j'(\breve{u})-j'(\bar{u}))(\breve{u}-\bar{u}).
\end{align}
The constants $\mu$ and $\mathfrak{D}$ are given as in \eqref{equivalence_condition} and \eqref{eq:lemma_j}, respectively.
\end{theorem}
\begin{proof} We divide the proof into two steps.

\emph{Step 1.} $\breve{u}-\bar{u} \in C_{\bar u}^{\tau}$. \EO{As a first step we show that $\breve{u}-\bar{u} \in C_{\bar u}^{\tau}$ for some $\tau > 0$.}
Since by definition $\breve{u} \in \mathbb{U}_{ad}$, we can conclude that $\breve{u}-\bar{u}$ satisfies the conditions in \eqref{eq:cone_condition}. To prove the remaining condition in \eqref{eq:Cutau}, we proceed as follows: Let $x\in \Omega$ be such that $\bar{\mathfrak{d}}(x) = \alpha \bar{u}(x) - \bar{\mathbf{y}}(x) \cdot \bar{\mathbf{z}}(x) > \tau$. The fact that $\tau > 0$ guarantees that $\bar{u}(x) > \alpha^{-1} \bar{\mathbf{y}}(x) \cdot \bar{\mathbf{z}}(x)$. The projection formula \eqref{eq:representation} therefore shows that $\bar{u}(x) = \GC{\mathsf{a}}$. On the other hand, we note that assumption \eqref{eq:bound_aux} shows that 
\[
 \| (\alpha \bar{u} - \bar{\mathbf{y}} \cdot \bar{\mathbf{z}}) - ( \alpha \breve{u}- \bar{\mathbf{y}}_h \cdot \bar{\mathbf{z}}_h )  \|_{L^{\infty}(\Omega)} 
 \leq 2 \| \bar{\mathbf{y}} \cdot \bar{\mathbf{z}} -  \bar{\mathbf{y}}_h \cdot \bar{\mathbf{z}}_h  \|_{L^{\infty}(\Omega)} < \tau.
\]
Since $\alpha \bar{u}(x) - \bar{\mathbf{y}}(x) \cdot \bar{\mathbf{z}}(x) > \tau$, we can conclude that $\alpha \breve{u}(x)- \bar{\mathbf{y}}_h(x) \cdot \bar{\mathbf{z}}_h(x) >0$ and therefore that $\breve{u}(x) > \alpha^{-1} \bar{\mathbf{y}}_h(x) \cdot \bar{\mathbf{z}}_h(x)$. This implies that $\breve{u}(x) = \GC{\mathsf{a}}$. As a result, we have proved that $(\breve{u} - \bar{u})(x) = 0$ if $x \in \Omega$ is such that $\bar{\mathfrak{d}}(x) > \tau$. Similar calculations show that $(\breve{u} - \bar{u})(x) = 0$ if $x \in \Omega$ is such that $\bar{\mathfrak{d}}(x) < -\tau$.

\emph{Step 2.} \emph{The bound \eqref{eq:bound_aux_j_final}}. \EO{Since $\breve{u}-\bar{u} \in C_{\bar u}^{\tau} \subset L^2(\Omega)$, we can set $v = \breve{u}-\bar{u}$ in \eqref{equivalence_condition} to obtain that $j''(\bar{u})(\breve{u}-\bar{u})^2 \geq \mu  \| \breve{u}-\bar{u} \|_{L^2(\Omega)}^2$ and thus that 
\[
\mu \| \breve{u}-\bar{u} \|_{L^2(\Omega)}^2 \leq j''(\bar{u})(\breve{u}-\bar{u})^2 =  (j''(\bar{u})-j''(\xi))(\breve{u}-\bar{u})^2 + (j'(\breve{u})-j'(\bar{u}))(\breve{u}-\bar{u}).
\]
To obtain the equality, we used that $(j'(\breve{u})-j'(\bar{u}))(\breve{u}-\bar{u})=j''(\xi)(\breve{u}-\bar{u})^2$, where $\xi=\bar{u}+\theta (\breve{u}-\bar{u})$ and $\theta \in (0,1)$, which follows from the mean value theorem.}
We now use \eqref{eq:lemma_j}, the projection formulas \eqref{eq:representation} and \eqref{eq:representation_aux}, the Lipschitz property of $\Pi_{[\mathsf{a},\mathsf{b}]}$, and the assumption \eqref{eq:bound_aux} to obtain
\begin{multline}
 (j''(\bar{u})-j''(\xi))(\breve{u}-\bar{u})^2 \leq \mathfrak{D} \| \breve{u} -\bar{u} \|_{L^2(\Omega)} \|\breve{u}-\bar{u}\|_{L^2(\Omega)}^2
 \\
 \leq \mathfrak{D} \alpha^{-1} \|\bar{\mathbf{y}}\cdot \bar{\mathbf{z}} - \bar{\mathbf{y}}_h\cdot \bar{\mathbf{z}}_h  \|_{L^2(\Omega)} \|\breve{u}-\bar{u}\|_{L^2(\Omega)}^2 \leq (\mu/2) \|\breve{u}-\bar{u}\|_{L^2(\Omega)}^2.
\end{multline}
If this bound is substituted into the one derived for $\mu \| \breve{u}-\bar{u} \|_{L^2(\Omega)}^2$, we arrive at \eqref{eq:bound_aux_j_final}. This concludes the proof.
\end{proof}


\subsection{Reliability analysis: the fully discrete scheme} 
\label{Real:fully_discrete} 
In this section, we propose an a posteriori error estimator for the fully discrete scheme and obtain a global reliability estimate. To this end, we introduce the following auxiliary variables. First, we introduce $(\check{\mathbf{y}},\check{\mathsf{p}})\in \mathbf{H}_0^1(\Omega)\times L_0^2(\Omega)$ as the solution of
\begin{equation}\label{eq:brinkamn_problem_state_aux}
(\nabla \check{\mathbf{y}},\nabla \mathbf{v})_{\mathbf{L}^2(\Omega)}+( \bar{u}_{h}\check{\mathbf{y}}, \mathbf{v})_{\mathbf{L}^2(\Omega)}-(\check{\mathsf{p}},\text{div }\mathbf{v})_{L^2(\Omega)} = 	(\mathbf{f},\mathbf{v})_{\mathbf{L}^2(\Omega)}
\end{equation}
and $(\mathsf{q},\text{div }\check{\mathbf{y}})_{L^2(\Omega)}  = 0$ for all $\mathbf{v} \in \mathbf{H}_0^1(\Omega)$ and $\mathsf{q} \in L_0^2(\Omega)$, respectively. We note that \AAF{since $\bar{u}_{h} \in \mathbb{U}_{ad,h} \subset \mathcal{A}_{0}$,} problem \eqref{eq:brinkamn_problem_state_aux} is well-posed. We also note that $(\bar{\mathbf{y}}_{h},\bar{\mathsf{p}}_{h})$, which solves \eqref{eq:brinkamn_problem_state_fd} with $u_{h}$ replaced by $\bar{u}_{h}$, corresponds to the finite element approximation of $(\check{\mathbf{y}},\check{\mathsf{p}})$ in the setting of \S\ref{fem}. Therefore, it is natural to define the following \emph{a posteriori error estimator}:
\begin{equation}\label{eq:estimator_st}
\mathcal{E}_{st,\T}^2:= \sum_{K\in\T}\mathcal{E}_{st,K}^2,
\end{equation}
where $\mathcal{E}_{st,K}^2:=h_K^2  \|\mathcal{R}_{K}^{st}\|_{\mathbf{L}^2(K)}^2
+h_K  \|\mathcal{J}_{\gamma}^{st}\|_{\mathbf{L}^2(\partial K\setminus\partial \Omega)}^2
+\|\text{div }\bar{\mathbf{y}}_{h}\|_{L^2(K)}^2$. Here,
\begin{align}\label{eq:residual_state}
\mathcal{R}_{K}^{st}:=(\mathbf{f}+\Delta \bar{\mathbf{y}}_{h}-\bar{u}_{h}\bar{\mathbf{y}}_{h}-\nabla\bar{\mathsf{p}}_{h})|_K,
\qquad 
\mathcal{J}_{\gamma}^{st}:=\llbracket(\nabla \bar{\mathbf{y}}_{h}-\bar{\mathsf{p}}_{h}\mathbf{I})\cdot \mathbf{n}\rrbracket.
\end{align}
An application of Theorem \ref{eq:reliability}, with $\mathfrak{u}$ and $\boldsymbol{\mathfrak{f}}$ replaced by $\bar{u}_{h}$ and $\mathbf{f}$, shows that
\begin{align}\label{eq:bound_realiability_1}
\|\nabla(\check{\mathbf{y}}-\bar{\mathbf{y}}_{h})\|_{\mathbf{L}^2(\Omega)}^2+\|\check{\mathsf{p}}-\bar{\mathsf{p}}_{h}\|_{L^2(\Omega)}^2\lesssim \mathcal{E}_{st,\T}^2.
\end{align}

We also introduce the pair $(\check{\mathbf{z}},\check{\mathsf{r}})\in\mathbf{H}_0^1(\Omega)\times L_0^2(\Omega)$ as the solution of
\begin{equation}
\label{eq:brinkamn_problem_adjoint_aux}
(\nabla \mathbf{v},\nabla \check{\mathbf{z}})_{\mathbf{L}^2(\Omega)}+( \bar{u}_{h}\check{\mathbf{z}}, \mathbf{v})_{\mathbf{L}^2(\Omega)}+(\check{\mathsf{r}},\text{div }\mathbf{v})_{L^2(\Omega)} = 	(\bar{\mathbf{y}}_{h}-\mathbf{y}_{\Omega},\mathbf{v})_{\mathbf{L}^2(\Omega)}
\end{equation}
and $(\mathsf{q},\text{div }\check{\mathbf{z}})_{L^2(\Omega)}  = 0$ for all $\mathbf{v} \in  \mathbf{H}_0^1(\Omega)$ and $\mathsf{q} \in L_0^2(\Omega)$, respectively. We note that \AAF{since $ \bar{u}_{h}\in\mathcal{A}_{0}$,} \eqref{eq:brinkamn_problem_adjoint_aux} is well-posed. \AAF{We also note that} $(\bar{\mathbf{z}}_{h},\bar{\mathsf{r}}_{h})$ corresponds to the finite element approximation of $(\check{\mathbf{z}},\check{\mathsf{r}})$ in the context of \S \ref{fem}. \AAF{It is therefore natural to} introduce the following \emph{a posteriori error estimator:}
\begin{equation}
\label{eq:estimator_AD}
\mathcal{E}_{adj,\T}^2:= \sum_{K\in\T}\mathcal{E}_{adj,K}^2, 
\end{equation}
where $\mathcal{E}_{adj,K}^2:=h_K^2  \|\mathcal{R}_{K}^{adj}\|_{\mathbf{L}^2(K)}^2+h_K  \|\mathcal{J}_{\gamma}^{adj}\|_{\mathbf{L}^2(\partial K\setminus\partial \Omega)}^2 + \|\text{div }\bar{\mathbf{z}}_{h}\|_{L^2(K)}^2$. Here,
\begin{align}\label{eq:residual_adjoint}
\mathcal{R}_{K}^{adj} := (\bar{\mathbf{y}}_{h}-\mathbf{y}_{\Omega}+\Delta \bar{\mathbf{z}}_{h}\!-\bar{u}_{h}\bar{\mathbf{z}}_{h} + \nabla\bar{\mathsf{r}}_{h})|_K,\qquad 
\mathcal{J}_{\gamma}^{adj} := \llbracket(\nabla \bar{\mathbf{z}}_{h}+\bar{\mathsf{r}}_{h}\mathbf{I})\cdot \mathbf{n}\rrbracket.
\end{align}
An application of Theorem \ref{eq:reliability}, with $\mathfrak{u}$ and $\boldsymbol{\mathfrak{f}}$ replaced by $\bar{u}_{h}$ and $\bar{\mathbf{y}}_{h}-\mathbf{y}_{\Omega}$, yields
\begin{eqnarray}\label{eq:bound_realiability_2}
\|\nabla(\check{\mathbf{z}}-\bar{\mathbf{z}}_{h})\|_{\mathbf{L}^2(\Omega)}^2+\|\check{\mathsf{r}}-\bar{\mathsf{r}}_{h}\|_{L^2(\Omega)}^2\lesssim \mathcal{E}_{adj,\T}^2.
\end{eqnarray}

It \AAF{is important to note that in the a posteriori error bounds \eqref{eq:bound_realiability_1} and \eqref{eq:bound_realiability_2}, the hidden constants depend on $\|\bar{u}_{h}\|_{L^{2}(\Omega)}$, as established in the proof of Theorem \ref{eq:reliability}. Despite this, since $\bar{u}_{h} \in \mathbb{U}_{ad,h}$, these constants can be uniformly controlled using the fact that $\|\bar{u}_{h}\|_{L^{2}(\Omega)} \leq \mathsf{b} |\Omega|^{1/2}$.}

After introducing a posteriori error estimators for the discretization of the state and adjoint equations, we define the following a posteriori local error indicator and error estimator for the discretization of the control variable:
\begin{eqnarray}\label{eq:estimator_control}
\mathcal{E}_{ct,K}:=\|\breve{u}-\bar{u}_{h}\|_{L^2(K)}, \qquad \mathcal{E}_{ct,\T}^2:= \sum_{K\in\T}\mathcal{E}_{ct,K}^2 = \|\breve{u}-\bar{u}_{h}\|_{L^2(\Omega)}^{2}.
\end{eqnarray}

Finally, we introduce an a posteriori error estimator for the fully discrete scheme, which consists of three contributions:
\begin{align}\label{eq:error_contributions}
\mathcal{E}_{ocp,\T}^2:=\mathcal{E}_{st,\T}^2 + \mathcal{E}_{adj,\T}^2 + \mathcal{E}_{ct,\T}^2.
\end{align}

To simplify the presentation and proof of the next result, we define the errors $\mathbf{e}_{\bar{\mathbf{y}}}:=\bar{\mathbf{y}}-\bar{\mathbf{y}}_{h}$, $e_{\bar{\mathsf{p}}}:=\bar{\mathsf{p}}-\bar{\mathsf{p}}_{h}$, $\mathbf{e}_{\bar{\mathbf{z}}}:=\bar{\mathbf{z}}-\bar{\mathbf{z}}_{h}$, $e_{\bar{\mathsf{r}}}:=\bar{\mathsf{r}}-\bar{\mathsf{r}}_{h}$, and $e_{\bar{u}}:=\bar{u}-\bar{u}_{h}$. We also define $(\breve{\mathbf{y}},\breve{\mathsf{p}})\in \mathbf{H}_0^1(\Omega)\times L_0^2(\Omega)$ as the solution of
\begin{equation}\label{eq:brinkamn_problem_state_auxA}
(\nabla \breve{\mathbf{y}},\nabla \mathbf{v})_{\mathbf{L}^2(\Omega)}+( \breve{u}\breve{\mathbf{y}}, \mathbf{v})_{\mathbf{L}^2(\Omega)}-(\breve{\mathsf{p}},\text{div }\mathbf{v})_{L^2(\Omega)} =(\mathbf{f},\mathbf{v})_{\mathbf{L}^2(\Omega)}
\end{equation}
and $(\mathsf{q},\text{div }\breve{\mathbf{y}})_{L^2(\Omega)}  = 0$ for all $\mathbf{v} \in \mathbf{H}_0^1(\Omega)$ and $\mathsf{q}\in L_0^2(\Omega)$, respectively. Finally, we define
$(\breve{\mathbf{z}},\breve{\mathsf{r}})\in\mathbf{H}_0^1(\Omega)\times L_0^2(\Omega)$ as the solution to
\begin{equation}\label{eq:brinkamn_problem_adjoint_auxA}
(\nabla \breve{\mathbf{z}},\nabla \mathbf{v})_{\mathbf{L}^2(\Omega)}+( \breve{u}\breve{\mathbf{z}}, \mathbf{v})_{\mathbf{L}^2(\Omega)}+(\breve{\mathsf{r}},\text{div }\mathbf{v})_{L^2(\Omega)} = 	(\breve{\mathbf{y}}-\mathbf{y}_{\Omega},\mathbf{v})_{\mathbf{L}^2(\Omega)}
\end{equation}
and $(\mathsf{q},\text{div }\breve{\mathbf{z}})_{L^2(\Omega)}  = 0$ for all $\mathbf{v} \in \mathbf{H}_0^1(\Omega)$ and $\mathsf{q} \in L_0^2(\Omega)$, respectively.

\begin{theorem}[global reliability of $\mathcal{E}_{ocp,\T}$]
\label{global_reliability_fully} 
Let $\bar{u}$ be a local solution of \eqref{eq:min_functional}--\eqref{eq:brinkamn_problem_state} such that \eqref{equivalence_condition} holds for $\tau > 0$, and let 
$(\bar{\mathbf{y}},\bar{\mathsf{p}})$ and $(\bar{\mathbf{z}},\bar{\mathsf{r}})$ be the corresponding state and adjoint state variables. Let $\bar{u}_{h}$ be a local minimum of the fully discrete scheme, and let $(\bar{\mathbf{y}}_h,\bar{\mathsf{p}}_h)$ and $(\bar{\mathbf{z}}_h,\bar{\mathsf{r}}_h)$ be the corresponding discrete state and adjoint state variables. If \eqref{eq:bound_aux} holds, then
\begin{align}\label{eq:estimates_global_reliability}
\|\nabla \mathbf{e}_{\bar{\mathbf{y}}} \|_{\mathbf{L}^2(\Omega)}^2
+
\|\nabla \mathbf{e}_{\bar{\mathbf{z}}} \|_{\mathbf{L}^2(\Omega)}^2
+
\| e_{\bar{\mathsf{p}}} \|_{L^2(\Omega)}^2
+
\| e_{\bar{\mathsf{r}}} \|_{L^2(\Omega)}^2 
+
\|e_{\bar{u}}\|_{L^2(\Omega)}^2
\lesssim \mathcal{E}_{ocp,\T}^2,
\end{align}
where the hidden constant is independent of the continuous and discrete optimal variables, the size of the elements in the mesh $\T$, and $\#\T$.
\end{theorem}
\begin{proof}
We divide the proof into six steps.

\emph{Step 1.} \emph{A bound for $\|e_{\bar{u}}\|_{L^2(\Omega)}$.} We use the variable $\breve{u}$ defined in \eqref{eq:representation_aux}, the definitions in \eqref{eq:estimator_control}, and a simple application of the triangle inequality to obtain
\begin{align}\label{eq:bound_00}
\|e_{\bar{u}}\|_{L^2(\Omega)}\leq \|\bar{u}-\breve{u}\|_{L^2(\Omega)}+\mathcal{E}_{ct,\T}.
\end{align}
It is therefore sufficient to bound $\|\bar{u}-\breve{u}\|_{L^2(\Omega)}$. To do this, we set $u=\breve{u}$ in \eqref{eq:ineq_variational} and $u=\bar{u}$ in \eqref{eq:variational_inequality_discrete_aux} to obtain $j'(\bar{u})(\breve{u}-\bar{u})\geq 0$ and $(\alpha \breve{u}-\bar{\mathbf{y}}_{h}\cdot \bar{\mathbf{z}}_{h},\bar{u}-\breve{u})_{L^2(\Omega)}\geq 0$.
With these estimates in hand, we use \eqref{eq:bound_aux_j_final} to obtain
\begin{multline}\label{eq:bound_01}
\mu\|\bar{u}-\breve{u}\|_{L^2(\Omega)}^2
\leq 
2 (j'(\breve{u})-j'(\bar{u}))(\breve{u}-\bar{u})
\leq 
2 j'(\breve{u})(\breve{u}-\bar{u})\\
= 
2 (\alpha \breve{u}-\breve{\mathbf{y}}\cdot \breve{\mathbf{z}},\breve{u}-\bar{u})_{L^2(\Omega)}
\leq 
2 (\bar{\mathbf{y}}_{h}\cdot \bar{\mathbf{z}}_{h}-\breve{\mathbf{y}}\cdot \breve{\mathbf{z}},\breve{u}-\bar{u})_{L^2(\Omega)},
\end{multline}
where $(\breve{\mathbf{y}},\breve{\mathsf{p}})$ and $(\breve{\mathbf{z}},\breve{\mathsf{r}})$ are the solutions of \eqref{eq:brinkamn_problem_state_auxA} and \eqref{eq:brinkamn_problem_adjoint_auxA}, respectively. \AAF{If we add and subtract the term $(\bar{\mathbf{y}}_{h}\cdot \breve{\mathbf{z}},\breve{u}-\bar{u})_{L^2(\Omega)}$ in \eqref{eq:bound_01} and use the Sobolev embedding $\mathbf{H}_0^1(\Omega) \hookrightarrow \mathbf{L}^4(\Omega)$, we obtain the bound}
$
	\|\bar{u} - \breve{u}\|_{L^2(\Omega)}\lesssim \|\nabla(\bar{\mathbf{z}}_{h} - \breve{\mathbf{z}})\|_{\mathbf{L}^2(\Omega)}\| \nabla\bar{\mathbf{y}}_{h}\|_{\mathbf{L}^2(\Omega)} +  \|\nabla(\bar{\mathbf{y}}_{h} - \breve{\mathbf{y}})\|_{\mathbf{L}^2(\Omega)}\| \nabla\breve{\mathbf{z}}\|_{\mathbf{L}^2(\Omega)}.
$
\AAF{Using the fact that $\bar{u}_{h} \in\mathcal{A}_0$, a basic stability bound for problem \eqref{eq:brinkamn_problem_state_fd} shows that $\| \nabla\bar{\mathbf{y}}_{h}\|_{\mathbf{L}^2(\Omega)}\lesssim \|\mathbf{f}\|_{\mathbf{L}^2(\Omega)}$. Similarly, the fact that $\breve{u} \in\mathcal{A}_0$ allows us to derive stability bounds for the solutions of problems \eqref{eq:brinkamn_problem_adjoint_auxA} and \eqref{eq:brinkamn_problem_state_auxA}. As a result, we can obtain that}
$
\|\nabla \breve{\mathbf{z}}\|_{\mathbf{L}^2(\Omega)}
\lesssim \|\mathbf{f}\|_{\mathbf{L}^2(\Omega)} + \|\mathbf{y}_{\Omega}\|_{\mathbf{L}^2(\Omega)}$.
\AAF{A collection of the derived stability bounds thus shows that} 
\begin{align}\label{eq:bound_03}
	\|\bar{u}-\breve{u}\|_{L^2(\Omega)}\lesssim \|\nabla(\bar{\mathbf{z}}_{h}-\breve{\mathbf{z}})\|_{\mathbf{L}^2(\Omega)}+ \|\nabla(\bar{\mathbf{y}}_{h}-\breve{\mathbf{y}})\|_{\mathbf{L}^2(\Omega)},
\end{align}
\EO{with a hidden constant that depends on $\|\mathbf{f}\|_{\mathbf{L}^2(\Omega)}$ and $\|\mathbf{y}_{\Omega}\|_{\mathbf{L}^2(\Omega)}$.} Let us now bound $\|\nabla(\bar{\mathbf{y}}_{h}-\breve{\mathbf{y}})\|_{\mathbf{L}^2(\Omega)}$. For this purpose, we use the pair $(\check{\mathbf{y}},\check{\mathsf{p}})$ defined as the solution to \eqref{eq:brinkamn_problem_state_aux} and the a posteriori error bound \eqref{eq:bound_realiability_1} to obtain
\begin{align}\label{eq:bound_04}
\|\nabla(\bar{\mathbf{y}}_{h}-\breve{\mathbf{y}})\|_{\mathbf{L}^2(\Omega)}\lesssim \|\nabla(\breve{\mathbf{y}}-\check{\mathbf{y}})\|_{\mathbf{L}^2(\Omega)}+\mathcal{E}_{st,\T}.
\end{align}
We now observe that $(\breve{\mathbf{y}}-\check{\mathbf{y}},\breve{\mathsf{p}}-\check{\mathsf{p}})\in \mathbf{H}_0^1(\Omega)\times L_0^2(\Omega)$ solves the following problem
\begin{multline*}\label{eq:brinkamn_problem_state_aux_error}
(\nabla (\breve{\mathbf{y}} - \check{\mathbf{y}}),\nabla \mathbf{v})_{\mathbf{L}^2(\Omega)}
+
( \breve{u}(\breve{\mathbf{y}}\!-\!\check{\mathbf{y}}), \mathbf{v})_{\mathbf{L}^2(\Omega)} 
-
(\breve{\mathsf{p}}\!-\!\check{\mathsf{p}},\text{div }\mathbf{v})_{L^2(\Omega)} 
\\
=  	(\check{\mathbf{y}}(\bar{u}_h\!-\!\breve{u}),\mathbf{v})_{\mathbf{L}^2(\Omega)},\quad
(\mathsf{q},\text{div }(\breve{\mathbf{y}}-\check{\mathbf{y}}))_{L^2(\Omega)} = 0\quad\forall (\mathbf{v},\mathsf{q})\in  \mathbf{H}_0^1(\Omega)\times L_0^2(\Omega).
\end{multline*}
\AAF{If we set $(\mathbf{v},\mathsf{q}) = (\breve{\mathbf{y}}-\check{\mathbf{y}},\breve{\mathsf{p}} - \check{\mathsf{p}})$ and use the fact that \EO{$\breve{u}\in \mathcal{A}_0$}, we can conclude that $\|\nabla (\breve{\mathbf{y}}-\check{\mathbf{y}})\|_{\mathbf{L}^2(\Omega)} \lesssim 
\|\nabla \check{\mathbf{y}}\|_{\mathbf{L}^2(\Omega)}\|\bar{u}_{h}-\breve{u}\|_{L^2(\Omega)}$. Similarly, we can use that $\bar{u}_{h} \in \mathcal{A}_0$ to obtain $\|\nabla \check{\mathbf{y}}\|_{\mathbf{L}^2(\Omega)}\lesssim\|\mathbf{f}\|_{\mathbf{L}^{2}(\Omega)}$. A combination of these estimates shows that}
% ...
$\|\nabla (\breve{\mathbf{y}}-\check{\mathbf{y}})\|_{\mathbf{L}^2(\Omega)} \lesssim 
\|\nabla \check{\mathbf{y}}\|_{\mathbf{L}^2(\Omega)}\|\bar{u}_{h}-\breve{u}\|_{L^2(\Omega)}\lesssim \mathcal{E}_{ct,\T}$, where we have also used \eqref{eq:estimator_control}.
Substituting this bound into \eqref{eq:bound_04} and the one obtained in \eqref{eq:bound_03}, we obtain
\begin{align}\label{eq:bound_05}
\|\bar{u}-\breve{u}\|_{L^2(\Omega)}\lesssim \|\nabla(\bar{\mathbf{z}}_{h}-\breve{\mathbf{z}})\|_{\mathbf{L}^2(\Omega)}
+
\AAF{\mathcal{E}_{st,\T}+\mathcal{E}_{ct,\T}}.
\end{align}

We now bound $\|\nabla(\bar{\mathbf{z}}_{h}-\breve{\mathbf{z}})\|_{\mathbf{L}^2(\Omega)}$. \EO{To do this, we use the pair $(\check{\mathbf{z}},\check{\mathsf{p}})$ that solves problem \eqref{eq:brinkamn_problem_adjoint_aux}, the a posteriori error bound \eqref{eq:bound_realiability_2}, and the triangle inequality to obtain} 
\begin{align}\label{eq:bound_06}
\|\nabla(\bar{\mathbf{z}}_{h}-\breve{\mathbf{z}})\|_{\mathbf{L}^2(\Omega)}\lesssim \mathcal{E}_{adj,\T}+\|\nabla(\breve{\mathbf{z}}-\check{\mathbf{z}})\|_{\mathbf{L}^2(\Omega)}.
\end{align}
To bound $\|\nabla(\breve{\mathbf{z}}-\check{\mathbf{z}})\|_{\mathbf{L}^2(\Omega)}$, we note that $(\breve{\mathbf{z}}-\check{\mathbf{z}},\breve{\mathsf{r}}-\check{\mathsf{r}})\in\mathbf{H}_0^1(\Omega)\times L_0^2(\Omega)$ solves
\begin{multline*}
(\nabla (\breve{\mathbf{z}}-\check{\mathbf{z}}),\nabla \mathbf{v})_{\mathbf{L}^2(\Omega)}+( \breve{u}(\breve{\mathbf{z}}-\check{\mathbf{z}}), \mathbf{v})_{\mathbf{L}^2(\Omega)}+(\breve{\mathsf{r}}-\check{\mathsf{r}},\text{div }\mathbf{v})_{L^2(\Omega)} = ((\bar{u}_{h}-\breve{u})\check{\mathbf{z}},\mathbf{v})_{\mathbf{L}^2(\Omega)}\\+(\breve{\mathbf{y}}-\bar{\mathbf{y}}_{h},\mathbf{v})_{\mathbf{L}^2(\Omega)},\quad
(\mathsf{q},\text{div }(\breve{\mathbf{z}}-\check{\mathbf{z}}))_{L^2(\Omega)}  =  0 \quad\forall (\mathbf{v},\mathsf{q})\in  \mathbf{H}_0^1(\Omega)\times L_0^2(\Omega). 
\end{multline*}
If we set \AAF{$(\mathbf{v},\mathsf{q}) = (\breve{\mathbf{z}}-\check{\mathbf{z}},\breve{\mathsf{r}}-\check{\mathsf{r}})$ and use that $\breve{u} \in\mathcal{A}_0$, we can conclude that 
$
\|\nabla(\breve{\mathbf{z}}-\check{\mathbf{z}})\|_{\mathbf{L}^2(\Omega)}\lesssim \|\nabla\check{\mathbf{z}}\|_{\mathbf{L}^2(\Omega)}\|\bar{u}_{h}-\breve{u}\|_{L^2(\Omega)}+\|\nabla(\breve{\mathbf{y}}-\bar{\mathbf{y}}_{h})\|_{\mathbf{L}^2(\Omega)}.
$
Note that $\|\nabla \check{\mathbf{z}}\|_{\mathbf{L}^2(\Omega)} \lesssim \|\mathbf{f}\|_{\mathbf{L}^2(\Omega)}+\|\mathbf{y}_{\Omega}\|_{\mathbf{L}^2(\Omega)}$ because $\bar{u}_{h} \in \mathcal{A}_0$ and $\|\nabla \bar{\mathbf{y}}_{h}\|_{\mathbf{L}^{2}(\Omega)}\lesssim\|\mathbf{f}\|_{\mathbf{L}^{2}(\Omega)}$.
%
%
On the other hand, based on previous results, we can conclude that
$\|\nabla (\breve{\mathbf{y}}-\bar{\mathbf{y}}_{h})\|_{\mathbf{L}^2(\Omega)}
\leq 
\|\nabla (\breve{\mathbf{y}}-\check{\mathbf{y}})\|_{\mathbf{L}^2(\Omega)}
+
\|\nabla (\check{\mathbf{y}}-\bar{\mathbf{y}}_{h})\|_{\mathbf{L}^2(\Omega)}
\lesssim \mathcal{E}_{st,\T}+\mathcal{E}_{ct,\T}.
$
Thus,} $\|\nabla(\breve{\mathbf{z}}-\check{\mathbf{z}})\|_{\mathbf{L}^2(\Omega)}\lesssim \mathcal{E}_{st,\T}+\mathcal{E}_{ct,\T}$. Substitute this estimate into \eqref{eq:bound_06} and the one obtained in \eqref{eq:bound_05} to obtain $\|\bar{u}-\breve{u}\|_{L^2(\Omega)}\lesssim \mathcal{E}_{st,\T}+\mathcal{E}_{adj,\T}+\mathcal{E}_{ct,\T}$. In view of \eqref{eq:bound_00}, we finally \AAF{conclude that}
\begin{align}\label{eq:bound_10}
\|e_{\bar{u}}\|_{L^2(\Omega)}\lesssim \mathcal{E}_{st,\T}+\mathcal{E}_{adj,\T}+\mathcal{E}_{ct,\T}.
\end{align}

\emph{Step 2.} \emph{A bound for} $\|\nabla \mathbf{e}_{\bar{\mathbf{y}}}\|_{\mathbf{L}^2(\Omega)}$. With the auxiliary pair $(\check{\mathbf{y}},\check{\mathsf{p}})$ and the a posteriori error bound \eqref{eq:bound_realiability_1} in hand, we bound $\|\nabla \mathbf{e}_{\bar{\mathbf{y}}}\|_{\mathbf{L}^2(\Omega)}$ as follows:
\begin{align}\label{eq:bound_11}
\|\nabla \mathbf{e}_{\bar{\mathbf{y}}}\|_{\mathbf{L}^2(\Omega)}\leq \|\nabla(\bar{\mathbf{y}}-\check{\mathbf{y}})\|_{\mathbf{L}^2(\Omega)}+\mathcal{E}_{st,\T}.
\end{align}
To bound $\|\nabla(\bar{\mathbf{y}}-\check{\mathbf{y}})\|_{\mathbf{L}^2(\Omega)}$, we note that $(\bar{\mathbf{y}}-\check{\mathbf{y}},\bar{\mathsf{p}}-\check{\mathsf{p}})\times \mathbf{H}_0^1(\Omega)\times L_0^2(\Omega)$ solves 
\begin{multline}\label{eq:brinkamn_problem_state_aux_errorB}
(\nabla (\bar{\mathbf{y}}-\check{\mathbf{y}}),\nabla \mathbf{v})_{\mathbf{L}^2(\Omega)}
+
( \bar{u}(\bar{\mathbf{y}}-\check{\mathbf{y}}), \mathbf{v})_{\mathbf{L}^2(\Omega)}
-
(\bar{\mathsf{p}}-\check{\mathsf{p}},\text{div }\mathbf{v})_{L^2(\Omega)}
\\
=
((\bar{u}_h-\bar{u})\check{\mathbf{y}},\mathbf{v})_{\mathbf{L}^2(\Omega)},
\quad
(\mathsf{q},\text{div }(\bar{\mathbf{y}}-\check{\mathbf{y}}))_{L^2(\Omega)}  =  0 
\quad
\forall (\mathbf{v},\mathsf{q})\in  \mathbf{H}_0^1(\Omega)\times L_0^2(\Omega). 
\end{multline}
If we set \AAF{$(\mathbf{v},\mathsf{q})=(\bar{\mathbf{y}}-\check{\mathbf{y}},\bar{\mathsf{p}}-\check{\mathsf{p}})$ and use $\|\nabla \check{\mathbf{y}}\|_{\mathbf{L}^{2}(\Omega)}\lesssim\|\mathbf{f}\|_{\mathbf{L}^{2}(\Omega)}$ and} \eqref{eq:bound_10}, we immediately obtain the bound
$
\|\nabla (\bar{\mathbf{y}}-\check{\mathbf{y}})\|_{\mathbf{L}^2(\Omega)}\lesssim 
\|\mathbf{f}\|_{\mathbf{L}^2(\Omega)} \|e_{\bar{u}}\|_{L^2(\Omega)}\lesssim \mathcal{E}_{st,\T}+\mathcal{E}_{adj,\T}+\mathcal{E}_{ct,\T}.
$
Finally, we replace this bound into \eqref{eq:bound_11} to obtain the a posteriori error estimate
\begin{align}\label{eq:bound_12}
\|\nabla \mathbf{e}_{\bar{\mathbf{y}}}\|_{\mathbf{L}^2(\Omega)}\lesssim \mathcal{E}_{st,\T}+\mathcal{E}_{adj,\T}+\mathcal{E}_{ct,\T}.
\end{align}

\emph{Step 3.} \emph{A bound for $\|\nabla \mathbf{e}_{\bar{\mathbf{z}}}\|_{\mathbf{L}^2(\Omega)}$}. Using similar arguments as in Step 2, we obtain
$
\|\nabla \mathbf{e}_{\bar{\mathbf{z}}}\|_{\mathbf{L}^2(\Omega)}\lesssim \mathcal{E}_{st,\T}+\mathcal{E}_{adj,\T}+\mathcal{E}_{ct,\T}.
$
For simplicity, we omit the details.

\emph{Step 4.} \emph{A bound for} $\|e_{\bar{\mathsf{p}}}\|_{L^2(\Omega)}$. A basic application of the bound \eqref{eq:bound_realiability_1} yields
\begin{align}\label{eq:bound_15}
\|e_{\bar{\mathsf{p}}}\|_{L^2(\Omega)}\lesssim \|\bar{\mathsf{p}}-\check{\mathsf{p}}\|_{L^2(\Omega)}+\mathcal{E}_{st,\T}.
\end{align}
We now control $\|\bar{\mathsf{p}}-\check{\mathsf{p}}\|_{L^2(\Omega)}$ using the inf--sup condition \eqref{eq:infsup}, the fact that $(\bar{\mathbf{y}}-\check{\mathbf{y}},\bar{\mathsf{p}}-\check{\mathsf{p}})$ solves problem \eqref{eq:brinkamn_problem_state_aux_errorB}, and the standard embedding $\mathbf{H}_0^1(\Omega) \hookrightarrow \mathbf{L}^4(\Omega)$. In fact,
\begin{multline}\label{eq:infsup_state_aux}
\|\bar{\mathsf{p}}-\check{\mathsf{p}}\|_{L^2(\Omega)}\lesssim \sup_{\mathbf{v}\in \mathbf{H}_0^1(\Omega)}\dfrac{(\bar{\mathsf{p}}-\check{\mathsf{p}},\text{div }\mathbf{v})_{L^2(\Omega)}}{\|\nabla \mathbf{v}\|_{\mathbf{L}^2(\Omega)}}\lesssim \|\nabla (\bar{\mathbf{y}}-\check{\mathbf{y}}) \|_{\mathbf{L}^2(\Omega)}\\
+\|\nabla (\bar{\mathbf{y}}-\check{\mathbf{y}}) \|_{\mathbf{L}^2(\Omega)} \|\bar{u}\|_{L^2(\Omega)}+\|\nabla \check{\mathbf{y}}\|_{\mathbf{L}^2(\Omega)}\|\bar{u}_{h}-\bar{u}\|_{L^2(\Omega)}.
\end{multline}
We \AAF{recall} that $\|\nabla \check{\mathbf{y}}\|_{\mathbf{L}^2(\Omega)}\lesssim \| \mathbf{f}\|_{\mathbf{L}^2(\Omega)}$. \AAF{On the other hand, $\|\bar{u}\|_{L^{2}(\Omega)}\leq \mathsf{b}|\Omega|^{1/2}$ because $\bar{u}\in \mathbb{U}_{ad}$.} We now refer to the bound $\|\nabla (\bar{\mathbf{y}}-\check{\mathbf{y}})\|_{\mathbf{L}^2(\Omega)} \lesssim \mathcal{E}_{ocp,\T}$, which was derived in Step 2, and the error estimate \eqref{eq:bound_10} to obtain $\|\bar{\mathsf{p}}-\check{\mathsf{p}}\|_{L^2(\Omega)}\lesssim \mathcal{E}_{ocp,\T}$. In view of  \eqref{eq:bound_15}, this estimate implies 
\begin{align}\label{eq:bound_16}
\|e_{\bar{\mathsf{p}}}\|_{L^2(\Omega)}\lesssim \mathcal{E}_{st,\T}+\mathcal{E}_{adj,\T}+\mathcal{E}_{ct,\T}.
\end{align}

\emph{Step 5.} \emph{A bound for $\|e_{\bar{\mathsf{r}}}\|_{L^2(\Omega)}$.} Using similar arguments as in Step 4, we obtain.
$
\|e_{\bar{\mathsf{r}}}\|_{L^2(\Omega)}\lesssim \mathcal{E}_{st,\T}+\mathcal{E}_{adj,\T}+\mathcal{E}_{ct,\T}.
$
For simplicity, we omit the details.

\emph{Step 6.} The desired estimate \eqref{eq:estimates_global_reliability}  follows from combining the bounds obtained in \AAF{all the previous steps}. This concludes the proof.
\end{proof}


\subsection{Reliability analysis: the semidiscrete scheme} 
\label{Real:semi_discrete} 
In this section, we propose an a posteriori error estimator for the semidiscrete scheme \eqref{eq:brinkamn_problem_state_sd}--\eqref{eq:brinkman_problem_adjoint_semi} and obtain a global reliability bound. 
As in the previous section, we introduce $(\check{\mathbf{y}},\check{\mathsf{p}})\in \mathbf{H}_0^1(\Omega)\times L_0^2(\Omega)$ as the solution to the following problem:
\begin{equation}
\label{eq:brinkamn_problem_state_aux_semi}
(\nabla \check{\mathbf{y}},\nabla \mathbf{v})_{\mathbf{L}^2(\Omega)}+(\bar{\mathsf{u}}\check{\mathbf{y}}, \mathbf{v})_{\mathbf{L}^2(\Omega)}-(\check{\mathsf{p}},\text{div }\mathbf{v})_{L^2(\Omega)} = 	(\mathbf{f},\mathbf{v})_{\mathbf{L}^2(\Omega)}
\end{equation}
and $(\mathsf{q},\text{div }\check{\mathbf{y}})_{L^2(\Omega)}  = 0$ for all $\mathbf{v} \in \mathbf{H}_0^1(\Omega)$ and $\mathsf{q} \in L_0^2(\Omega)$, respectively.
We define the following error estimator associated with the discretization of the state equations:
\begin{equation*}\label{eq:estimator_st_semi}
\mathfrak{E}_{st,\T}^2:= \sum_{K\in\T}\mathfrak{E}_{st,K}^2, 
\end{equation*}
where $\mathfrak{E}_{st,K}^2:=h_K^2  \|\mathfrak{R}_{K}^{st}\|_{\mathbf{L}^2(K)}^2 + h_K  \|\mathfrak{J}_{\gamma}^{st}\|_{\mathbf{L}^2(\partial K\setminus\partial \Omega)}^2+\|\text{div }\bar{\mathbf{y}}_{h}\|_{L^2(K)}^2$. Here,
\begin{align*}\label{eq:residual_state_semi}
\mathfrak{R}_{K}^{st}\ = (\mathbf{f}+\Delta \bar{\mathbf{y}}_{h} - \bar{\mathsf{u}}\bar{\mathbf{y}}_{h} - \nabla\bar{\mathsf{p}}_{h})|_K,
\qquad 
\mathfrak{J}_{\gamma}^{st} := \llbracket(\nabla \bar{\mathbf{y}}_{h}-\bar{\mathsf{p}}_{h}\mathbf{I})\cdot \mathbf{n}\rrbracket.
\end{align*}
We note that the result of Theorem \ref{eq:reliability} allows us to conclude the following a posteriori error bound:
$\|\nabla(\check{\mathbf{y}}-\bar{\mathbf{y}}_{h})\|_{\mathbf{L}^2(\Omega)} +\|\check{\mathsf{p}}-\bar{\mathsf{p}}_{h}\|_{L^2(\Omega)} \lesssim \mathfrak{E}_{st,\T}$.

We also introduce the pair $(\check{\mathbf{z}},\check{\mathsf{r}})\in\mathbf{H}_0^1(\Omega)\times L_0^2(\Omega)$ as the solution to 
\begin{equation}\label{eq:brinkamn_problem_adjoint_aux_semi}
(\nabla \check{\mathbf{z}},\nabla \mathbf{v})_{\mathbf{L}^2(\Omega)}+( \bar{\mathsf{u}}\check{\mathbf{z}}, \mathbf{v})_{\mathbf{L}^2(\Omega)}+(\check{\mathsf{r}},\text{div }\mathbf{v})_{L^2(\Omega)} = (\bar{\mathbf{y}}_{h}-\mathbf{y}_{\Omega},\mathbf{v})_{\mathbf{L}^2(\Omega)}
\end{equation}
and $(\mathsf{q},\text{div }\check{\mathbf{z}})_{L^2(\Omega)}  =  0$ for all $\mathbf{v} \in \mathbf{H}_0^1(\Omega)$ and $\mathsf{q} \in L_0^2(\Omega)$, respectively.
We define the following error estimator associated with the discretization of the adjoint equations:
\begin{equation*}\label{eq:estimator_AD_semi}
\mathfrak{E}_{adj,\T}^2:= \sum_{K\in\T}\mathfrak{E}_{adj,K}^2, 
\end{equation*}
where $\mathfrak{E}_{adj,K}^2:=h_K^2  \|\mathfrak{R}_{K}^{adj}\|_{\mathbf{L}^2(K)}^2+h_K  \|\mathfrak{J}_{\gamma}^{adj}\|_{\mathbf{L}^2(\partial K\setminus\partial \Omega)}^2+\|\text{div }\bar{\mathbf{z}}_{h}\|_{L^2(K)}^2$. Here,
\begin{align*}\label{eq:residual_adjoint_semi}
\mathfrak{R}_{K}^{adj} := (\bar{\mathbf{y}}_{h}-\mathbf{y}_{\Omega}+\Delta \bar{\mathbf{z}}_{h} -\bar{\mathsf{u}}\bar{\mathbf{z}}_{h} + \nabla\bar{\mathsf{r}}_{h})|_K,
\qquad \mathfrak{J}_{\gamma}^{adj} := \llbracket(\nabla \bar{\mathbf{z}}_{h} +\bar{\mathsf{r}}_{h}\mathbf{I})\cdot \mathbf{n}\rrbracket.
\end{align*}
An application of Theorem \ref{eq:reliability} immediately yields the following a posteriori error estimate: $
\|\nabla(\check{\mathbf{z}}-\bar{\mathbf{z}}_{h})\|_{\mathbf{L}^2(\Omega)} +\|\check{\mathsf{r}}-\bar{\mathsf{r}}_{h}\|_{L^2(\Omega)} \lesssim \mathfrak{E}_{adj,\T}$. 

As \AAF{mentioned in Section \ref{Real:fully_discrete}, in the previously derived a posteriori error bounds, the hidden constants depend on $\|\bar{\mathsf{u}}\|_{L^{2}(\Omega)}$. Since $\bar{\mathsf{u}}\in\mathbb{U}_{ad}$, these constants can be uniformly controlled using the fact that $\|\bar{\mathsf{u}}\|_{L^{2}(\Omega)}\leq \mathsf{b}|\Omega|^{1/2}$.}

To present the following result, we define $\mathsf{e}_{\bar{u}}:=\bar{u}-\bar{\mathsf{u}}$ and $\mathfrak{E}_{ocp,\T}:=\mathfrak{E}_{st,\T}+\mathfrak{E}_{adj,\T}$.

\begin{theorem}[global reliability $\mathfrak{E}_{ocp,\T}$] 
Let $\bar{u}$ be a local solution of \eqref{eq:min_functional}--\eqref{eq:brinkamn_problem_state} such that \eqref{equivalence_condition} holds for $\tau > 0$, and let 
$(\bar{\mathbf{y}},\bar{\mathsf{p}})$ and $(\bar{\mathbf{z}},\bar{\mathsf{r}})$ be the corresponding state and adjoint state variables. Let $\bar{\mathsf{u}}$ be a local minimum of the semidiscrete scheme, and let $(\bar{\mathbf{y}}_h,\bar{\mathsf{p}}_h)$ and $(\bar{\mathbf{z}}_h,\bar{\mathsf{r}}_h)$ be the corresponding discrete state and adjoint state variables. If \eqref{eq:bound_aux} holds, then
%
\begin{align}\label{eq:estimates_global_reliability_semi}
\|\nabla \mathbf{e}_{\bar{\mathbf{y}}} \|_{\mathbf{L}^2(\Omega)}^2 + \|\nabla \mathbf{e}_{\bar{\mathbf{z}}} \|_{\mathbf{L}^2(\Omega)}^2 + \| e_{\bar{\mathsf{p}}} \|_{L^2(\Omega)}^2 + \| e_{\bar{\mathsf{r}}} \|_{L^2(\Omega)}^2 + \|\mathsf{e}_{\bar{u}}\|_{L^2(\Omega)}^2\lesssim \mathfrak{E}_{ocp,\T}^2,
\end{align}
where the hidden constant is independent of the continuous and discrete optimal variables, the size of the elements in the mesh $\T$, and $\#\T$.
\end{theorem}
\begin{proof}
We note that under the particular conditions entailed by the semidiscrete scheme, the auxiliary variable $\breve{u}$ defined in \eqref{eq:representation_aux} coincides with $\bar{\mathsf{u}}$. \EO{Consequently}, the variable $(\breve{\mathbf{y}},\breve{\mathsf{p}})$ defined as the solution \EO{to} \eqref{eq:brinkamn_problem_state_auxA} coincides with $(\check{\mathbf{y}},\check{\mathsf{p}})$ defined as the solution \EO{to} \eqref{eq:brinkamn_problem_state_aux_semi}. Applying the same arguments as in the proof of Theorem \ref{global_reliability_fully} to this simplified situation, we can derive the bound \eqref{eq:estimates_global_reliability_semi}.  For brevity, we omit the details.
\end{proof}


\subsection{Efficiency analysis: the fully discrete scheme}\label{efficiency_fully}

In this section, we derive local efficiency estimates for the local a posteriori error indicators $\mathcal{E}_{st,K}$ and $\mathcal{E}_{adj,K}$, as well as a global efficiency estimate for the a posteriori error estimator $\mathcal{E}_{ct,\T}$. To do this, we use standard bubble function arguments. 

We start our analysis by introducing the following notation: for an edge, triangle, or tetrahedron $G$, let $\mathcal{V}(G)$ be the set of vertices of $G$. Using this notation, we introduce the standard element and edge bubble functions: For $K\in\mathscr{T}$ and $\gamma\in\mathscr{S}$, 
\begin{equation}\label{def:standard_bubbles}
\varphi^{}_{K}=
(d+1)^{d+1}\prod_{\textsc{v} \in \mathcal{V}(K)} \lambda^{}_{\textsc{v}},
\qquad
\varphi^{}_{\gamma}=
d^d \prod_{\textsc{v} \in \mathcal{V}(\gamma)}\lambda^{}_{\textsc{v}}|^{}_{K'},
\quad 
K' \subset \mathcal{N}_{\gamma}.
\end{equation} 
Here, $\lambda_{\textsc{v}}$ denotes the barycentric coordinate function associated with the vertex $\textsc{v}$, and $\mathcal{N}_{\gamma}$ corresponds to the patch composed of the two elements of $\mathscr{T}$ that share $\gamma$.

The following identities are essential for performing an efficiency analysis. Since the pair $(\bar{\mathbf{y}},\bar{\mathsf{\mathsf{p}}})\in\mathbf{H}_0^1(\Omega)\times L_0^2(\Omega)$ solves \eqref{eq:brinkamn_problem_state}, where $u$ is replaced by $\bar{u}$, an elementwise integration by parts formula yields
\begin{multline}\label{eq:brinkamn_problem01_aux_error}
(\nabla \mathbf{e}_{\bar{\mathbf{y}}},\nabla \mathbf{v})_{\mathbf{L}^2(\Omega)}
+
( e_{\bar{u}}\bar{\mathbf{y}}, \mathbf{v})_{\mathbf{L}^2(\Omega)}
-
(e_{\bar{\mathsf{p}}},\text{div }\mathbf{v})_{L^2(\Omega)} 
+
(\mathsf{q},\text{div }\mathbf{e}_{\bar{\mathbf{y}}})_{L^2(\Omega)}
\\
= 
-(\bar{u}_{h}\mathbf{e}_{\bar{\mathbf{y}}},\mathbf{v})_{\mathbf{L}^2(\Omega)}
+
\sum_{\gamma\in \mathscr{S}} (\llbracket(\nabla \bar{\mathbf{y}}_{h}-\bar{\mathsf{p}}_{h}\mathbf{I})\cdot \mathbf{n}\rrbracket,\mathbf{v})_{\mathbf{L}^2(\gamma)}
-
\sum_{K\in\T}(\mathsf{q},\text{div }\bar{\mathbf{y}}_{h})_{\mathbf{L}^2(K)}\\
+
\sum_{K\in\T}\left[(\mathsf{P}_{K}\mathbf{f}+\Delta \bar{\mathbf{y}}_{h}-\bar{u}_{h}\bar{\mathbf{y}}_{h}-\nabla\bar{\mathsf{p}}_{h},\mathbf{v})_{\mathbf{L}^2(K)}
+
(\mathbf{f}-\mathsf{P}_{K}\mathbf{f},\mathbf{v})_{\mathbf{L}^2(K)}\right],
\end{multline}
for all $\mathbf{v}\in \mathbf{H}_0^1(\Omega)$ and $\mathsf{q}\in L_0^2(\Omega)$. \EO{Here,} $\mathsf{P}_K$ denotes the $L^2$-projection operator onto piecewise constant functions over $K\in\T$. Similarly, since $(\bar{\mathbf{z}},\bar{\mathsf{\mathsf{r}}})\in\mathbf{H}_0^1(\Omega)\times L_0^2(\Omega)$ solves the adjoint problem \eqref{eq:brinkamn_problem_adjoint} with $\mathbf{y}=\bar{\mathbf{y}}$ and $u=\bar{u}$, we have
\begin{multline}\label{eq:brinkamn_problem02_aux_error}
(\nabla \mathbf{e}_{\bar{\mathbf{z}}},\nabla \mathbf{v})_{\mathbf{L}^2(\Omega)}+( e_{\bar{u}}\bar{\mathbf{z}}, \mathbf{v})_{\mathbf{L}^2(\Omega)}+(e_{\bar{\mathsf{r}}},\text{div }\mathbf{v})_{L^2(\Omega)} + (\mathsf{q},\text{div }\mathbf{e}_{\bar{\mathbf{z}}})_{L^2(\Omega)}\\= -(\bar{u}_{h}\mathbf{e}_{\bar{\mathbf{z}}},\mathbf{v})_{\mathbf{L}^2(\Omega)}+ (\mathbf{e}_{\bar{\mathbf{y}}},\mathbf{v})+ \sum_{\gamma\in \mathscr{S}} (\llbracket(\nabla \bar{\mathbf{z}}_{h}+\bar{\mathsf{r}}_{h}\mathbf{I})\cdot \mathbf{n}\rrbracket,\mathbf{v})_{\mathbf{L}^2(\gamma)}-\sum_{K\in\T}(\mathsf{q},\text{div }\bar{\mathbf{z}}_{h})_{\mathbf{L}^2(K)}
\\
+
\sum_{K\in\T}
\left[
(\bar{\mathbf{y}}_{h}-\mathsf{P}_{K}\mathbf{y}_{\Omega}+\Delta \bar{\mathbf{z}}_{h}-\bar{u}_{h}\bar{\mathbf{z}}_{h}+\nabla\bar{\mathsf{r}}_{h},\mathbf{v})_{\mathbf{L}^2(K)}+(\mathsf{P}_{K}\mathbf{y}_{\Omega}-\mathbf{y}_{\Omega},\mathbf{v})_{\mathbf{L}^2(K)}
\right]
\end{multline}
for all $\mathbf{v}\in \mathbf{H}_0^1(\Omega)$ and $\mathsf{q}\in L_0^2(\Omega)$.
As a final element, we introduce, for $\mathbf{v}\in \mathbf{L}^2(\Omega)$ and $\mathcal{M}\subset\T$, the oscillation term
\begin{align*}
\text{osc}_{\mathscr{T}}(\mathbf{v},\mathcal{M}):= \left( \sum_{K\in\mathcal{M}}h_{K}^2 \|\mathbf{v}-\mathsf{P}_K \mathbf{v}\|_{\mathbf{L}^2(K)}^2 \right)^{\frac{1}{2}}.
\end{align*}

With these ingredients in hand, we now proceed to obtain local efficiency properties for the indicator $\mathcal{E}_{st,K}$ defined in \S\ref{Real:fully_discrete}.

\begin{theorem}[local efficiency of $\mathcal{E}_{st,K}$]\label{eq:efficiency_est}
Let $\bar{u}$ be a local solution of \eqref{eq:min_functional}--\eqref{eq:brinkamn_problem_state} such that \eqref{equivalence_condition} holds for $\tau > 0$, and let $(\bar{\mathbf{y}},\bar{\mathsf{p}})$ and $(\bar{\mathbf{z}},\bar{\mathsf{r}})$ be the corresponding state and ajoint state variables. Let $\bar{u}_h$ be a local minimum of the fully discrete scheme, and let $(\bar{\mathbf{y}}_h,\bar{\mathsf{p}}_h)$ and $(\bar{\mathbf{z}}_h,\bar{\mathsf{r}}_h)$ be the corresponding discrete state and adjoint state variables. Then, for $K\in\T$, the local error indicator $\mathcal{E}_{st,K}$ satisfies the bound
\begin{multline}\label{eq:efficiency_estimates_est}
\mathcal{E}_{st,K} \lesssim\|\nabla \mathbf{e}_{\bar{\mathbf{y}}} \|_{\mathbf{L}^2(\mathcal{N}_{K})}+\| e_{\bar{\mathsf{p}}} \|_{L^2(\mathcal{N}_{K})}
+ h_{K}\|e_{\bar{u}}\|_{L^2(\mathcal{N}_{K})}\\
+h_{K}\|\mathbf{e}_{\bar{\mathbf{y}}}\|_{\mathbf{L}^2(\mathcal{N}_{K})}+\mathrm{osc}_{\mathscr{T}}(\mathbf{f},\mathcal{N}_{K}),
\end{multline}
where $\mathcal{N}_{K}$ is defined in \eqref{eq:patch}. The hidden constant is independent of the continuous and discrete optimal variables, the size of the elements in the mesh $\T$, and $\#\T$.
\end{theorem}
\begin{proof}
We proceed in four steps.

\emph{Step 1.} Let $K\in\T$. Define $\tilde{\mathcal{R}}_{K}^{st}:=(\mathsf{P}_{K}\mathbf{f}+\Delta \bar{\mathbf{y}}_{h}-\bar{u}_{h}\bar{\mathbf{y}}_{h}-\nabla\bar{\mathsf{p}}_{h})|_K$. A first bound for the term $h_K^2\|\mathcal{R}_{K}^{st}\|_{\mathbf{L}^2(K)}^2$ follows from the triangle inequality:
\begin{align}\label{eq:bound01_eff}
h_K^2\|\mathcal{R}_{K}^{st}\|_{\mathbf{L}^2(K)}^2\lesssim h_K^2 \|\tilde{\mathcal{R}}_{K}^{st}\|_{\mathbf{L}^2(K)}^2 + \text{osc}_{\mathscr{T}}(\mathbf{f},K)^2.
\end{align}
To bound $h_K^2 \|\tilde{\mathcal{R}}_{K}^{st}\|_{\mathbf{L}^2(K)}^2$, we set $\mathbf{v}=\varphi_K \tilde{\mathcal{R}}_{K}^{st}$ and $\mathsf{q}=0$ in \eqref{eq:brinkamn_problem01_aux_error}, where $\varphi_K$ denotes the bubble function defined in \eqref{def:standard_bubbles}, and use standard properties of $\varphi_K$ and inverse inequalities. These arguments show that
\begin{multline}\label{eq:bound02_eff}
\|\tilde{\mathcal{R}}_{K}^{st}\|_{\mathbf{L}^2(K)} \lesssim h_{K}^{-1} [\|\nabla \mathbf{e}_{\bar{\mathbf{y}}}\|_{\mathbf{L}^2(K)}
+
\| e_{\bar{\mathsf{p}}} \|_{L^2(K)}]+\GC{\|\bar{\mathbf{y}}\|_{\mathbf{L}^{\infty}(K)}\|e_{\bar{u}}\|_{L^2(K)}}
\\
+
\GC{h_K^{-d/2}\|\bar{u}_h\|_{L^2(K)}\| \mathbf{e}_{\bar{\mathbf{y}}}\|_{\mathbf{L}^2(K)}}+\|\mathbf{f}-\mathsf{P}_{K}\mathbf{f}\|_{\mathbf{L}^2(K)}.
\end{multline}
\GC{We note that $\|\bar{u}_{h}\|_{L^{2}(K)}\leq\mathsf{b}|K|^{1/2}\lesssim h_{K}^{d/2}$ because $\bar{u}_{h}\in\mathbb{U}_{ad,h}$. On the other hand, Theorem \ref{theorem_regularity_Linfty} guarantees that $\|\bar{\mathbf{y}}\|_{\mathbf{L}^{\infty}(K)}\leq \|\bar{\mathbf{y}}\|_{\mathbf{L}^{\infty}(\Omega)}\lesssim \|\mathbf{f}\|_{\mathbf{L}^2(\Omega)}$. We can thus obtain}
\begin{multline}\label{eq:bound03_eff}
h_K\|\tilde{\mathcal{R}}_{K}^{st}\|_{\mathbf{L}^2(K)}  
\lesssim  
\|\nabla \mathbf{e}_{\bar{\mathbf{y}}}\|_{\mathbf{L}^2(K)}
+\| e_{\bar{\mathsf{p}}} \|_{L^2(K)}
+\AAF{h_K\|e_{\bar{u}}\|_{L^2(K)}}
\\
+
h_K\| \mathbf{e}_{\bar{\mathbf{y}}}\|_{\mathbf{L}^2(K)}
+
\text{osc}_{\mathscr{T}}(\mathbf{f},K).
\end{multline}
If we replace this bound into \eqref{eq:bound01_eff}, we obtain the desired bound for \GC{$h_K^2 \|\mathcal{R}_{K}^{st}\|_{\mathbf{L}^2(K)}^2$.}

\emph{Step 2.} Let $K\in\T$ and $\gamma\in \mathscr{S}$. We now control $h_K^{\frac{1}{2}}\|\mathcal{J}_{\gamma}^{st}\|_{\mathbf{L}^2(\gamma)}$. To do so, we set $\mathbf{v}=\varphi_{\gamma}\mathcal{J}_{\gamma}^{st}$ and $\mathsf{q}=0$ in \eqref{eq:brinkamn_problem01_aux_error} and \GC{proceed with arguments similar to those used to obtain estimate} \eqref{eq:bound03_eff}:
\begin{multline}\label{eq:bound04_eff}
	h_{K}^{\frac{1}{2}}\|\mathcal{J}_{\gamma}^{st}\|_{\mathbf{L}^2(\gamma)}\lesssim \sum_{K'\in \mathcal{N}_{\gamma}} 
	\Big(\|\nabla \mathbf{e}_{\bar{\mathbf{y}}}\|_{\mathbf{L}^2(K')}
	+\| e_{\bar{\mathsf{p}}} \|_{L^2(K')}\AAF{+h_{K'}}\GC{\|e_{\bar{u}}\|_{L^2(K')}}
\\
+ h_{K'}\| \mathbf{e}_{\bar{\mathbf{y}}}\|_{\mathbf{L}^2(K')}+\text{osc}_{\mathscr{T}}(\mathbf{f},K')\Big),
\end{multline}
\AAF{where we have also used that $h_{K}^{-\frac{1}{2}}\|\varphi_{\gamma} \mathcal{J}_{\gamma}^{st}\|_{\mathbf{L}^{2}(K)}+h_{K}^{\frac{1}{2}}\|\nabla (\varphi_{\gamma} \mathcal{J}_{\gamma}^{st}) \|_{\mathbf{L}^{2}(K)}\lesssim \|\mathcal{J}_{\gamma}^{st}\|_{\mathbf{L}^{2}(\gamma)}$.}

\emph{Step 3.} Let $K\in \T$. The control of the term $\|\text{div }\bar{\mathbf{y}}_{h}\|_{L^2(K)}^2$ follows easily from the incompressibility condition $\text{div }\bar{\mathbf{y}}=0$. In fact,
\begin{align}\label{eq:bound05_eff}
	\|\text{div }\bar{\mathbf{y}}_{h}\|_{L^2(K)}=\|\text{div }\mathbf{e}_{\bar{\mathbf{y}}}\|_{L^2(K)}\lesssim \|\nabla\mathbf{e}_{\bar{\mathbf{y}}}\|_{\mathbf{L}^2(K)}.
\end{align}

\emph{Step 4.} The proof concludes by collecting the estimates \eqref{eq:bound03_eff}, \eqref{eq:bound04_eff}, and \eqref{eq:bound05_eff}.
\end{proof}

In the following, we study local efficiency properties of the indicator $\mathcal{E}_{adj,K}$ defined in \S\ref{Real:fully_discrete}.

\begin{theorem}[local efficiency of $\mathcal{E}_{adj,K}$]\label{eq:efficiency_adj}
In the framework of Theorem \ref{eq:efficiency_est}, we have that, for $K\in\T$, the local error indicator $\mathcal{E}_{adj,K}$ satisfies the bound 
\begin{multline}\label{eq:efficiency_estimates_adj}
\mathcal{E}_{adj,K} \lesssim\|\nabla \mathbf{e}_{\bar{\mathbf{z}}} \|_{\mathbf{L}^2(\mathcal{N}_{K})}+\| e_{\bar{\mathsf{r}}} \|_{L^2(\mathcal{N}_{K})}
+
\AAF{h_{K}}
\GC{\|e_{\bar{u}}\|_{L^2(\mathcal{N}_{K})}}\\+h_{K}\|\mathbf{e}_{\bar{\mathbf{z}}}\|_{\mathbf{L}^2(\mathcal{N}_{K})}+h_{K}\|\mathbf{e}_{\bar{\mathbf{y}}}\|_{\mathbf{L}^2(\mathcal{N}_{K})}+\mathrm{osc}_{\mathscr{T}}(\mathbf{y}_{\Omega},\mathcal{N}_{K}),
\end{multline}
where $\mathcal{N}_{K}$ is defined in \eqref{eq:patch}. The hidden constant is independent of the continuous and discrete optimal variables, the size of the elements in the mesh $\T$, and $\#\T$.
\end{theorem}

\begin{proof}
We proceed in four steps.

\emph{Step 1:} Let $K\in\T$. Define $\tilde{\mathcal{R}}_{K}^{adj}:=(\bar{\mathbf{y}}_{h}-\mathsf{P}_{K}\mathbf{y}_{\Omega}+\Delta \bar{\mathbf{z}}_{h}\!-\bar{u}_{h}\bar{\mathbf{z}}_{h}\!+\!\nabla\bar{\mathsf{r}}_{h})|_K$. To bound $h_K^2\|\mathcal{R}_{K}^{adj}\|_{\mathbf{L}^2(K)}^2$, we first apply the triangle inequality to obtain the basic bound
\begin{align}\label{eq:bound06_eff}
h_K^2\|\mathcal{R}_{K}^{adj}\|_{\mathbf{L}^2(K)}^2\lesssim h_K^2 \|\tilde{\mathcal{R}}_{K}^{adj}\|_{\mathbf{L}^2(K)}^2 + \text{osc}_{\mathscr{T}}(\mathbf{y}_{\Omega},K)^2.
\end{align}
To bound $h_K^2 \|\tilde{\mathcal{R}}_{K}^{adj}\|_{\mathbf{L}^2(K)}^2$, we set $\mathbf{v}=\varphi_K \tilde{\mathcal{R}}_{K}^{st}$ and $\mathsf{q}=0$ in \eqref{eq:brinkamn_problem02_aux_error}. \AAF{Applying similar arguments to those used to obtain the estimate \eqref{eq:bound03_eff}, we can obtain that}
\begin{multline}\label{eq:bound08_eff}
h_K\|\tilde{\mathcal{R}}_{K}^{adj}\|_{\mathbf{L}^2(K)}  \lesssim  \|\nabla \mathbf{e}_{\bar{\mathbf{z}}}\|_{\mathbf{L}^2(K)}
+\| e_{\bar{\mathsf{r}}} \|_{L^2(K)}
+ h_{K} \|e_{\bar{u}}\|_{L^2(K)} 
\\+h_K\| \mathbf{e}_{\bar{\mathbf{z}}}\|_{\mathbf{L}^2(K)}+h_K\| \mathbf{e}_{\bar{\mathbf{y}}}\|_{\mathbf{L}^2(K)}+\text{osc}_{\mathscr{T}}(\mathbf{y}_{\Omega},K).
\end{multline}
If we replace this bound into \eqref{eq:bound06_eff}, we obtain the desired bound for $h_K^2 \|\mathcal{R}_{K}^{adj}\|_{\mathbf{L}^2(K)}^2$.

\emph{Step 2.} Let $K\in\T$ and $\gamma\in \mathscr{S}$. We now control $h_K^{\frac{1}{2}}\|\mathcal{J}_{\gamma}^{adj}\|_{\mathbf{L}^2(\gamma)}$. To do so, we set $\mathbf{v}=\varphi_{\gamma}\mathcal{J}_{\gamma}^{adj}$ and $\mathsf{q}=0$ in \eqref{eq:brinkamn_problem02_aux_error} and continue with similar arguments as we developed for \eqref{eq:bound08_eff}. In fact, we have the bound
\begin{multline}\label{eq:bound09_eff}
h_{K}^{\frac{1}{2}}\|\mathcal{J}_{\gamma}^{adj}\|_{\mathbf{L}^2(\gamma)}\lesssim \sum_{K'\in \mathcal{N}_{\gamma}} \Big(\|\nabla \mathbf{e}_{\bar{\mathbf{z}}}\|_{\mathbf{L}^2(K')} +\| e_{\bar{\mathsf{r}}} \|_{L^2(K')} +
\AAF{h_{K'}}\GC{\|e_{\bar{u}}\|_{L^2(K')}}
\\
+
h_{K'}\| \mathbf{e}_{\bar{\mathbf{y}}}\|_{\mathbf{L}^2(K')}
+
h_{K'}\| \mathbf{e}_{\bar{\mathbf{z}}}\|_{\mathbf{L}^2(K')}+\text{osc}_{\mathscr{T}}(\mathbf{y}_{\Omega},K')\Big).
\end{multline}

\emph{Step 3:} Let $K\in \T$. Since $\text{div }\bar{\mathbf{z}}=0$, we immediately obtain that
\begin{align}\label{eq:bound10_eff}
	\|\text{div }\bar{\mathbf{z}}_{h}\|_{L^2(K)}=\|\text{div }\mathbf{e}_{\bar{\mathbf{z}}}\|_{L^2(K)}\lesssim \|\nabla\mathbf{e}_{\bar{\mathbf{z}}}\|_{\mathbf{L}^2(K)}.
\end{align}

\emph{Step 4:} The proof concludes by gathering the bounds \eqref{eq:bound08_eff}, \eqref{eq:bound09_eff}, and \eqref{eq:bound10_eff}.
\end{proof}

Finally, we present a global efficiency result for the estimator $\mathcal{E}_{ocp,\T}$.

\begin{theorem}[global efficiency of $\mathcal{E}_{ocp,\T}$]\label{eq:efficiency_ocp}
In the framework of Theorem \ref{eq:efficiency_est}, we have that, for $K\in\T$, the estimator $\mathcal{E}_{ocp,\T}$ satisfies the global bound
\begin{multline}\label{eq:efficiency_estimates_ocp}
\mathcal{E}_{ocp,\T} \lesssim\|\nabla \mathbf{e}_{\bar{\mathbf{y}}} \|_{\mathbf{L}^2(\Omega)}+\|\nabla \mathbf{e}_{\bar{\mathbf{z}}} \|_{\mathbf{L}^2(\Omega)}+\| e_{\bar{\mathsf{p}}} \|_{L^2(\Omega)}+\| e_{\bar{\mathsf{r}}} \|_{L^2(\Omega)}\\+\|e_{\bar{u}}\|_{L^2(\Omega)}
+
\mathrm{osc}_{\mathscr{T}}(\mathbf{f},\T)
+
\mathrm{osc}_{\mathscr{T}}(\mathbf{y}_{\Omega},\T).
\end{multline}
The hidden constant is independent of the continuous and discrete optimal variables, the size of the elements in the mesh $\T$, and $\#\T$.
\end{theorem}

\begin{proof}
We begin the proof by considering the definition of the error estimator $\mathcal{E}_{st,\T}$ given in \eqref{eq:estimator_st} and the local efficiency estimate \eqref{eq:efficiency_estimates_est} to conclude that
\begin{multline}\label{eq:efficiency_estimates_est_global}
\mathcal{E}_{st,\T} \lesssim\|\nabla \mathbf{e}_{\bar{\mathbf{y}}} \|_{\mathbf{L}^2(\Omega)}+\| e_{\bar{\mathsf{p}}} \|_{L^2(\Omega)}
+
\GC{\text{diam}(\Omega)\|e_{\bar{u}}\|_{L^2(\Omega)}}
\\
+\text{diam}(\Omega)\|\mathbf{e}_{\bar{\mathbf{y}}}\|_{\mathbf{L}^2(\mathcal{N}_{K})}+\text{osc}_{\mathscr{T}}(\mathbf{f},\T).
\end{multline}
Similarly, considering the definition of the error estimator $\mathcal{E}_{adj,\T}$ in \eqref{eq:estimator_AD} and the local efficiency estimate \eqref{eq:efficiency_estimates_adj}, we reach the following conclusion:
\begin{multline}\label{eq:efficiency_estimates_adj_global}
\mathcal{E}_{adj,\T} \lesssim\|\nabla \mathbf{e}_{\bar{\mathbf{z}}} \|_{\mathbf{L}^2(\Omega)}
+
\| e_{\bar{\mathsf{r}}} \|_{L^2(\Omega)}
+
\GC{\text{diam}(\Omega)\|e_{\bar{u}}\|_{L^2(\Omega)}}\\
+
\text{diam}(\Omega)\|\mathbf{e}_{\bar{\mathbf{z}}}\|_{\mathbf{L}^2(\Omega)}+\text{diam}(\Omega)\|\mathbf{e}_{\bar{\mathbf{y}}}\|_{\mathbf{L}^2(\Omega)}+\text{osc}_{\mathscr{T}}(\mathbf{y}_{\Omega},\T).
\end{multline}
Our aim now is to control the term $\mathcal{E}_{ct,\T}$ with the above estimates. To this end, we invoke \eqref{eq:estimator_control} and use a simple application of a triangle inequality to obtain
\begin{eqnarray*}
\mathcal{E}_{ct,\T}	&\leq &\|\breve{u}-\bar{u}\|_{L^2(\Omega)}+\|e_{\bar{u}}\|_{L^2(\Omega)}
\\
&\leq &\|\Pi_{[\mathsf{a},\mathsf{b}]}(\alpha^{-1}\bar{\mathbf{y}}_{h}\cdot\bar{\mathbf{z}}_{h})-\Pi_{[\mathsf{a},\mathsf{b}]}(\alpha^{-1}\bar{\mathbf{y}}\cdot\bar{\mathbf{z}})\|_{L^2(\Omega)}+\|e_{\bar{u}}\|_{L^2(\Omega)}.
\end{eqnarray*}
Applying the Lipschitz property of the projection $\Pi_{[\mathsf{a},\mathsf{b}]}$, the Cauchy--Schwarz inequality, and the standard Sobolev embedding $\mathbf{H}_0^1(\Omega)\hookrightarrow\mathbf{L}^4(\Omega)$ in the previous estimate shows that
\begin{align*}
\mathcal{E}_{ct,\T}\leq \alpha^{-1} (\|\nabla \mathbf{e}_{\bar{\mathbf{z}}}\|_{\mathbf{L}^2(\Omega)}\|\nabla \bar{\mathbf{y}}_{h}\|_{\mathbf{L}^2(\Omega)}+\|\nabla \mathbf{e}_{\bar{\mathbf{y}}}\|_{\mathbf{L}^2(\Omega)}\|\nabla \bar{\mathbf{z}}\|_{\mathbf{L}^2(\Omega)})+\|e_{\bar{u}}\|_{L^2(\Omega)}.
\end{align*}
The bounds $\|\nabla \bar{\mathbf{y}}_{h}\|_{\mathbf{L}^2(\Omega)}\lesssim \|\mathbf{f}\|_{\mathbf{L}^2(\Omega)}$ and $\|\nabla \bar{\mathbf{z}}\|_{\mathbf{L}^2(\Omega)}\lesssim  \|\mathbf{f}\|_{\mathbf{L}^2(\Omega)}+ \|\mathbf{y}_{\Omega}\|_{\mathbf{L}^2(\Omega)}$ therefore result in the bound
\begin{align}\label{eq:efficiency_estimates_ct_global}
\mathcal{E}_{ct,\T}\lesssim \|\nabla \mathbf{e}_{\bar{\mathbf{z}}}\|_{\mathbf{L}^2(\Omega)}+\|\nabla \mathbf{e}_{\bar{\mathbf{y}}}\|_{\mathbf{L}^2(\Omega)}+\|e_{\bar{u}}\|_{L^2(\Omega)}.
\end{align}
The desired estimate \eqref{eq:efficiency_estimates_ocp} follows from the application of Poincare's inequality and a collection of the bounds \eqref{eq:efficiency_estimates_est_global}, \eqref{eq:efficiency_estimates_adj_global}, and \eqref{eq:efficiency_estimates_ct_global}. This concludes the proof.
\end{proof}

\subsection{Efficiency analysis: the semidiscrete scheme}\label{efficiency_semi}
We begin the section by defining the local error indicator $\mathfrak{E}_{ocp,K}^2:=\mathfrak{E}_{st,K}^2+\mathfrak{E}_{adj,K}^2$ for $K\in\T$.

We now present an analysis of local efficiency by adapting the arguments from \S\ref{efficiency_fully}. For brevity, we omit the details of the proof.

\begin{theorem}[local efficiency of $\mathfrak{E}_{ocp,K}$]\label{eq:efficiency_est_semi}
Let $\bar{u}$ be a local solution of \eqref{eq:min_functional}--\eqref{eq:brinkamn_problem_state} such that \eqref{equivalence_condition} holds for $\tau > 0$, and let $(\bar{\mathbf{y}},\bar{\mathsf{p}})$ and $(\bar{\mathbf{z}},\bar{\mathsf{r}})$ be the corresponding state and ajoint state variables. Let $\bar{\mathsf{u}}$ be a local minimum of the semidiscrete scheme, and let $(\bar{\mathbf{y}}_h,\bar{\mathsf{p}}_h)$ and $(\bar{\mathbf{z}}_h,\bar{\mathsf{r}}_h)$ be the corresponding discrete state and adjoint state variables, respectively. Then, for $K\in\T$, the local error indicator $\mathfrak{E}_{ocp,K}$ satisfies the bound 
\begin{multline}\label{eq:efficiency_estimates_ocp_semi}
\mathfrak{E}_{ocp,K} \lesssim\|\nabla \mathbf{e}_{\bar{\mathbf{y}}} \|_{\mathbf{L}^2(\mathcal{N}_{K})}+\|\nabla \mathbf{e}_{\bar{\mathbf{z}}} \|_{\mathbf{L}^2(\mathcal{N}_{K})}
+
\| e_{\bar{\mathsf{p}}} \|_{L^2(\mathcal{N}_{K})}
+
\| e_{\bar{\mathsf{r}}} \|_{L^2(\mathcal{N}_{K})}
\\
+
h_{K}\|e_{\bar{u}}\|_{L^2(\mathcal{N}_{K})}
+
h_K \| \mathbf{e}_{\bar{\mathbf{y}}} \|_{\mathbf{L}^2(\mathcal{N}_{K})}
+
h_K\| \mathbf{e}_{\bar{\mathbf{z}}} \|_{\mathbf{L}^2(\mathcal{N}_{K})}
+
\mathrm{osc}_{\mathscr{T}}(\mathbf{f},\mathcal{N}_{K})\\
+
\mathrm{osc}_{\mathscr{T}}(\mathbf{y}_{\Omega},\mathcal{N}_{K})
+
\mathrm{osc}_{\mathscr{T}}(\bar{\mathsf{u}}\bar{\mathbf{y}}_h,\mathcal{N}_{K})
+
\mathrm{osc}_{\mathscr{T}}(\bar{\mathsf{u}}\bar{\mathbf{z}}_h,\mathcal{N}_{K}),
\end{multline}
where $\mathcal{N}_{K}$ is defined in \eqref{eq:patch}. The hidden constant is independent of the continuous and discrete optimal variables, the size of the elements in the mesh $\T$, and $\#\T$.
\end{theorem}


\section{Numerical experiments}
\label{sec:numericalexperiments}

\EO{In this section, we present a series of numerical experiments in two dimensions that illustrate the performance of the fully discrete scheme and the semidiscrete scheme described in \S \ref{fully_discrete_framework} and \S \ref{semi_discrete_framework}, respectively. The experiments were performed with a code implemented in \texttt{C++}, where the global linear systems were solved using the multifrontal massively parallel sparse direct solver (MUMPS) \cite{MUMPS2}, and the required integrals were computed using a quadrature rule that is exact for polynomials up to degree 19.}

Before \EO{discussing implementation details, we introduce some basic ingredients. To simplify the presentation of the algorithms, we first define the forms
\[
a: \mathbf{H}_0^1(\Omega) \times \mathbf{H}_0^1(\Omega) \rightarrow \mathbb{R},
\,\,
b: L_0^2(\Omega) \times \mathbf{H}_0^1(\Omega) \rightarrow \mathbb{R},
\,\,
c: L^2(\Omega) \times \mathbf{H}_0^1(\Omega) \times \mathbf{H}_0^1(\Omega) \rightarrow \mathbb{R},
\]
by $a(\mathbf{w},\mathbf{v}) := (\nabla\mathbf{w},\nabla\mathbf{v})_{\mathbf{L}^{2}(\Omega)}$, $b(r,\mathbf{v})=(r,\text{div}~\mathbf{v})_{L^{2}(\Omega)}$, and $c(u,\mathbf{w},\mathbf{v}) = (u\mathbf{w},\mathbf{v})_{\mathbf{L}^{2}(\Omega)}$. As a second ingredient, we note that the projection operator $\Pi_{[\mathsf{a},\mathsf{b}]}$ can be equivalently written as follows:
\begin{equation}
\Pi_{[\mathsf{a},\mathsf{b}]}(\xi) =
\xi+\max\{\mathsf{a}-\xi,0\}-\max\{-\mathsf{b}+\xi,0\}
\quad
\forall \xi \in \mathbb{R}.
\label{eq:new_projection_formula}
\end{equation}
Finally, we set
$\mathrm{E}_{ocp,\T}=
(
\|\nabla \mathbf{e}_{\bar{\mathbf{y}}} \|_{\mathbf{L}^2(\Omega)}^2
+
\|\nabla \mathbf{e}_{\bar{\mathbf{z}}} \|_{\mathbf{L}^2(\Omega)}^2
+
\|e_{\bar{\mathsf{p}}}\|_{L^2(\Omega)}^2
+
\| e_{\bar{\mathsf{r}}} \|_{L^2(\Omega)}^2
+
\|e_{\bar{u}}\|_{L^2(\Omega)}^2
)^{\frac{1}{2}}$.}

\textbf{The fully discrete scheme:} \EO{For a given partition $\T$, we seek $(\bar{\mathbf{y}}_h,\bar{\mathsf{p}}_h,\bar{\mathbf{z}}_h,\bar{\mathsf{r}}_h,\bar{u}_h)$ $\in \mathbf{X}_h \times M_h\times \mathbf{X}_h \times M_h\times \mathbb{U}_{ad,h}$ that solves the discrete optimality system \eqref{eq:brinkamn_problem_state_fd}--\eqref{eq:variational_inequality_discrete}. This system is solved using the adaptive algorithm described in \textbf{Algorithm 1}, which involves the semi-smooth Newton method described in \textbf{Algorithm 2}. Below, we briefly explain how \textbf{Algorithm 2} is obtained (see \cite[Appendix A.1]{MR2971171} and \cite{MR3523574}). We first note that an optimal control $\bar{u}_h$ is such that
\[
 \bar{u}_{h}|_{K} =
\Pi_{[\mathsf{a},\mathsf{b}]}
\left(
\frac{(\bar{\mathbf{y}}_{h},\bar{\mathbf{z}}_{h})_{\mathbf{L}^{2}(K)}}{\alpha|K|}
\right)
\quad
\forall K \in \T.
\]
We rewrite this expression using \eqref{eq:new_projection_formula} and apply the standard Newton method to the system composed of \eqref{eq:brinkamn_problem_state_fd}, \eqref{eq:brinkman_problem_adjoint_fully}, and the derived expression. When deriving the semi-smooth Newton method, the following characteristic functions appear:
\[
\tilde{\chi}_{[\mathsf{a}]}
=
\left\{
\begin{array}{cl}
1, & \textrm{if }\mathsf{a}-\frac{(\mathbf{y}_{h}^{n},\mathbf{z}_{h}^{n})_{\mathbf{L}^{2}(K)}}{\alpha|K|}\geq 0,\\
0, & \textrm{otherwise},
\end{array}
\right.
\quad
\tilde{\chi}_{[\mathsf{b}]}
=
\left\{
\begin{array}{cl}
1, & \textrm{if }-\mathsf{b}+\frac{(\mathbf{y}_{h}^{n},\mathbf{z}_{h}^{n})_{\mathbf{L}^{2}(K)}}{\alpha|K|}\geq 0,\\
0, & \textrm{otherwise}.
\end{array}
\right.
\]
}

Finally, we introduce the \EO{\emph{total number of degrees of freedom} $\mathsf{Ndof}=2\text{dim} (\mathbf{X}_h)+2\text{dim}(M_h)+\dim(\mathbb{U}_{ad,h})$ and the effectivity index $\mathtt{Eff}_{F}=\mathrm{E}_{ocp}/\mathcal{E}_{ocp,\T}$.}

\textbf{The semidiscrete scheme:} \EO{For a given partition $\T$, we seek $(\bar{\mathbf{y}}_h,\bar{\mathsf{p}}_h,\bar{\mathbf{z}}_h,\bar{\mathsf{r}}_h)\in \mathbf{X}_h \times M_h\times \mathbf{X}_h \times M_h$ that solves the discrete optimality system \eqref{eq:brinkamn_problem_state_sd}--\eqref{eq:brinkman_problem_adjoint_semi}. This system is solved using the adaptive algorithm described in \textbf{Algorithm 1}, which involves the semi-smooth Newton method described in \textbf{Algorithm 3} (see \cite[Appendix A.1]{MR2971171} and \cite{MR3523574}). We note that, by the construction of the scheme, a discrete optimal control $\bar{\mathsf{u}}$ satisfies the following projection formula:
\[
 \bar{\mathsf{u}}=\Pi_{[\mathsf{a},\mathsf{b}]}\left(\frac{1}{\alpha}\bar{\mathbf{y}}_{h}\cdot\bar{\mathbf{z}}_{h}\right).
\]
We rewrite this expression using formula \eqref{eq:new_projection_formula} and substitute it into \eqref{eq:brinkamn_problem_state_sd} and \eqref{eq:brinkman_problem_adjoint_semi}. Then, we apply the standard Newton method to the resulting system. When deriving the semi-smooth Newton method, the following characteristic functions appear:}
\[
\chi_{[\mathsf{a}]}
=
\left\{
\begin{array}{cl}
1, & \textrm{if }\mathsf{a}-\frac{\mathbf{y}_{h}^{n}\cdot\mathbf{z}_{h}^{n}}{\alpha}\geq 0,\\
0, & \textrm{otherwise},
\end{array}
\right.
\quad
\chi_{[\mathsf{b}]}
=
\left\{
\begin{array}{cl}
1, & \textrm{if }-\mathsf{b}+\frac{\mathbf{y}_{h}^{n}\cdot\mathbf{z}_{h}^{n}}{\alpha}\geq 0,\\
0, & \textrm{otherwise}.
\end{array}
\right.
\]

We \EO{note that the numerical implementation of the semidiscrete scheme requires assembling and exact computation of the terms involving the control $\bar{\mathsf{u}} = \Pi_{[\mathsf{a},\mathsf{b}]}(\alpha^{-1}\bar{\mathbf{y}}_{h}\cdot\bar{\mathbf{z}}_{h})$. In particular, exact integration is needed on the elements $K \in \T$ where the control $\bar{\mathsf{u}}$ has kinks. Given the FEM spaces we are using and the fact that the multiplication $\bar{\mathbf{y}}_{h}\cdot\bar{\mathbf{z}}_{h}$ is involved, the regions where the control variable is active/inactive generally have curved boundaries. The exact computation of the terms involving the control is therefore far from trivial. We compute them using a quadrature formula that is exact for polynomials up to degree 19, resulting in an approximate version of the variational discretization scheme. Nonetheless, this approximation achieves optimal experimental rates of convergence for all the variables involved.}

Finally, \EO{we introduce the \emph{total number of degrees of freedom} $\mathsf{Ndof}=2\text{dim} (\mathbf{X}_h)+2\text{dim}(M_h)$ and the effectivity index $\mathtt{Eff}_{S}=\mathrm{E}_{ocp}/\mathfrak{E}_{ocp,\T}$.}

\begin{algorithm}[!h]
\caption{\textbf{: Adaptive algorithm.}}
\label{Algorithm1}
\footnotesize
\textbf{Input:} Initial mesh $\mathscr{T}_{h}$, $\alpha$, $\mathsf{a}$, $\mathsf{b}$. Set $i=0$;
\\
$\boldsymbol{1}$: Solve the discrete problems:
\\
- \eqref{eq:brinkamn_problem_state_fd}--\eqref{eq:brinkman_problem_adjoint_fully}, for the fully discrete scheme,  by using \textbf{Algorithm 2};\\
- \eqref{eq:brinkamn_problem_state_sd}--\eqref{eq:brinkman_problem_adjoint_semi}, for the semidiscrete scheme, by using \textbf{Algorithm 3};
\\
$\boldsymbol{2}$: For each $K\in\mathscr{T}$ compute the local error indicators:\\
- For the fully discrete scheme, $\mathrm{E}_{K}^2:=\mathcal{E}_{st,K}^2 + \mathcal{E}_{adj,K}^2 + \mathcal{E}_{ct,K}^2$ (\S \ref{Real:fully_discrete});
\\
- For the semidiscrete scheme, $\mathrm{E}_{K}^2:=\mathfrak{E}_{st,K}^2 + \mathfrak{E}_{adj,K}^2$ (\S \ref{Real:semi_discrete});
\\
$\boldsymbol{3}$: Mark an element $K\in\mathscr{T}$ for refinement if: $\mathrm{E}_{K} \geq 0.5 \max_{K'\in\T_h}\mathrm{E}_{K'}$.
\\
$\boldsymbol{4}$: From step $\boldsymbol{3}$ construct a new mesh using a longest edge bisection algorithm. Set $i \leftarrow i + 1$ and go to step $\boldsymbol{1}$.
\end{algorithm}

\begin{algorithm}[!h]
\caption{\textbf{: Semi-smooth Newton for the fully discrete scheme.}}
\label{Algorithm2.1}
\footnotesize
\textbf{Input:} Initial guess $(\bar{\mathbf{y}}_{h}^{0},\bar{\mathsf{p}}_{h}^{0},\bar{\mathbf{z}}_{h}^{0},\bar{\mathsf{r}}_{h}^{0},\bar{u}_{h}^{0})$. For $n=0,1,...$
\\
$\boldsymbol{1}$: Compute the incremental $(\delta\mathbf{y},\delta\mathsf{p},\delta\mathbf{z},\delta\mathsf{r},\delta u)\in\mathbf{X}_{h}\times M_{h}\times\mathbf{X}_{h}\times M_{h}\times \mathbb{U}_{ad,h}$, such that
\begin{equation*}
\begin{array}{l}
a\left(\delta\mathbf{y},\mathbf{v}_{h}\right)
+
c(u_{h}^{n},\delta\mathbf{y},\mathbf{v}_{h})
+
c(\delta u,\mathbf{y}_{h}^{n},\mathbf{v}_{h})
-
b(\delta \mathsf{p},\mathbf{v}_{h})
\\
=~
(\mathbf{f},\mathbf{v}_{h})_{\mathbf{L}^2(\Omega)}
-
a(\mathbf{y}_{h}^{n},\mathbf{v}_{h})
-
c(u_{h}^{n},\mathbf{y}_{h}^{n},\mathbf{v}_{h})
+
b(\mathsf{p}_{h}^{n},\text{div }\mathbf{v}_{h}),
\\
b(\mathsf{q}_{h},\text{div }\delta\mathbf{y})
=   
-b(\mathsf{q}_{h},\text{div }\mathbf{y}_{h}^{n})
, 
\\
-(\delta\mathbf{y},\mathbf{w}_{h})_{\mathbf{L}^2(\Omega)}
+
a(\delta\mathbf{z},\mathbf{w}_{h})
+
c(u_{h}^{n},\delta\mathbf{z},\mathbf{w}_{h})
+
c(\delta u,\mathbf{z}_{h}^{n},\mathbf{w}_{h})
+
b(\delta \mathsf{r},\text{div }\mathbf{w}_{h})
\\
=  
-(\mathbf{y}_{\Omega},\mathbf{w}_{h})_{\mathbf{L}^2(\Omega)}
+
(\mathbf{y}_{h}^{n},\mathbf{w}_{h})_{\mathbf{L}^2(\Omega)}
-
a(\mathbf{z}_{h}^{n},\mathbf{w}_{h})
-
c(u_{h}^{n},\mathbf{z}_{h}^{n},\mathbf{v}_{h})
-
b(\mathsf{r}_{h}^{n},\text{div }\mathbf{w}_{h}),
\\
b(\mathsf{s}_{h},\text{div }\delta\mathbf{z})
=   
-b(\mathsf{s}_{h},\text{div }\mathbf{z}_{h}^{n}),
\\
\left(
-1+\tilde{\chi}_{[\mathsf{a}]}+\tilde{\chi}_{[\mathsf{b}]}
\right)
\left(
\frac{(\mathbf{y}_{h}^{n},\delta\mathbf{z})_{\mathbf{L}^{2}(K)}}{\alpha|K|}
+\frac{(\delta\mathbf{y},\mathbf{z}_{h}^{n})_{\mathbf{L}^{2}(K)}}{\alpha|K|}
\right)
+
\delta u_{|K}
\\
=~
\frac{(\mathbf{y}_{h}^{n},\mathbf{z}_{h}^{n})_{\mathbf{L}^{2}(K)}}{\alpha|K|}
+
\max\left\{\mathsf{a}-
\frac{(\mathbf{y}_{h}^{n},\mathbf{z}_{h}^{n})_{\mathbf{L}^{2}(K)}}{\alpha|K|}
,0\right\}
-
\max\left\{-\mathsf{b}+
\frac{(\mathbf{y}_{h}^{n},\mathbf{z}_{h}^{n})_{\mathbf{L}^{2}(K)},0}{\alpha|K|}
\right\}
-
u_{h|K}^{n}
, 
\end{array}
\end{equation*} 
for all $(\mathbf{v}_{h},\mathsf{q}_{h},\mathbf{w}_{h},\mathsf{s}_{h})\in\mathbf{X}_{h}\times M_{h}\times\mathbf{X}_{h}\times M_{h}$, where the last equation holds for all $K\in\mathscr{T}$. \\
$\boldsymbol{2}$: Update: $\mathbf{y}_{h}^{n+1}=\mathbf{y}_{h}^{n}+\delta\mathbf{y}$, $\mathsf{p}_{h}^{n+1}=\mathsf{p}_{h}^{n}+\delta\mathsf{p}$, $\mathbf{z}_{h}^{n+1}=\mathbf{z}_{h}^{n}+\delta\mathbf{z}$, $\mathsf{r}_{h}^{n+1}=\mathsf{r}_{h}^{n}+\delta\mathsf{r}$, $u_{h}^{n+1}=u_{h}^{n}+\delta u$.\\
$\boldsymbol{3}$: If $\|(\delta\mathbf{y},\delta\mathsf{p},\delta\mathbf{z},\delta\mathsf{r},\delta u)\|_{2}\leq 10^{-8}$ stop.  
\end{algorithm}

\begin{algorithm}[!h]
\caption{\textbf{: Semi-smooth Newton for the semidiscrete scheme.}}
\label{Algorithm2.2}
\footnotesize
\textbf{Input:} Initial guess $(\bar{\mathbf{y}}_{h}^{0},\bar{\mathsf{p}}_{h}^{0},\bar{\mathbf{z}}_{h}^{0},\bar{\mathsf{r}}_{h}^{0})$. For $n=0,1,...$
\\
$\boldsymbol{1}$: Compute the incremental $(\delta\mathbf{y},\delta\mathsf{p},\delta\mathbf{z},\delta\mathsf{r})\in\mathbf{X}_{h}\times M_{h}\times\mathbf{X}_{h}\times M_{h}$, such that
\begin{equation*}
\begin{array}{l}
a\left(\delta\mathbf{y},\mathbf{v}_{h}\right)
+
\alpha^{-1}
c\left(
(1-\chi_{[\mathsf{a}]}-\chi_{[\mathsf{b}]})
(\mathbf{y}_{h}^{n}\cdot\delta\mathbf{z}+\delta\mathbf{y}\cdot\mathbf{z}_{h}^{n}),\mathbf{y}_{h}^{n}, \mathbf{v}_{h}
\right)
\\
\quad
-
b(\delta \mathsf{p},\mathbf{v}_{h})
+
c\left(\Pi_{[\mathsf{a},\mathsf{b}]}(\alpha^{-1} \mathbf{y}_{h}^{n}\cdot\mathbf{z}_{h}^{n}),\delta\mathbf{y},\mathbf{v}_{h}\right) 
\\
=~(\mathbf{f},\mathbf{v}_{h})_{\mathbf{L}^2(\Omega)}
-
a(\mathbf{y}_{h}^{n},\mathbf{v}_{h})
-
c(\Pi_{[\mathsf{a},\mathsf{b}]}(\alpha^{-1} \mathbf{y}_{h}^{n}\cdot\mathbf{z}_{h}^{n}),\mathbf{y}_{h}^{n},\mathbf{v}_{h}) 
+
b(\mathsf{p}_{h}^{n},\text{div }\mathbf{v}_{h}),
\\
\vspace*{0.1cm}
b(\mathsf{q}_{h},\text{div }\delta\mathbf{y})
=   
-b(\mathsf{q}_{h},\text{div }\mathbf{y}_{h}^{n})
, 
\\
-(\delta\mathbf{y},\mathbf{w}_{h})_{\mathbf{L}^2(\Omega)}
+
a(\delta\mathbf{z},\mathbf{w}_{h})
+
\alpha^{-1}
c\left((1-\chi_{[\mathsf{a}]}-\chi_{[\mathsf{b}]})
(\mathbf{y}_{h}^{n}\cdot\delta\mathbf{z}+\delta\mathbf{y}\cdot\mathbf{z}_{h}^{n}),\mathbf{y}_{h}^{n}),\mathbf{z}_{h}^{n},\mathbf{w}_{h}\right)
\\
+~c\left(\Pi_{[\mathsf{a},\mathsf{b}]}(\alpha^{-1} \mathbf{y}_{h}^{n}\cdot\mathbf{z}_{h}^{n}),\delta\mathbf{z},\mathbf{w}_{h}\right) 
+
b(\delta \mathsf{r},\text{div }\mathbf{w}_{h})
\\
=  
-(\mathbf{y}_{\Omega},\mathbf{w}_{h})_{\mathbf{L}^2(\Omega)}
+
(\mathbf{y}_{h}^{n},\mathbf{w}_{h})_{\mathbf{L}^2(\Omega)}
-
a(\mathbf{z}_{h}^{n},\mathbf{w}_{h})
-
c\left(\Pi_{[\mathsf{a},\mathsf{b}]}(\alpha^{-1} \mathbf{y}_{h}^{n}\cdot\mathbf{z}_{h}^{n}),\mathbf{y}_{h}^{n},\mathbf{w}_{h}\right) 
-
b(\mathsf{r}_{h}^{n},\text{div }\mathbf{w}_{h}),
\\
b(\mathsf{s}_{h},\text{div }\delta\mathbf{z})
=   
-b(\mathsf{s}_{h},\text{div }\mathbf{z}_{h}^{n}),
\end{array}
\end{equation*} 
for all $(\mathbf{v}_{h},\mathsf{q}_{h},\mathbf{w}_{h},\mathsf{s}_{h})\in\mathbf{X}_{h}\times M_{h}\times\mathbf{X}_{h}\times M_{h}$.\\
$\boldsymbol{2}$: Update: $\mathbf{y}_{h}^{n+1}=\mathbf{y}_{h}^{n}+\delta\mathbf{y}$, $\mathsf{p}_{h}^{n+1}=\mathsf{p}_{h}^{n}+\delta\mathsf{p}$, $\mathbf{z}_{h}^{n+1}=\mathbf{z}_{h}^{n}+\delta\mathbf{z}$, $\mathsf{r}_{h}^{n+1}=\mathsf{r}_{h}^{n}+\delta\mathsf{r}$.\\
$\boldsymbol{3}$: If $\|(\delta\mathbf{y},\delta\mathsf{p},\delta\mathbf{z},\delta\mathsf{r})\|_{2}\leq 10^{-8}$ stop.  
\end{algorithm}


The initial meshes used for the resolution are shown in Fig. \ref{fig:meshes}.

\begin{figure}[!h]
\centering
\begin{minipage}[b]{0.4\textwidth}
\centering
\includegraphics[trim={0 0 0 0},clip,width=1.5cm,height=1.5cm,scale=0.66]{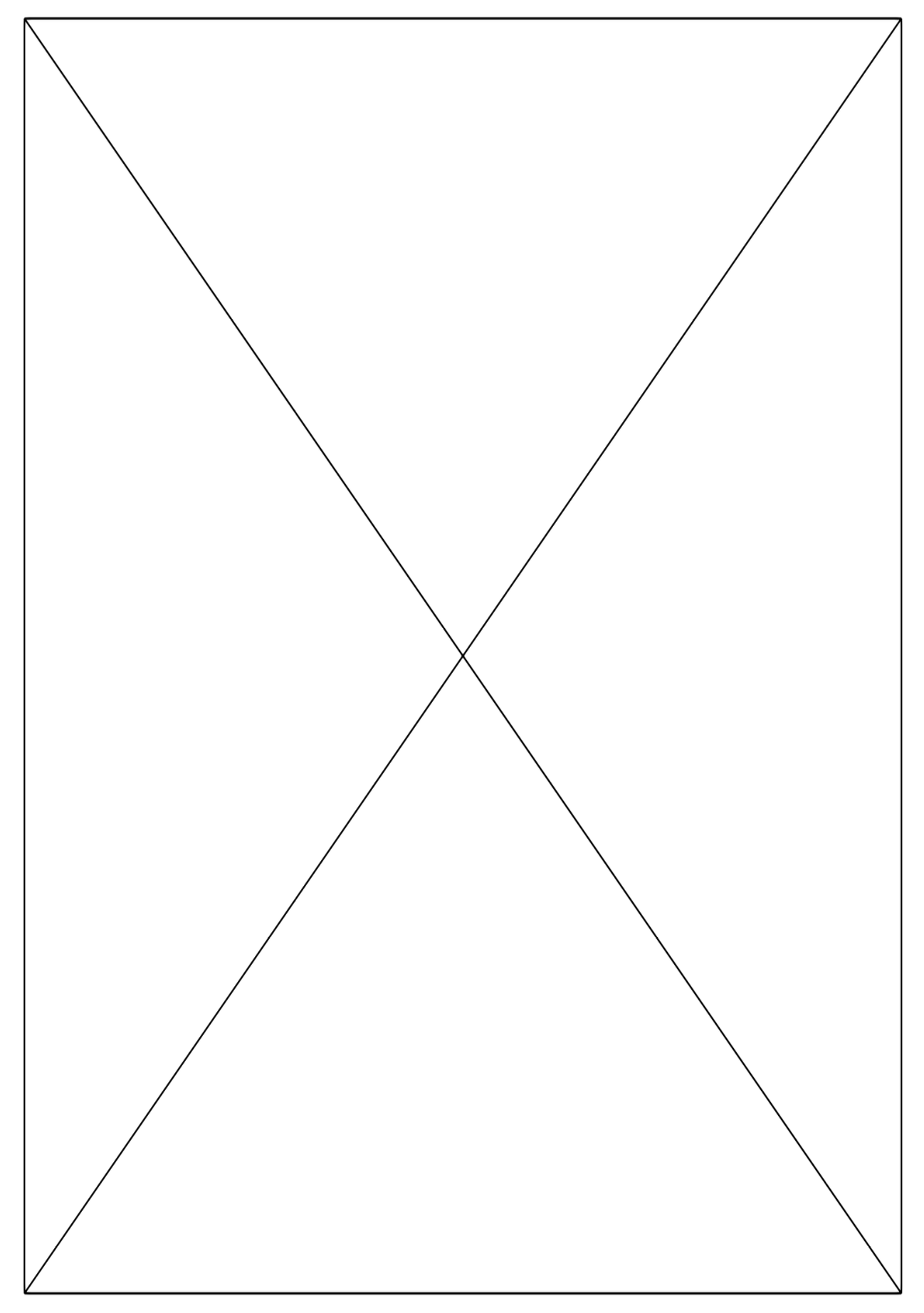} 
\end{minipage}
\begin{minipage}[b]{0.4\textwidth}
\centering
\includegraphics[trim={0 0 0 0},clip,width=2cm,height=1.5cm,scale=0.66]{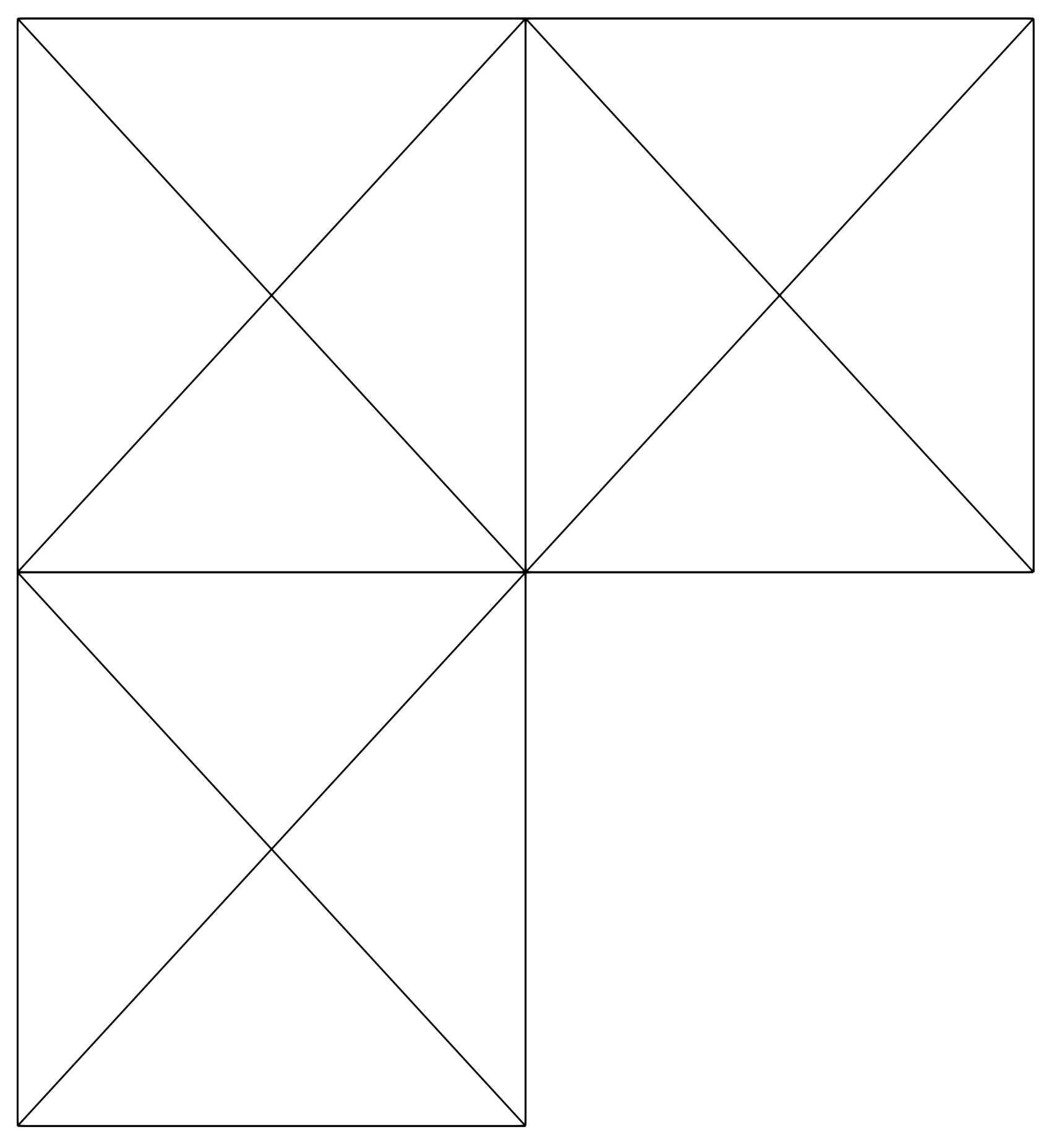} 
\end{minipage}
\caption{Initial meshes for $\Omega=(0,1)^{2}$ (left) and $\Omega=(-1, 1) \times (-1, 1) \setminus [0, 1] \times [-1,0]$ (right).}
\label{fig:meshes}
\end{figure}


\subsection{A sharp interior layer: uniform vs.~adaptive refinement}

In this section, we present a numerical example with a known exact solution: we fix the state and the adjoint velocities and pressures and obtain an exact optimal control. With these ingredients, we can thus compute $\mathbf{f}$ and the desired state $\mathbf{y}_{\Omega}$.

\textbf{Example 1.} We set $\Omega=(0,1)^{2}$, $\mathsf{a}=0.1$, $\mathsf{b}=0.2$, and $\alpha=1.0$. The state and adjoint variables are given by $\bar{\mathbf{y}}=50( \partial_{y}\xi, -\partial_{x}\xi )$, $\bar{\mathsf{p}}=xy(x - 1)(y - 1) - 1/36$, $\bar{\mathbf{z}}=50( \partial_{y}\varsigma , -\partial_{x}\varsigma )$, and $\bar{\mathsf{r}} =\sin(2\pi x)\sin(2\pi y)$. Here, $\xi=(xy(1-x)(1-y))^2$ and $\varsigma=(xy(1-x)(1-y))^2\tan^{-1}((x-0.5)/0.01)$. Due to the structure of the problem, we emphasize that the adjoint velocity field has a sharp interior layer at $x=0.5$.

In \AAF{Figures \ref{fig:test_01_1} (fully discrete scheme) and \ref{fig:test_01_1_2} (semidiscrete scheme), we present the experimental convergence rates obtained with uniform and adaptive refinement for the errors $\mathrm{E}_{ocp,\T}$ and their individual contributions, as well as for the a posteriori error estimators $\mathcal{E}_{ocp,\T}$ and $\mathfrak{E}_{ocp,\T}$ and their individual contributions. In Figures \ref{fig:test_01_2} (fully discrete scheme) and \ref{fig:test_01_2_new} (semidiscrete scheme), we show the meshes obtained after 35 iterations of adaptive refinement, the delineations of the discrete boundaries of the active sets, for the discrete and exact solutions, and the computed optimal discrete adjoint velocities and control variables.}

The \AAF{following observations and remarks are therefore appropriate:
~\\
$\bullet$ \textit{Fully discrete scheme}: \EO{It is observed that with uniform refinement, the experimental convergence rates for $\|\nabla \mathbf{e}_{\bar{\mathbf{z}}}\|_{\mathbf{L}^2(\Omega)}$ (see panel (A.1)) and the associated error estimator $\mathcal{E}_{adj,\T}$ (see panel (A.6)) are not optimal. This behavior also occurs for the total error $\mathrm{E}_{ocp,\T}$ (see panel (A.4)) and the total error estimator $\mathcal{E}_{ocp,\T}$ (see panel (A.8)).} In contrast, with adaptive refinement, all corresponding individual contributions from the total approximation error and the error estimator achieve optimal convergence rates.
~\\
$\bullet$ \textit{Semi-discrete scheme}: As in the fully discrete scheme, it is observed that the experimental convergence rates achieved for $\|\nabla \mathbf{e}_{\bar{\mathbf{z}}}\|_{\mathbf{L}^2(\Omega)}$ (see panel (B.1), $\mathfrak{E}_{adj,\T}$ (see panel (B.6)), the total error $E_{ocp,\T}$ (see panel (B.4)), and the total error estimator $\mathfrak{E}_{ocp,\T}$ (see panel (A.7)) are not optimal. In contrast, when adaptive refinement is employed, all corresponding individual contributions from the total approximation error and the error estimator achieve optimal convergence rates.
~\\
$\bullet$  \textit{Effectivity indices}: All calculated effectivity indices remain stable around the value of $10$ (see panels (A.9) and (B.8)).
~\\
$\bullet$  \textit{Mesh refinement}: It is observed that when adaptive refinement is used, most of the refinement is concentrated around the area where the singular behavior of the adjoint velocity occurs.
~\\
$\bullet$  \textit{Boundary of the discrete active set}: As shown in panels (C.1) and (D.1), the discrete level sets obtained using the approximate variational discretization scheme provide highly accurate approximations to that of the continuous active set, shown in panels (C.2) and (D.2). Notably, this accuracy is achieved without the additional computational cost of the fully discrete scheme.
~\\
$\bullet$  \textit{Experimental convergence rates for the optimal control}:  Panel (A.3) corroborates the linear rate of convergence  predicted in \eqref{eq:error_estimate_fd} for the fully discrete scheme. In contrast, panel (B.3) shows a quadratic rate of convergence for the semidiscrete scheme, which agrees with the estimate derived in \eqref{eq:error_estimate_sd}.}

\begin{figure}[!h]
\centering
\begin{minipage}[b]{0.23\textwidth}
\centering
\includegraphics[trim={0 0 0 0},clip,width=3cm,height=3.0cm,scale=0.66]{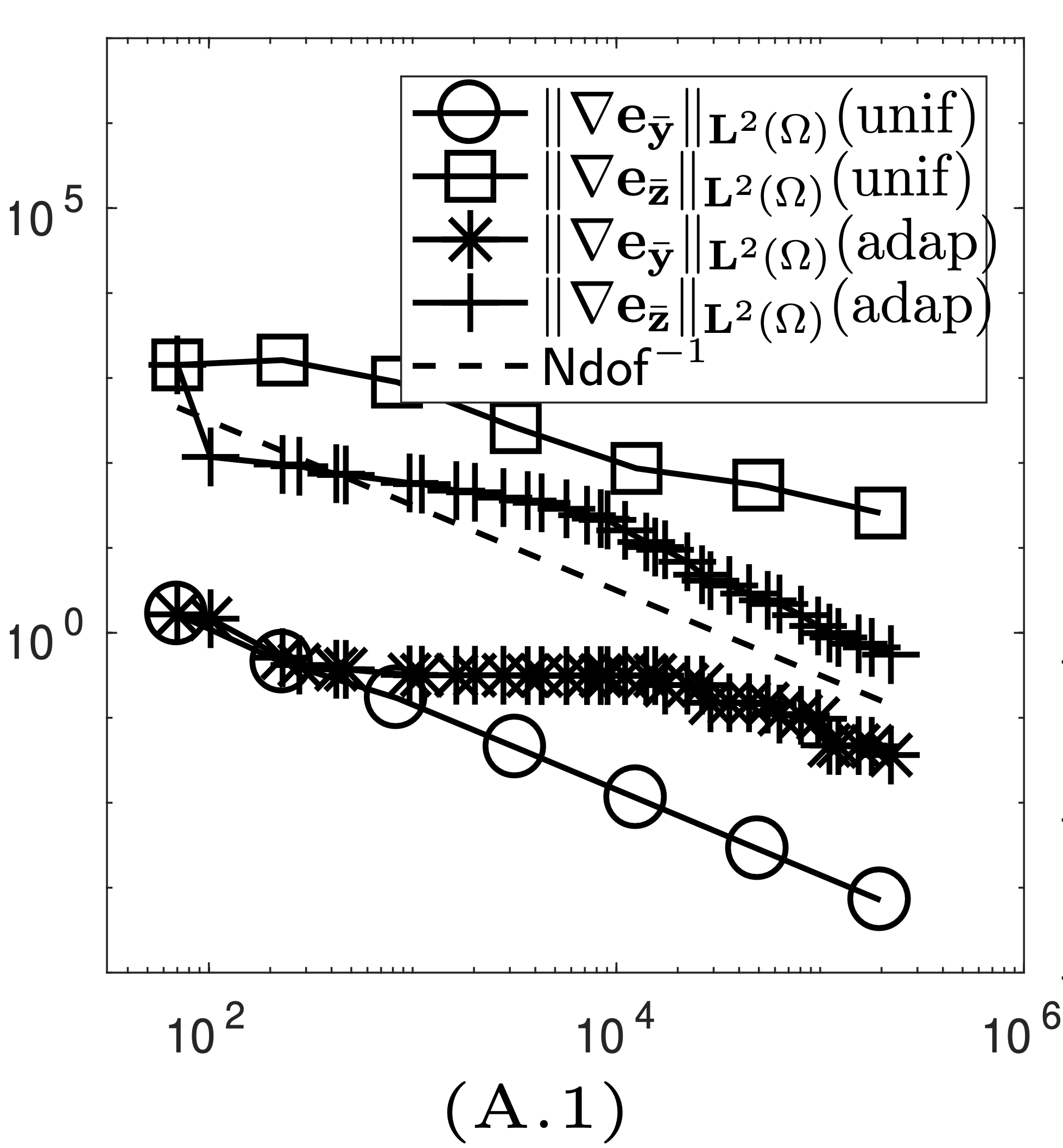} \\
\end{minipage}
\begin{minipage}[b]{0.23\textwidth}
\centering
\includegraphics[trim={0 0 0 0},clip,width=3cm,height=3cm,scale=0.66]{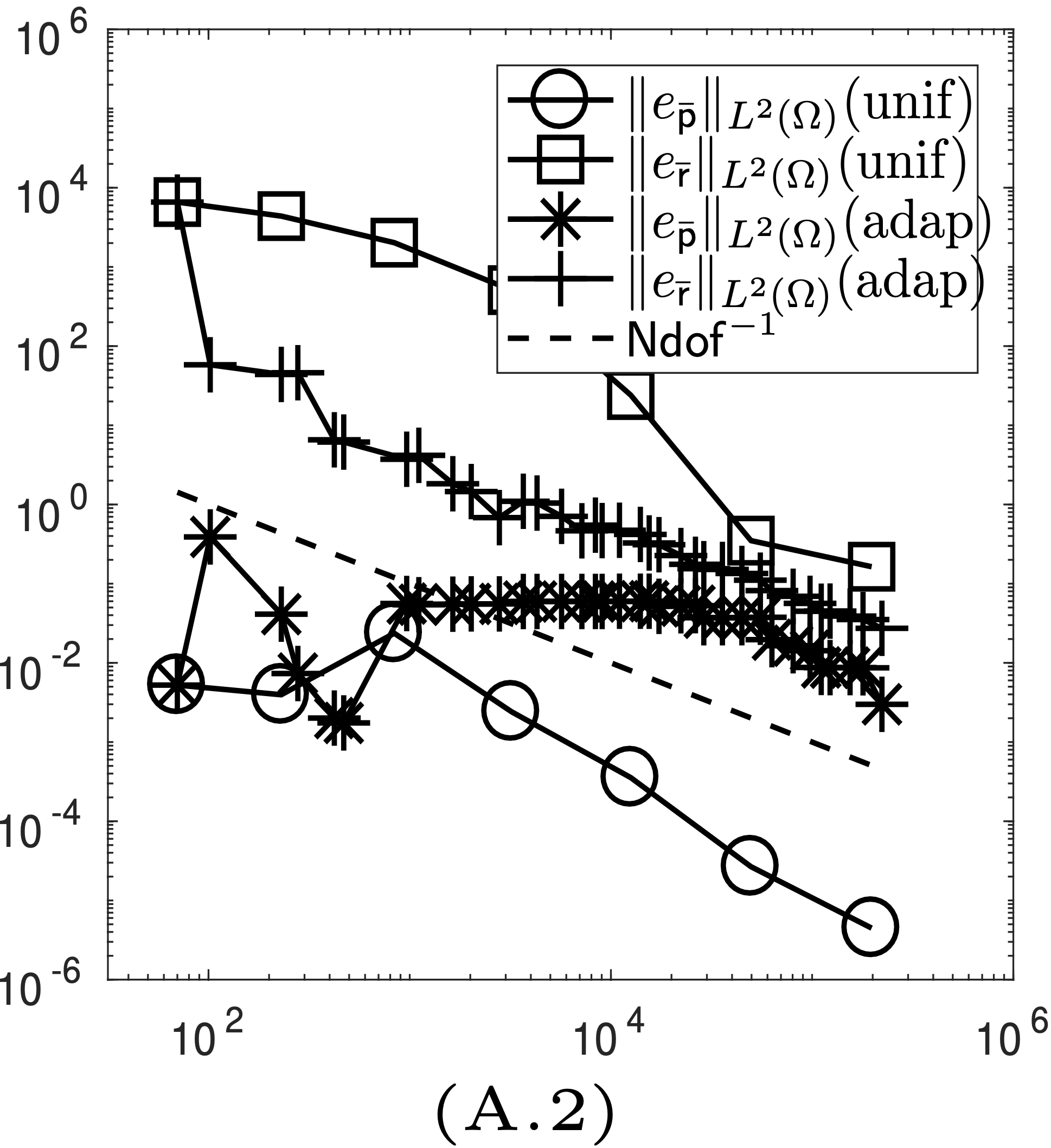} \\
\end{minipage}
\begin{minipage}[b]{0.23\textwidth}
\centering
\includegraphics[trim={0 0 0 0},clip,width=3cm,height=3cm,scale=0.66]{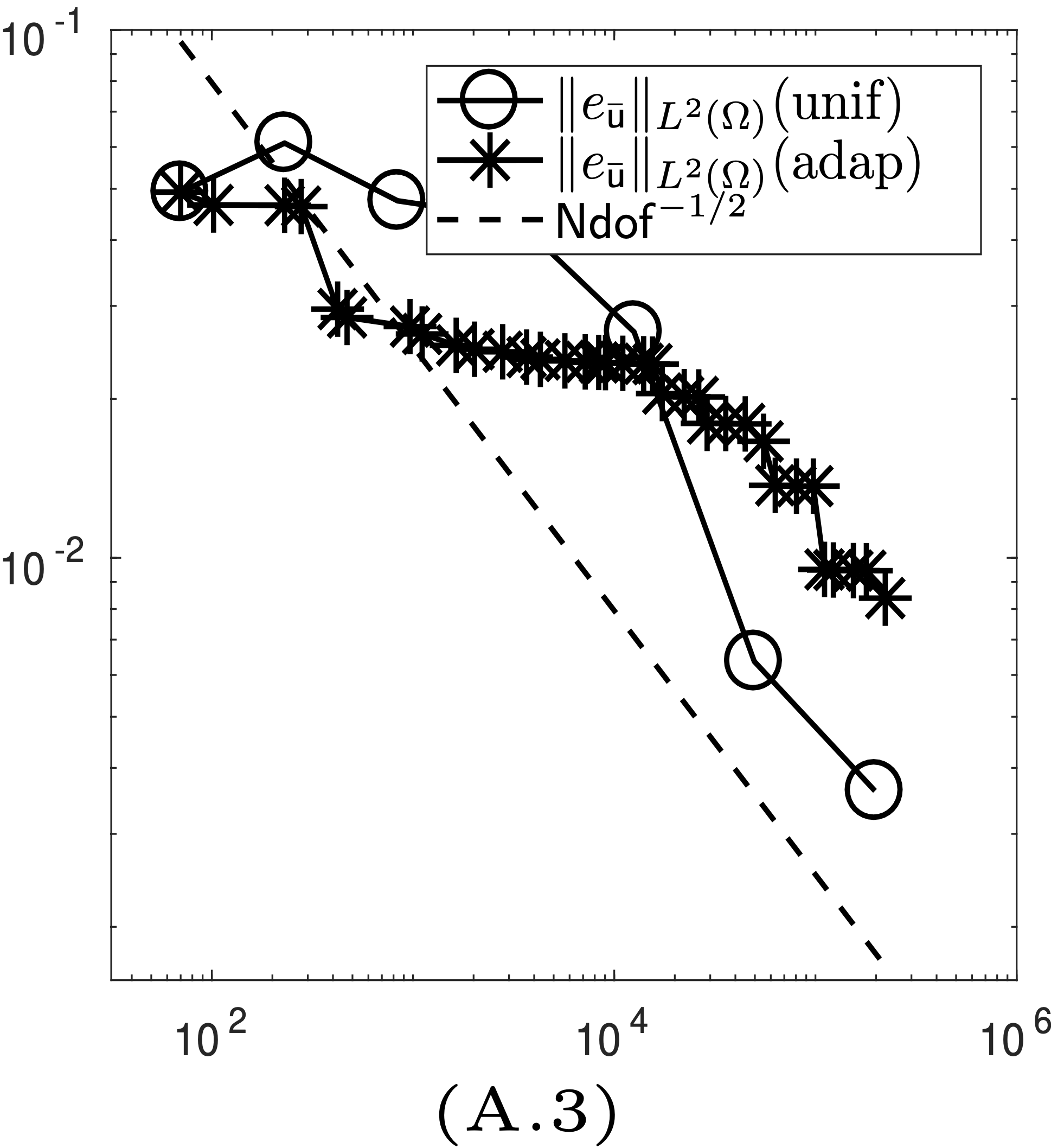} \\
\end{minipage}
\begin{minipage}[b]{0.23\textwidth}
\centering
\includegraphics[trim={0 0 0 0},clip,width=3cm,height=3cm,scale=0.66]{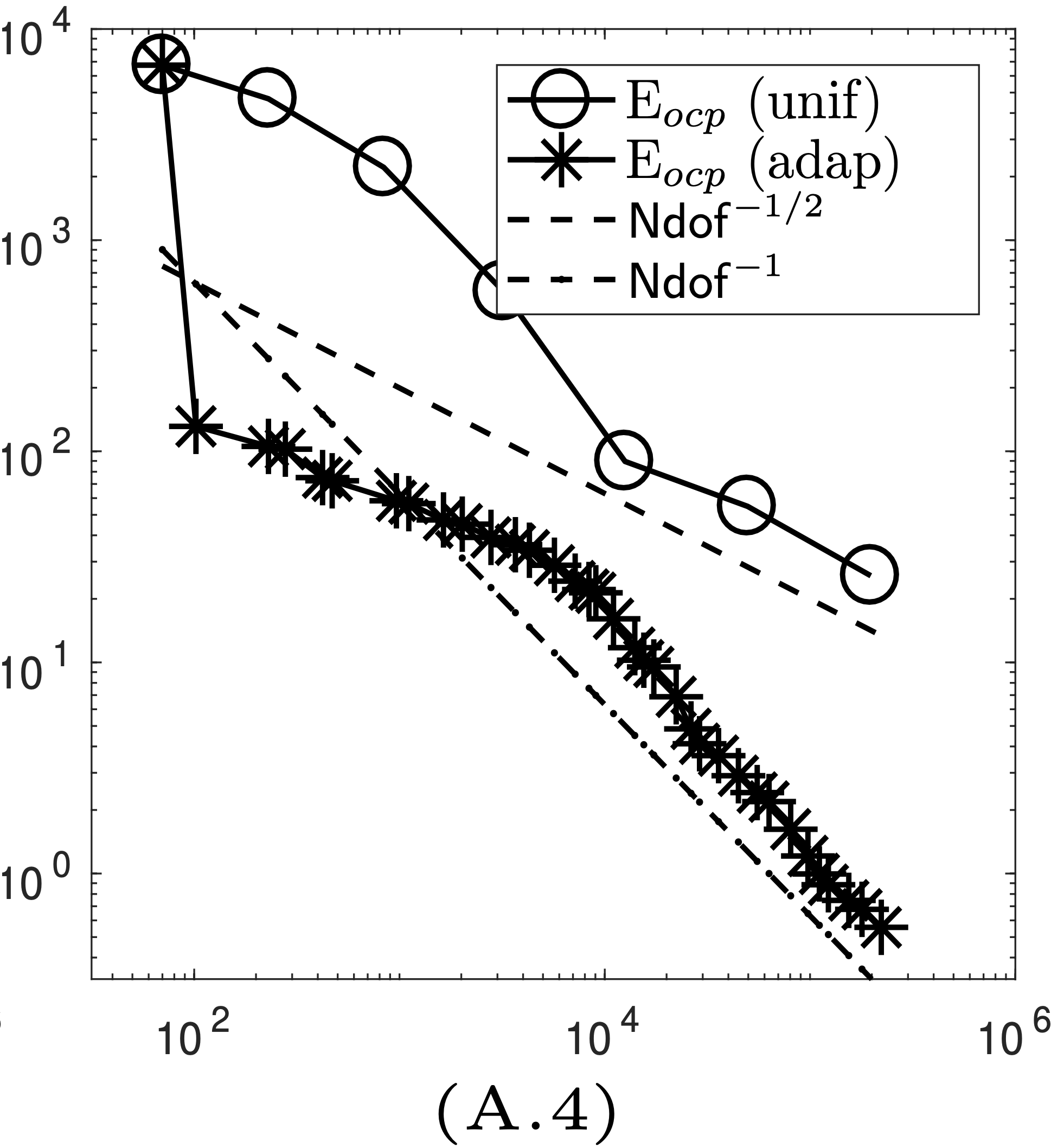} \\
\end{minipage}
\\
\begin{minipage}[b]{0.23\textwidth}
\centering
\includegraphics[trim={0 0 0 0},clip,width=3cm,height=3cm,scale=0.66]{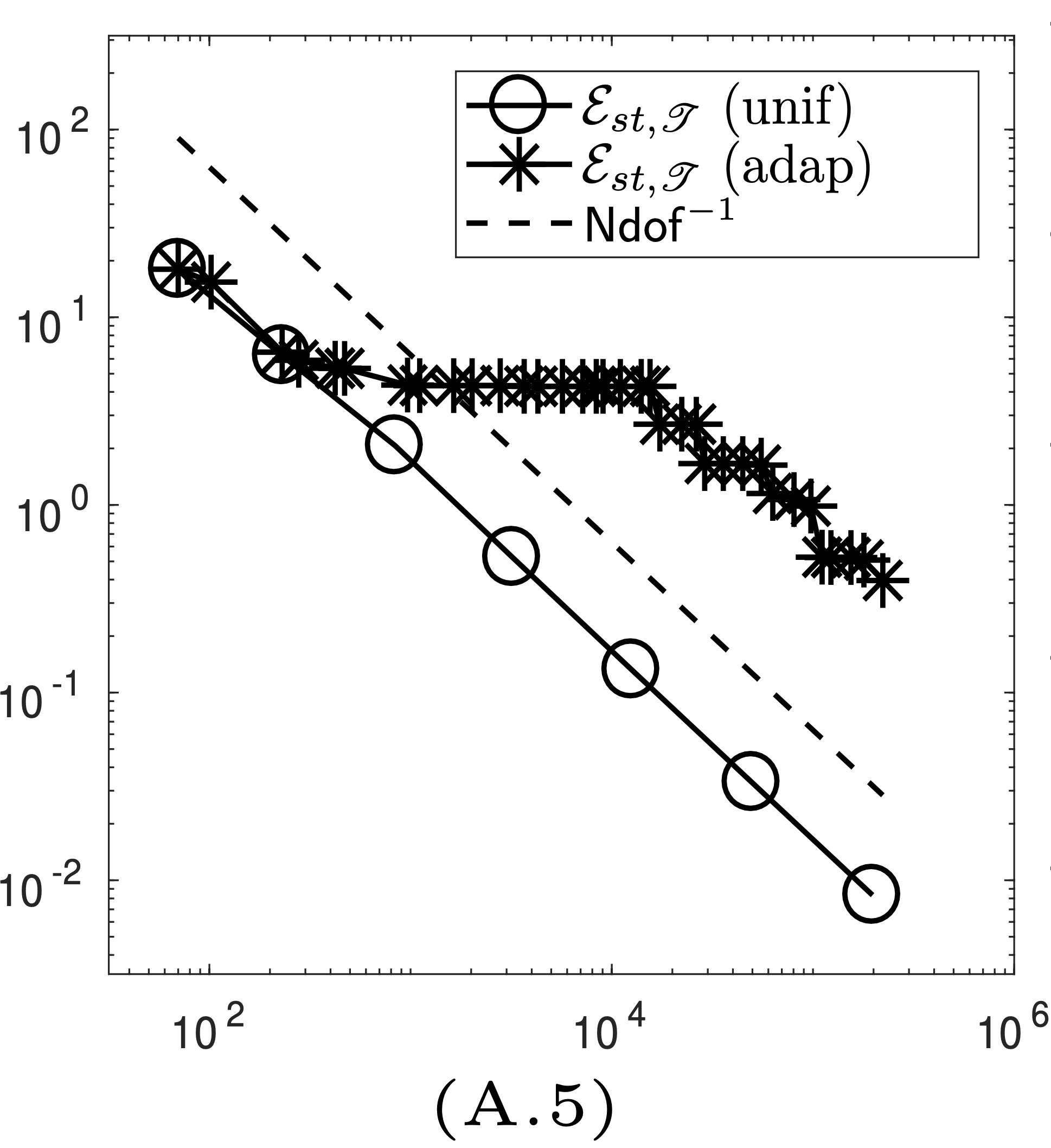} \\
\end{minipage}
\begin{minipage}[b]{0.23\textwidth}
\centering
\includegraphics[trim={0 0 0 0},clip,width=3cm,height=3cm,scale=0.66]{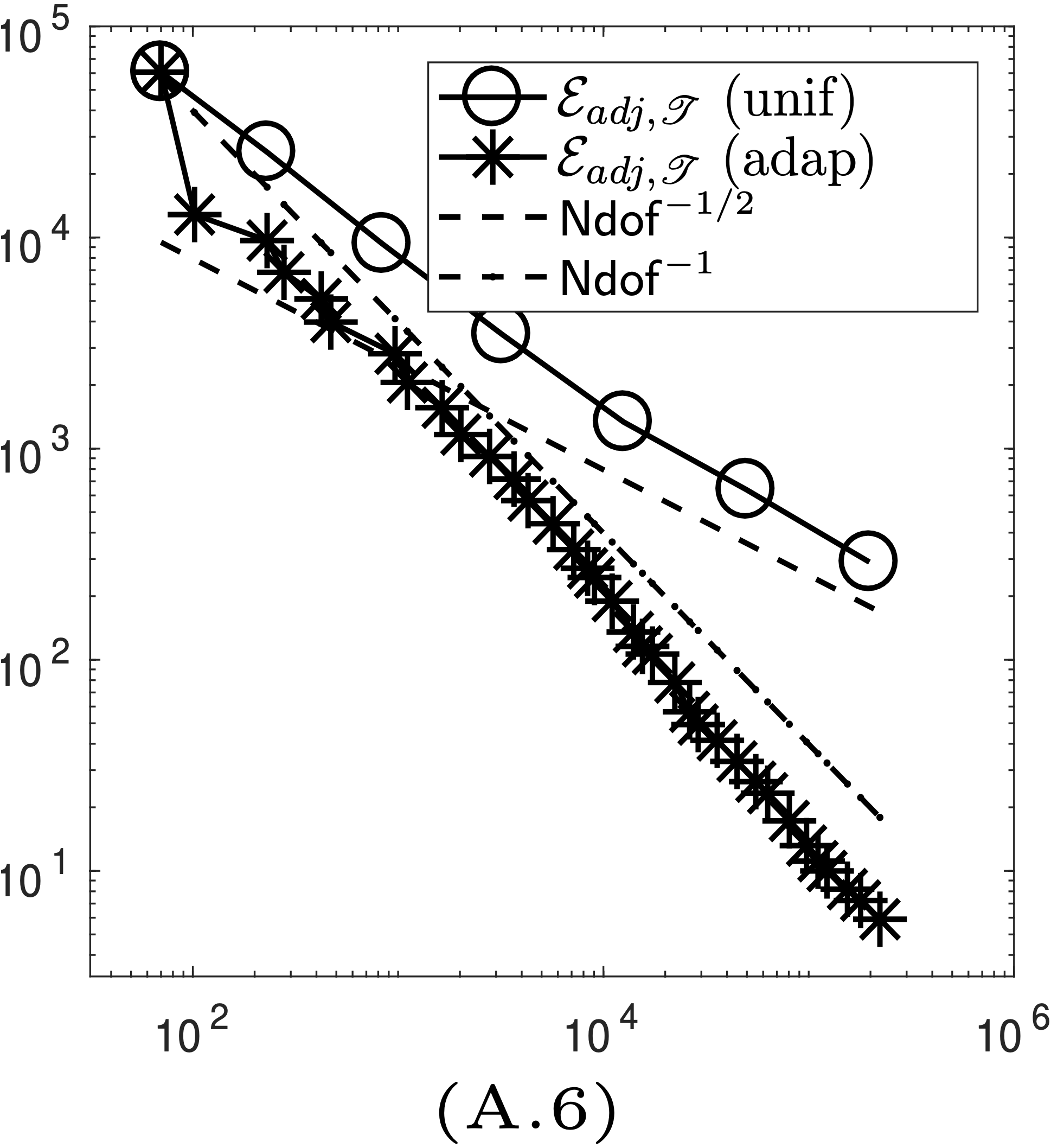} \\
\end{minipage}
\begin{minipage}[b]{0.23\textwidth}
\centering
\includegraphics[trim={0 0 0 0},clip,width=3cm,height=3cm,scale=0.66]{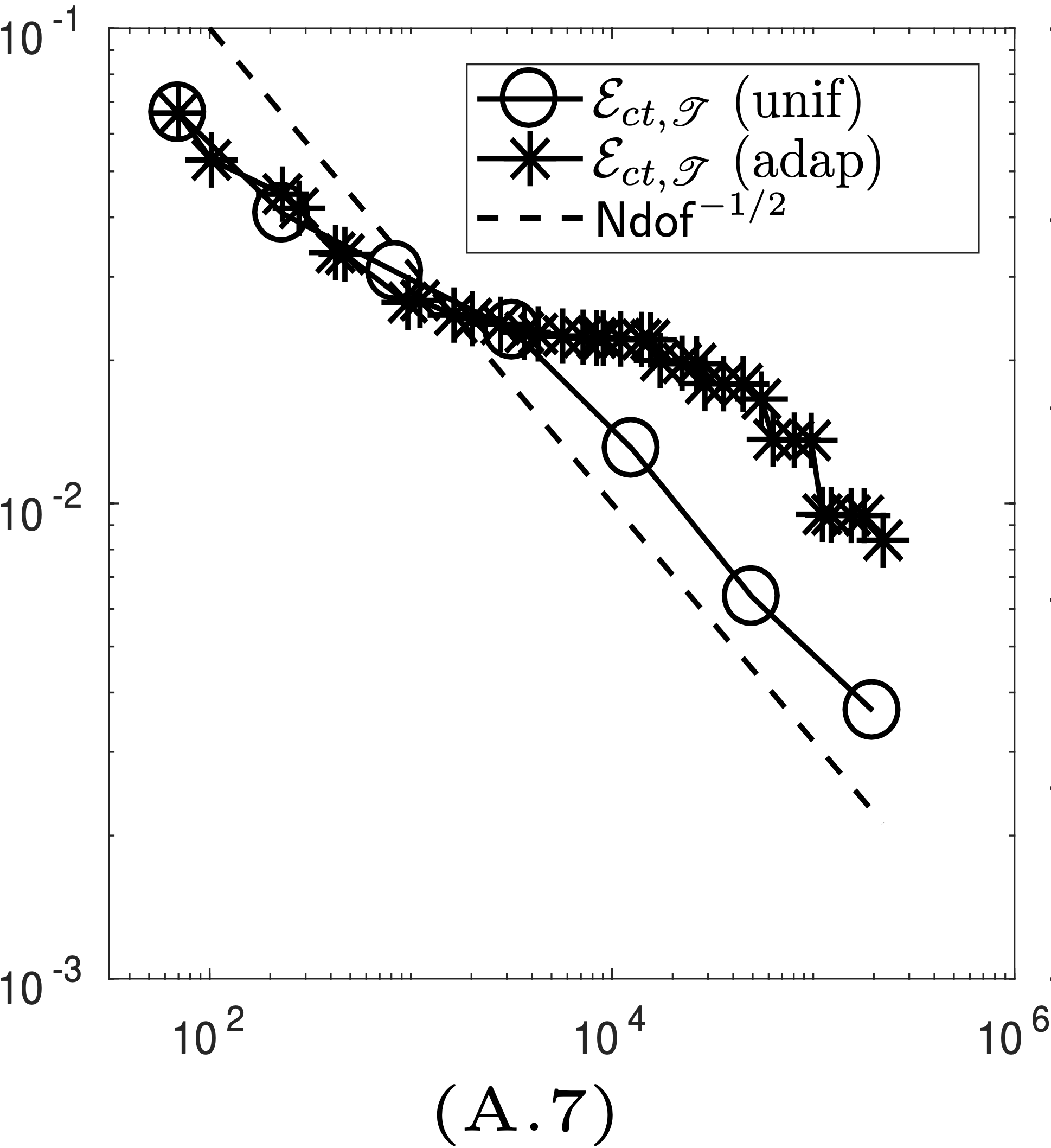} \\
\end{minipage}
\begin{minipage}[b]{0.23\textwidth}
\centering
\includegraphics[trim={0 0 0 0},clip,width=3cm,height=3cm,scale=0.66]{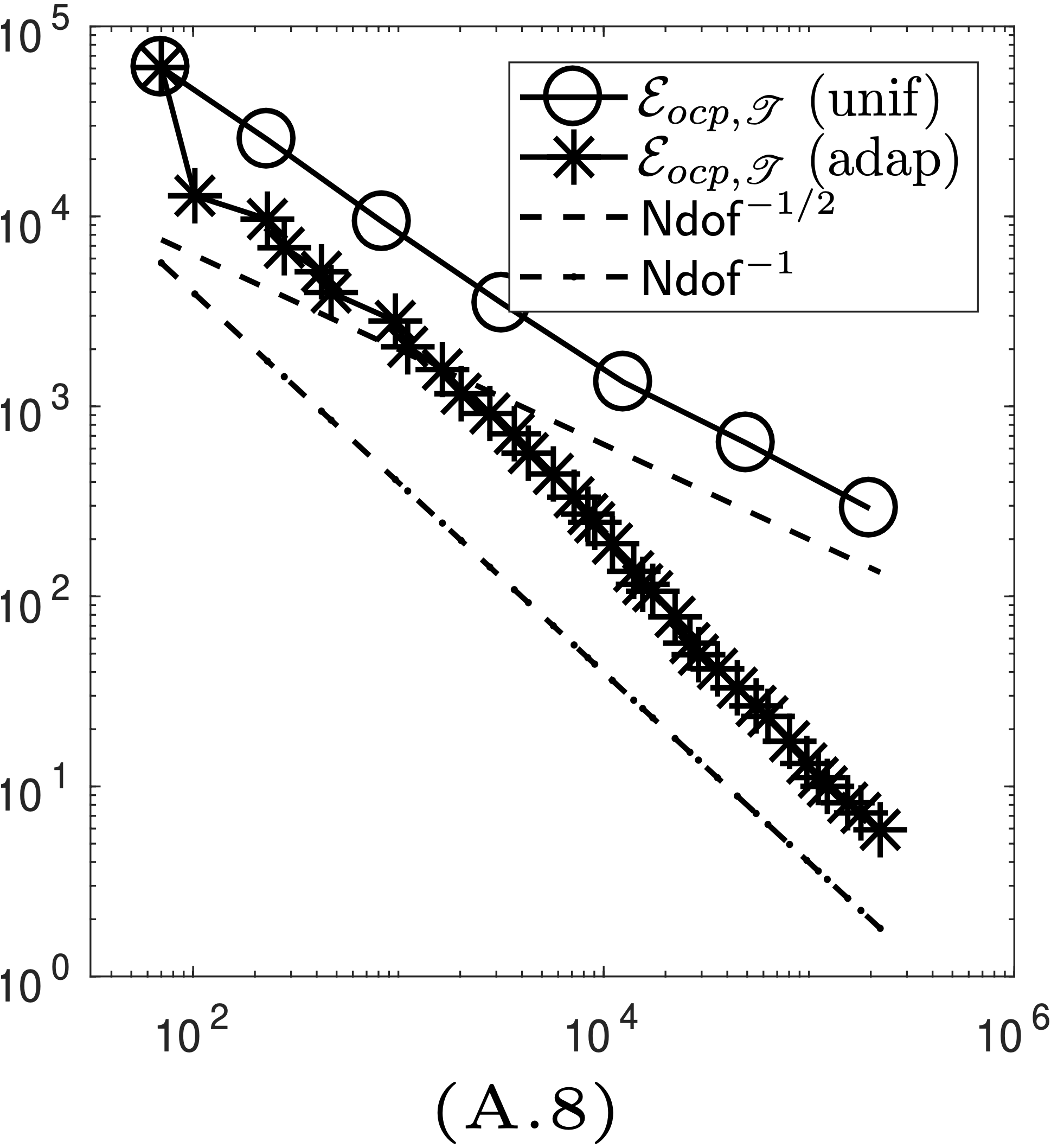} \\
\end{minipage}
\\
\begin{minipage}[b]{0.3\textwidth}
\centering
\includegraphics[trim={0 0 0 0},clip,width=3cm,height=3cm,scale=0.66]{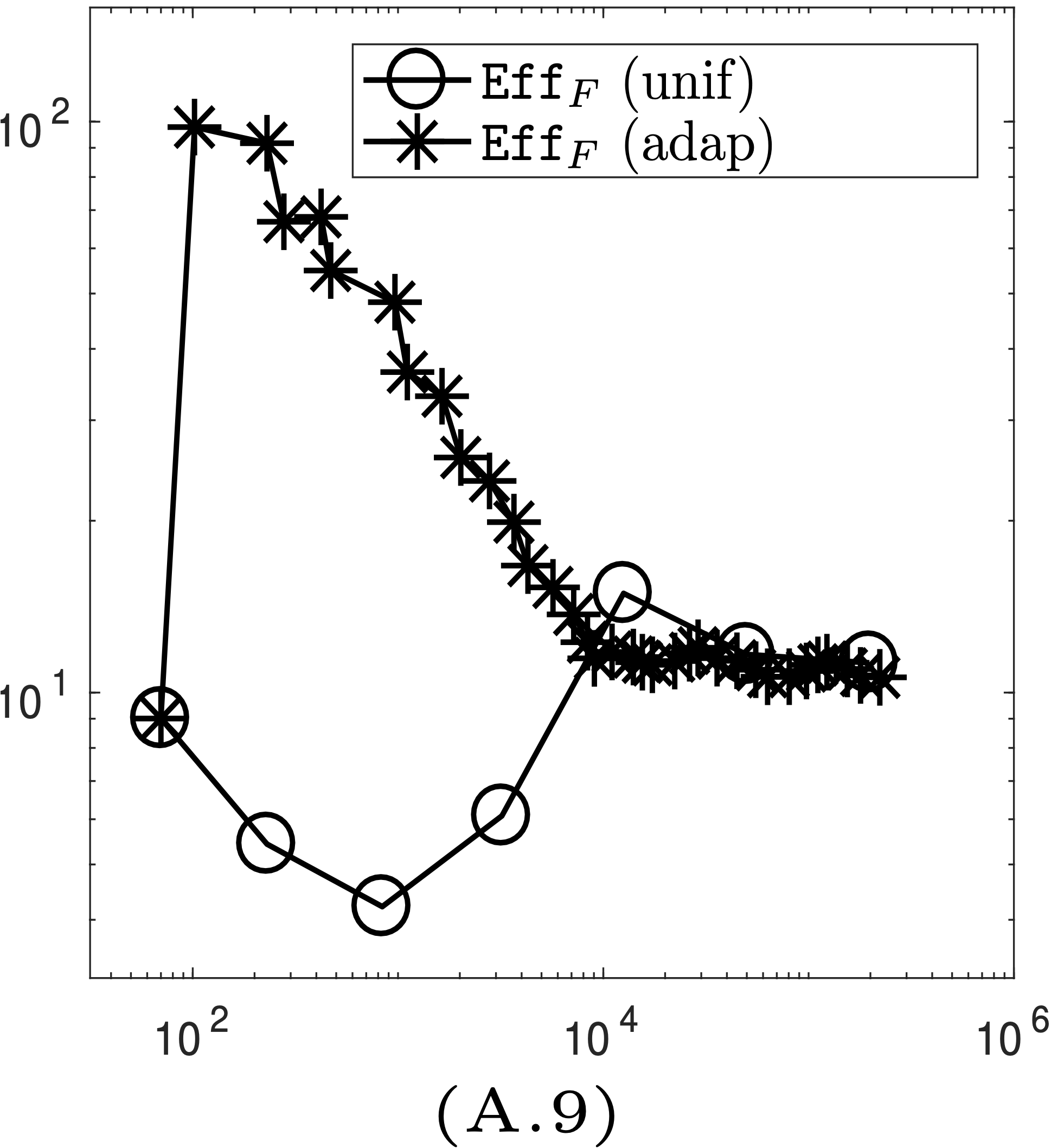} \\
\end{minipage}
\caption{\AAF{Example 1. Experimental convergence rates for the fully discrete scheme with uniform and adaptive refinement: state and adjoint state velocity errors (A.1); state and adjoint state pressure errors (A.2); control errors (A.3); total errors (A.4); state error estimators (A.5); adjoint state error estimators (A.6); control error estimators (A.7); total error estimators (A.8); and effectivity indices (A.9).}}
\label{fig:test_01_1}
\end{figure}

\begin{figure}[!h]
\centering
\begin{minipage}[b]{0.23\textwidth}
\centering
\includegraphics[trim={0 0 0 0},clip,width=3cm,height=3cm,scale=0.66]{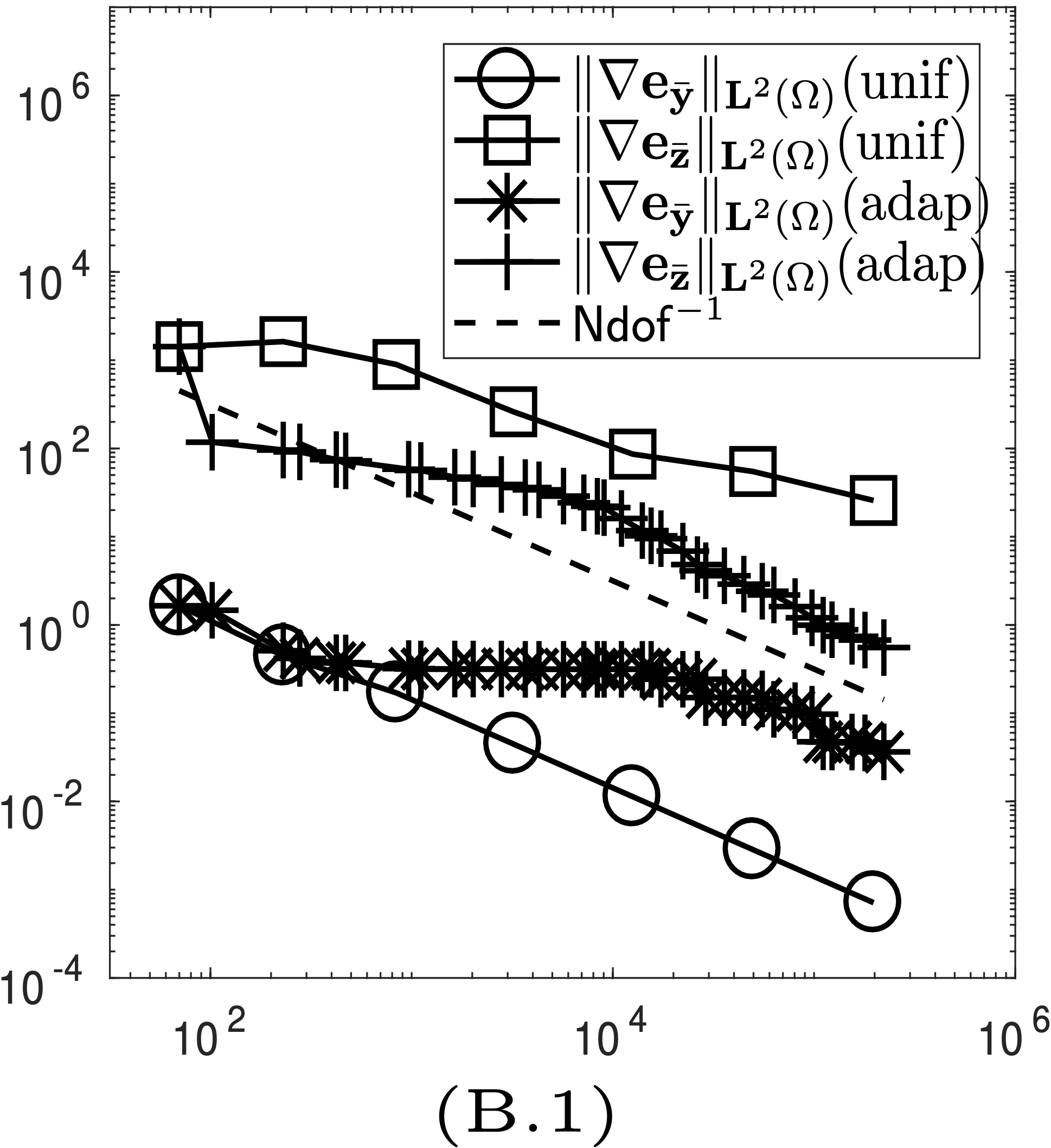} \\
\end{minipage}
\begin{minipage}[b]{0.23\textwidth}
\centering
\includegraphics[trim={0 0 0 0},clip,width=3cm,height=3cm,scale=0.66]{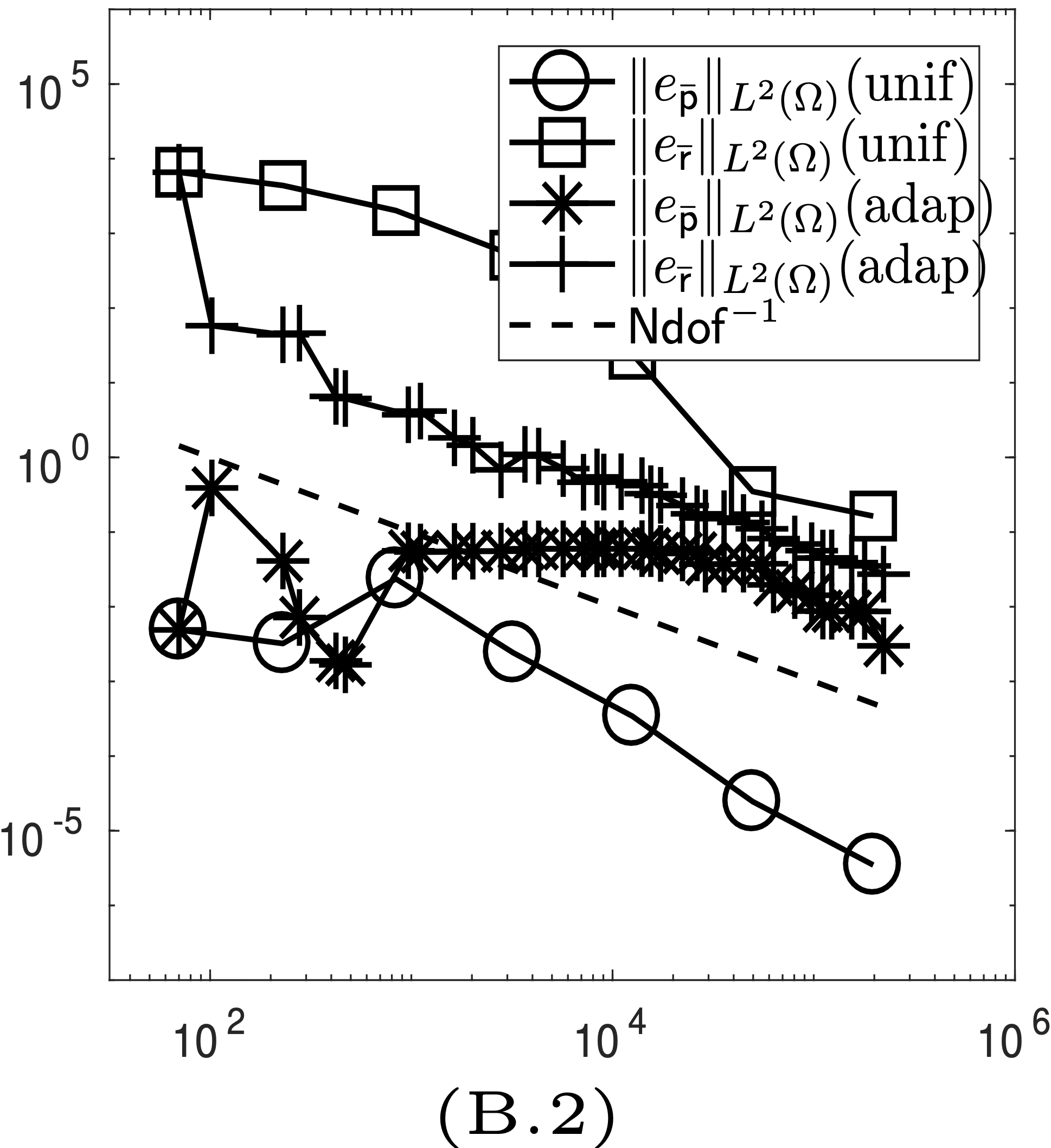} \\
\end{minipage}
\begin{minipage}[b]{0.23\textwidth}
\centering
\includegraphics[trim={0 0 0 0},clip,width=3cm,height=3cm,scale=0.66]{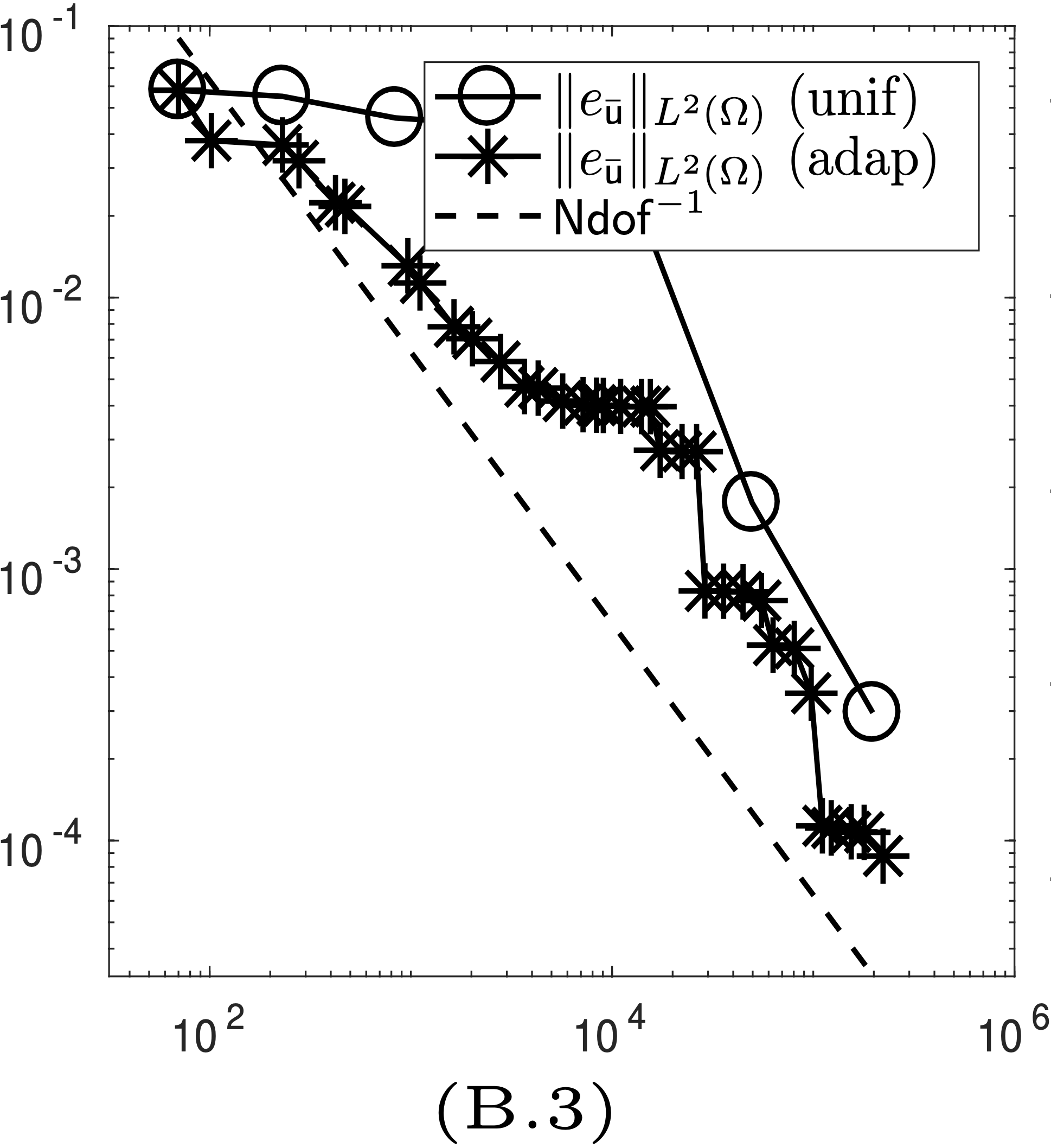} \\
\end{minipage}
\begin{minipage}[b]{0.23\textwidth}
\centering
\includegraphics[trim={0 0 0 0},clip,width=3cm,height=3cm,scale=0.66]{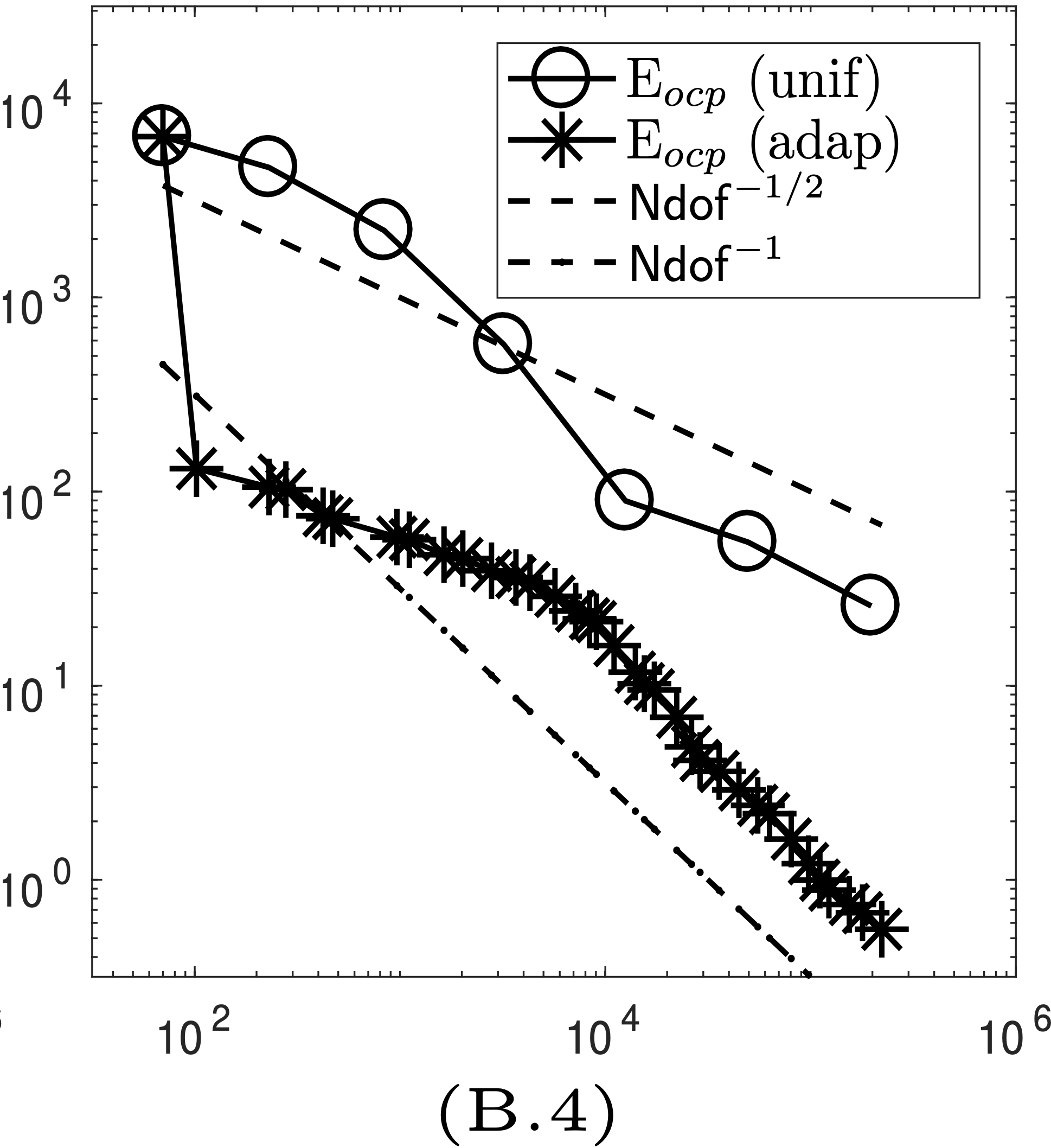} \\
\end{minipage}
\\
\begin{minipage}[b]{0.23\textwidth}
\centering
\includegraphics[trim={0 0 0 0},clip,width=3cm,height=3cm,scale=0.66]{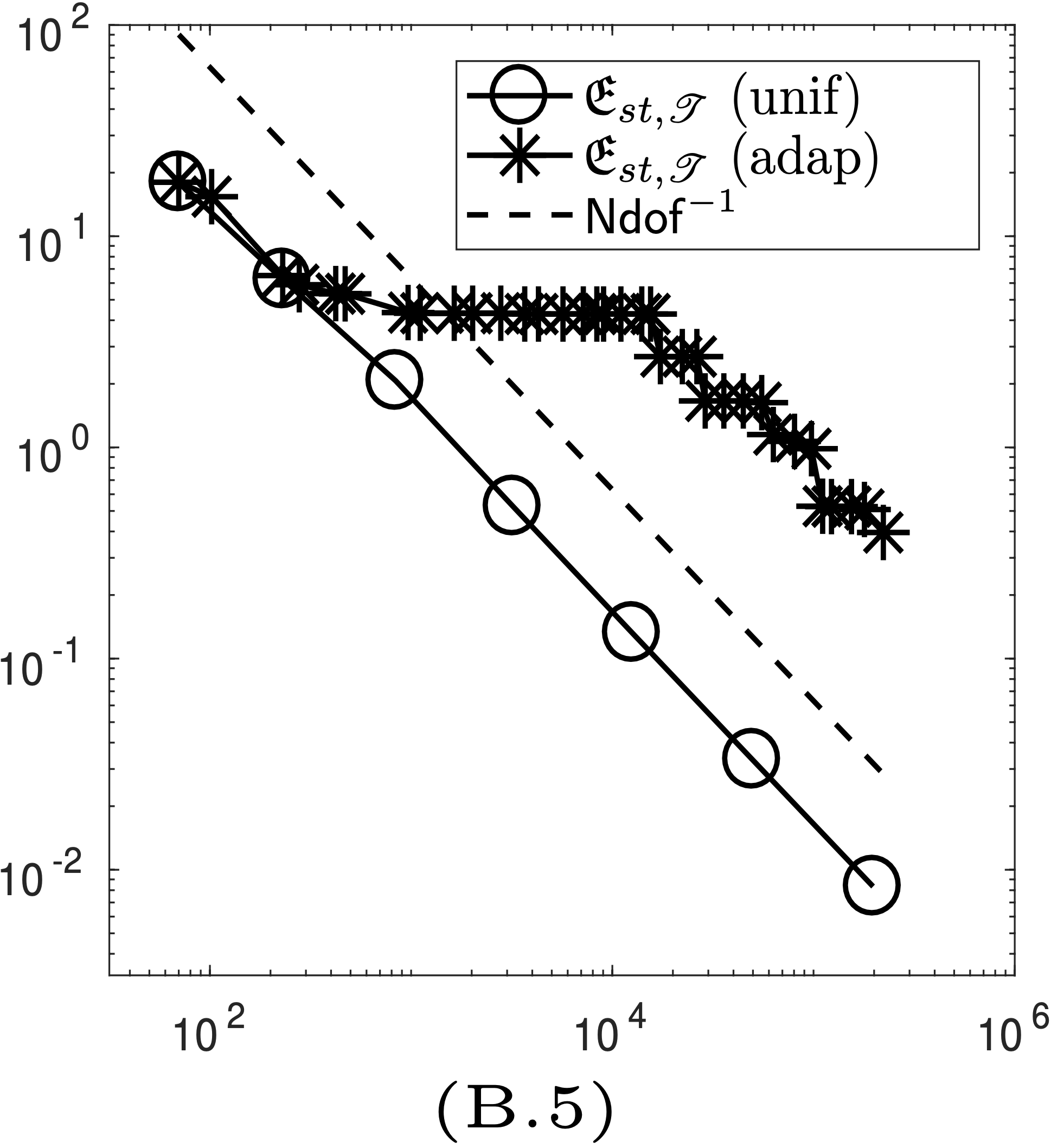} \\
\end{minipage}
\begin{minipage}[b]{0.23\textwidth}
\centering
\includegraphics[trim={0 0 0 0},clip,width=3cm,height=3cm,scale=0.66]{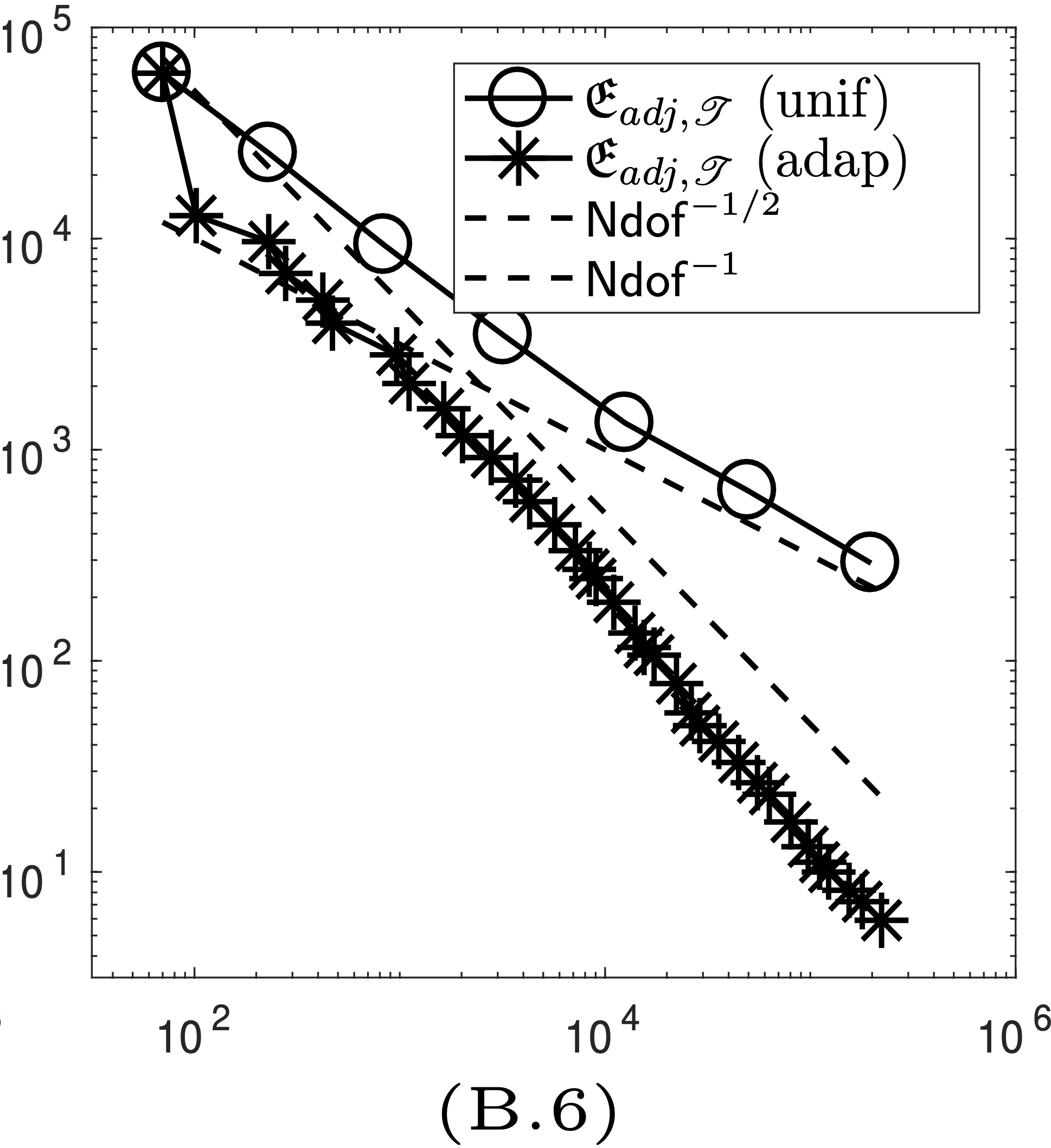} \\
\end{minipage}
\begin{minipage}[b]{0.23\textwidth}
\centering
\includegraphics[trim={0 0 0 0},clip,width=3cm,height=3cm,scale=0.66]{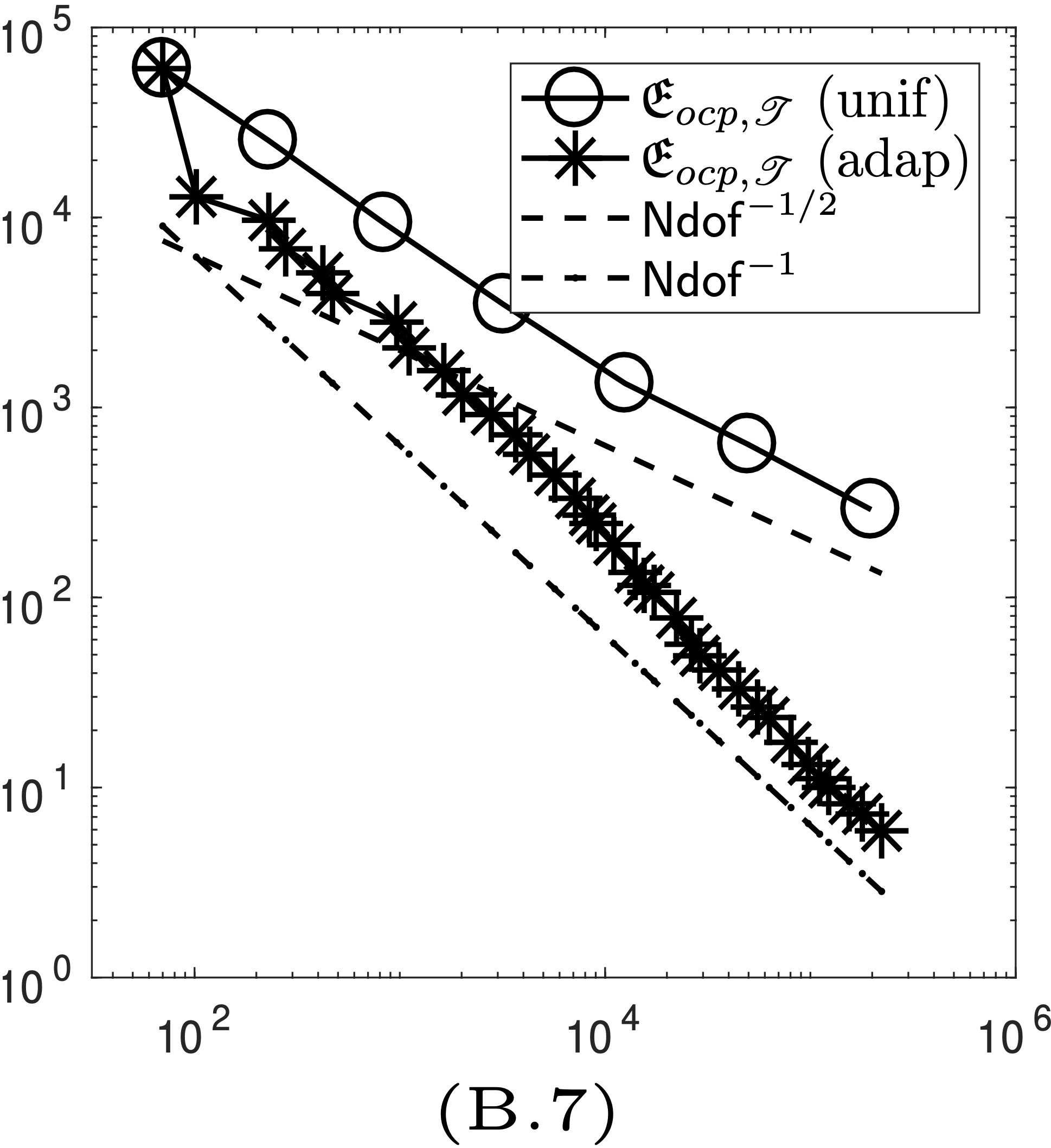} \\
\end{minipage}
\begin{minipage}[b]{0.23\textwidth}
\centering
\includegraphics[trim={0 0 0 0},clip,width=3cm,height=2.9cm,scale=0.66]{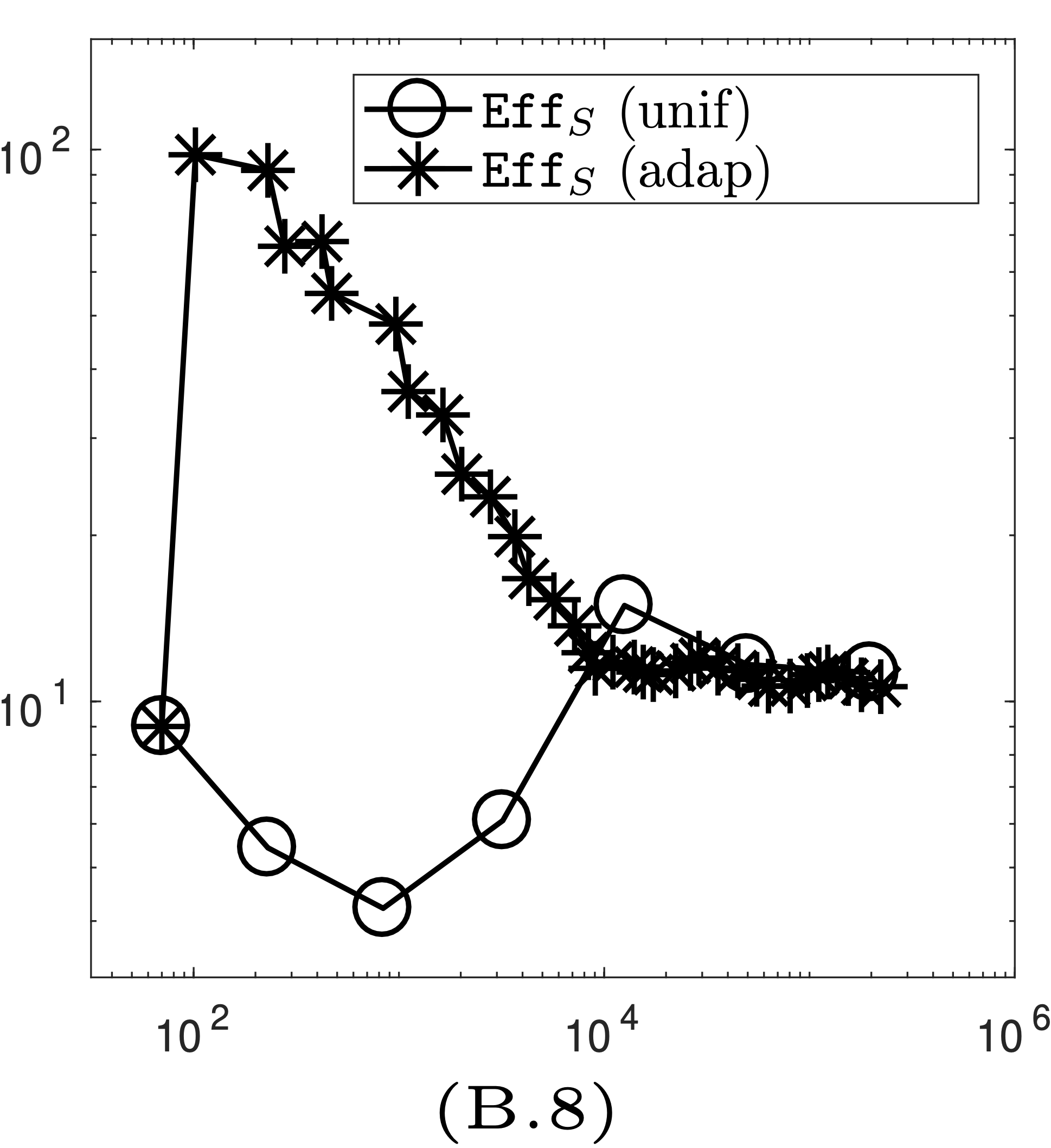} \\
\end{minipage}
\\
\caption{\AAF{Example 1. Experimental convergence rates for the semidiscrete scheme with uniform and adaptive refinement: state and adjoint state velocity errors (B.1); state and adjoint state pressure errors (B.2); control errors (B.3); total errors (B.4); state error estimators (B.5); adjoint state error estimators (B.6); total error estimators (B.7); and effectivity indices (B.8).}}
\label{fig:test_01_1_2}
\end{figure}

\begin{figure}[!h]
\centering
\begin{minipage}[b]{0.24\textwidth}
\centering
\includegraphics[trim={0cm 0cm 0cm 0cm},clip,width=2.7cm,height=3.2cm,scale=0.66]{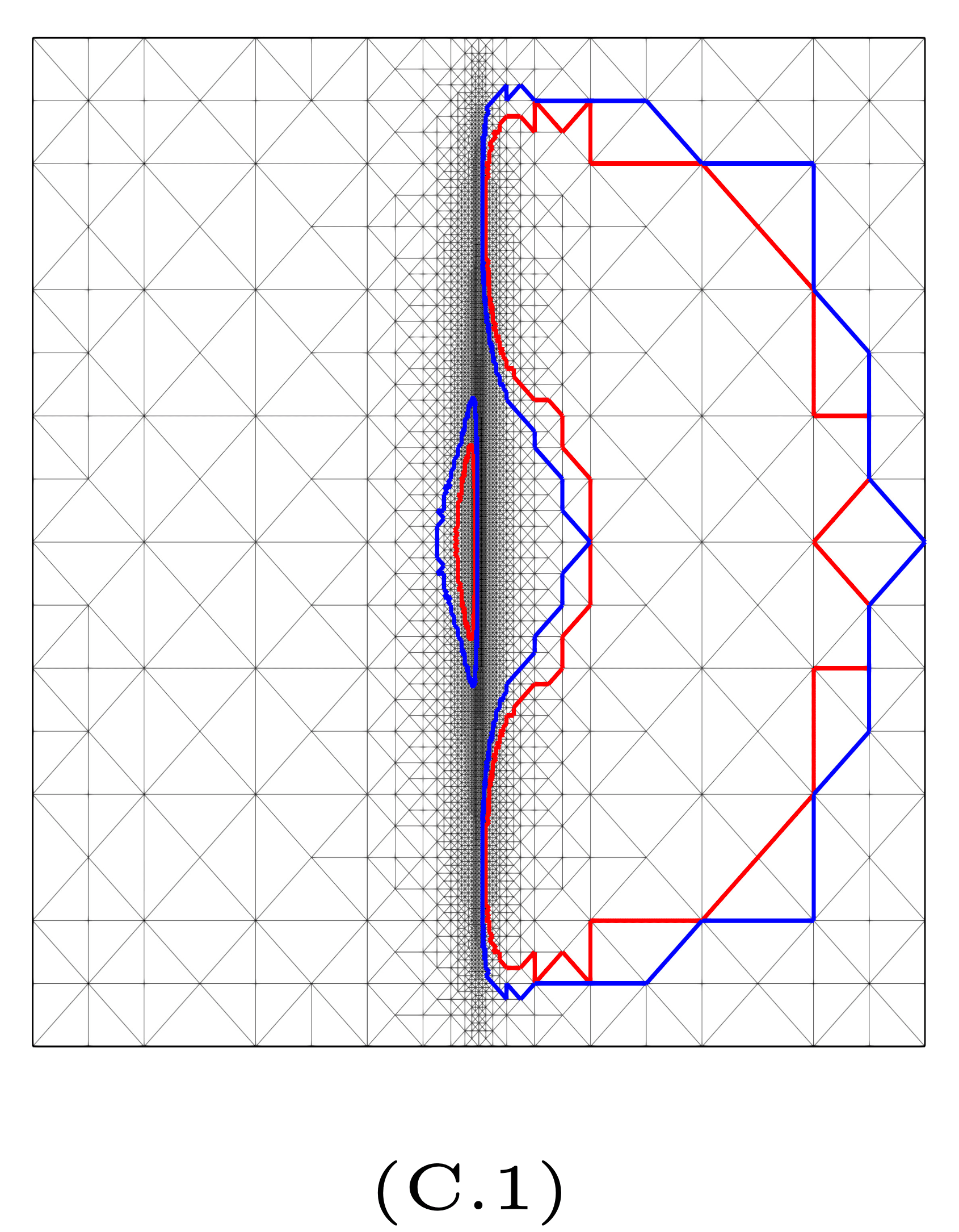} \\
\end{minipage}
\begin{minipage}[b]{0.24\textwidth}
\centering
\includegraphics[trim={0cm 0cm 0cm 0cm},clip,width=2.7cm,height=3.2cm,scale=0.66]{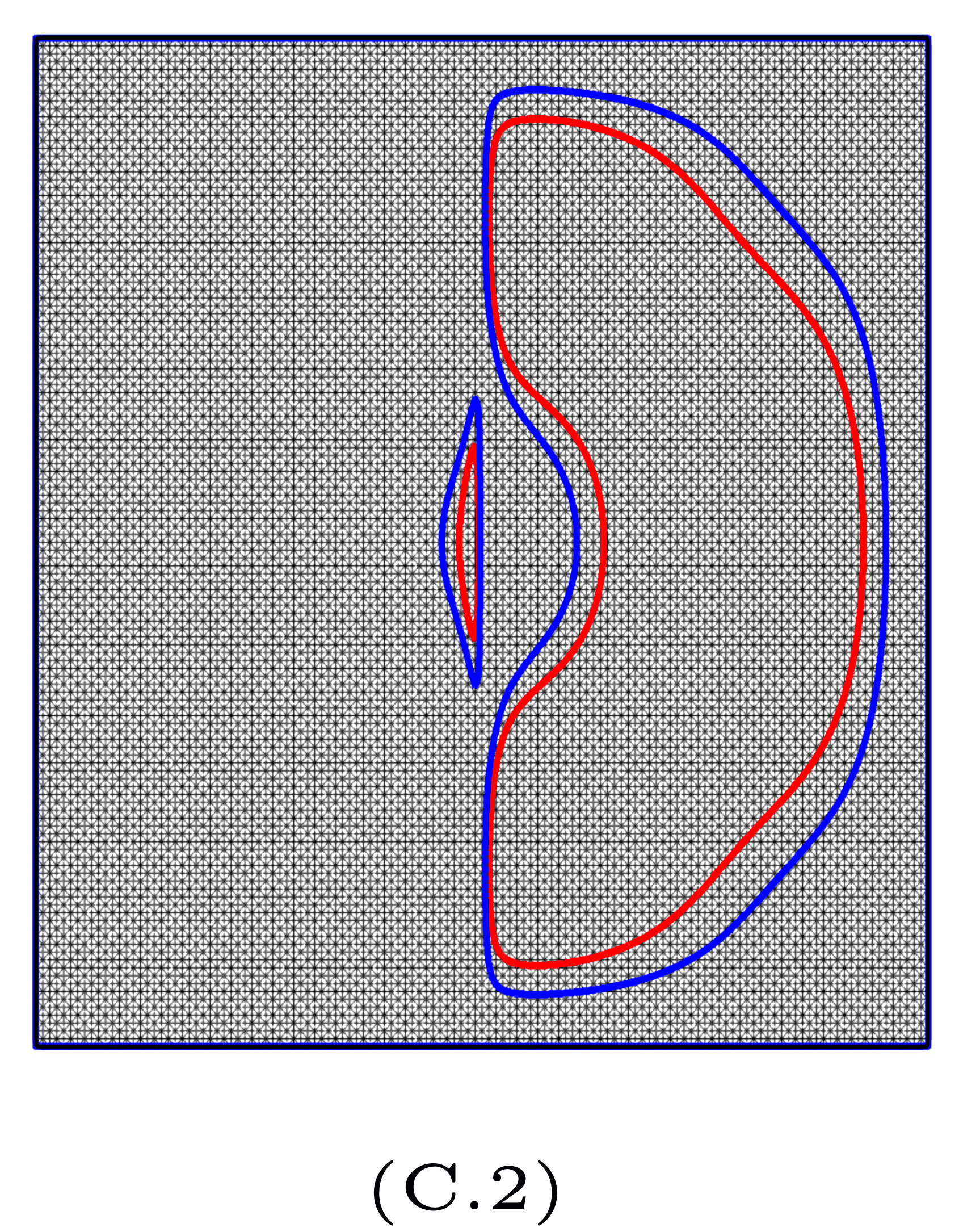}  \\
\end{minipage}
\begin{minipage}[b]{0.24\textwidth}
\centering
\includegraphics[trim={0 0 0 0},clip,width=3cm,height=3.2cm,scale=0.66]{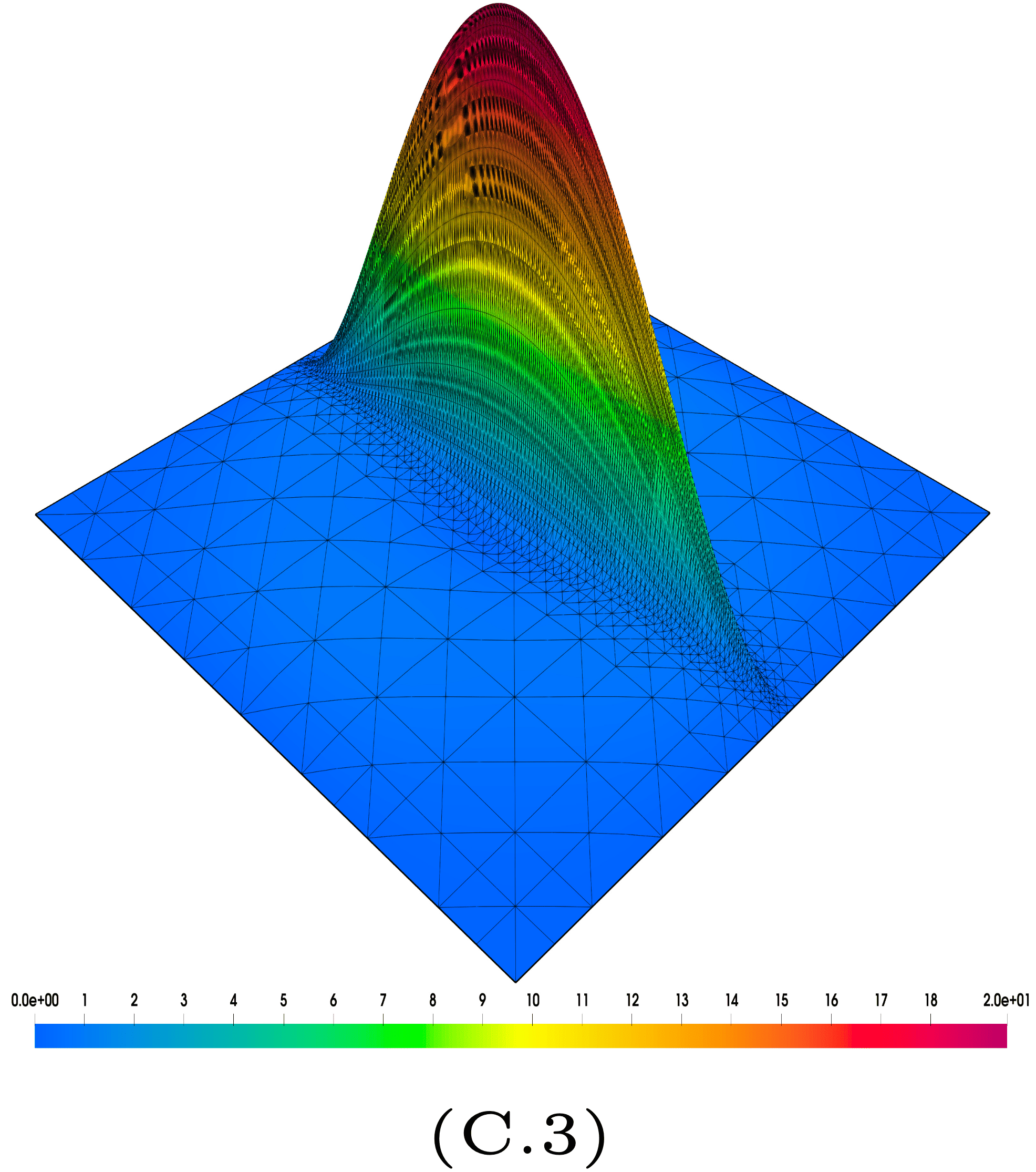}  \\
\end{minipage}
\begin{minipage}[b]{0.24\textwidth}
\centering
\includegraphics[trim={0 0 0 0},clip,width=3cm,height=3.2cm,scale=0.66]{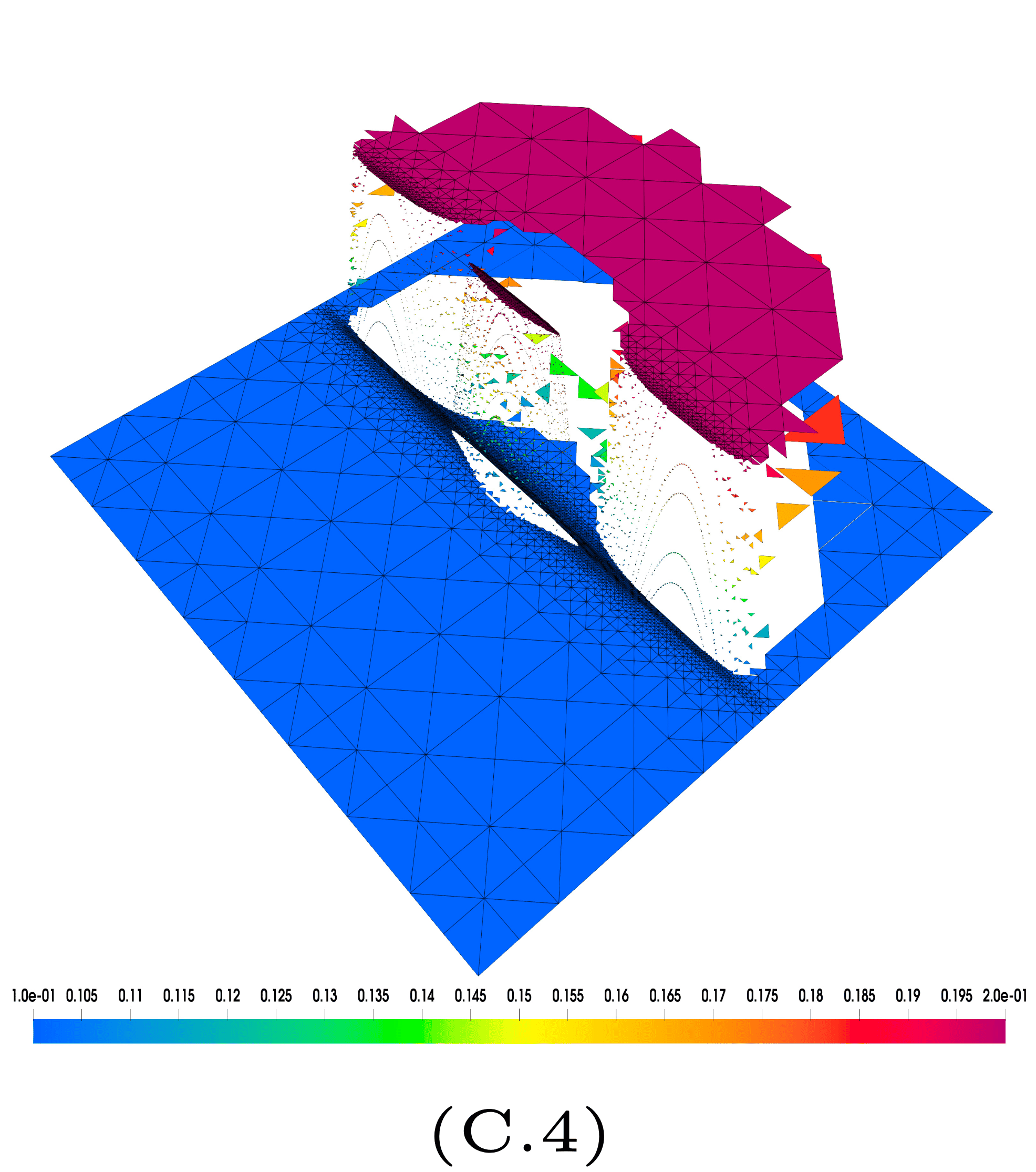}  \\
\end{minipage}
\caption{\AAF{Example 1. For the fully discrete scheme we present the mesh obtained after 35 iterations of adaptive refinement (18.504 triangles and 9.285 vertices), the discrete boundary of the active sets, where the blue line corresponds to the one for $\mathsf{a}=0.1$, and the red line corresponds to the one for $\mathsf{b}=0.2$ (C.1); the exact boundary of the corresponding active sets (C.2),
the discrete optimal adjoint velocity $|\bar{\mathbf{z}}_{h}|$ (C.3); and the discrete optimal control $\bar{u}_{h}$ (C.4).}}
\label{fig:test_01_2}
\end{figure}

\begin{figure}[!h]
\centering
\begin{minipage}[b]{0.24\textwidth}
\centering
\includegraphics[trim={0 0 0 0},clip,width=2.7cm,height=3.2cm,scale=0.66]{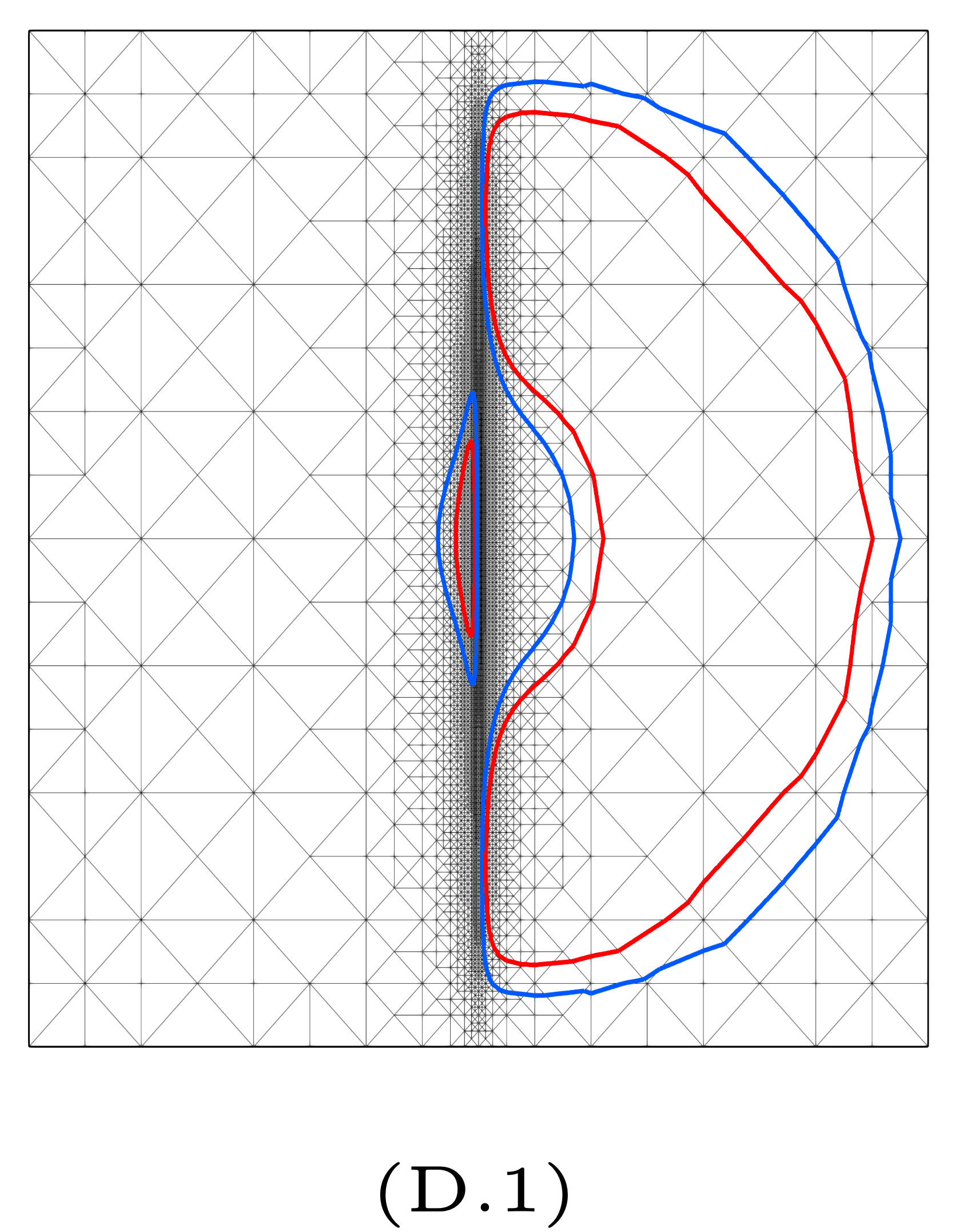} \\
\end{minipage}
\begin{minipage}[b]{0.24\textwidth}
\centering
\includegraphics[trim={0cm 0cm 0cm 0cm},clip,width=2.7cm,height=3.2cm,scale=0.66]{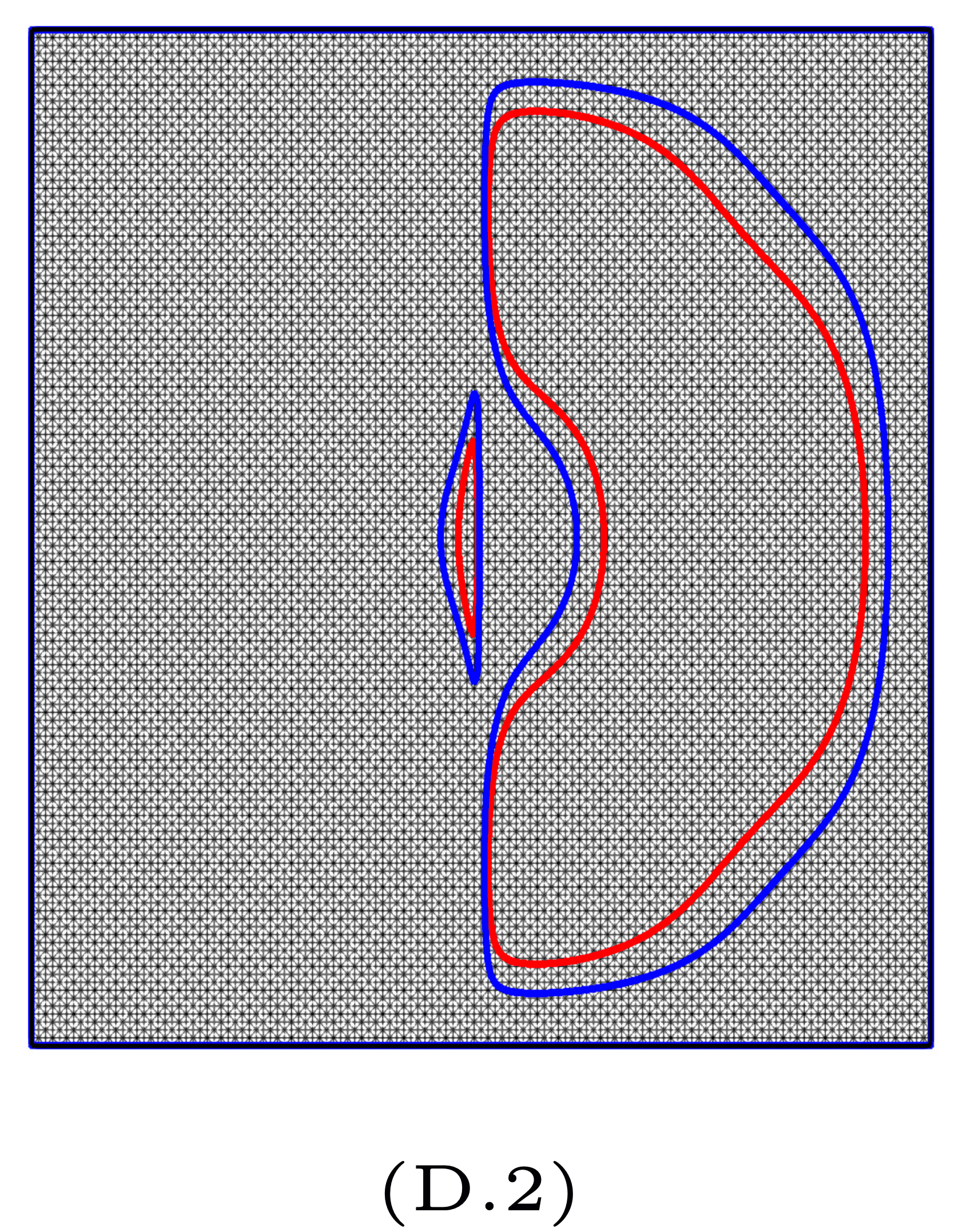} \\
\end{minipage}
\begin{minipage}[b]{0.24\textwidth}
\centering
\includegraphics[trim={0 0 0 0},clip,width=3cm,height=3cm,scale=0.66]{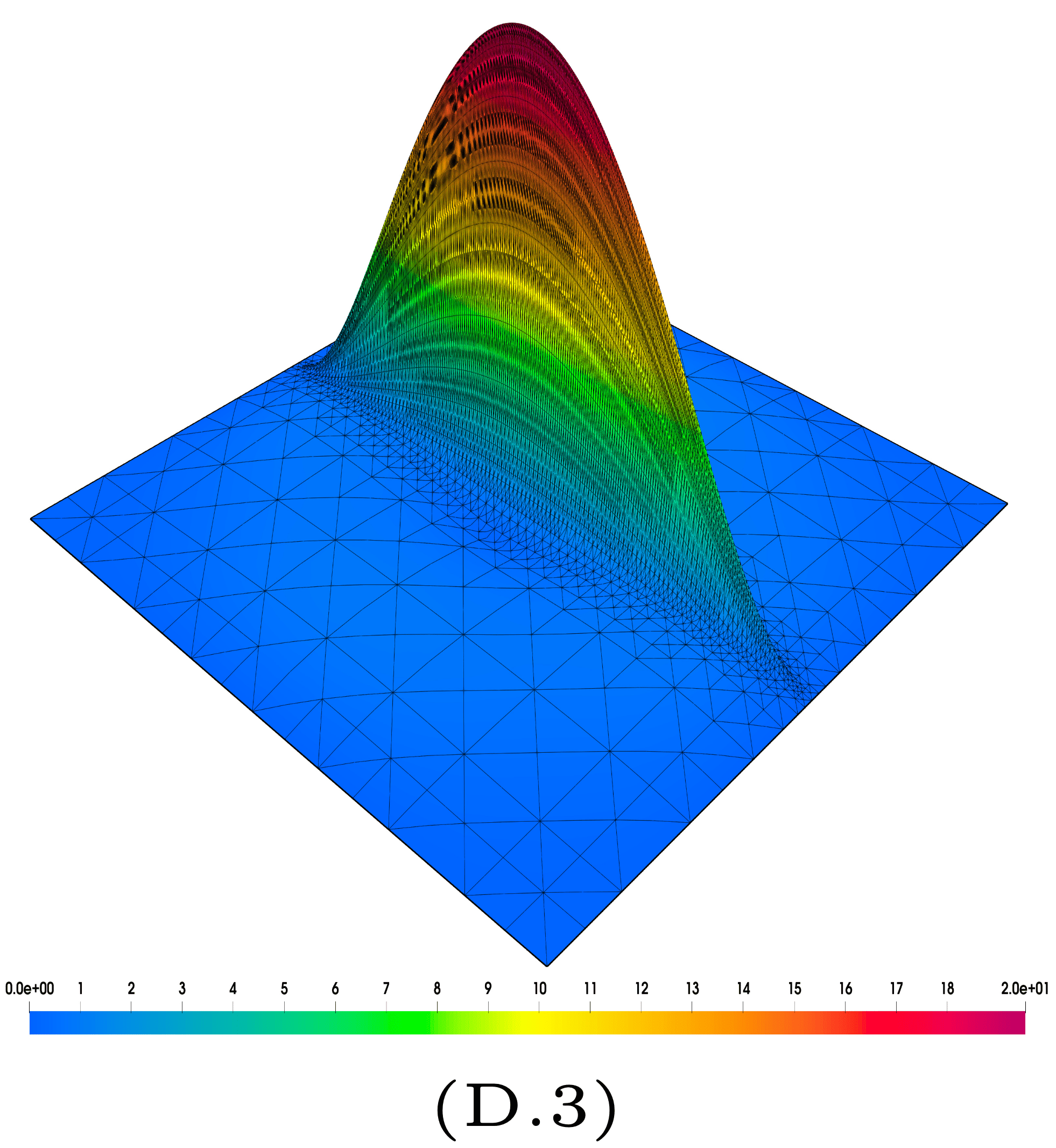} \\
\end{minipage}
\begin{minipage}[b]{0.24\textwidth}
\centering
\includegraphics[trim={0 0 0 0},clip,width=3cm,height=3cm,scale=0.66]{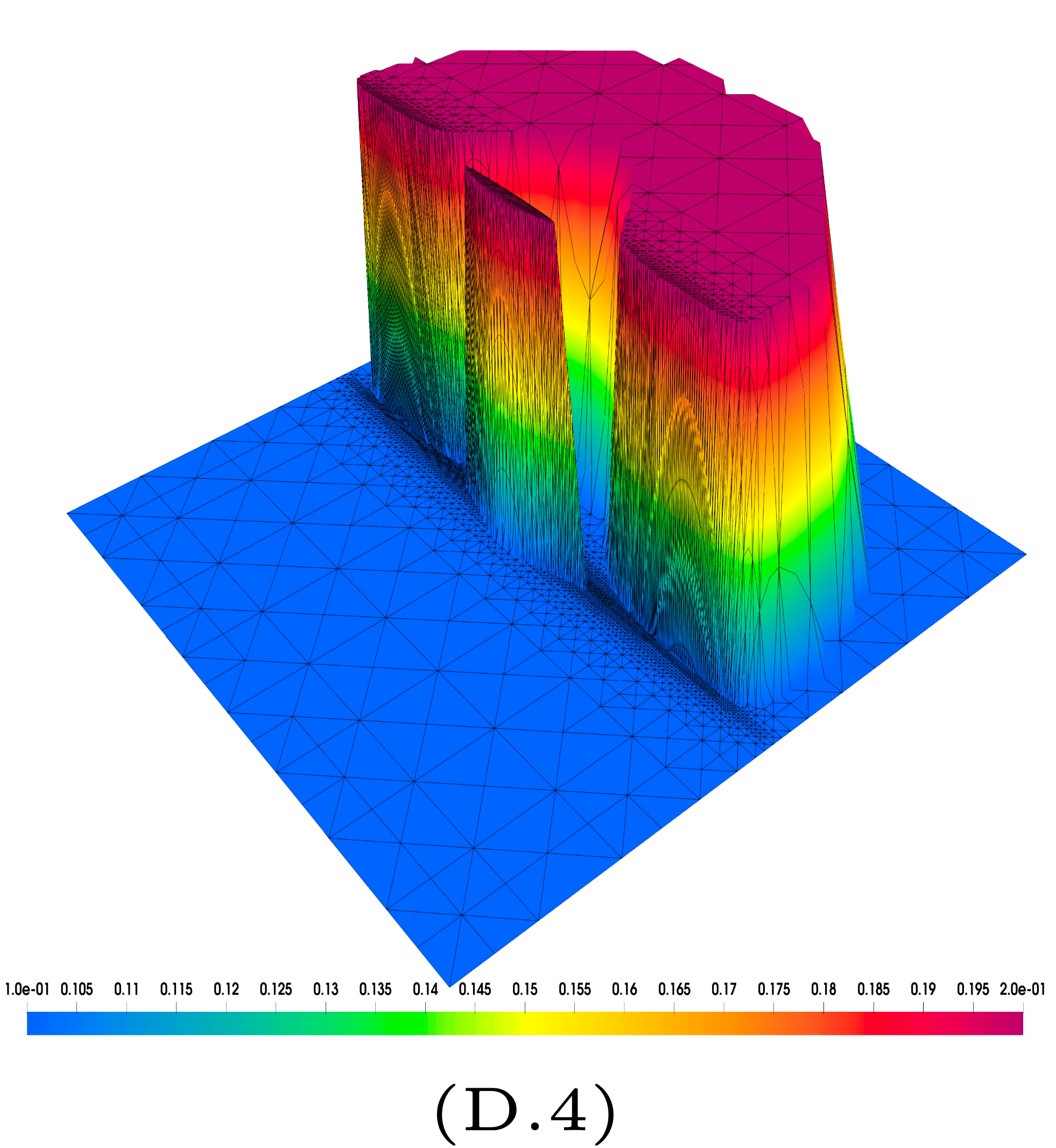} \\
\end{minipage}
\caption{\AAF{Example 1. For the semidiscrete scheme we present the mesh obtained after 35 iterations of adaptive refinement (18.504 triangles and 9.285 vertices), the discrete boundary of the active sets, where the blue line corresponds to the one for $\mathsf{a}=0.1$, and the red line corresponds to the one for $\mathsf{b}=0.2$ (D.1); the exact boundary of the corresponding active sets (D.2); the discrete optimal adjoint velocity $|\bar{\mathbf{z}}_{h}|$ (D.3); and the discrete optimal control $\bar{\mathsf{u}}=\Pi_{[\mathsf{a},\mathsf{b}]}(\alpha^{-1}\bar{\mathbf{y}}_{h}\cdot\bar{\mathbf{z}}_{h})$ (D.4).}}
\label{fig:test_01_2_new}
\end{figure}

\subsection{A non-convex domain}

 In this section, we apply the adaptive strategy to a problem posed on an L-shaped domain for which no analytical solution is known.
 
\textbf{Example 2.} We set $\Omega=(-1, 1) \times (-1, 1) \setminus [0,1] \times [-1,0]$, $\mathsf{a}=1.0$, $\mathsf{b}=5.0$, $\alpha=1.0$, $\mathbf{f}=1000((x+y)^4,\sin(2\pi x)\sin(2\pi y))^{\mathrm{T}}$, and $\mathbf{y}_{\Omega}=1000(\sin(2\pi x)\sin(2\pi y),xy(1-x)(1-y))^{\mathrm{T}}$.

In \AAF{Figures \ref{fig:test_0_1_2} (fully discrete scheme) and \ref{fig:test_0_1_2_1} (semidiscrete scheme), we present experimental rates of convergence for the individual contributions of the total errors and the error estimators. We observe that optimal experimental convergence rates are achieved. In Figures \ref{fig:test_02} (fully discrete scheme) and \ref{fig:test_02_SEMI} (semi discrete scheme), we show the meshes obtained after 50 iterations of adaptive refinement, the delineations of the discrete boundaries
of the active sets (see panels (G.1) and (H.1)), and the discrete optimal control. The figures clearly illustrate that the mesh refinement predominantly concentrates around the geometric singularity.}

\begin{figure}[!h]
\centering
\begin{minipage}[b]{0.22\textwidth}
\centering\hspace*{-0.8cm}
\includegraphics[trim={0 0 0 0},clip,width=2.7cm,height=2.7cm,scale=0.66]{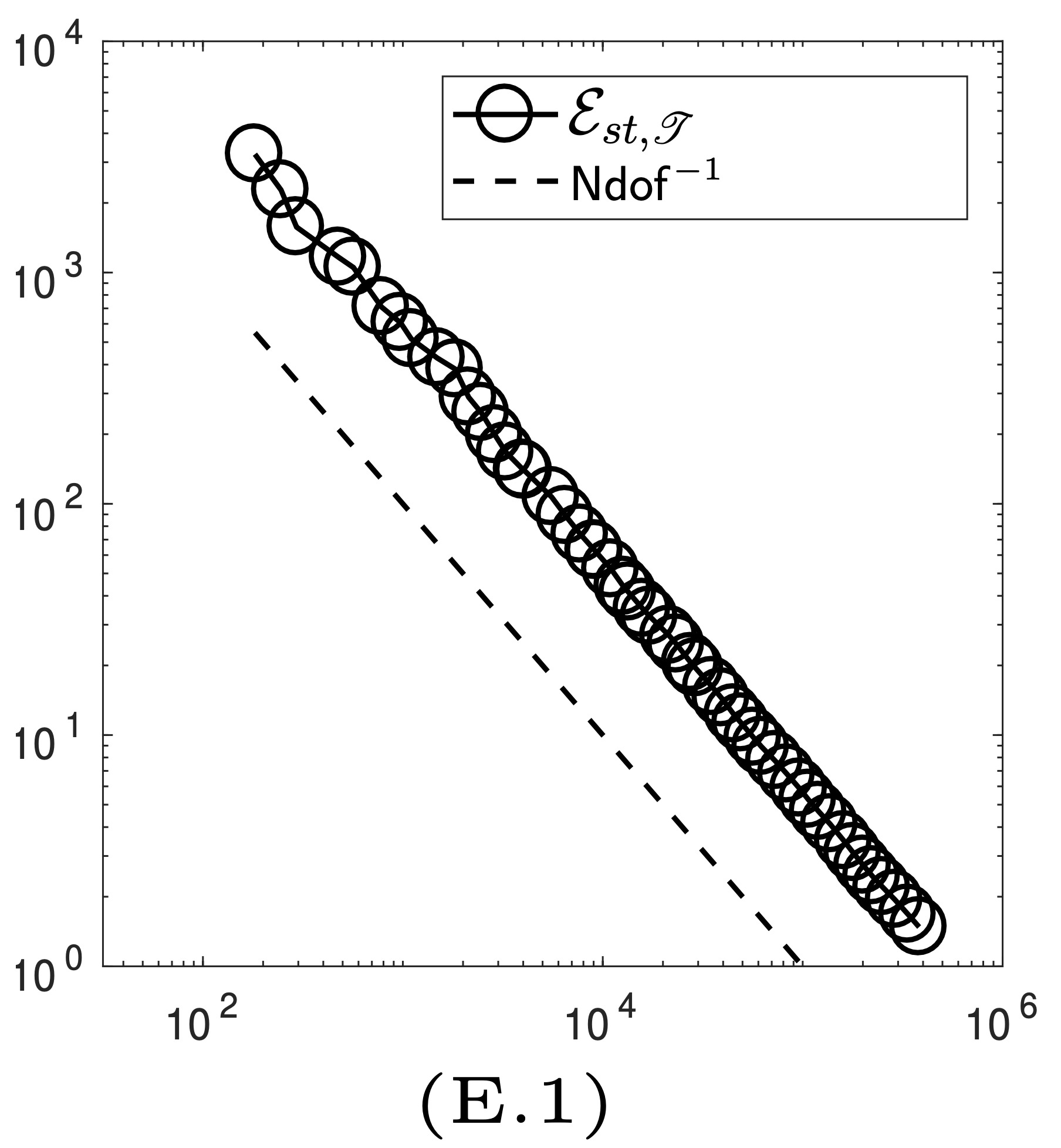} \\
\end{minipage}
\begin{minipage}[b]{0.22\textwidth}
\centering\hspace*{-0.3cm}
\includegraphics[trim={0 0 0 0},clip,width=2.7cm,height=2.7cm,scale=0.66]{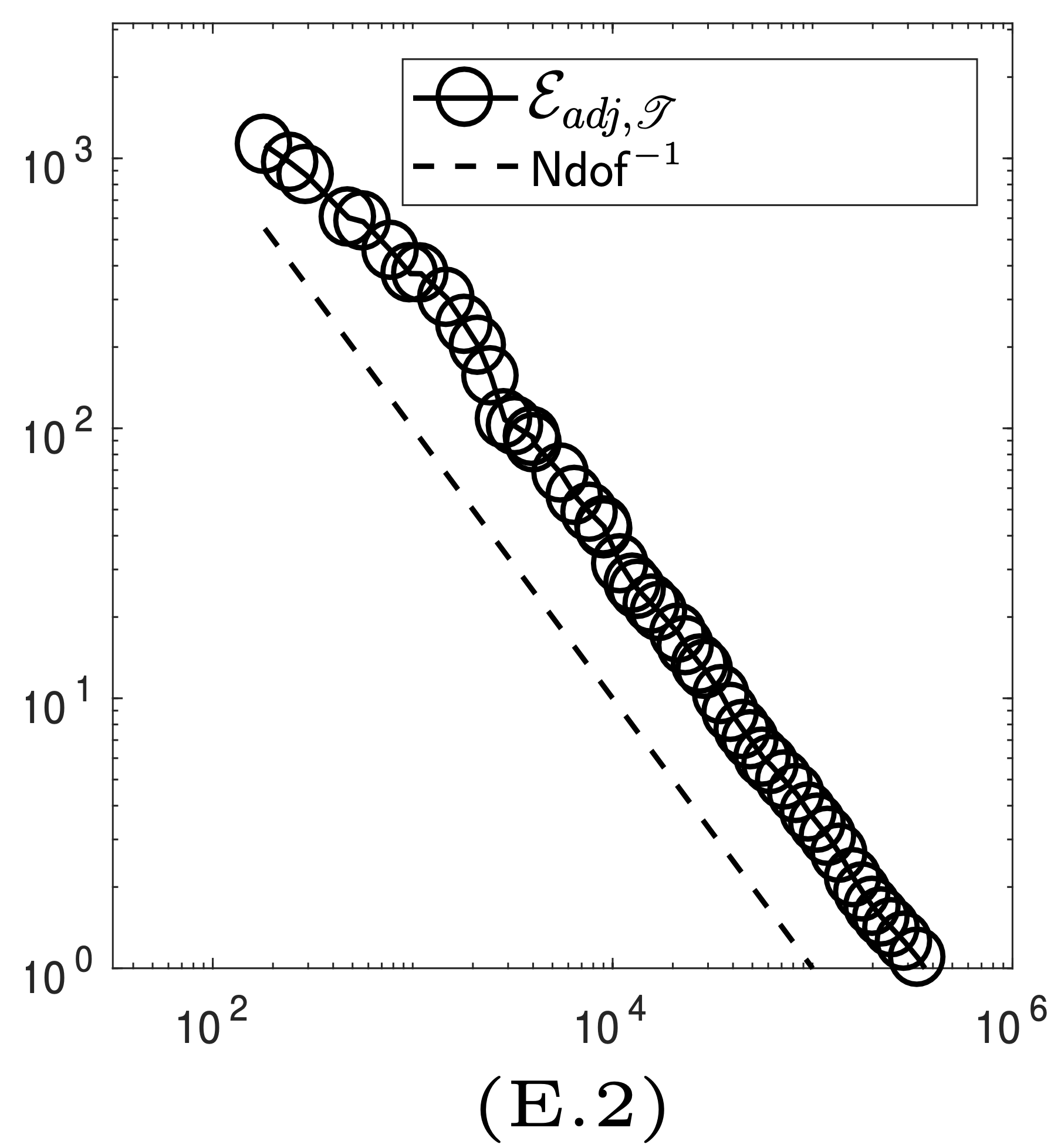} \\
\end{minipage}
\begin{minipage}[b]{0.22\textwidth}
\centering
\includegraphics[trim={0 0 0 0},clip,width=2.7cm,height=2.7cm,scale=0.66]{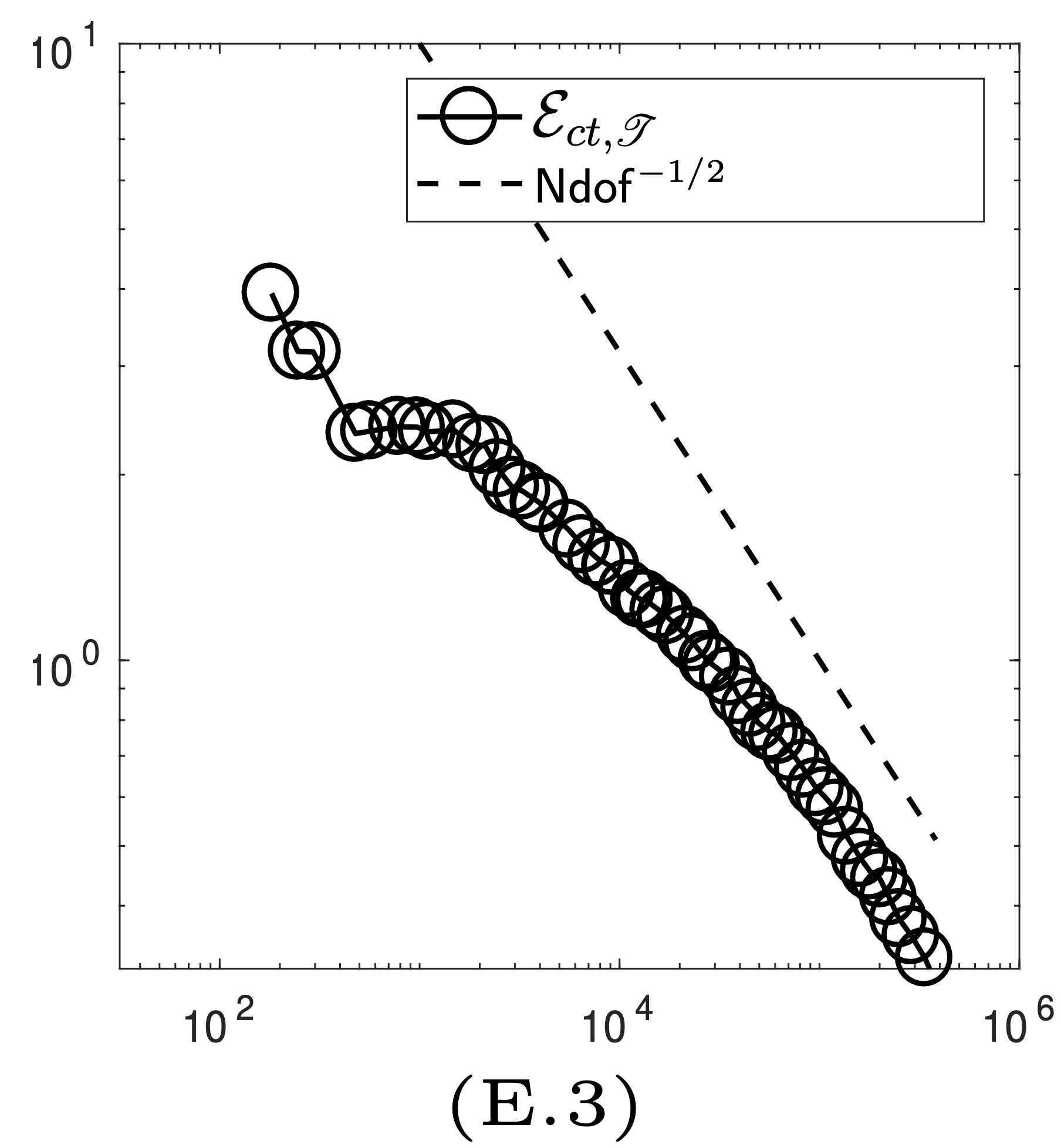} \\
\end{minipage}
\begin{minipage}[b]{0.22\textwidth}
\centering\hspace*{0.3cm}
\includegraphics[trim={0 0 0 0},clip,width=2.7cm,height=2.7cm,scale=0.66]{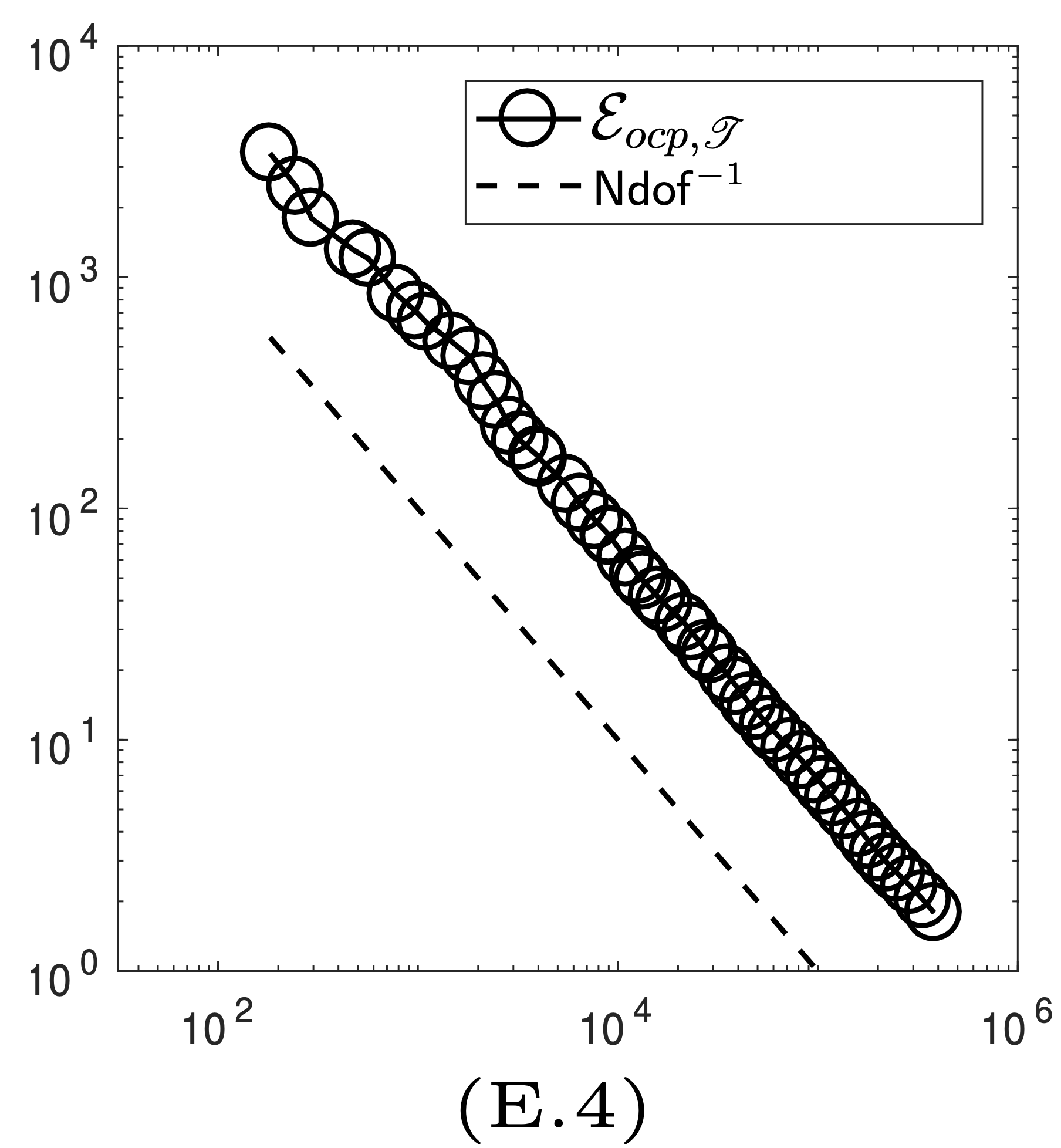} \\
\end{minipage}
\caption{\AAF{Example 2. We present experimental convergence rates for the fully discrete scheme: state error estimator (E.1); adjoint error estimator (E.2); control error estimator (E.3); and total error estimator (E.4).}}
\label{fig:test_0_1_2}
\end{figure}


\begin{figure}[!h]
\centering
\begin{minipage}[b]{0.25\textwidth}
\centering
\includegraphics[trim={0 0 0 0},clip,width=2.7cm,height=2.7cm,scale=0.66]{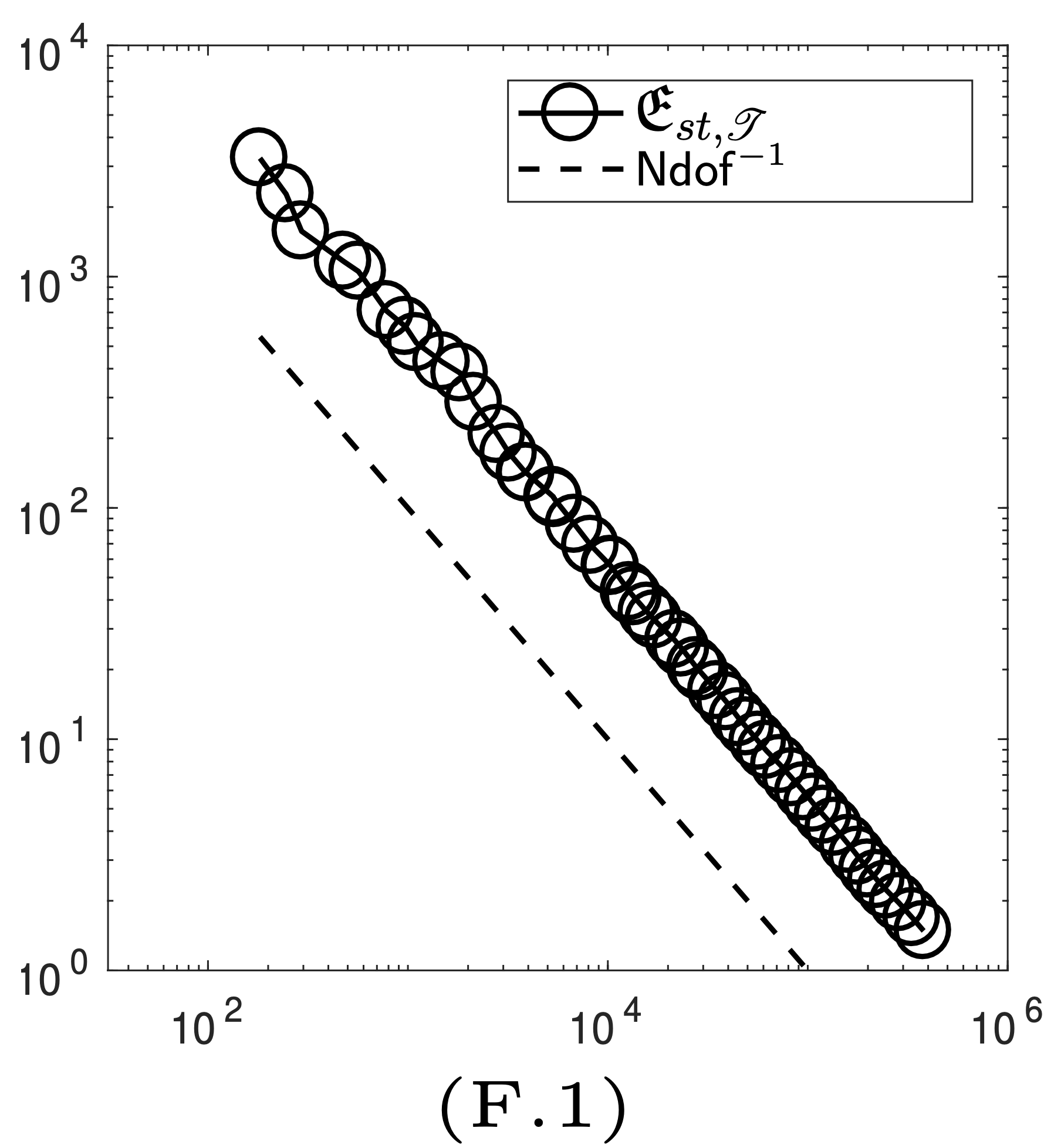} \\
\end{minipage}
\begin{minipage}[b]{0.25\textwidth}
\centering
\includegraphics[trim={0 0 0 0},clip,width=2.7cm,height=2.7cm,scale=0.66]{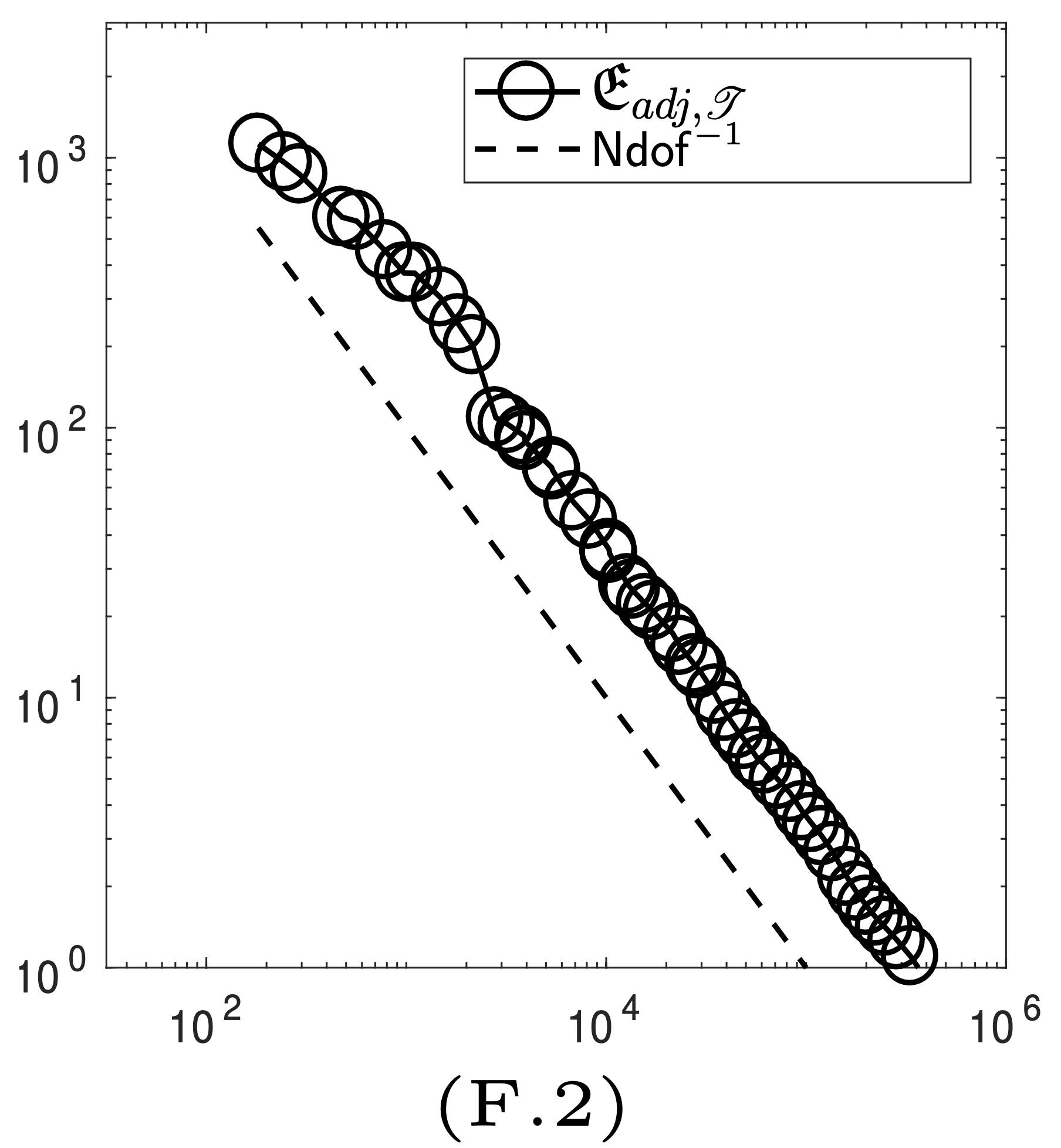} \\
\end{minipage}
\begin{minipage}[b]{0.25\textwidth}
\centering
\includegraphics[trim={0 0 0 0},clip,width=2.7cm,height=2.7cm,scale=0.66]{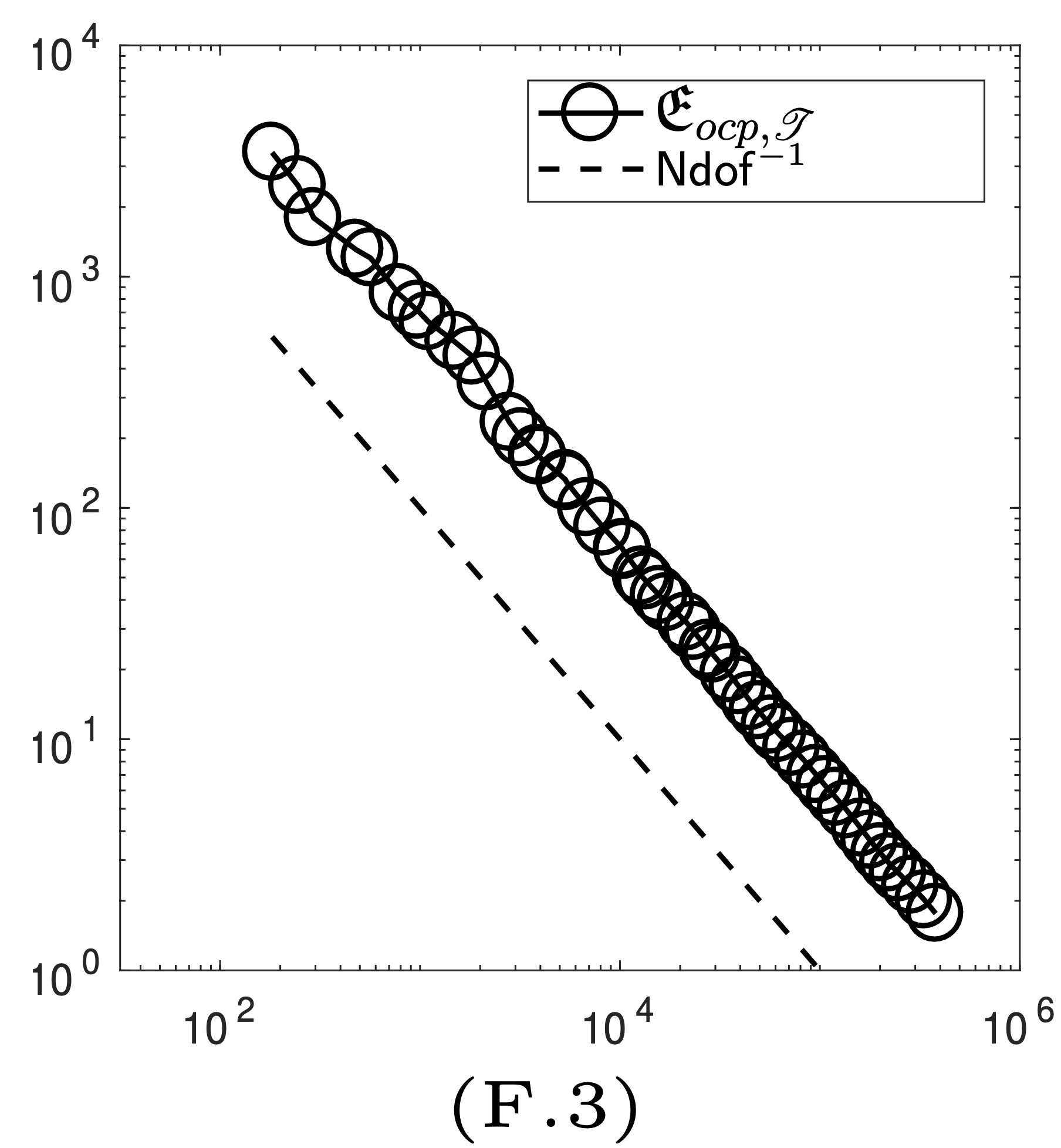} \\
\end{minipage}
\caption{\AAF{Example 2: We present the experimental convergence rates for the semidiscrete scheme: state error estimator (F.1); adjoint error estimator (F.2); the total error estimator (F.3).}}
\label{fig:test_0_1_2_1}
\end{figure}


\begin{figure}[!h]
\centering
\begin{minipage}[b]{0.4\textwidth}
\centering
\includegraphics[trim={0 0 0 0},clip,width=2.7cm,height=3cm,scale=0.66]{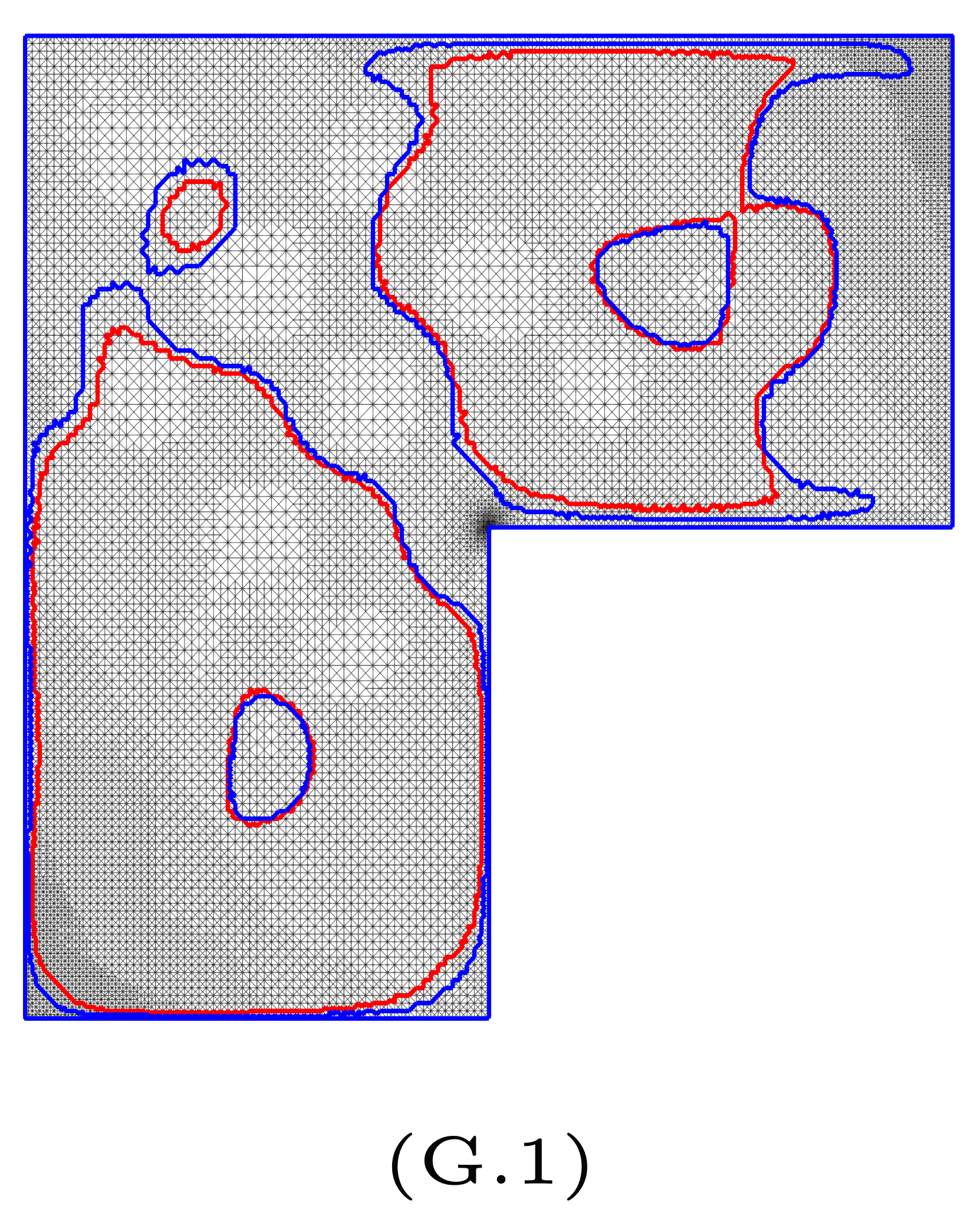} \\
\end{minipage}
\begin{minipage}[b]{0.4\textwidth}
\centering
\includegraphics[trim={0 0 0 0},clip,width=3.5cm,height=3cm,scale=0.66]{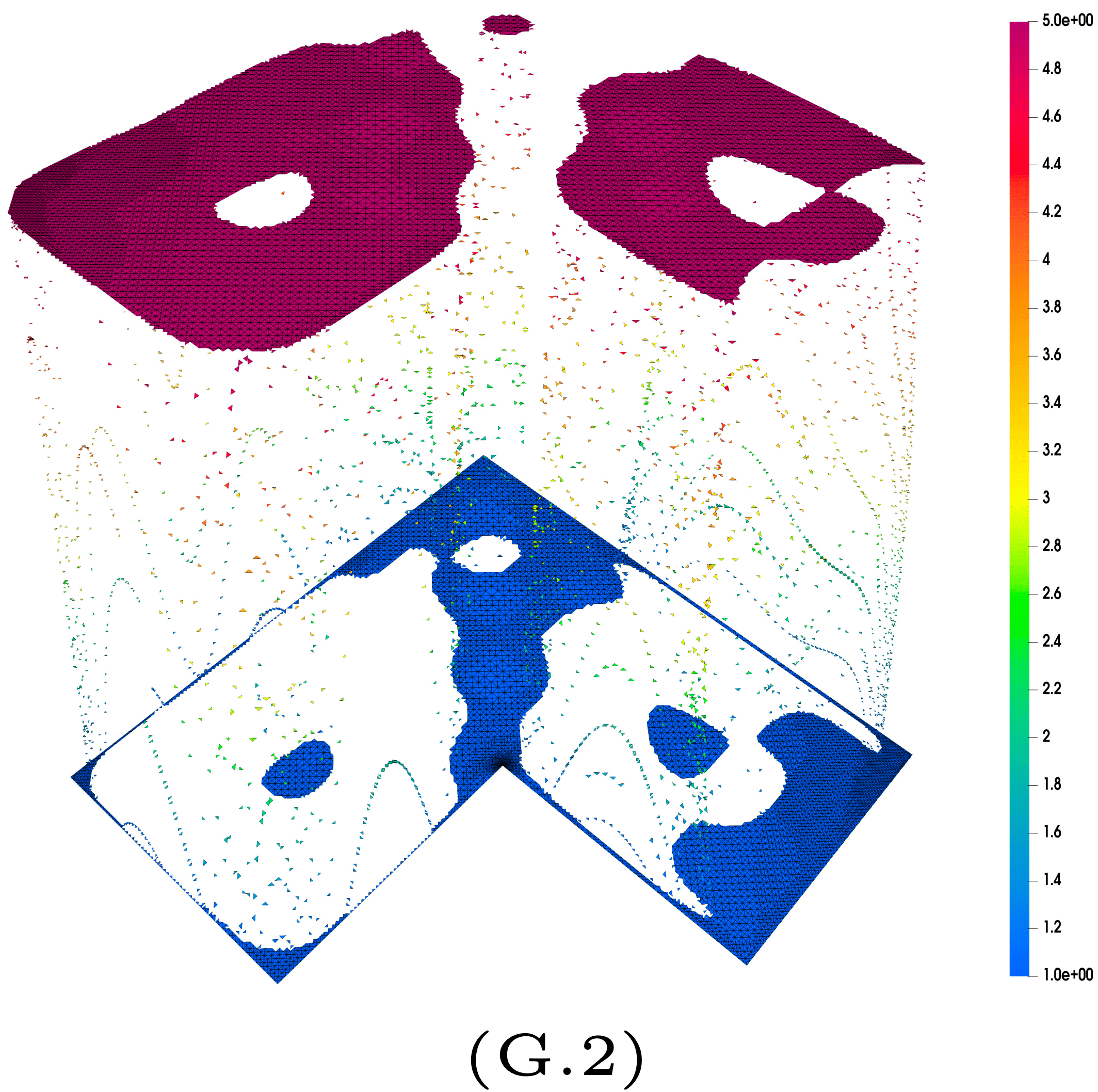} \\
\end{minipage}
\caption{\AAF{Example 2: For the fully discrete scheme we present the mesh obtained after 50 adaptive iterations (31.630 triangles and 94.890 vertices), the discrete boundary of the active sets, where the blue line corresponds to the one for $\mathsf{a}=1.0$, and the red line corresponds to the one for $\mathsf{b}=5.0$ (G.1); and the discrete optimal control $\bar{u}_{h}$ (G.2).}}
\label{fig:test_02}
\end{figure}

\begin{figure}[!h]
\centering
\begin{minipage}[b]{0.4\textwidth}
\centering
\includegraphics[trim={0 0 0 0},clip,width=2.7cm,height=3cm,scale=0.66]{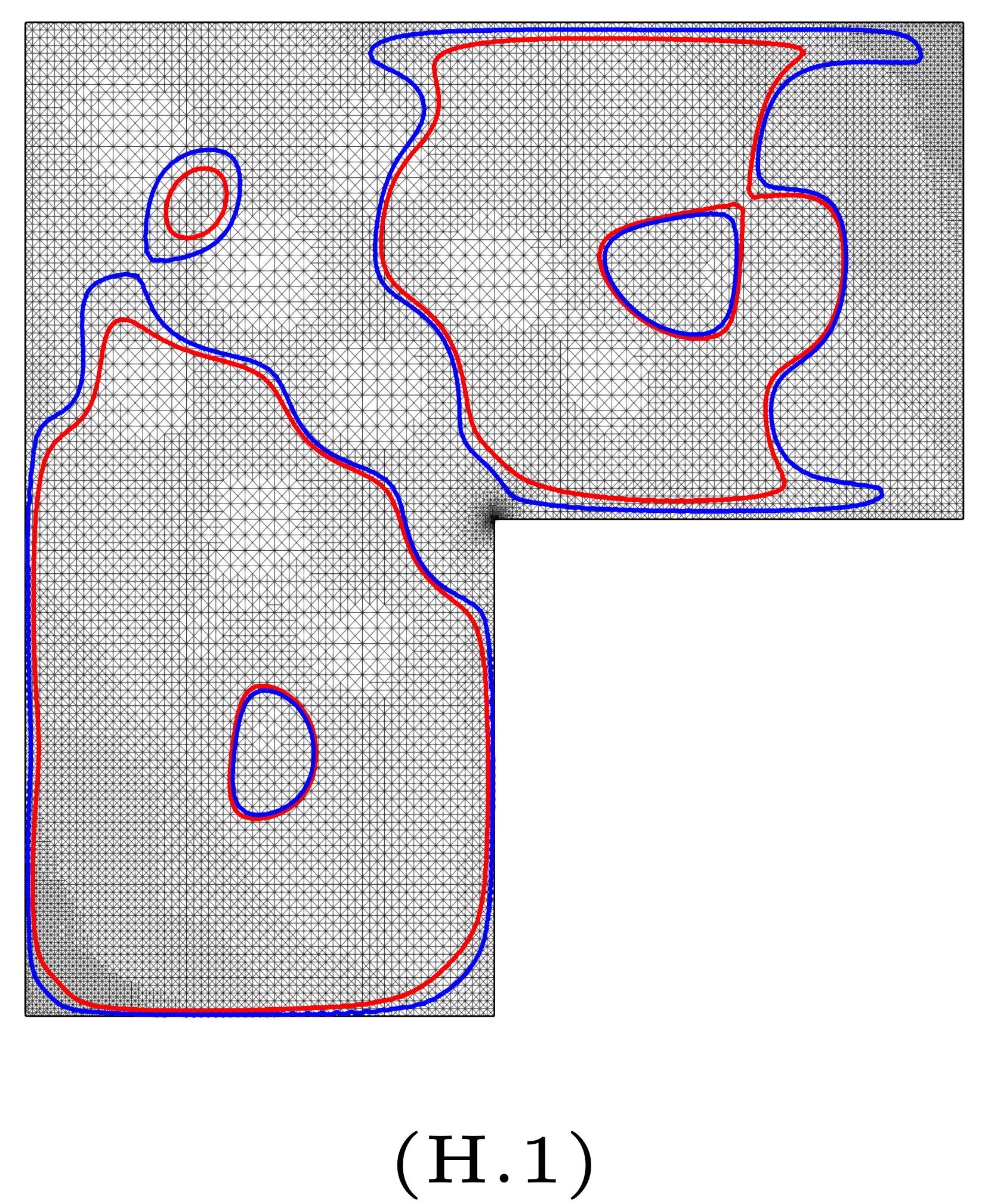} \\
\end{minipage}
\begin{minipage}[b]{0.4\textwidth}
\centering
\includegraphics[trim={0 0 0 0},clip,width=3.5cm,height=3cm,scale=0.66]{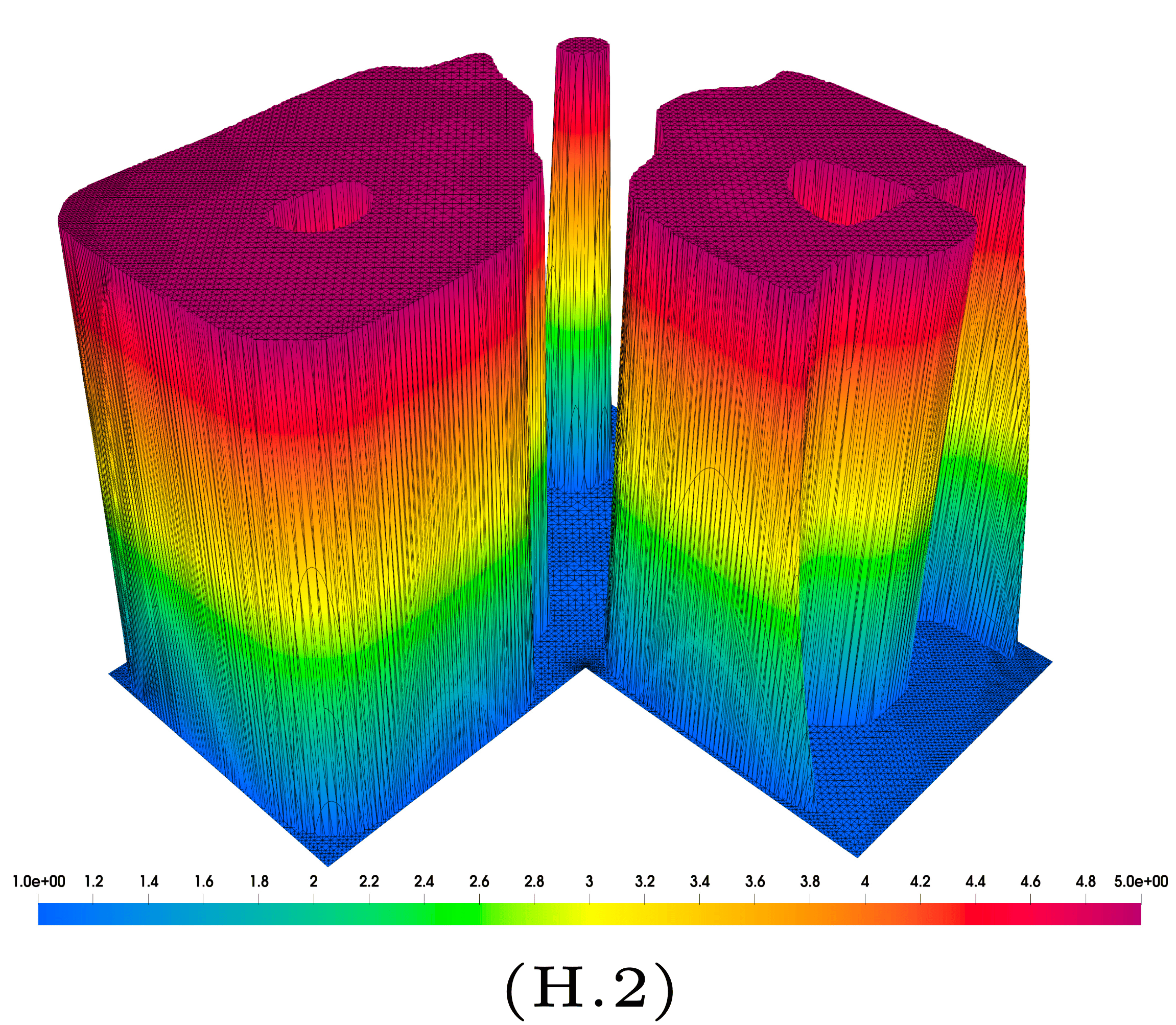} \\
\end{minipage}
\caption{\AAF{Example 2. For the semidiscrete scheme we present the mesh obtained after 50 adaptive iterations (31.434 triangles and 16.046 vertices); the discrete boundary of the active sets, where the blue line corresponds to the one for $\mathsf{a}=1.0$, and the red line correspond to the one for $\mathsf{b}=5.0$ (H.1); and the discrete optimal control $\bar{\mathsf{u}}=\Pi_{[\mathsf{a},\mathsf{b}]}(\alpha^{-1}\bar{\mathbf{y}}_{h}\cdot\bar{\mathbf{z}}_{h})$ (H.2).}}
\label{fig:test_02_SEMI}
\end{figure}

Following \AAF{the series of numerical experiments, several observations are in order:
~\\
$\bullet$ \emph{Error estimators and control approximation:} The adaptive procedure for the fully discrete scheme relies on the computation of three error indicators associated with the discretization of the state, adjoint, and control variables. In contrast, the adaptive strategy for the semidiscrete scheme involves only two error estimators associated with the discretization of the state and adjoint variables. The fully discrete scheme can therefore be influenced by the control approximation, guiding mesh refinement toward the regions where the restrictions of the control variable become active. Conversely, this influence is absent in the semidiscrete approach. Despite this, the semidiscrete scheme delivers better approximation properties for the control variable.
~\\
$\bullet$ \emph{Computational complexity:} An important consideration is the computational complexity associated with each methodology. The fully discrete approach approximates the control variable, resulting in a larger system matrix to solve compared to that in the semidiscrete scheme, which does not discretize the admissible control set. In summary, the fully discrete scheme requires more degrees of freedom than the semidiscrete scheme. It is also important to note that the implementation and resolution of the semi-smooth Newton method used in the semidiscrete method involve higher computational complexity than that for the fully discrete scheme; compare \textbf{Algorithm 2} with \textbf{Algorithm 3}.
}

\bibliographystyle{siamplain}
\bibliography{biblio}

\begin{thebibliography}{10}

\bibitem{MR2424078}
{\sc R.~A. Adams and J.~J.~F. Fournier}, {\em Sobolev spaces}, vol.~140 of Pure
  and Applied Mathematics (Amsterdam), Elsevier/Academic Press, Amsterdam,
  second~ed., 2003.

\bibitem{MR1445736}
{\sc M.~Ainsworth and J.~T. Oden}, {\em A posteriori error estimators for the
  {S}tokes and {O}seen equations}, SIAM J. Numer. Anal., 34 (1997),
  pp.~228--245, \url{http://dx.doi.org/10.1137/S0036142994264092}.

\bibitem{MUMPS2}
{\sc P.~R. Amestoy, I.~S. Duff, J.-Y. L'Excellent, and J.~Koster}, {\em A fully
  asynchronous multifrontal solver using distributed dynamic scheduling}, SIAM
  J. Matrix Anal. Appl., 23 (2001), pp.~15--41 (electronic),
  \url{http://dx.doi.org/10.1137/S0895479899358194}.

\bibitem{MR3097958}
{\sc D.~Boffi, F.~Brezzi, and M.~Fortin}, {\em Mixed finite element methods and
  applications}, vol.~44 of Springer Series in Computational Mathematics,
  Springer, Heidelberg, 2013,
  \url{http://dx.doi.org/10.1007/978-3-642-36519-5}.

\bibitem{MR2373954}
{\sc S.~C. Brenner and L.~R. Scott}, {\em The mathematical theory of finite
  element methods}, vol.~15 of Texts in Applied Mathematics, Springer, New
  York, third~ed., 2008, \url{http://dx.doi.org/10.1007/978-0-387-75934-0}.

\bibitem{MR3523574}
{\sc C.~Brett, A.~Dedner, and C.~Elliott}, {\em Optimal control of elliptic
  {PDE}s at points}, IMA J. Numer. Anal., 36 (2016), pp.~1015--1050,
  \url{http://dx.doi.org/10.1093/imanum/drv040},
  \url{https://doi.org/10.1093/imanum/drv040}.

\bibitem{MR4679130}
{\sc E.~Casas, K.~Chrysafinos, and M.~Mateos}, {\em Semismooth {N}ewton method
  for boundary bilinear control}, IEEE Control Syst. Lett., 7 (2023),
  pp.~3549--3554, \url{http://dx.doi.org/10.1109/lcsys.2023.3337747}.

\bibitem{MR4858133}
{\sc E.~Casas, K.~Chrysafinos, and M.~Mateos}, {\em Bilinear control of
  semilinear elliptic {PDE}s: convergence of a semismooth {N}ewton method},
  Numer. Math., 157 (2025), pp.~143--163,
  \url{http://dx.doi.org/10.1007/s00211-024-01448-1}.

\bibitem{MR4956345}
{\sc E.~Casas, K.~Chrysafinos, and M.~Mateos}, {\em Error estimates for the
  discretization of bilinear control problems governed by semilinear elliptic
  {PDE}s}, Math. Control Relat. Fields, 15 (2025), pp.~1320--1345,
  \url{http://dx.doi.org/10.3934/mcrf.2024049}.

\bibitem{MR3586845}
{\sc E.~Casas and M.~Mateos}, {\em Optimal control of partial differential
  equations}, in Computational mathematics, numerical analysis and
  applications, vol.~13 of SEMA SIMAI Springer Ser., Springer, Cham, 2017,
  pp.~3--59.

\bibitem{MR2272157}
{\sc E.~Casas and J.-P. Raymond}, {\em Error estimates for the numerical
  approximation of {D}irichlet boundary control for semilinear elliptic
  equations}, SIAM J. Control Optim., 45 (2006), pp.~1586--1611,
  \url{http://dx.doi.org/10.1137/050626600}.

\bibitem{MR2902693}
{\sc E.~Casas and F.~Tr\"{o}ltzsch}, {\em Second order analysis for optimal
  control problems: improving results expected from abstract theory}, SIAM J.
  Optim., 22 (2012), pp.~261--279, \url{http://dx.doi.org/10.1137/110840406}.

\bibitem{MR3103238}
{\sc Y.~Chen, Z.~Lu, and Y.~Huang}, {\em Superconvergence of triangular
  {R}aviart-{T}homas mixed finite element methods for a bilinear constrained
  optimal control problem}, Comput. Math. Appl., 66 (2013), pp.~1498--1513,
  \url{http://dx.doi.org/10.1016/j.camwa.2013.08.019}.

\bibitem{CiarletBook}
{\sc P.~G. Ciarlet}, {\em The finite element method for elliptic problems},
  SIAM, Philadelphia, PA, 2002,
  \url{http://dx.doi.org/10.1137/1.9780898719208}.

\bibitem{MR3136903}
{\sc P.~G. Ciarlet}, {\em Linear and nonlinear functional analysis with
  applications}, Society for Industrial and Applied Mathematics, Philadelphia,
  PA, 2013.

\bibitem{MR2971171}
{\sc C.~Clason and B.~Jin}, {\em A semismooth {N}ewton method for nonlinear
  parameter identification problems with impulsive noise}, SIAM J. Imaging
  Sci., 5 (2012), pp.~505--538, \url{http://dx.doi.org/10.1137/110826187},
  \url{https://doi.org/10.1137/110826187}.

\bibitem{MR0977489}
{\sc M.~Dauge}, {\em Stationary {S}tokes and {N}avier-{S}tokes systems on two-
  or three-dimensional domains with corners. {I}. {L}inearized equations}, SIAM
  J. Math. Anal., 20 (1989), pp.~74--97,
  \url{http://dx.doi.org/10.1137/0520006}.

\bibitem{Guermond-Ern}
{\sc A.~Ern and J.-L. Guermond}, {\em Theory and practice of finite elements},
  vol.~159 of Applied Mathematical Sciences, Springer-Verlag, New York, 2004,
  \url{http://dx.doi.org/10.1007/978-1-4757-4355-5}.

\bibitem{MR3693332}
{\sc H.~Fu, H.~Guo, J.~Hou, and J.~Zhang}, {\em A stabilized mixed finite
  element approximation of bilinear state optimal control problems}, Comput.
  Math. Appl., 74 (2017), pp.~1246--1261,
  \url{http://dx.doi.org/10.1016/j.camwa.2017.06.010}.

\bibitem{MR4450052}
{\sc F.~Fuica and E.~Ot\'{a}rola}, {\em A posteriori error estimates for an
  optimal control problem with a bilinear state equation}, J. Optim. Theory
  Appl., 194 (2022), pp.~543--569,
  \url{http://dx.doi.org/10.1007/s10957-022-02039-6}.

\bibitem{MR3422453}
{\sc V.~Girault, R.~H. Nochetto, and L.~R. Scott}, {\em Max-norm estimates for
  {S}tokes and {N}avier-{S}tokes approximations in convex polyhedra}, Numer.
  Math., 131 (2015), pp.~771--822,
  \url{http://dx.doi.org/10.1007/s00211-015-0707-8}.

\bibitem{MR2121575}
{\sc V.~Girault, R.~H. Nochetto, and R.~Scott}, {\em Maximum-norm stability of
  the finite element {S}tokes projection}, J. Math. Pures Appl. (9), 84 (2005),
  pp.~279--330, \url{http://dx.doi.org/10.1016/j.matpur.2004.09.017}.

\bibitem{MR851383}
{\sc V.~Girault and P.-A. Raviart}, {\em Finite element methods for
  {N}avier-{S}tokes equations}, vol.~5 of Springer Series in Computational
  Mathematics, Springer-Verlag, Berlin, 1986,
  \url{http://dx.doi.org/10.1007/978-3-642-61623-5}.
\newblock Theory and algorithms.

\bibitem{MR2122182}
{\sc M.~Hinze}, {\em A variational discretization concept in control
  constrained optimization: the linear-quadratic case}, Comput. Optim. Appl.,
  30 (2005), pp.~45--61, \url{http://dx.doi.org/10.1007/s10589-005-4559-5}.

\bibitem{MR0404849}
{\sc R.~B. Kellogg and J.~E. Osborn}, {\em A regularity result for the {S}tokes
  problem in a convex polygon}, J. Functional Analysis, 21 (1976),
  pp.~397--431, \url{http://dx.doi.org/10.1016/0022-1236(76)90035-5}.

\bibitem{MR1301452}
{\sc V.~A. Kozlov, V.~G. Maz'ya, and C.~Schwab}, {\em On singularities of
  solutions to the {D}irichlet problem of hydrodynamics near the vertex of a
  cone}, J. Reine Angew. Math., 456 (1994), pp.~65--97.

\bibitem{MR2536007}
{\sc A.~Kr\"{o}ner and B.~Vexler}, {\em A priori error estimates for elliptic
  optimal control problems with a bilinear state equation}, J. Comput. Appl.
  Math., 230 (2009), pp.~781--802,
  \url{http://dx.doi.org/10.1016/j.cam.2009.01.023}.

\bibitem{MR2680928}
{\sc K.~Kunisch, W.~Liu, Y.~Chang, N.~Yan, and R.~Li}, {\em Adaptive finite
  element approximation for a class of parameter estimation problems}, J.
  Comput. Math., 28 (2010), pp.~645--675,
  \url{http://dx.doi.org/10.4208/jcm.2009.10-m1003}.

\bibitem{MR2987056}
{\sc M.~Mitrea and M.~Wright}, {\em Boundary value problems for the {S}tokes
  system in arbitrary {L}ipschitz domains}, Ast\'{e}risque,  (2012),
  pp.~viii+241.

\bibitem{MR1259620}
{\sc J.~T. Oden, W.~Wu, and M.~Ainsworth}, {\em An a posteriori error estimate
  for finite element approximations of the {N}avier-{S}tokes equations},
  Comput. Methods Appl. Mech. Engrg., 111 (1994), pp.~185--202,
  \url{http://dx.doi.org/10.1016/0045-7825(94)90045-0}.

\bibitem{OSQ:25}
{\sc E.~Ot\'arola, D.~Quero, and M.~Sasso}, {\em A bilinear pointwise tracking
  optimal control problem for a semilinear elliptic {PDE}},  (2025).
\newblock Preprint.

\bibitem{MR2583281}
{\sc F.~Tr\"{o}ltzsch}, {\em Optimal control of partial differential
  equations}, vol.~112 of Graduate Studies in Mathematics, American
  Mathematical Society, Providence, RI, 2010,
  \url{http://dx.doi.org/10.1090/gsm/112}.

\bibitem{MR4081184}
{\sc M.~Winkler}, {\em Error estimates for the finite element approximation of
  bilinear boundary control problems}, Comput. Optim. Appl., 76 (2020),
  pp.~155--199, \url{http://dx.doi.org/10.1007/s10589-020-00171-5}.

\end{thebibliography}
\end{document}